\documentclass[10pt]{article}

\usepackage{amsmath,amssymb,amsthm,mathtools}
\usepackage{aliascnt}
\usepackage[left=.90in,right=.90in,top=.85in,bottom=.85in]{geometry}
\usepackage{enumitem}
\usepackage{array}
\usepackage[hidelinks]{hyperref}
\newtheoremstyle{plain}{6pt plus 2pt minus 2pt}{6pt plus 2pt minus 2pt}
{\itshape}{}{\bfseries}{.}{.5em}{}
\newtheoremstyle{definition}{6pt plus 2pt minus 2pt}{6pt plus 2pt minus 2pt}
{\normalfont}{}{\bfseries}{.}{.5em}{}

\theoremstyle{plain}
\newtheorem{theorem}{Theorem}[section]
\newaliascnt{lemma}{theorem}
\newtheorem{lemma}[lemma]{Lemma}
\aliascntresetthe{lemma}
\newaliascnt{proposition}{theorem}
\newtheorem{proposition}[proposition]{Proposition}
\aliascntresetthe{proposition}
\newaliascnt{corollary}{theorem}
\newtheorem{corollary}[corollary]{Corollary}
\aliascntresetthe{corollary}
\theoremstyle{definition}
\newaliascnt{definition}{theorem}
\newtheorem{definition}[definition]{Definition}
\aliascntresetthe{definition}
\newaliascnt{problem}{theorem}
\newtheorem{problem}[problem]{Problem}
\aliascntresetthe{problem}
\theoremstyle{remark}
\newaliascnt{remark}{theorem}
\newtheorem{remark}[remark]{Remark}
\aliascntresetthe{remark}
\numberwithin{equation}{section}

\newcolumntype{G}[1]{>{\raggedright\arraybackslash}p{#1}}

\newcommand{\PSL}{\operatorname{PSL}}
\newcommand{\PGL}{\operatorname{PGL}}
\newcommand{\PSU}{\operatorname{PSU}}
\newcommand{\PGU}{\operatorname{PGU}}
\newcommand{\PSp}{\operatorname{PSp}}

\newcommand{\POmega}{\mathrm{P}\Omega}

\newcommand{\GL}{\operatorname{GL}}
\newcommand{\SL}{\operatorname{SL}}
\newcommand{\SU}{\operatorname{SU}}
\newcommand{\GU}{\operatorname{GU}}
\newcommand{\Sp}{\operatorname{Sp}}
\newcommand{\Sz}{\operatorname{Sz}}
\newcommand{\AGL}{\operatorname{AGL}}

\newcommand{\Soc}{\operatorname{soc}}
\newcommand{\Aut}{\operatorname{Aut}}
\newcommand{\Out}{\operatorname{Out}}
\newcommand{\Hom}{\operatorname{Hom}}
\newcommand{\Core}{\operatorname{Core}}
\newcommand{\Stab}{\operatorname{Stab}}
\newcommand{\Fix}{\operatorname{Fix}}
\newcommand{\Sym}{\operatorname{Sym}}
\newcommand{\Tr}{\operatorname{Tr}}
\newcommand{\ord}{\operatorname{ord}}
\newcommand{\diag}{\operatorname{diag}}

\newcommand{\F}{\mathbb{F}}
\newcommand{\Z}{\mathbb{Z}}
\newcommand{\gen}[1]{\langle #1\rangle}

\title{Ordinary $G$-permutable subgroups of finite simple groups\thanks{ \scriptsize\emph{E-mail address:}
 shengmin.zhang.math@outlook.com\,(S. Zhang).}}

\author{Shengmin Zhang\thanks{Corresponding author.} \\ \small School of Mathematical Sciences, Peking University}
\date{}

\begin{document}
\maketitle

\setlength{\abovedisplayskip}{6pt plus 2pt minus 2pt}
\setlength{\belowdisplayskip}{6pt plus 2pt minus 2pt}
\setlength{\abovedisplayshortskip}{0pt plus 2pt}
\setlength{\belowdisplayshortskip}{4pt plus 2pt minus 2pt}

\begin{abstract}
We completely solve part \textup{(a)} of Problem 17.112 in the Kourovka Notebook.
We prove that a non-abelian finite simple group $G$ has a proper non-trivial ordinary $G$-permutable subgroup if and only if $G$ is isomorphic to $\PSL_2(q)$ where $q\geq4$ is a prime power and $q\not\equiv3\pmod4$, ${}^2B_2(2^{2a+1})$ where $a\geq1$, $J_1$, $\Sp_4(4)$, or $\PSp_4(p)$ where $p$ is prime and $p\equiv1,13,17\pmod{24}$.
Moreover, we classify all proper non-trivial ordinary $G$-permutable subgroups of these groups.
Every other non-abelian finite simple group $G$ has no proper non-trivial ordinary $G$-permutable subgroup.

\noindent{\bf Keywords} finite simple groups, permutable subgroups, group factorizations, maximal subgroups.

\noindent{\bf MSC} 20D05, 20D40, 20G40.
\end{abstract}

\setcounter{tocdepth}{2}
\tableofcontents

\section{Introduction}\label{sec:introduction}

Our classification concerns finite groups.
For $B\leq G$ and $x\in G$, let $B^x=x^{-1}Bx$.

Subgroup permutability is a basic and important property in finite group theory.
For example, Guo and coauthors \cite{GuoShumSkiba2007} used generalized permutability conditions to characterize soluble, supersoluble and nilpotent groups.
Earlier, Ore \cite[\S1]{GT25} studied permutable subgroups.
In the classical definition, we call $A\leq G$ permutable if $AB=BA$ for every $B\leq G$.
Skiba \cite{Skiba2003,Skiba2006} later introduced generalized forms of subgroup permutability.
More precisely, we call $A\leq G$ ordinary $G$-permutable if every $B\leq G$ satisfies $AB^x=B^xA$ for some $x\in G$.
In addition, we call $A$ hereditarily $G$-permutable if $A$ is ordinary $E$-permutable for every $A\leq E\leq G$.

In 2006, Skiba and Tyutyanov \cite[Question 2.41]{Skiba2006} asked which nonabelian finite simple groups $G$ have an ordinary $G$-permutable subgroup $A$ with $1<A<G$.
Both the ordinary and hereditary questions appeared as Problem 17.112 in the 2010 issue of the Kourovka Notebook \cite[p.~124]{KN}.
Galt and Tyutyanov \cite{GT22} answered the ordinary question for the sporadic groups and the Tits group ${}^2F_4(2)'$ in 2022.
In 2023, they classified ordinary permutable subgroups in the Suzuki groups $G={}^2B_2(2^{2a+1})$ with $a\geq1$ \cite{GT23}.
Also in 2023, Yang and Galt \cite{YG} proved that the alternating group $A_n$ has no proper nontrivial ordinary $A_n$-permutable subgroup for prime $n\geq7$.
Moreover, there exists an integer $N\geq7$ such that no $A_n$ with $n\geq N$ has a proper nontrivial ordinary $A_n$-permutable subgroup \cite{YG}.
For every prime power $q>4$, Galt and Tyutyanov \cite[Theorem 1]{GT25} completed the ordinary classification in $G=\PSL_2(q)$ in 2025.

\begin{problem}[{\cite[Question 2.41]{Skiba2006}, \cite[Problem 17.112(a)]{KN}}]\label{prob:ordinary}
Which nonabelian finite simple groups $G$ have an ordinary $G$-permutable subgroup $A$ with $1<A<G$?
\end{problem}

In \cite{ZhangGT}, we proved that a nontrivial finite group $G$ is simple if and only if $G$ has no proper nontrivial hereditarily $G$-permutable subgroup.
Thus, we proved the Galt--Tyutyanov conjecture \cite{GT23} and answered part \textup{(b)}.
In contrast, every subgroup $A\leq A_5$ with $|A|=2$ is ordinary $A_5$-permutable \cite[p.~18]{Skiba2006}.
We now solve part \textup{(a)} completely, 20 years after the 2006 question.
More precisely, we classify all ordinary $G$-permutable subgroups $A$ with $1<A<G$ for every nonabelian finite simple group $G$.

\begin{theorem}\label{thm:main}
Let $G$ be a nonabelian finite simple group, and let $1<A<G$.
Up to isomorphism, $A$ is ordinary $G$-permutable if and only if one of the following holds.
\begin{itemize}
\item[(1)] $G=\PSL_2(q)$ with $q\geq4$ and $q\not\equiv3\pmod4$. We have $|A|=2$.
\item[(2)] $G={}^2B_2(2^{2a+1})$ with $a\geq1$. We have $|A|=2$.
\item[(3)] $G=J_1$ and $|A|=2$.
\item[(4)] $G=\Sp_4(4)$ and $A=\gen t$. We have $|t|=2$ and $|C_G(t)|=256$.
\item[(5)] $G=\PSp_4(p)$, where $p$ is prime and $p\equiv1,13,17\pmod{24}$.
The subgroup $A=\gen t$ has order $2$.
There is $s\in\Sp_4(p)$ with $s^2=-I$ and $\pi(s)=t$,
where $\pi:\Sp_4(p)\to G$ is the quotient map.
\end{itemize}
\end{theorem}

\section{Preliminaries}\label{sec:preliminaries}

\subsection{Ordinary permutability and group actions}

\begin{definition}[{\cite[Definitions 2.1--2.2]{Skiba2006}}]\label{def:g-permutable}
Let $A\leq G$.
We call $A$ \emph{ordinary $G$-permutable} if every $B\leq G$ satisfies $AB^x=B^xA$ for some $x\in G$.
In addition, we call $A$ \emph{hereditarily $G$-permutable} if $A$ is ordinary $E$-permutable for every $A\leq E\leq G$.
\end{definition}

For a finite set $\Omega$, let $S_{\Omega}$ be the group of bijections $\Omega\to\Omega$.
For $n\geq1$, let $S_n=S_{\{1,\ldots,n\}}$.
Let $\operatorname{sgn}:S_n\to\{1,-1\}$ be the sign homomorphism and let $A_n=\ker\operatorname{sgn}$.
Let $C_n=\langle c\mid c^n=1\rangle$. For $n\geq2$, let $D_{2n}=\langle a,b\mid a^n=b^2=1,\ bab^{-1}=a^{-1}\rangle$.
For $D\leq K$, let $\Core_K(D)=\bigcap_{k\in K}D^k$.
For a finite group $H$, let $\Soc(H)=\gen{N:N\text{ is a minimal normal subgroup of }H}$.
For a prime $p$, let $O_{p'}(H)=\gen{N\unlhd H:p\nmid|N|}$, the largest normal subgroup of $H$ of order prime to $p$.
For a prime power $q$ and a prime $\ell\nmid q$, let $\ord_\ell(q)=\min\{e\geq1:q^e\equiv1\pmod\ell\}$.
For $e\geq1$, a prime $\ell$ is a primitive prime divisor of $q^e-1$ if $\ell\mid q^e-1$ and $\ord_\ell(q)=e$.
For $e\geq1$, let $\Phi_e(X)$ be the $e$th cyclotomic polynomial, so $X^n-1=\prod_{e\mid n}\Phi_e(X)$.
For a positive integer $a$ and a prime $p$, let $v_p(a)=\max\{k\geq0:p^k\mid a\}$ and $a_p=p^{v_p(a)}$.
We call $G=XY$ a proper factorization when $X,Y<G$.
For $X<G$, we call $X$ a factor of $G$ when $G=XY$ for some $Y<G$.

For an action $\rho:H\to S_{\Omega}$ with $\Omega\ne\varnothing$, we call the action faithful if $\ker\rho=1$.
It is transitive if $\{\rho(h)(\omega):h\in H\}=\Omega$ for every $\omega\in\Omega$.
A transitive action is regular if $\{h\in H:\rho(h)(\omega)=\omega\}=1$ for every $\omega\in\Omega$.

\begin{lemma}\label{lem:max}
Let $A<G$ be ordinary $G$-permutable, and let $M$ be maximal in $G$.
Then, either $A\leq M^x$ or $G=AM^x$ for some $x\in G$.
If $MY\ne G$ for every $Y<G$, then $|A|\mid|M|$.
If $XY\ne G$ for all $X,Y<G$, then $|A|\mid\gcd\{|M|:M<G\text{ is maximal}\}$.
\end{lemma}
\begin{proof}
We choose $x\in G$ with $AM^x=M^xA$.
Since $M^x$ is maximal and $M^x\leq AM^x\leq G$, we have $AM^x=M^x$ or $AM^x=G$.
If $AM^x=G$, then we have $G=(M^xA)^{x^{-1}}=MA^{x^{-1}}$.
If $MY\ne G$ for all $Y<G$, then we must have $A\leq M^x$.
Hence, $|A|\mid|M|$.
If $XY\ne G$ for all $X,Y<G$, then this divisibility holds for every maximal subgroup $M<G$.
Thus, $|A|\mid\gcd\{|M|:M<G\text{ is maximal}\}$.
\end{proof}

\begin{lemma}\label{lem:involution-normalizers}
Let $t\in G$ have order $2$.
Then, $\gen{t}$ is ordinary $G$-permutable if and only if for every $B\leq G$ there exists $g\in G$ such that $t^g\in N_G(B)$.
\end{lemma}
\begin{proof}
Suppose that $K=\gen{t}B^x$ is a subgroup for some $x\in G$.
If $t\in B^x$, then $t\in N_G(B^x)$.
If $t\notin B^x$, then $[K:B^x]=2$ and hence $B^x\unlhd K$.
In both situations, we obtain that $t\in N_G(B^x)$.
Since $N_G(B^x)^{x^{-1}}=N_G(B)$, we have $t^{x^{-1}}\in N_G(B)$.

Conversely, suppose that $t^g\in N_G(B)$ for some $g\in G$.
Then, $t\in N_G(B^{g^{-1}})$, so $\gen{t}B^{g^{-1}}=B^{g^{-1}}\gen{t}$.
\end{proof}

\begin{lemma}\label{lem:involution}
Let $A=\gen t$ have order $2$.
Suppose that every involution of $G$ equals $t^g$ for some $g\in G$.
Then, $A$ is ordinary $G$-permutable if and only if $|N_G(B)|$ is even for every $B\leq G$ with $|B|$ odd.
\end{lemma}
\begin{proof}
Suppose that $A$ is ordinary $G$-permutable.
By Lemma \ref{lem:involution-normalizers}, every $B\leq G$ has $t^g\in N_G(B)$ for some $g\in G$.
Thus, $|N_G(B)|$ is even.
Conversely, suppose that $|N_G(B)|$ is even whenever $|B|$ is odd.
Since $B\leq N_G(B)$, we also have $|N_G(B)|$ even for $|B|$ even.
Hence, every $B\leq G$ has an involution $s\in N_G(B)$.
As every involution is conjugate to $t$, we choose $g\in G$ with $s=t^g$.
By Lemma \ref{lem:involution-normalizers}, we conclude that $A$ is ordinary $G$-permutable.
\end{proof}

\begin{lemma}\label{lem:index3}
Let $D$ be cyclic and have index $3$ in a finite group $K$. Then, the
subgroup $D^2=\gen{d^2:d\in D}$ is normal in $K$.
\end{lemma}
\begin{proof}
Let $C=\Core_K(D)$.
Let $\rho:K\to S_{K/D}$ satisfy $\rho(k)(gD)=kgD$ for $k,g\in K$.
Since $\ker\rho=C$ and $[K:D]=3$, we have $K/C\leq S_3$.
Note that $\rho(D)$ fixes the coset $D$, so $[D:C]\leq2$.
Because $D$ is cyclic and $[D:C]\leq2$, we have $C=D$ or $C=D^2$.
As $D^2$ is characteristic in $D$, both $C=D$ and $C=D^2$ imply that $D^2\unlhd K$.
\end{proof}

\begin{lemma}\label{lem:abelian-index}
Let $B$ be an abelian subgroup of prime index $r$ in a finite group $K$,
and let $C=\Core_K(B)$. Then, $[B:C]$ divides $r-1$.
\end{lemma}
\begin{proof}
Let $L=K/C$ and $D=B/C$.
Next, we consider the faithful action of $L$ on the $r$ cosets of $D$.
The stabilizer $D$ is abelian. If $D=1$, then we have $[B:C]=1$.
Suppose that $D\ne1$.
Since $[L:D]=r$ is prime, $D$ is maximal in $L$.
If $N_L(D)=L$, then the subgroup $D$ lies in the kernel of the action.
Hence, $N_L(D)=D$.
For $g\in L$ with $D^g\ne D$, maximality of $D$ yields $\gen{D,D^g}=L$.
As $D$ and $D^g$ are abelian, both centralize $D\cap D^g$.
Hence, $D\cap D^g\leq Z(L)$.
Let $z\in Z(L)$ fix a coset $\alpha\in L/D$.
For $g\in L$, we have $z(g\alpha)=g(z\alpha)=g\alpha$.
Because $L$ is transitive and acts faithfully, we have $z=1$.
The subgroup $D\cap D^g$ fixes the coset $D$ and lies in $Z(L)$, so it is trivial.
For $g,h\in L$ with $D^g\ne D^h$, we have $(D^g\cap D^h)^{g^{-1}}=D\cap D^{hg^{-1}}=1$.
Hence, $D^g\cap D^h=1$.

Since $N_L(D)=D$, we have $r$ distinct subgroups $D^g$ with $g\in L$.
Every nonidentity element of a conjugate of $D$ fixes exactly one coset.
Therefore, we count $r|D|-1-r(|D|-1)=r-1$ elements of $L$ without a fixed point.
Every nonidentity element of a Sylow $r$-subgroup is an $r$-cycle.
Distinct Sylow $r$-subgroups of $L\leq S_r$ would contribute disjoint
sets of $r-1$ such elements. Therefore, we have a unique Sylow $r$-subgroup $R$.
As $|R|=r$ and $R\cap D=1$, we have $L=R\rtimes D$.
We consider $\theta:D\to\Aut(R)$ with $\theta(d)(x)=dxd^{-1}$ for $d\in D$ and $x\in R$.
Because $D$ is abelian and $R$ centralizes $\ker\theta$, we have $\ker\theta\unlhd L$ and $\ker\theta\leq D$.
As $\Core_L(D)=1$, we have $\ker\theta=1$.
Hence, $\theta$ is injective and $|D|\mid|\Aut(C_r)|=r-1$.
\end{proof}

\subsection{Matrices and unipotent elements}

Let $V$ be a finite-dimensional vector space over a field $F$.
The group $\GL_F(V)$ consists of invertible $F$-linear maps $V\to V$.
Let $\SL_F(V)=\{g\in\GL_F(V):\det_F g=1\}$.
For $h\in\GL_F(V)$ and $W\leq V$, $h$-invariant means $hW=W$.
For $H\leq\GL_F(V)$, $H$-invariant means $hW=W$ for every $h\in H$.
For $X\leq\GL_F(V)$, let $Z_{\rm sc}=X\cap\{\lambda I_V:\lambda\in F^*\}$ and let $\pi:X\to X/Z_{\rm sc}$ satisfy $\pi(g)=gZ_{\rm sc}$.
The scalar matrix $\lambda I_V$ is the map $v\mapsto\lambda v$.
We call $\pi(B)$ the projective image of $B\leq X$.
A lift of $\bar g\in X/Z_{\rm sc}$ is a matrix $g\in X$ with $\pi(g)=\bar g$.
Since $\lambda W=W$ for $\lambda\in F^*$, the action $\bar gW=gW$ is independent of the lift.
For $X=\GL_F(V)$, the quotient is $\PGL_F(V)$.

Let $\overline F$ be an algebraic closure of $F$ and let $d=\dim_F V$.
A matrix $a\in\GL_F(V)$ is semisimple if it is diagonalizable over $\overline F$.
In $\GL_F(V)$, it is regular semisimple if it has $d$ distinct eigenvalues over $\overline F$.
A matrix $u$ is unipotent if $(u-I_V)^d=0$.
A transvection is a unipotent matrix with $\operatorname{rank}_F(u-I_V)=1$.
For $\SL_d(q)$ and $\SU_d(q)$, a regular unipotent element has $1$ Jordan block of size $d$ on the natural module.
For unipotent $u$, every reference to rank means $\operatorname{rank}(u-I_V)$.
For a basis $(e_1,\ldots,e_d)$, the matrix unit $E_{ij}$ satisfies $E_{ij}e_j=e_i$ and $E_{ij}e_k=0$ for $k\ne j$.

\begin{lemma}
\label{lem:transvection}
Let $r\geq5$ be prime and $p\ne r$.
Let $a\in\GL_r(\overline{\F}_p)$.
Suppose that every $r$th root of unity is an eigenvalue of $a$ with multiplicity $1$.
Let $\pi:\GL_r(\overline{\F}_p)\to\PGL_r(\overline{\F}_p)$ be the quotient map.
Let $\bar a=\pi(a)$.
Then, no transvection $t$ satisfies $\pi(t)\in N_{\PGL_r(\overline{\F}_p)}(\gen{\bar a})$.
\end{lemma}
\begin{proof}
Suppose that $t$ is a transvection and that $\pi(t)$ normalizes $\gen{\bar a}$.
We choose $\zeta\in\overline{\F}_p^*$ and $1\leq e<r$ with $tat^{-1}=\zeta a^e$.
By comparing eigenvalues, we have $\zeta^r=1$ and $e\ne0\pmod r$.
Let $\omega$ be a primitive $r$th root of unity and let $\zeta=\omega^b$ with $b\in\F_r$.
Let $L_j=\ker(a-\omega^jI)$ for $j\in\F_r$.
Since $tat^{-1}=\omega^b a^e$, we have $tL_j=L_{e^{-1}(j-b)}$.
Thus, $t$ induces $\sigma:j\mapsto e^{-1}(j-b)$ on the eigenlines.
As $t^p=I$, the order of $\sigma$ divides $p$.
If $\sigma=1$, then $t$ is diagonal on the eigenlines of $a$.
Hence, $t$ is semisimple, contrary to $t$ being a transvection.
Suppose that $\sigma\ne1$.
The nonidentity affine map $\sigma$ has order $p$.
Because $p\ne r$, it has $1$ fixed point and $(r-1)/p$ cycles of length $p$.
As $t^p=I$, the product of the scalar entries around each cycle is $1$.
Thus, $t$ has $1$ Jordan block of size $p$ on each cycle.
Consequently, we have $\operatorname{rank}(t-I)\geq(r-1)(p-1)/p\geq2$.
However, a transvection satisfies $\operatorname{rank}(t-I)=1$.
\end{proof}

\begin{lemma}\label{lem:regular-unipotent}
Let $q$ have odd prime characteristic $r$, and let
$u$ be regular unipotent in $G=\SL_r(q)$ or $\SU_r(q)$.
Then, $C_G(u)$ is an $r$-group. If $r\geq5$, then
$\gen u$ is not ordinary $G$-permutable.
\end{lemma}
\begin{proof}
Let $E=\F_q$ for $G=\SL_r(q)$ and $E=\F_{q^2}$ for $G=\SU_r(q)$.
We choose a basis of $E^r$ in which $u=I+J$, with $J$ a single nilpotent Jordan block.
Since $u=I+J$ has a single Jordan block, every $c\in C_G(u)$ has the form $c=\alpha I+\sum_{i=1}^{r-1}\beta_iJ^i$.
Let $\alpha\in E^*$ and $\beta_i\in E$ be the coefficients of $c$.
Since $\det c=\alpha^r=1$ in characteristic $r$, we have $\alpha=1$.
As every $c\in C_G(u)$ has order a power of $r$, we obtain from Cauchy's theorem that $C_G(u)$ is an $r$-group.

Suppose that $r\geq5$ and that $\gen u$ is ordinary $G$-permutable.
In both groups, we choose an involution $t$ with eigenvalue $1$ of multiplicity $1$ and eigenvalue $-1$ of multiplicity $r-1$.
We choose $x\in G$ with $K=\gen u\gen{t^x}\leq G$.
Since $|u|=r$, we have $|K|=2r$ and $[K:\gen u]=2$.
Hence, $\gen u\unlhd K$.
As $C_G(u)$ has odd order, $t^x$ does not centralize $u$.
The automorphism $u\mapsto(t^x)^{-1}u(t^x)$ of $\gen u\cong C_r$ has order $2$.
Therefore, $(t^x)^{-1}u(t^x)=u^{-1}$.

Let $L=\log u=\sum_{i=1}^{r-1}(-1)^{i+1}(u-I)^i/i$.
Since $1\leq i<r$, every denominator $i$ is nonzero in $E$.
The coefficient of $J$ in $L$ is $1$.
Hence, $L$ has $1$ nilpotent Jordan block of size $r$.
As $(t^x)u(t^x)^{-1}=u^{-1}$, we have $(t^x)L(t^x)^{-1}=-L$.
Because $t^x$ is semisimple and $LE^r$ is $t^x$-invariant, we choose a $t^x$-eigenvector $v\notin LE^r$.
Since $v\notin LE^r$ and $L$ has one nilpotent Jordan block, we have $L^{r-1}v\ne0$.
Because $v$ is a $t^x$-eigenvector, $t^xv=\varepsilon v$ for some $\varepsilon\in\{1,-1\}$.
For $0\leq j<r$, $t^x(L^jv)=\varepsilon(-1)^jL^jv$.
As $(v,Lv,\ldots,L^{r-1}v)$ is a basis, the eigenvalue multiplicities are $(r+1)/2$ and $(r-1)/2$.
Because $r\geq5$, neither multiplicity equals $r-1$.
However, $t^x$ has multiplicities $1,r-1$. We have a contradiction.
\end{proof}

\subsection{Involutions in odd characteristic}

Let $p$ be a prime and $k=\overline{\F}_p$.
An algebraic torus $\mathbf T$ is isomorphic to $(k^*)^r$ as an algebraic group.
Let $\mathbf G$ be a connected semisimple algebraic group and let $\mathbf T$ be a maximal torus.
The rank of $\mathbf G$ is $\dim\mathbf T$.
Let $\Phi=\Phi(\mathbf G,\mathbf T)$ be the root system, with positive roots $\Phi^+$ and simple roots $\Delta$.
An element $a\in\mathbf G$ is regular if $\dim C_{\mathbf G}(a)=\operatorname{rank}\mathbf G$.
Each root $\alpha\in\Phi$ is a character $\alpha:\mathbf T\to k^*$.
For $D\leq\mathbf T$, $\alpha|_D=1$ means $\alpha(d)=1$ for every $d\in D$.
For characters $\alpha,\beta:\mathbf T\to k^*$, addition means $(\alpha+\beta)(t)=\alpha(t)\beta(t)$ for $t\in\mathbf T$.
For $\alpha=\sum_{\delta\in\Delta}n_\delta\delta$, let $\operatorname{ht}(\alpha)=\sum_{\delta\in\Delta}n_\delta$.
The root lattice is $\Z\Phi$.
For $\alpha\in\Phi$, let $X_\alpha$ be the corresponding root subgroup of $\mathbf G$.
Choose an algebraic group isomorphism $x_\alpha:(k,+)\to X_\alpha(k)$.
For $u,v\in k$ and $t\in\mathbf T$, it satisfies $x_\alpha(u)x_\alpha(v)=x_\alpha(u+v)$ and $t x_\alpha(u)t^{-1}=x_\alpha(\alpha(t)u)$.
The positive root subgroups generate $\mathbf U=\gen{X_\alpha(k):\alpha\in\Phi^+}$.
The corresponding Borel subgroup is $\mathbf B=\mathbf T\mathbf U$, with unipotent radical $\mathbf U$.

\begin{lemma}\label{lem:split-root-normalizer}
Let $q=p^f$, where $p$ is an odd prime and $f\geq1$.
Let $\mathbf G$ be a split irreducible semisimple algebraic group of rank $r\geq3$ over $\F_q$.
Let $\Phi$ be the root system of $\mathbf G$, with positive roots $\Phi^+$.
Let $G_0=\langle X_\alpha(\F_q)\mid\alpha\in\Phi\rangle$ and
$U=\langle X_\alpha(\F_q)\mid\alpha\in\Phi^+\rangle$.
There is a subgroup $H\leq U$ with
$N_{G_0}(H)\leq Z(G_0)U$.
If $f\geq2$, then $N_{G_0}(D)\leq Z(G_0)U$ for every $D\leq H$ with $[H:D]=p$.
\end{lemma}
\begin{proof}
Let $T\leq G_0$ be the split maximal torus normalizing $U$.
Next, we choose a path $A$--$B$--$C$ of $3$ simple roots.
Let $D=A+B$ and $E=B+C$.
Now, let $U_3=\gen{X_\alpha(\F_q):\alpha\in\Phi^+,\ \operatorname{ht}(\alpha)\geq3}$.
The coefficients of $ab$ and $bc$ in the $D,E$ commutator coordinates belong to $\{1,-1\}$.
We choose signs so that $[X_A(a),X_B(b)]\equiv X_D(ab)\pmod{U_3}$ for $a,b\in\F_q$.
Then, we also have $[X_B(b),X_C(c)]\equiv X_E(bc)\pmod{U_3}$ for $b,c\in\F_q$.
All root groups with $\operatorname{ht}(\alpha)=2$ are central in $U/U_3$.
Next, let $Q_2=\gen{X_\alpha(\F_q):\alpha\in\Phi^+,\ \operatorname{ht}(\alpha)=2,\ \alpha\notin\{D,E\}}$.
Let $Q=\gen{U_3,Q_2,X_D(t)X_E(t):t\in\F_q}$.
As the root groups of height $2$ are central in $U/U_3$, we have $Q\unlhd U$.
We identify $U/Q$ with $\F_q^r\times\F_q$ as a set.
For $a\in\F_q^r$, we use $a_1,a_2,a_3$ for the coordinates at $A,B,C$.
The multiplication is $(a,z)(a',z')=(a+a',z+z'-a'_1a_2+a'_2a_3)$.
The other simple coordinates do not enter the central term.
For $a,b\in\F_q^r$, the commutator form is $\beta(a,b)=b_2(a_1+a_3)-a_2(b_1+b_3)$.
We have $\operatorname{rad}\beta=\{a\in\F_q^r:a_2=0,\ a_1+a_3=0\}$.

Let $(e_1,\ldots,e_r)$ be the standard basis of $\F_q^r$.
Let $v=\sum_{i=1}^r e_i$ and $w=e_2+2e_3$.
From the multiplication formula, we obtain the abelian subgroup
$S=\{(tv+sw,t+s(t+s)):t,s\in\F_q\}$
with additive coordinates $t,s$.
Let $\eta:U\to U/Q$ be the quotient map and let $H=\eta^{-1}(S)$.
Let $Z_1=\{(c(v-w),c):c\in\F_q\}\leq U/Q$.
Then, $S\cap Z(U/Q)=Z_1$.
Let $\pi:U/Q\to\F_q^r$ send $(a,z)$ to $a$.
Let $P=\pi(S)=\gen{v,w}_{\F_q}$.
For $u\in U$, we consider $c_u:U/Q\to U/Q$ with $c_u(z)=\eta(u)^{-1}z\eta(u)$.
We have $c_u(s)=s$ for $s\in Z_1$ and $(\pi\circ c_u)|_S=\pi|_S$.

Now, we choose $h_0\in H$ with $\eta(h_0)=(v,1)$.
Every coordinate of $v$ equals $1$, so every simple-root coordinate of $h_0$ is nonzero.
The element $h_0$ lies in a unique algebraic Borel subgroup $\mathbf B$, with $U\leq\mathbf B$
\cite[Lemma 3.2 and Theorem 3.3]{SteinbergRegular}.
For $g\in N_{G_0}(H)$, the element $gh_0g^{-1}\in U$ is also regular.
Hence, $g\mathbf B g^{-1}=\mathbf B$ and we have $N_{G_0}(H)\leq TU$.
For $g\in N_{G_0}(H)$, we choose $\tau\in T$ and $u\in U$ with $g=\tau u$.
For $h\in H$, we have $(\pi\circ\eta)(uhu^{-1})=(\pi\circ\eta)(h)$.
Since $gHg^{-1}=H$, we obtain that $(\pi\circ\eta)(\tau H\tau^{-1})=P$.
The coordinate columns of $P$ at $A,B,C$ are $(1,0),(1,1),(1,2)$.
As the $3$ corresponding lines are distinct, every linear map fixing them is scalar.
Every other column equals $(1,0)$.
Hence, we have $\alpha(\tau)=\lambda\in\F_q^*$ for every simple root $\alpha$.
We have $D(\tau)=E(\tau)=\lambda^2$, so $\tau Q\tau^{-1}=Q$.
On $U/Q$, the map $hQ\mapsto\tau h\tau^{-1}Q$ sends $(v-w,1)$ to $(\lambda(v-w),\lambda^2)$.
The map $c_u$ fixes $Z_1$ pointwise.
Because $gHg^{-1}=H$, we have $(\lambda(v-w),\lambda^2)\in Z_1$.
Thus, we have $\lambda=\lambda^2=1$.
We conclude that $\tau\in Z(G_0)$ and $N_{G_0}(H)\leq Z(G_0)U$.

Now, we suppose that $f\geq2$ and choose $D_0\leq H$ with $[H:D_0]=p$.
Let $W=\eta(D_0)\leq S$.
The restriction $\pi|_S:S\to P$ is injective.
Because $[H:D_0]=p$ and $\eta(H)=S$, we have $[S:W]\leq p$.
As $\pi|_S$ is injective, $|\pi(W)|=|W|\geq q^2/p>q$, so $\pi(W)$ spans $P$ over $\F_q$.
At most $3$ lines in $P$ have a vanishing simple coordinate.
The union of those lines has at most $3q-2$ elements, whereas
$|\pi(W)|\geq q^2/p>3q-2$.
Thus, $D_0$ contains a regular element and $N_{G_0}(D_0)\leq TU$.
Also, $|W\cap Z_1|\geq q/p>1$.
For $g\in N_{G_0}(D_0)$, we choose $\tau\in T$ and $u\in U$ with $g=\tau u$.
For $h\in D_0$, we have $(\pi\circ\eta)(uhu^{-1})=(\pi\circ\eta)(h)$.
Since $gD_0g^{-1}=D_0$, we obtain that $(\pi\circ\eta)(\tau D_0\tau^{-1})=\pi(W)$.
As $\pi(W)$ spans $P$, we obtain that $\alpha(\tau)=\lambda$ for every simple root $\alpha$.
Hence, $\tau Q\tau^{-1}=Q$.
We choose $z_1=(c(v-w),c)\in W\cap Z_1$ with $c\ne0$.
Because $z_1\in Z(U/Q)$, we have $\eta(u)z_1\eta(u)^{-1}=z_1$.
As $gD_0g^{-1}=D_0$, we have $(\lambda c(v-w),\lambda^2c)\in W\cap Z_1$.
Hence, we have $\lambda=\lambda^2=1$.
We conclude that $\tau\in Z(G_0)$ and $N_{G_0}(D_0)\leq Z(G_0)U$.
\end{proof}

\begin{corollary}\label{cor:split-odd-involutions}
Let $G$ be a finite simple split group of Lie type of rank at least
$3$ in odd characteristic.
Then, no nontrivial cyclic $2$-subgroup is ordinary $G$-permutable.
\end{corollary}
\begin{proof}
We take $G_0$ and $H$ from Lemma \ref{lem:split-root-normalizer}.
Let $\pi:G_0\to G$ be the quotient map and let $B=\pi(H)$.
Now, let $p=\operatorname{char}\F_q$.
We have $|B|$ odd and $\pi^{-1}(B)=HZ(G_0)$.
Since $p\nmid|Z(G_0)|$, $H$ is the unique Sylow $p$-subgroup of $\pi^{-1}(B)$.
Therefore, $\pi^{-1}(N_G(B))\leq N_{G_0}(H)$.
By Lemma \ref{lem:split-root-normalizer}, we have $\pi^{-1}(N_G(B))\leq Z(G_0)U$.
Hence, $|N_G(B)|$ is odd.
Suppose that a nontrivial cyclic subgroup $A$ of order a power of $2$ is ordinary $G$-permutable.
We choose $x\in G$ with $K=AB^x\leq G$.
Since $A$ is a cyclic Sylow $2$-subgroup of $K$, Burnside's normal complement theorem yields $H_0\unlhd K$ with $|H_0|$ odd and $[K:H_0]=|A|$.
As $K/H_0$ is a $2$-group, we have $B^x\leq H_0$.
Because $|B^x|=|H_0|$, we have $B^x=H_0\unlhd K$.
Hence, $A\leq N_G(B^x)$, contrary to $|N_G(B^x)|$ being odd.
\end{proof}

\begin{lemma}\label{lem:odd-normalizer-fusion}
Let $G$ contain a $p$-subgroup $H$, with $p$ odd and $N_G(H)$ of odd order.
Then, no cyclic subgroup of order $4$ is ordinary $G$-permutable.
If an ordinary $G$-permutable subgroup $A$ is isomorphic to $C_2\times C_2$, then for $s,t\in A\setminus\{1\}$ there is $g\in G$ with $s^g=t$.
\end{lemma}
\begin{proof}
Let $A\leq G$ with $|A|=4$. Suppose that $A$ is ordinary $G$-permutable.
We choose $x\in G$ with $K=AH^x\leq G$.
Next, let $D=\Core_K(H^x)$.
Let $\rho:K\to S_{K/H^x}$ satisfy $\rho(k)(gH^x)=kgH^x$ for $k,g\in K$.
Then, we have $\ker\rho=D$ and $K/D\leq S_4$.
Since $D$ has odd order, $|\rho(A)|=4$.
The subgroup $\rho(H^x)$ lies in a point stabilizer $S_3$.
If $\rho(H^x)=1$, then we have $H^x\unlhd K$ and $A\leq N_G(H^x)$.
However, $|A|=4$ and $|N_G(H^x)|$ is odd. Hence, $\rho(H^x)\ne1$.
If $\rho(H^x)\ne1$, then we have $p=3$ and $|\rho(H^x)|=3$.
Hence, $|K/D|=12$ and $K/D=A_4$.
Since $\rho|_A$ is injective and $A_4$ has no cyclic subgroup of order $4$, $A$ cannot be cyclic.
If $A\cong C_2\times C_2$, then for $s,t\in A\setminus\{1\}$ we choose $k_0\in K$ with $\rho(s)^{\rho(k_0)}=\rho(t)$.

We lift the equation $\rho(s)^{\rho(k_0)}=\rho(t)$ through $D$.
Let $t_1=s^{k_0}$ and $t_2=t$.
Then, we have $\rho(t_1)=\rho(t_2)$.
Let $L_0=\rho^{-1}(\gen{\rho(t_1)})$.
Since $|D|$ is odd, we have $L_0=D\rtimes\gen{t_1}=D\rtimes\gen{t_2}$.
By Schur--Zassenhaus, we choose $d\in D$ with $t_1^d=t_2$.
Consequently, $s^{k_0d}=t$ with $k_0d\in K\leq G$.
\end{proof}

\section{Linear groups}\label{sec:linear}

Let $V\ne0$ be a vector space over a field $F$ and let $H\leq\GL_F(V)$.
The natural $H$-module $V$ is irreducible if its only $H$-invariant subspaces are $0$ and $V$.
It is indecomposable if no nonzero $H$-invariant subspaces $V_1,V_2$ satisfy $V=V_1\oplus V_2$.
We call $V$ completely reducible if it is a direct sum of irreducible $H$-submodules.
An irreducible constituent is an irreducible quotient $U/W$ with $W\leq U\leq V$ both $H$-invariant.

Let $q$ be a prime power and $n\geq2$.
We use $\PSL_n(q)$ for the quotient of $\SL_n(q)$ by its scalar matrices.
For a partition $n=n_1+\cdots+n_t$, we identify $V=\F_q^n=\bigoplus_{i=1}^t\F_{q^{n_i}}$ over $\F_q$.
Let $M_\lambda(v)=\lambda v$ and $N_i(\lambda)=\lambda^{(q^{n_i}-1)/(q-1)}$ for $\lambda\in\F_{q^{n_i}}^*$.
A torus of type $(n_1,\ldots,n_t)$ in $\SL_n(q)$ is
\[
\widetilde T=\left\{\diag(M_{\lambda_1},\ldots,M_{\lambda_t}):
\lambda_i\in\F_{q^{n_i}}^*,\ \prod_i N_i(\lambda_i)=1\right\}.
\]
Its image in $\PSL_n(q)$ has the corresponding type.
For type $(n)$, we call $\widetilde T$ and its projective image Singer tori.
Their orders are $(q^n-1)/(q-1)$ and $(q^n-1)/((q-1)(n,q-1))$.
For type $(1,\ldots,1)$, we call the torus split.
A monomial subgroup stabilizes the set of coordinate lines $\{\F_qe_i:1\leq i\leq n\}$.

\subsection{\texorpdfstring{The groups $\PSL_3(q)$}{The groups PSL3(q)}}

\begin{theorem}\label{thm:l3}
For every prime power $q\geq2$, the group $G=\PSL_3(q)$ has no proper
nontrivial ordinary $G$-permutable subgroup.
\end{theorem}
\begin{proof}
Suppose that $1<A<G$ is ordinary $G$-permutable.
For $q=2$, use $\PSL_3(2)\cong\PSL_2(7)$ and \cite{GT25}.
We now assume that $q\geq3$ and $q\ne4$.
Let $q=p^f$ with $p$ prime, and let $d=(3,q-1)$.
The split monomial subgroup $M$ is maximal and has order
$6(q-1)^2/d$, including $M\cong S_4$ for $q=3$
\cite[Table 8.3]{BHR}. By \cite[Theorem 3.3]{LiXia}, we have $MY\ne G$ for every $Y<G$.
Suppose that $G=XY$ with $X,Y<G$ in a generic row.
A factor has Singer type.
Moreover, the other factor contains $\F_q^2\rtimes\SL_2(q)$.
Neither listed order divides $|M|$.
For $q=3$, the listed factors $X,Y<G$ have orders $39$ and $144$.
Neither order divides $|M|=24$.
For $q=8$, the exceptional equalities $H=XY$ in the cited list have $G<H\leq\Aut(G)$.
The group $H$ contains an outer field automorphism.

By Lemma \ref{lem:max}, we have $|A|\mid|M|$.
We take a Singer torus $T$ with $|T|=(q^2+q+1)/d$.
Then, we have $(|T|,|M|)=1$ and $|G|/(|M||T|)=dq^3(q+1)/(6(q-1))>1$.
Since $A$ is ordinary $G$-permutable, we choose $x\in G$ such that $AT^x$ is a subgroup.
As $|A||T|<|G|$, we have $AT^x<G$.
We choose a maximal subgroup $H<G$ with $AT^x\leq H$.
By \cite{Bereczky}, we have $H=N_G(T^x)=T^x\rtimes C_3$.
Hence, $AT^x\leq N_G(T^x)$ and $|A|\mid3$.
Let $w:\F_{q^3}\to\F_{q^3}$ be the $\F_q$-linear map with $w(z)=z^q$.
On a normal basis of $\F_{q^3}/\F_q$, $w$ is a $3$-cycle and lies in $\SL_3(q)$.
If $A\ne1$, then we choose $g\in G$ with $A^g=\gen{\overline w}$ by Schur--Zassenhaus.
For $p=3$, the lift $w$ is regular unipotent.
Otherwise, $w$ has eigenvalues $1,\omega,\omega^2$, where $\omega^3=1\ne\omega$.

For $p=3$, we obtain from Lemma \ref{lem:regular-unipotent} that $C_{\SL_3(q)}(w)$ is a $3$-group.
The group $G$ has a cyclic subgroup $D$ of order $4$ in a torus of type $(2,1)$.
A subgroup $AD^x$ would have order $12$.
By Lemma \ref{lem:index3}, we would have $(D^x)^2\unlhd AD^x$.
Because $\Aut(C_2)=1$, $A$ would centralize the involution generating $(D^x)^2$.
This would place an involution in $C_G(A)$, contrary to the preceding centralizer calculation.

Suppose that $q\equiv2\pmod3$. Then, $d=1$.
We choose a Sylow $p$-subgroup $U$ of order $q^3$.
Since $A$ is ordinary $G$-permutable, we have a subgroup $K=AU^x$ for some $x\in G$.
If $p>3$, then $[K:U^x]=3<p$.
Hence, $U^x\lhd K$.
Therefore, $A\leq N_G(U^x)$.
But $N_G(U^x)$ is a Borel subgroup of order $q^3(q-1)^2$, which is prime to $3$.
As $|A|=3$, we have a contradiction.

For $p=2$, we have $q\geq8$.
Suppose that $K=AU^x$ is a subgroup.
Let $C=\Core_K(U^x)$.
Since $K$ acts on $3$ cosets, we have $[U^x:C]\leq2$.
To bound $N_G(C)$, consider an arbitrary subgroup $U_0\leq U$ with $[U:U_0]\leq2$.
The group $U$ is
$$
 U=\left\{\begin{pmatrix}1&a&c\\0&1&b\\0&0&1\end{pmatrix}:
 a,b,c\in\F_q\right\}.
$$
As $[U:U_0]\leq2$, we have $[U,U]\leq U_0$.
We have $[U,U]=\{I+cE_{13}:c\in\F_q\}$.
The coordinate maps $\pi_a,\pi_b:U\to(\F_q,+)$ send the displayed matrix to $a,b$.
Since $[U:U_0]\leq2$, each coordinate image of $U_0$ has size at least $q/2>1$.
Therefore, $\Fix_{\F_q^3}(U_0)=\gen{e_1}$ and $\Fix_{(\F_q^3)^*}(U_0)=\gen{e_3^*}$.
Hence, $N_G(U_0)\leq B_0$, where $B_0$ stabilizes the flag $\langle e_1\rangle<\langle e_1,e_2\rangle$.
For $U_0=C^{x^{-1}}\leq U$, we have $N_G(U_0)\leq B_0$.
Then, $A\leq N_G(C)\leq B_0^x$, contrary to $3\nmid|B_0|$.

Finally, suppose that $q\equiv1\pmod3$.
Then, $d=3$.
We take a cyclic torus $B$ of type $(2,1)$ with $|B|=(q^2-1)/3$.
The subgroup $A$ has a noncentral lift of order $3$.
In contrast, every element of order $3$ in the inverse image of $B$ is scalar.
Thus, $A\cap B^x=1$ for every $x\in G$.
Suppose that $K=AB^x$ is a subgroup, and let $C=\Core_K(B^x)$.
By Lemma \ref{lem:abelian-index}, we have $[B^x:C]\mid2$.
Let $\widetilde C$ be the inverse image of $C$ in $\SL_3(q)$.
Since $B^x$ is cyclic and $[B^x:C]\mid2$, we have $\widetilde B^2\leq\widetilde C$, where $\widetilde B$ is the inverse image of $B^x$ and $|\widetilde B|=q^2-1$.
Let $z$ generate $\F_{q^2}^*$. Since $q+1\nmid2$, we have $z^2\notin\F_q$.
Let $\F_q^3=W\oplus L$ be the decomposition for $\widetilde B$.
We have $\dim_{\F_q}W=2$ and $\dim_{\F_q}L=1$.
Because $z^2\notin\F_q$ and $\widetilde B^2\leq\widetilde C$, the group $\widetilde C$ acts irreducibly on $W$ and stabilizes $L$.
Relative to $\F_q^3=W\oplus L$, the normalizer $N_{\GL_3(q)}(\widetilde C)$ lies in $(\GL_1(q^2)\rtimes C_2)\times\GL_1(q)$.
Next, a generator of $A$ lifts to a noncentral element $a$ of order $3$.
Since $C\unlhd K$, the lift $a$ normalizes $\widetilde C$.
As $3\nmid2$, the lift $a$ centralizes $\widetilde C$.
Because the $3$ eigenvalues of $a$ lie in $\F_q$, we have $a|_W=\lambda I_W$ for some $\lambda\in\F_q$.
Note that $|a|=3$, so $\lambda^3=1$.
Because $\det a=1$, we have $a|_L=\lambda^{-2}I_L=\lambda I_L$.
Hence, $a$ is scalar on $\F_q^3$, contrary to the choice of a noncentral lift.

We now consider $q=4$.
The group $G=\PSL_3(4)$ has order $20160$ and maximal subgroup
orders $960,360,168,72$ \cite{Atlas1985}.
Next, we claim that $XY\ne G$ for all $X,Y<G$.
In \cite[Theorem 3.3]{LiXia}, the exceptional equalities $H=XY$ with socle $G$ have $G<H\leq\Aut(G)$.
If $G=XY$ with $X,Y<G$ in a generic row, then the Singer normalizer of order $21$ would act regularly on the $21$ projective points.
A generator of the complement of order $3$ lifts to a matrix with $3$ distinct eigenvalues in $\F_4$.
Hence, the Singer normalizer has nontrivial point stabilizers and is not regular.

Therefore, $|A|\mid(960,360,168,72)=24$.
We choose $T\leq G$ of order $7$ and a subgroup $K=AT^x$ with $x\in G$.
Since $7\nmid|A|$, we have $|K|=7|A|\leq168$.
As $168$ is the only maximal subgroup order divisible by $7$, we have $K\leq L$ for some $L\cong\PSL_2(7)$.
By \cite{Dickson1901}, $|K|\in\{7,21,168\}$.
Hence, $|A|\in\{1,3,24\}$.
Suppose that $|A|=24$.
Since $|K|=168$, we have $K=L$.
Dickson's subgroup theorem then yields $A\cong S_4$.
Let $M_{72}\leq G$ be a maximal subgroup with $M_{72}\cong C_3^2\rtimes Q_8$.
By Lemma \ref{lem:max}, we choose $g\in G$ with $A^g\leq M_{72}$.
Let $\rho:M_{72}\to Q_8$ be the quotient map with kernel $C_3^2$.
Because $S_4$ has no nontrivial normal $3$-subgroup, we have $A^g\cap\ker\rho=1$.
Thus, $\rho|_{A^g}:A^g\to Q_8$ is injective.
This contradicts $24>8$.

If $|A|=3$, then we choose $B\leq G$ of order $5$.
Since $A$ is ordinary $G$-permutable, $AB^x$ would be a cyclic subgroup of order $15$ for some $x\in G$.
But $G$ has no element of order $15$.
A semisimple lift with order divisible by $5$ acts irreducibly on $W\leq\F_4^3$ with $\dim_{\F_4}W=2$.
As the lift is semisimple, it stabilizes a line complementary to $W$.
Thus, its projective image lies in a torus of type $(2,1)$ and order $5$.
Hence, no semisimple element has order $15$.
We conclude that $|A|\ne3$ and $A=1$.
\end{proof}

\subsection{\texorpdfstring{The groups $\PSL_r(q)$ with prime $r\geq5$}{The groups PSLr(q) with prime r at least 5}}

Let $V=W\oplus U$ over $\F_q$, with $W,U\ne0$, and let $P=\Stab_{\SL(V)}(W)$.
The block diagonal group $L=\{\diag(a,b):a\in\GL(W),\ b\in\GL(U),\ \det(a)\det(b)=1\}$ is a Levi subgroup of $P$.
In $\PSL(V)$, we use the image of $L$ under the scalar quotient map.

\begin{theorem}\label{thm:linear-prime}
For every odd prime $r\geq5$ and every prime power $q$, the group
$G=\PSL_r(q)$ has no proper nontrivial ordinary $G$-permutable subgroup.
\end{theorem}
\begin{proof}
For ease of reading, we divide the proof of the theorem into the following steps.

\setlist[itemize]{leftmargin=4em,labelsep=.5em,topsep=.2\baselineskip,partopsep=0pt,parsep=0pt,itemsep=0pt}
\begin{itemize}
\item[{\bf Step 1.} ] For $(r,q)\ne(5,2)$, $A$ stabilizes a subspace of $\F_q^r$ of dimension $2$.
\end{itemize}
Suppose that $1<A<G$ is ordinary $G$-permutable.
Let $q=p^f$ with $p$ prime, and let $d=(r,q-1)$.
Next, we take a Singer torus $T$ with $|T|=(q^r-1)/((q-1)d)$.
Let $N=N_G(T)=T\rtimes C_r$.
By \cite{Bereczky}, we have $H=N$ for every maximal subgroup $H<G$ with $T\leq H$.
Let $P_i=\operatorname{Stab}_G(W_i)$ for $W_i\leq\F_q^r$ with $\dim_{\F_q}W_i=i$.
Since $r$ is prime, we use \cite[Theorem 3.3]{LiXia} for $G=XY$ with $X,Y<G$.
In each generic row, we label the factors $X,Y$.
We then choose $u,v\in G$ with $X^u\leq N$ and $Y^v\leq P_1$ or $P_{r-1}$.
For $(r,q)=(5,2)$, the list also permits $G=NP_2$ and $G=NP_3$.
By this factorization list, $P_2Y\ne G$ for every $Y<G$ when $(r,q)\ne(5,2)$.

Firstly, assume that $(r,q)\ne(5,2)$.
By Lemma \ref{lem:max}, we have $A\leq P_2^g$ for some $g\in G$.
If $r\mid q-1$, then $v_r(\Phi_r(q))=1$.
Otherwise, $r\nmid\Phi_r(q)$.
Since $d=(r,q-1)$, we have $r\nmid|T|$.
Every prime $\ell\mid|T|$ has $\ord_\ell(q)=r$.
The factors of $|P_2|$ prime to $p$ have the form $q^i-1$ with $i<r$.
As $\operatorname{ord}_\ell(q)=r$ for every $\ell\mid|T|$, we have $(|T|,|P_2|)=1$.
Furthermore,
$\frac{[G:P_2]}{|T|} =d\frac{q^{r-1}-1}{q^2-1}>1$.
\begin{itemize}
\item[{\bf Step 2.} ] For $(r,q)\ne(5,2)$, we have $|A|=r$ and $q\equiv1\pmod r$.
\end{itemize}
Because $A$ is ordinary $G$-permutable, we have a subgroup $K=AT^x$ for some $x\in G$.
Since $A\leq P_2^g$, we have $|K|\leq|P_2||T|<|G|$.
As $K<G$, we choose a maximal subgroup $H<G$ with $K\leq H$.
Because $T^x\leq H$, we have $H=N^x$.
Therefore, $|A|\mid (|P_2|,|N|)\mid r$.
If $A\ne1$, then $|A|=r$.
By Schur--Zassenhaus, choose $g\in G$ with $A^g$ equal to a complement of $T$ in $N$.
Let $V=\F_{q^r}$ as an $\F_q$-vector space, so $\dim_{\F_q}V=r$.
A generator of $A^g$ lifts to $w:V\to V$ with $w(z)=z^q$.
In a normal basis, $w$ is an $r$-cycle of determinant $1$.
For $p\ne r$, every $r$th root is an eigenvalue of $w$ with multiplicity $1$.
For $p=r$, the lift $w$ is regular unipotent.

If $p\ne r$ and $q\not\equiv1\pmod r$, then let $t$ be a transvection of order $p$.
Since $A$ is ordinary $G$-permutable, $A\gen{t^x}$ would be a subgroup of order $rp$ for some $x\in G$.
For $p<r$, the Sylow $r$-subgroup $A$ would be normal.
Hence, $t^x\in N_G(A)$.
By Lemma \ref{lem:transvection}, we have a contradiction.
For $p>r$, the subgroup $\gen{t^x}$ would be normal.
As $r\nmid p-1$, $A$ would centralize $t^x$.
Since $d=1$ in this case, $G=\SL_r(q)$.
A generator of $A$ is regular semisimple, with a torus centralizer containing no transvection.
This contradicts its centralization of $t^x$.
For $p=r$, we have $G=\SL_r(q)$.
By Lemma \ref{lem:regular-unipotent}, we have a contradiction.

\begin{itemize}
\item[{\bf Step 3.} ] If $q\equiv1\pmod r$, then $A=1$.
\end{itemize}
Now, let $q\equiv1\pmod r$, so $d=r$.
We take the cyclic torus $B$ of type $(r-1,1)$ with $|B|=(q^{r-1}-1)/r$.
For $z\in\F_{q^{r-1}}^*$, let $M_z:\F_{q^{r-1}}\to\F_{q^{r-1}}$ satisfy $M_z(v)=zv$.
The inverse image of $B$ in $\SL_r(q)$ consists of the matrices $\diag(M_z,N(z)^{-1})$.
Let $N(z)=z^{(q^{r-1}-1)/(q-1)}$ be the norm from $\F_{q^{r-1}}$ to $\F_q$.
If $z^r=1$, then $z\in\F_q$.
Since $r\mid q-1$, we have $N(z)^{-1}=z^{-(r-1)}=z$.
Hence, every element of order $r$ in the inverse image of $B$ is scalar.
As $A$ has a noncentral lift of order $r$, we have $A\cap B^x=1$ for every $x\in G$.

Suppose that $K=AB^x$ is a subgroup.
Then, $[K:B^x]=r$.
Let $C=\Core_K(B^x)$.
By Lemma \ref{lem:abelian-index}, we have $[B^x:C]\mid r-1$.
We choose a primitive prime divisor $\ell$ of $q^{r-1}-1$.
Since $r-1\geq4$ and $q\geq r+1$, Zsigmondy's theorem \cite{Zsigmondy1892} has no exception.
We have $\ell>r-1$ and $\ell\ne r$, so $\ell\mid|C|$.
Let $R\leq C$ have order $\ell$.
As $C$ is cyclic, $R$ is characteristic in $C$.
Because $C\unlhd K$, we have $R\unlhd K$.

The inverse image of $R$ has a unique subgroup $\widetilde R$ of order $\ell$ in $\SL_r(q)$.
Next, let $\F_q^r=W\oplus L$ be the decomposition for $\widetilde R$.
We have $\dim_{\F_q}W=r-1$ and $\dim_{\F_q}L=1$.
Since $\ord_\ell(q)=r-1$, the group $\widetilde R$ acts irreducibly on $W$ and fixes $L$ pointwise.
Relative to $\F_q^r=W\oplus L$, the normalizer $N_{\GL_r(q)}(\widetilde R)$ lies in
$(\GL_1(q^{r-1})\rtimes C_{r-1})\times\GL_1(q)$.
Now, a generator of $A$ lifts to a noncentral element $a$ of order $r$.
Since $R\unlhd K$, the lift $a$ normalizes $\widetilde R$.
As $r\nmid r-1$, the lift $a$ centralizes $\widetilde R$.
The eigenvalues of $a$ are $r$th roots in $\F_q$, so $a|_W$ is multiplication by an element of $\F_q$.
Let $\lambda\in\F_q$ satisfy $a|_W=\lambda I_W$.
As $|a|=r$, we have $\lambda^r=1$.
Because $\det a=1$, we have $a|_L=\lambda^{-(r-1)}I_L=\lambda I_L$.
Hence, $a$ is scalar on $\F_q^r$, a contradiction.

\begin{itemize}
\item[{\bf Step 4.} ] If $G=\PSL_5(2)$, then $A=1$.
\end{itemize}
Now, we consider $G=\PSL_5(2)$.
Then, we have $|G|=9999360$ and $|N|=155$.
The maximal subgroups are $P_1,P_4$ of order $322560$ and $P_2,P_3$ of order $64512$.
The other maximal subgroup is
$N=C_{31}\rtimes C_5$ \cite{Atlas1985}.
By Lemma \ref{lem:max}, we have $A\leq N^x$ or $G=AN^x$ for some $x\in G$.

Suppose that $A\leq N^g$ for some $g\in G$ and replace $A$ by $A^{g^{-1}}$.
Then, $A\leq N$ and $|A|\in\{1,5,31,155\}$.
If $31\mid|A|$, then we choose a Sylow $2$-subgroup $U$ of order $1024$.
Since $A$ is ordinary $G$-permutable, we have a subgroup $K=AU^x$ for some $x\in G$.
We have $|K|\leq155\cdot1024<|G|$.
As $K<G$, we choose a maximal subgroup $H<G$ with $K\leq H$.
Because $U^x\leq H$, we obtain from the maximal subgroup list that $H$ is a parabolic subgroup of $G$.
Therefore, $A$ stabilizes a proper nonzero subspace of $\F_2^5$.
But a subgroup of order $31$ in $A$ acts irreducibly on $\F_2^5$.
Therefore, we have a contradiction.
If $|A|=5$, then a generator has eigenvalue multiplicities $1+4$ over $\F_2$.
Over the algebraic closure of $\F_2$, the eigenvalues are the $5$ distinct $5$th roots.
We choose a transvection $t$ of order $2$.
If $A\gen{t^x}$ were a subgroup for some $x\in G$, then its Sylow $5$-subgroup $A$ would be normal.
This contradicts Lemma \ref{lem:transvection}.
Thus, no $A\gen{t^x}$ with $x\in G$ is a subgroup.

Suppose that $G=AN^x$ with $A,N^x<G$.
By \cite[Theorem 3.3]{LiXia}, we choose $g\in G$ with $A^g=P_i$ for some $i\in\{1,2,3,4\}$.
By duality, we need only consider $P_1$ and $P_2$.
Let $L\leq\F_2^5$ be the line stabilized by $P_1$.
The subspaces $W\leq\F_2^5$ with $\dim_{\F_2}W=2$ and $L\leq W$ form a proper $P_1$-invariant set.
Hence, $P_1$ is not transitive on the subspaces of dimension $2$.
Since $|P_1|>|P_2|$, we have $P_1\nleq P_2^h$ for every $h\in G$.
Since $P_1$ is not transitive on the subspaces of dimension $2$, we have $P_1P_2^h\ne G$ for every $h\in G$.
Together with $P_1\nleq P_2^h$, Lemma \ref{lem:max} excludes every product $P_1P_2^h$ as a subgroup.
As $P_2$ stabilizes a proper nonzero subspace, $P_2$ is not transitive on points.
Let $\F_2^5=W_2\oplus W_3$ be a decomposition for a Levi subgroup of $P_2$.
We have $\dim_{\F_2}W_i=i$ for $i=2,3$.
Because the Levi subgroup acts irreducibly on $W_2$, $P_2$ fixes no line in $W_2$.
Let $U$ be the unipotent radical of $P_2$.
Its matrices on $W_2\oplus W_3$ are $u_X=\begin{pmatrix}I_2&X\\0&I_3\end{pmatrix}$ with $X\in M_{2\times3}(\F_2)$.
As $\bigcap_X\ker(u_X-I)=W_2$, no line outside $W_2$ is fixed by every element of $U$.
Thus, $P_2$ has no fixed line.
Since $P_2$ is not transitive on points and fixes no line, we have $P_2P_1^h\ne G$ and $P_2\nleq P_1^h$ for every $h\in G$.
By Lemma \ref{lem:max}, no product $P_2P_1^h$ with $h\in G$ is a subgroup.
\end{proof}

\subsection{\texorpdfstring{The groups $\PSL_n(q)$ with composite $n$}{The groups PSLn(q) with composite n}}

Let $F\leq E$ be finite fields and let $b\geq1$.
For $v=(v_1,\ldots,v_b)\in E^b$ and $\sigma\in\operatorname{Gal}(E/F)$, let $v^\sigma=(\sigma(v_1),\ldots,\sigma(v_b))$.
For $(g,\sigma)\in\GL_b(E)\rtimes\operatorname{Gal}(E/F)$, we use the action $(g,\sigma)v=gv^\sigma$.

The wreath product $H\wr S_t=H^t\rtimes S_t$ uses the action $\sigma(h_1,\ldots,h_t)=(h_{\sigma^{-1}(1)},\ldots,h_{\sigma^{-1}(t)})$.
An imprimitive subgroup in this section stabilizes a set $\mathcal D=\{V_1,\ldots,V_t\}$ with $V=V_1\oplus\cdots\oplus V_t$, $t\geq2$ and $\dim V_i=d\geq1$ for every $i$.
Its elements satisfy $g\mathcal D=\mathcal D$.

\begin{lemma}\label{lem:linear-singer-containing-subgroups}
Let $T\leq J\leq\SL_m(q)$, where $T$ is a Singer subgroup of order $(q^m-1)/(q-1)$.
Suppose that $m\geq2$ and $(m,q)\ne(4,3)$.
Suppose that $q\geq16$ for $m=2$, $q\geq9$ for $m=3$, and $q\geq3$ for $m\geq4$.
If $J\nleq N_{\SL_m(q)}(T)$, then $J$ contains the natural subgroup $\SL_b(q^s)$ under an $\F_q$-linear identification $\F_q^m\cong(\F_{q^s})^b$ for integers $b,s\geq1$ with $bs=m$ and $b\geq2$.
\end{lemma}
\begin{proof}
We induct on $m$.
For $m=2$ and $q\geq16$, we have $|T|=q+1$.
By Dickson's subgroup theorem \cite{Dickson1901}, $J\leq N_{\SL_2(q)}(T)$ or $J=\SL_2(q)$.
See also \cite[Tables 8.1--8.2]{BHR}.

Suppose that $m\geq3$.
If $J=\SL_m(q)$, then take $b=m$ and $s=1$.
We assume that $J$ is proper.
By Bereczky's theorem \cite{Bereczky}, we have $J\leq E$.
See also \cite[Theorem 2.1.7]{Brookfield2014}.
We take $Q=q^s$ and $V=(\F_Q)^b$ with $bs=m$ and $s$ prime.
Then, on the $\F_q$-vector space $V$, we have
$E=\bigl(\GL_b(Q)\rtimes\operatorname{Gal}(\F_Q/\F_q)\bigr)\cap\SL_m(q)$.
The hypotheses imply that $(m,q)\ne(3,4)$.

We choose a primitive prime divisor $\ell$ of $q^m-1$.
By Zsigmondy's theorem, $\ell$ exists and $\ell>m$.
Since $\ell>s$, the subgroup $T_\ell$ of order $\ell$ lies in $E\cap\GL_b(Q)$.
Let $E_T=C_{\operatorname{End}_{\F_q}(V)}(T_\ell)$.
Since $T_\ell$ acts irreducibly on $V$, we have $E_T\cong\F_{q^m}$.
Both $T$ and the scalar maps from $\F_Q$ lie in $E_T$.
As $E_T$ is commutative, every element of $T$ is $\F_Q$-linear.
If $b=1$, then we have $E=N_{\SL_m(q)}(T)$.

Suppose that $b\geq2$.
Let $J_0=J\cap\SL_b(Q)$.
Then, $J_0\lhd J$ and $T_0=T\cap\SL_b(Q)\leq J_0$.
We have $|T_0|=(q^m-1)/(Q-1)$, so $T_0$ is a Singer subgroup of $\SL_b(Q)$.
The inductive hypotheses hold for $(b,Q)$.
Since $Q=q^s$ with $s$ prime, $Q=9$ would imply that $s=2$ and $q=3$.
If also $b=2$, then $m=4$.
Our hypotheses prohibit $(m,q)=(4,3)$.
Suppose that $J_0$ contains $\SL_a(Q^t)$ on $(\F_{Q^t})^a$, where $at=b$ and $a\geq2$.
As $J_0\leq J$, we have $\SL_a(Q^t)\leq J$.
Otherwise, $J_0\leq N_{\SL_b(Q)}(T_0)$.
The normalizer quotient has order dividing $b<\ell$.
Thus, the Sylow $\ell$-subgroup $R_\ell$ of $J_0$ is unique.
Hence, $R_\ell$ is characteristic in $J_0$ and normal in $J$.
Since $R_\ell\unlhd J$, the group $J$ normalizes $C_{\SL_m(q)}(R_\ell)=T$.
Hence, $J\leq N_{\SL_m(q)}(T)$.
\end{proof}

\begin{corollary}\label{lem:linear-singer-small-p-part}
Under the hypotheses of \textup{Lemma \ref{lem:linear-singer-containing-subgroups}},
let $q=p^f$ with $p$ prime.
If $|J|_p<q^{m/2}$, then $J\leq N_{\SL_m(q)}(T)$.
\end{corollary}
\begin{proof}
If $J\nleq N_{\SL_m(q)}(T)$, then Lemma \ref{lem:linear-singer-containing-subgroups} supplies a subgroup $\SL_b(q^s)\leq J$ with
$|\SL_b(q^s)|_p=q^{m(b-1)/2}\geq q^{m/2}$.
This contradicts $|J|_p<q^{m/2}$.
Therefore, $J\leq N_{\SL_m(q)}(T)$.
\end{proof}

\begin{lemma}\label{lem:linear-singer-line}
Let $n\geq4$ be composite and let $q$ be a prime power.
Let $G=\PSL_n(q)$.
Suppose that $A<G$ is ordinary $G$-permutable.
Suppose that $A$ stabilizes a line in $\F_q^n$ and that $A\leq N_G(T)$ for a Singer torus $T\leq G$.
Then, $A=1$.
\end{lemma}
\begin{proof}
Let $T$ be the projective Singer torus with $A\leq N_G(T)$.
Since a field multiplier in $T$ stabilizing a line lies in $\F_q$, the projective element is trivial.
Thus, $A\cap T=1$.
Let $\nu:N_G(T)\to N_G(T)/T\cong C_n$ be the quotient map.
As $\ker(\nu|_A)=A\cap T=1$, the subgroup $A$ is cyclic and $|A|\mid n$.

We choose a primitive prime divisor $\ell$ of $q^{n-1}-1$.
Since $n$ is composite and $n\geq4$, we have $n-1\geq3$ and $n-1\ne6$.
Therefore, there is no Zsigmondy exception.
We have $\ell\geq n$.
As $n$ is composite and $\ell$ is prime, we have $\ell>n$.
We choose $B\leq G$ of order $\ell$.
A lift of $B$ acts irreducibly on $U\leq\F_q^n$ with $\dim_{\F_q}U=n-1$ and fixes a complementary line.
Because $A$ is ordinary $G$-permutable, we have a subgroup $K=AB^x$ for some $x\in G$.
As $[K:B^x]\leq n<\ell$, the subgroup $B^x$ is a normal Sylow $\ell$-subgroup of $K$.
The quotient $N_G(B^x)/C_G(B^x)$ has order dividing $n-1$.
Let $\eta:A\to N_G(B^x)/C_G(B^x)$ satisfy $\eta(a)=aC_G(B^x)$.
By the homomorphism $\eta$, the order $|\eta(A)|$ divides both $|A|$ and $n-1$.
Since $|A|\mid n$ and $(n,n-1)=1$, we have $\eta(A)=1$.
Therefore, $A$ centralizes $B^x$.
Let $\pi:\SL_n(q)\to G$ be the quotient map.
As its kernel has order prime to $\ell$, the inverse image of $B^x$ has a unique subgroup $\widetilde B_x$ of order $\ell$.
The restriction $\pi|_{\widetilde B_x}:\widetilde B_x\to B^x$ is an isomorphism.
Every lift of $A$ normalizes $\widetilde B_x$ and acts trivially on $B^x$.
Therefore, every lift of $A$ centralizes $\widetilde B_x$.

Suppose that $A\ne1$, and choose an element $a\in A$
of prime order $r$.
The image of $a$ in $N_G(T)/T$ also has order $r$.
Next, we choose a lift $\widetilde a\in\SL_n(q)$.
Let $\overline F$ be an algebraic closure of $\F_q$.
The inverse image of $T$ has $n$ eigenlines on $\overline F\otimes_{\F_q}\F_q^n$.
Because the image of $a$ in $N_G(T)/T$ has order $r$, the lift $\widetilde a$ permutes these eigenlines in $n/r$ cycles of length $r$.
Along each cycle, $\widetilde a$ sends a weight vector to a nonzero scalar multiple of the next vector.
An eigenvector is determined by one coordinate on each cycle.
Since $\widetilde a^r$ is scalar, every geometric eigenspace has dimension at most $n/r$.
This argument also applies in characteristic $r$.

We choose a lift $\widetilde x\in\SL_n(q)$ of $x$.
Let $U_x=\widetilde x^{-1}U$.
Since $\widetilde a$ centralizes $\widetilde B_x$, its restriction $\widetilde a|_{U_x}$ is multiplication by some $\lambda\in\F_{q^{n-1}}^*$.
Since $\widetilde a^r=\zeta I$ for some $\zeta\in\F_q^*$, we have $\lambda^r=\zeta$.
The degree $e=[\F_q(\lambda):\F_q]$ divides $n-1$ and is at most $r$.
As $r\mid n$, we have $r\nmid n-1$.
Therefore, $e\leq r-1$.
Because $n$ is composite and $r\mid n$, we have $n>r$.
Since $\widetilde a|_{U_x}$ is multiplication by $\lambda$ of degree $e$, some geometric eigenspace has dimension at least $(n-1)/e\geq(n-1)/(r-1)$.
As $n>r$, we have $(n-1)/(r-1)>n/r$.
The bounds $n/r$ and $(n-1)/(r-1)$ contradict one another.
\end{proof}

\begin{theorem}\label{thm:linear-composite-large-field}
Let $n\geq4$ be composite and let $q\geq5$ be a prime power.
Then, $G=\PSL_n(q)$ has no proper nontrivial ordinary
$G$-permutable subgroup.
\end{theorem}
\begin{proof}
Let $A<G$ be ordinary $G$-permutable.
Let $q=p^f$ with $p$ prime.
Next, let $P_2$ stabilize $W\leq\F_q^n$ with $\dim_{\F_q}W=2$.
We have $P_2Y\ne G$ for every $Y<G$ by \cite[Theorem C]{LPS}.
See also \cite[Theorem 2.15 and Table A.1]{LiXia}.
By Lemma \ref{lem:max}, we have $A\leq P_2^g$ for some $g\in G$.
Hence, $G\ne AB$ for every proper subgroup $B<G$.
Otherwise, $G=P_2^gB$ with $P_2^g,B<G$, contrary to the cited theorem.
For every maximal subgroup $M<G$, we obtain from Lemma \ref{lem:max} that $A\leq M^x$ for some $x\in G$.
In particular, $A$ stabilizes a line.
Let $T$ be a projective Singer torus.

Firstly, suppose that $(n,q)\ne(4,5)$.
The monomial subgroup $M$ of type $\GL_1(q)\wr S_n$
is maximal \cite[Proposition 2.2.9]{Uffelmann2017}.
Hence, $A\leq M^g$ for some $g\in G$, and
$|A|_p\leq(n!)_p<q^{n/2}$.
For odd $p$, use $v_p(n!)<n/(p-1)\leq n/2$.
For $p=2$, we have $q\geq8$ and $v_2(n!)<n<fn/2$.

Since $A$ is ordinary $G$-permutable, we have a subgroup $K=AT^x$ for some $x\in G$.
As $AB\ne G$ for every $B<G$, we have $K<G$.
We choose a lift $\widetilde x\in\SL_n(q)$ of $x$.
Next, let $\widetilde T$ be the inverse image of $T$ under $\SL_n(q)\to G$.
The inverse image $\widetilde K$ contains $\widetilde T^{\widetilde x}$ of order $(q^n-1)/(q-1)$.
As both $|\widetilde T|$ and the central kernel order are prime to $p$, we have $|\widetilde K|_p=|A|_p$.
By Corollary \ref{lem:linear-singer-small-p-part}, we have $A\leq N_G(T^x)$.
Since $A$ stabilizes a line, Lemma \ref{lem:linear-singer-line} yields $A=1$.

For $(n,q)=(4,5)$, we take a maximal subgroup $M$ of type $\GL_2(5)\wr S_2$
\cite[Proposition 2.2.9]{Uffelmann2017}.
By Lemma \ref{lem:max}, we choose $g_0\in G$ with $A\leq M^{g_0}$.
Hence, $|A|_5\leq25$.
Let $\pi:\SL_4(5)\to G$ be the quotient map, and let $\widetilde A=\pi^{-1}(A)$.
Since $A$ is ordinary $G$-permutable, we choose $x\in G$ with $K=AT^x$ a subgroup.
As $AB\ne G$ for every $B<G$, we have $K<G$.
Let $\widetilde K=\pi^{-1}(K)$.
By Bereczky's theorem, we choose an $\F_{25}$-vector space structure on $V=\F_5^4$ with
$\widetilde K\leq E=\bigl(\GL_2(25)\rtimes C_2\bigr)\cap\SL_4(5)$.
Since the field automorphism part of $E$ has order $2$, every element $u\in E$ of order $5$ is $\F_{25}$-linear.
Hence, $\operatorname{rank}_{\F_5}(u-I)$ is even.
Suppose that $|A|_5=25$.
A Sylow $5$-subgroup of $\widetilde A$ is then a Sylow subgroup of $\pi^{-1}(M^{g_0})$.
Hence, $\widetilde A$ contains a transvection $u$ with $\operatorname{rank}_{\F_5}(u-I)=1$.
Because $u\in\widetilde A\leq\widetilde K\leq E$, we have a contradiction.
Thus, $|A|_5\leq5<25$.
By Corollary \ref{lem:linear-singer-small-p-part} and Lemma \ref{lem:linear-singer-line}, we conclude that $A=1$.
\end{proof}

\begin{lemma}\label{lem:unitriangular-forms}
Let $F\leq E$ be finite fields and let $b\geq2$.
Let $V=E^b$ be an $F$-vector space.
Let $P=\gen{I+tE_{ij}:t\in E,\ 1\leq i<j\leq b}$.
No nondegenerate form $h:V\times V\to F$ satisfies $h(gv,gw)=h(v,w)$ for all $g\in P$ and $v,w\in V$ in the following cases.
\begin{itemize}
\item[(1)] $F$ has odd characteristic and $h$ is symmetric.
\item[(2)] $F=\F_4$ and $h$ is Hermitian.
\item[(3)] $b\geq3$ and $h$ is alternating.
\end{itemize}
\end{lemma}
\begin{proof}
Suppose that a form $h:E^b\times E^b\to F$ satisfies $h(gv,gw)=h(v,w)$ for all $g\in P$ and $v,w\in E^b$.
Let $Y=Ee_1$ and $W=\ker E_{12}$.
The root group $\{I+tE_{12}:t\in E\}$ fixes $W$ pointwise.
For $t\in E$ and $w\in W$, we have $h(e_2+te_1,w)=h(e_2,w)$.
Hence, $h(te_1,w)=0$.
Thus, $Y\perp W$.
In particular, $h(Y,Y)=0$.

Firstly, suppose that $h$ is symmetric and the characteristic is odd.
For $x\ne0$ in $E$ and $t\in E$, let $g=I+tE_{12}$.
Since $g\in P$, we have $h(gxe_2,gxe_2)=h(xe_2,xe_2)$.
Hence, $2h(tx e_1,xe_2)=0$.
Varying $t$ shows that $Y\perp xe_2$.
Therefore, $Y\perp Ee_2$.
Since $Y\perp W$, the nonzero subspace $Y$ lies in the radical of $h$.
As $h$ is nondegenerate, we have a contradiction.

For a Hermitian form over $\F_4$, let $a\in\F_4$ and $v,w\in E^b$.
We use $h(av,w)=a h(v,w)$ and $h(v,aw)=a^2h(v,w)$.
Now, let $x,y\in E$ and $t\in\F_4$.
We take $v=xe_2$, $w=ye_2$ and $g=I+tE_{12}$.
Since $g\in P$, we have $h(gv,gw)=h(v,w)$.
As $h(Y,Y)=0$, we obtain that $t h(xe_1,ye_2)+t^2 h(xe_2,ye_1)=0$.
For $t=1$, we have $h(xe_1,ye_2)+h(xe_2,ye_1)=0$.
For $t\in\F_4\setminus\F_2$, we have $t-t^2\ne0$.
Subtracting $t^2$ times the equation for $t=1$ from the preceding identity, we obtain $(t-t^2)h(xe_1,ye_2)=0$.
Since $t-t^2\ne0$, we have $h(xe_1,ye_2)=h(xe_2,ye_1)=0$.
Because $Y\perp Ee_2$ and $Y\perp W$, the nonzero subspace $Y$ lies in the radical of $h$.
We have a contradiction.

Finally, suppose that $b\geq3$ and $h$ is alternating.
For every $t\in E$, the map $I+tE_{23}$ fixes $\ker E_{23}$ pointwise.
For $w\in\ker E_{23}$, we have $h(e_3+te_2,w)=h(e_3,w)$.
Thus, $h(te_2,w)=0$ and $Ee_2\perp\ker E_{23}$.
Since $Y\leq\ker E_{23}$, we have $Y\perp Ee_2$.
As $Y\perp W$ and $Y\perp Ee_2$, the nonzero subspace $Y$ lies in the radical of $h$.
Because $h$ is nondegenerate, we have a contradiction.
\end{proof}

\begin{theorem}\label{thm:linear-composite-small-fields}
Let $n\geq4$ be composite and let $q\in\{3,4\}$.
Then, $G=\PSL_n(q)$ has no proper nontrivial ordinary
$G$-permutable subgroup.
\end{theorem}
\begin{proof}
For $G=\PSL_4(3)$, we use Proposition \ref{prop:finite-tables}.
Suppose that $(n,q)\ne(4,3)$ and that $A<G$
is ordinary $G$-permutable.
Let $p$ be the characteristic.
As in Theorem \ref{thm:linear-composite-large-field}, we have $P_2Y\ne G$ for every $Y<G$.
Therefore, $AB\ne G$ for every $B<G$.
As $AB\ne G$ for every $B<G$, Lemma \ref{lem:max} yields $A\leq M^g$ for some $g\in G$ for every maximal subgroup $M<G$.
In particular, $A$ stabilizes a line.

For $q=3$, the orthogonal form stabilizer is maximal
\cite[Theorem 4.15(b)]{King2005}.
For $q=4$, the Hermitian form stabilizer is maximal
\cite[Theorem 4.15(c)]{King2005}.
We take a Sylow $p$-subgroup $P$ of the inverse image $\widetilde A$ in $\SL_n(q)$.
For $q=3$, we choose a nondegenerate symmetric form $h$ on $V=\F_3^n$.
For $q=4$, we choose a nondegenerate Hermitian form $h$ on $V=\F_4^n$.
By Lemma \ref{lem:max} for the form stabilizer, we choose $h$ so that $\widetilde A$ acts by similarities of $h$.
Let $\mu:\widetilde A\to\F_q^*$ satisfy $h(gv,gw)=\mu(g)h(v,w)$ for $g\in\widetilde A$ and $v,w\in V$.
Since $p\nmid|\F_q^*|$, we have $\mu(P)=1$.
Hence, $h(gv,gw)=h(v,w)$ for every $g\in P$ and $v,w\in V$.

Let $T$ be a projective Singer torus.
Since $A$ is ordinary $G$-permutable, we have a subgroup $K=AT^x$ for some $x\in G$.
As $AB\ne G$ for every $B<G$, we have $K<G$.
The inverse image $\widetilde K$ in $\SL_n(q)$ contains the inverse image of $T^x$ of order $(q^n-1)/(q-1)$.
We have $|\widetilde K|_p=|A|_p$, so $P$ is a Sylow $p$-subgroup of $\widetilde K$.
Every $p$-subgroup $P_0\leq\widetilde K$ satisfies $P_0^g\leq P$ for some $g\in\widetilde K$.
Let $h_g:V\times V\to\F_q$ satisfy $h_g(v,w)=h(g^{-1}v,g^{-1}w)$ for $V=\F_q^n$.
Then, $h_g(uv,uw)=h_g(v,w)$ for every $u\in P_0$.
Consider $\SL_b(q^s)$ acting on $V=(\F_{q^s})^b$ with $bs=n$ and $b\geq2$.
If this subgroup lay in $\widetilde K$, then its upper unitriangular subgroup would preserve one of the nondegenerate forms $h_g$.
This contradicts Lemma \ref{lem:unitriangular-forms}.
Thus, $\widetilde K$ contains no such subgroup.
From Lemma \ref{lem:linear-singer-containing-subgroups}, we obtain that $A\leq N_G(T^x)$.
Then, we conclude from Lemma \ref{lem:linear-singer-line} that $A=1$.
\end{proof}

\begin{lemma}\label{lem:unitriangular-indecomposable}
Let $F\leq E$ be finite fields and let $b\geq2$.
The group $P=\gen{I+tE_{ij}:t\in E,\ 1\leq i<j\leq b}$ acts indecomposably on the $F$-vector space $E^b$.
\end{lemma}
\begin{proof}
Let $X\in\operatorname{End}_F(E^b)$ commute with $P$.
Relative to $E^b=\bigoplus_{i=1}^b Ee_i$, the block matrix of $X$ is $(X_{ij})_{1\leq i,j\leq b}$.
Moreover, each block $X_{ij}$ belongs to $\operatorname{End}_F(E)$.
Since $X$ commutes with $P$, we have $X(I+tE_{j,j+1})=(I+tE_{j,j+1})X$ for every $t\in E$.
From the block comparison in this identity, we obtain $X_{ij}=0$ for $i\ne j$ and $(i,j)\ne(1,b)$.
The diagonal blocks also coincide.
Since the common diagonal map commutes with multiplication by every $t\in E$, it is $x\mapsto ax$ for some $a\in E$.
Thus, every commuting endomorphism $X$ has $X_{ii}(x)=ax$ for $1\leq i\leq b$ and $x\in E$, with $a\in E$.
The only possible nondiagonal block is $X_{1b}=f\in\operatorname{End}_F(E)$.
An idempotent has $a^2=a$, so $a=0$ or $a=1$.
For $a=0$ or $a=1$, the $(1,b)$ block of $X^2=X$ is $2af=f$.
Thus, $f=0$.
Hence, the only idempotents are $0$ and $I$.
If $E^b=V_1\oplus V_2$ with nonzero $P$-invariant $F$-subspaces $V_1,V_2$, then the projection onto $V_1$ is an idempotent.
Because both summands are $P$-invariant, the projection lies in $\operatorname{End}_{F[P]}(E^b)$ and differs from $0,I$.
Therefore, no such decomposition exists.
\end{proof}

\begin{theorem}\label{thm:linear-composite-two}
Let $n\geq4$ be composite.
Then, $G=\PSL_n(2)$ has no proper nontrivial ordinary
$G$-permutable subgroup.
\end{theorem}
\begin{proof}
For $n=4$, use $\PSL_4(2)\cong A_8$ and Proposition
\ref{prop:finite-tables}.
Suppose that $n\geq6$ and that $A<G$ is ordinary
$G$-permutable.
We have $G=\GL_n(2)$.
By the argument in Theorem \ref{thm:linear-composite-large-field}, we have $P_2Y\ne G$ for every $Y<G$.
Therefore, $AB\ne G$ for every $B<G$.
By Lemma \ref{lem:max}, we have $A\leq P_1^g$ for some $g\in G$.
Hence, $A$ stabilizes a line.

We take a Singer subgroup $T$ of order $2^n-1$.
Since $A$ is ordinary $G$-permutable, we have a subgroup $K=AT^x$ for some $x\in G$.
As $AB\ne G$ for every $B<G$, we have $K<G$.
By Kantor's theorem \cite{Kantor1980}, we have
$\GL_b(2^s)\lhd K\leq\GL_b(2^s)\rtimes\operatorname{Gal}(\F_{2^s}/\F_2)$
with $bs=n$.
We use $V=(\F_{2^s})^b$ as an $\F_2$-vector space.
Because $T$ has odd order, every Sylow $2$-subgroup of $A$ is a Sylow $2$-subgroup of $K$.
We prove that $b=1$.

Firstly, suppose that $n$ is odd.
We choose $n=dt$ with odd $d,t\geq3$.
The imprimitive subgroup of type $\GL_d(2)\wr S_t$
is maximal \cite[Proposition 2.2.9]{Uffelmann2017}.
By Lemma \ref{lem:max}, $A$ lies in a conjugate of this maximal subgroup.
Thus, a Sylow $2$-subgroup $P$ of $A$ permutes $t$ blocks of dimension $d$.
Since $t>1$ is odd, $P$ is intransitive on the blocks.
We sum the blocks in each $P$-orbit and obtain the decomposition $V=W_1\oplus\cdots\oplus W_k$ with $k>1$.
For every $u\in P$, we have $uW_i=W_i$.
If $b\geq2$, then we take $P_0=\gen{I+tE_{ij}:t\in\F_{2^s},\ 1\leq i<j\leq b}\leq\GL_b(2^s)$.
Then, $P_0^g\leq P$ for some $g\in K$.
As $P_0^g\leq P$, we have $u(gW_i)=gW_i$ for every $u\in P_0$ and $1\leq i\leq k$.
Thus, $V=gW_1\oplus\cdots\oplus gW_k$ contradicts Lemma \ref{lem:unitriangular-indecomposable}.

Now, suppose that $n=2m$ is even, so $m\geq3$.
The subgroup $\Sp_n(2)$ is maximal
\cite[Theorem 4.15(a)]{King2005}.
Let $P$ be a Sylow $2$-subgroup of $A$.
By Lemma \ref{lem:max}, $A$ lies in a conjugate of $\Sp_n(2)$.
Therefore, we choose a nondegenerate alternating form $h$ on $V$ with
$h(uv,uw)=h(v,w)$ for all $u\in P$ and $v,w\in V$.

If $b\geq3$, then a conjugate of the upper unitriangular subgroup of $\GL_b(2^s)$ lies in $P$.
It would preserve the corresponding conjugate of $h$, contrary to Lemma \ref{lem:unitriangular-forms}.
Thus, $b<3$.
If $b=2$, then $s=m$ and $K$ contains the root group
$R=\{I+tE_{12}:t\in\F_{2^m}\}\cong C_2^m$.
Every nonidentity element $g\in R$ satisfies
$(g-I)^2=0$ and $\operatorname{rank}(g-I)=m$.
There is $g\in K$ such that $R^g$ lies in a Sylow $2$-subgroup of $A$.
Then, we replace $R$ by $R^g$.

The subgroup of type $\GL_m(2)\wr S_2$ is maximal
\cite[Proposition 2.2.9]{Uffelmann2017}.
By Lemma \ref{lem:max}, $A$ lies in a conjugate of this maximal subgroup.
Hence, $V=V_1\oplus V_2$ with $\dim_{\F_2}V_i=m$, and $R$ permutes $V_1,V_2$.
Let $R_0=\{g\in R:gV_i=V_i\text{ for }i=1,2\}$.
Since $[R:R_0]\leq2$, we have $\dim_{\F_2}R_0\geq m-1$.
For odd $m$, every involution on $1$ block has rank
at most $(m-1)/2$.
Hence, every nonidentity element of $R_0$ has total rank at most $m-1$.
But every nonidentity element of $R$ has rank $m$.
Hence, we have a contradiction.

For even $m$, each block rank is at most $m/2$.
Since the total rank is $m$, both block ranks equal $m/2$.
For $i=1,2$, let $\pi_i:R_0\to\GL(V_i)$ satisfy $\pi_i(g)=g|_{V_i}$.
As every nonidentity element of $R_0$ has rank $m/2$ on each block, each $\pi_i$ is injective.
Let $d=\dim_{\F_2}R_0$.
By the bound for $R_0$, we have $d\geq m-1$.
Every nonidentity element fixes exactly $2^{m/2}$ vectors
on either block.
Using Burnside's formula for the $R_0$-orbits on $V_1$, we obtain
$$
 \frac{2^m+(2^d-1)2^{m/2}}{2^d}
 =2^{m/2}+\frac{2^{m/2}(2^{m/2}-1)}{2^d}\in\Z.
$$
Since $2^{m/2}-1$ is odd, the integrality of the displayed expression implies $d\leq m/2$.
As $d\geq m-1$ and $m\geq4$, we have a contradiction.
Therefore, $b\ne2$.
We conclude that $b=1$ and $A\leq N_G(\GL_1(2^n))$.
By Lemma \ref{lem:linear-singer-line}, we have $A=1$.
\end{proof}

\begin{corollary}\label{cor:all-linear}
For every $n\geq3$ and every prime power $q$, the group
$G=\PSL_n(q)$ has no proper nontrivial ordinary
$G$-permutable subgroup.
\end{corollary}
\begin{proof}
Let $A<G$ be ordinary $G$-permutable.
By Theorems \ref{thm:l3} and \ref{thm:linear-prime}, we have $A=1$ in odd prime dimensions.
By Theorem \ref{thm:linear-composite-large-field}, we have $A=1$ for composite $n$ and $q\geq5$.
For composite $n$ and $q\in\{2,3,4\}$, we obtain from Theorems \ref{thm:linear-composite-small-fields} and \ref{thm:linear-composite-two} that $A=1$.
\end{proof}

\section{Unitary groups}\label{sec:unitary}

Let $q$ be a prime power and $n\geq2$.
We use $V=\F_{q^2}^n$ with a Hermitian form $h$ that is additive in both variables.
For $a,b\in\F_{q^2}$ and $v,w\in V$, we require $h(av,bw)=ab^q h(v,w)$ and $h(w,v)=h(v,w)^q$.
For $W\leq V$, let $W^\perp=\{v\in V:h(v,w)=0\text{ for all }w\in W\}$.
We assume $V^\perp=0$.
A subspace $W$ is nondegenerate if $W\cap W^\perp=0$, and totally isotropic if $W\leq W^\perp$.
The notation $V=V_1\perp V_2$ means $V=V_1\oplus V_2$ and $h(V_1,V_2)=0$.
An orthonormal basis satisfies $h(e_i,e_j)=\delta_{ij}$.
Let $\GU(V,h)=\{g\in\GL_{\F_{q^2}}(V):h(gv,gw)=h(v,w)\text{ for all }v,w\in V\}$.
Let $\SU(V,h)=\GU(V,h)\cap\SL_{\F_{q^2}}(V)$.
We also use $\GU_n(q)$ and $\SU_n(q)$, with scalar quotients $\PGU_n(q)$ and $\PSU_n(q)$.

A torus of type $(n_1,\ldots,n_t)$ in $\GU_n(q)$ acts on $V=V_1\perp\cdots\perp V_t$ with $\dim_{\F_{q^2}}V_i=n_i$.
Its cyclic factor on $V_i$ has order $q^{n_i}-(-1)^{n_i}$.
For odd $n_i$, that factor acts irreducibly on $V_i$.
For even $n_i$, we use $V_i=W_i\oplus W_i'$ with totally isotropic summands of dimension $n_i/2$.
The factor $T_i$ on $V_i$ satisfies $T_iW_i=W_i$ and $T_iW_i'=W_i'$.
We intersect the torus with $\SU_n(q)$ to impose determinant $1$, and take the image in $\PSU_n(q)$ under the scalar quotient map.
For odd $n$, type $(n)$ is a unitary Singer torus.
Its orders in $\GU_n(q)$, $\SU_n(q)$ and $\PSU_n(q)$ are $q^n+1$, $(q^n+1)/(q+1)$ and $(q^n+1)/((q+1)(n,q+1))$.

\subsection{\texorpdfstring{The groups $\PSU_3(q)$}{The groups PSU3(q)}}

Let $Q$ be a prime power, $d\geq1$, $E=\F_{Q^d}$ and $F=\F_Q$.
The trace $\Tr_{E/F}:E\to F$ satisfies $\Tr_{E/F}(z)=\sum_{j=0}^{d-1}z^{Q^j}$.
The norm $N_{E/F}:E^*\to F^*$ satisfies $N_{E/F}(z)=z^{(Q^d-1)/(Q-1)}$.
For a group $H$, let $H'=\gen{a^{-1}b^{-1}ab:a,b\in H}$.
A $p$-group $P$ is extraspecial if $Z(P)=P'=\Phi(P)$ has order $p$, where $\Phi(P)$ is the intersection of the maximal subgroups of $P$.
For $p=3,5$, $p^{1+2}$ denotes an extraspecial group of order $p^3$ and exponent $p$.

\begin{lemma}\label{lem:unitary-three-subgroups}
Let $q\geq3$ be a prime power.
Let $G=\PSU_3(q)$ and $d=(3,q+1)$.
If $q\notin\{3,5\}$, then $XY\ne G$ for all $X,Y<G$.
For $q\notin\{3,5\}$, a Singer torus normalizer $N$ and a Borel subgroup $P$ are maximal.
Their orders are $|N|=3(q^2-q+1)/d$ and $|P|=q^3(q^2-1)/d$.
If $d=1$, then the split torus normalizer is also maximal and has
order $6(q+1)^2$.

The group $M_{10}$ is the Mathieu group of degree $10$.
In the table, $C_3\leq\GU_2(5)$ denotes the scalar subgroup of order $3$.
For $q=3,5$, the table lists the types and orders of maximal subgroups of $G$.
If $G=XY$ with $X,Y<G$ maximal, then one factor has the first type and the other has the second type in the row for $q$.
\begin{center}
\renewcommand{\arraystretch}{1.12}
\begin{tabular}{|c|c|c|c|c|}
\hline
$q$ & \multicolumn{4}{c|}{Maximal subgroups and orders}\\\hline
$3$ & $3^{1+2}\rtimes C_8$ & $\PSL_2(7)$ & $\GU_2(3)$ & $(C_4\times C_4)\rtimes S_3$\\
    & $216$ & $168$ & $96$ & $96$\\\hline
$5$ & $5^{1+2}\rtimes C_8$ & $A_7$ & $M_{10}$ & $\GU_2(5)/C_3$\\
    & $1000$ & $2520$ & $720$ & $240$\\
\hline
\end{tabular}
\end{center}
Every proper subgroup of $\PSL_2(7)$ containing a subgroup of order
$7$ lies in $C_7\rtimes C_3$.
Every subgroup of $\PSL_2(7)$ of order $24$ is isomorphic to $S_4$.
Every proper subgroup of $A_7$ containing an element of order $7$ lies in a subgroup isomorphic to $\PSL_2(7)$.
\end{lemma}
\begin{proof}
We use the maximal subgroup lists for $G=\PSU_3(q)$ in \cite{BHR}.
For $G=XY$ with $X,Y<G$ maximal, we use the list in \cite{LPS}.
The orders are also in \cite[Section 4]{GT25}.
We also use the small subgroup lists in \cite{Atlas1985}.
\end{proof}

\begin{proposition}\label{prop:unitary-three-order-three}
Let $q\geq3$ be a prime power and let $G=\PSU_3(q)$.
Let $N$ be a Singer torus normalizer and $P$ a Borel subgroup.
Let $A\leq G$ have order $3$.
Suppose that $A^g\leq N$ and $A^h\leq P$ for some $g,h\in G$.
Then, $A$ is not ordinary $G$-permutable.
\end{proposition}
\begin{proof}
Suppose that $A$ is ordinary $G$-permutable.
Let $d=(3,q+1)$.
Choose $g\in G$ with $A^g\leq N$ and replace $A$ by $A^g$.

Let $\pi:\SU_3(q)\to G$ be the quotient map.
We choose a field model for the Singer torus normalized by $N$.
Let $V=\F_{q^6}$ as an $\F_{q^2}$-vector space, so $\dim_{\F_{q^2}}V=3$.
We take the Hermitian form $h(x,y)=\Tr_{\F_{q^6}/\F_{q^2}}(xy^{q^3})$.
For $a,b\in\F_{q^2}$, $h(ax,by)=ab^q h(x,y)$ and $h(y,x)=h(x,y)^q$.
We take $\widetilde T=\{M_z:z^{q^2-q+1}=1\}$, where $M_z(x)=zx$ for $x,z\in\F_{q^6}$.
Then, $\widetilde T$ is a Singer torus of $\SU_3(q)$.
Let $T=\pi(\widetilde T)$, so $N=N_G(T)$.
Let $w:V\to V$ satisfy $w(x)=x^{q^2}$.
We have $h(wx,wy)=h(x,y)$ for $x,y\in V$ and $\det w=1$.
Moreover, the map $w$ is $\F_{q^2}$-linear.
Furthermore, on a normal basis, the matrix of $w$ is a $3$-cycle permutation matrix.
The normal basis calculation yields $|w|=3$.
Since $wM_zw^{-1}=M_{z^{q^2}}$, we also have $w\in N_{\SU_3(q)}(\widetilde T)$.

We have $|T|=(q^2-q+1)/d$ and $(|T|,3)=1$.
Moreover, we have $N_G(T)=T\rtimes\gen{\overline w}$.
Let $C\leq N_G(T)$ have order $3$.
Since $(|T|,3)=1$, we choose $g\in N_G(T)$ with $C=\gen{\overline w}^{\,g}$ by Schur--Zassenhaus.
If the characteristic is $3$, then $w$ has $1$ Jordan block of size $3$.
Otherwise, $w$ has the $3$ distinct eigenvalues $1,\omega,\omega^2$, with $\omega^3=1\ne\omega$.
For every scalar $\zeta\ne0$, the eigenvalue multiplicities of $\zeta w$ equal those of $w$.

Firstly, suppose that $q\equiv1\pmod3$.
Then, $d=1$ and $q\geq4$.
We choose $s\in\F_{q^2}^{\times}$ of order $m=q+1$, so $s^{q+1}=1$.
In an orthonormal basis, we choose $b=\diag(s,s^{-1},1)$.
Let $D=\gen b$.
Since $A$ is ordinary $G$-permutable, $K=AD^x$ is a subgroup for some $x\in G$.
As $(3,m)=1$, we have $|K|=3m$.
By Lemma \ref{lem:index3}, we have $(D^x)^2\unlhd K$.
Therefore, $A\leq N_G((D^x)^2)$.

Since $q\geq4$ for even $q$ and $q\geq7$ for odd $q$, we have $|s^2|=m/(2,m)>2$.
Thus, $b^2$ has $3$ distinct eigenvalues.
Let $L_i=\F_{q^2}e_i$ in the chosen orthonormal basis.
For $g\in N_G(D^2)$, we have $gL_3=L_3$ and $\{gL_1,gL_2\}=\{L_1,L_2\}$.
As $|C_G(D^2)|=(q+1)^2$, we have $|N_G(D^2)|\mid2(q+1)^2$.
Because $3\nmid2(q+1)^2$, the inclusion $A\leq N_G((D^x)^2)$ is impossible.

Now, suppose that $q\equiv2\pmod3$.
Then, $d=3$, and the characteristic is not $3$.
Suppose that $v_3(q+1)=1$.
We have $|P|=q^3(q^2-1)/3$ and $3\nmid|P|$.
Hence, $A^h\nleq P$ for every $h\in G$, contrary to the hypothesis.

Suppose that $9\mid q+1$.
We take the inverse image $\widetilde P$ of $P$ under $\SU_3(q)\to G$.
We choose a basis $(e_1,e_2,e_3)$ with $h(e_1,e_3)=h(e_3,e_1)=h(e_2,e_2)=1$.
All other values $h(e_i,e_j)$ are $0$.
In this basis, we have $C=\{z(a):a\in\F_{q^2}^{\times}\}\leq\widetilde P$.
With this basis, we use $z(a)=\diag(a,a^{q-1},a^{-q})$.
Let $\overline C$ be the projective image of $C$.
Since $|C\cap Z(\SU_3(q))|=3$, the unique subgroup $D_3\leq\overline C$ of order $3$ has inverse image of order $9$ in $C$.
For $a$ of order $9$, we have $q\equiv-1\pmod9$ and $z(a)=\diag(a,a^7,a)$.
Hence, $z(a)$ has a repeated eigenvalue.
Let $B\leq P$ have order $3$.
Since the unipotent radical of $P$ is normal and has order prime to $3$, we choose $u\in P$ with $B^u\leq\overline C$ by Schur--Zassenhaus.
Every subgroup of order $3$ in $N$ has a lift with $3$ distinct eigenvalues, by the calculation of $w$.
Scalar multiplication does not change eigenvalue multiplicities.
Thus, a generator of $B$ has a repeated eigenvalue in every lift.
By the preceding comparison with lifts of subgroups of order $3$ in $N$, we have $B^g\nleq N$ for every $g\in G$.
Therefore, the $2$ inclusions $A^g\leq N$ and $A^h\leq P$ cannot both hold.

Finally, suppose that $3\mid q$.
Then, $d=1$, and a generator of $A$ is regular unipotent by the computed Jordan block of $w$.
The centralizer of a generator of $A$ is a $3$-group by Lemma \ref{lem:regular-unipotent}.
We choose a cyclic subgroup $D$ of order $4$ in the Borel complement.
Suppose that $AD^x$ is a subgroup.
By Lemma \ref{lem:index3}, we have $(D^x)^2\unlhd AD^x$.
Because a normal subgroup of order $2$ is central, $A$ would centralize the involution generating $(D^x)^2$.
This contradicts the preceding conclusion that $C_G(A)$ is a $3$-group.\qedhere
\end{proof}

\begin{theorem}\label{thm:u3}
For every prime power $q\geq3$, the simple group $G=\PSU_3(q)$ has no
proper nontrivial ordinary $G$-permutable subgroup.
\end{theorem}

\begin{proof}
Let $A<G$ be nontrivial and ordinary $G$-permutable.
Let $d=(3,q+1)$.
Next, let $N$ be a Singer torus normalizer and $P$ a Borel subgroup,
as in Lemma \ref{lem:unitary-three-subgroups}.
Firstly, suppose that $q\notin\{3,5\}$.
If $d=1$, then $q^2-q+1$ is odd and coprime to $3(q+1)$.
By Lemma \ref{lem:max}, we compare $|N|$ with the split torus normalizer order.
We have $|A|\mid\gcd\bigl(3(q^2-q+1),6(q+1)^2\bigr)=3$.
If $d=3$, then let $s=(q^2-q+1)/3$.
Then,
$$
 (s,3q(q-1)(q+1))=1, |A|\mid\gcd\left(3s,\frac{q^3(q^2-1)}3\right)\mid3.
$$
Thus, $|A|=3$ for $q\notin\{3,5\}$.
Since $XY\ne G$ for all $X,Y<G$, we have $A^g\leq N$ and $A^h\leq P$ for some $g,h\in G$.
By Proposition \ref{prop:unitary-three-order-three}, we conclude that $|A|\ne3$.

For $q=3$, we have $|G|=6048$.
We choose a maximal subgroup $M<G$ of order $96$.
By Lemma \ref{lem:unitary-three-subgroups}, we have $MY\ne G$ for every $Y<G$.
It follows from Lemma \ref{lem:max} that $|A|\mid96$.
For a maximal subgroup $L\cong\PSL_2(7)$, $G=AL^x$ would require $36=[G:L]\mid|A|$.
Hence, $A\leq L^x$ for some $x\in G$ and $|A|\mid24$.
Since $28=[G:P]\nmid|A|$, we also have $A\leq P^y$ for some $y\in G$.
The Sylow $2$-subgroups of $P\cong3^{1+2}\rtimes C_8$ are cyclic of order $8$.

We take $T\leq G$ of order $7$.
Since $A$ is ordinary $G$-permutable, we have a subgroup $R=AT^x$ for some $x\in G$.
We have $|R|=7|A|\leq168$.
Only maximal subgroups isomorphic to $L$ have order divisible by $7$.
Hence, we choose $L_0\leq G$ with $R\leq L_0$ and $L_0\cong L$.
By the subgroup list of $L_0$, we have $|R|\in\{7,21,168\}$ and $|A|\in\{1,3,24\}$.
If $|A|=24$, then $A\leq L_0$ implies that $A\cong S_4$.
As $A\leq P^y$ and every Sylow $2$-subgroup of $P^y$ is cyclic, $A\cong S_4$ is impossible.
If $|A|=3$, then $R=C_7\rtimes C_3=N_G(T^x)$.
We have $A\leq P^y$ with $y\in G$.
By Proposition \ref{prop:unitary-three-order-three}, we have a contradiction.
Therefore, $A=1$ for $q=3$.

For $q=5$, we have $|G|=126000$.
We choose a maximal subgroup $M\cong\GU_2(5)/C_3$ of order $240$.
By Lemma \ref{lem:unitary-three-subgroups}, we have $MY\ne G$ for every $Y<G$.
Then, we obtain from Lemma \ref{lem:max} that $|A|\mid240$.
Since $50=[G:L]\nmid240$ for $L\cong A_7$, we have $A\leq L^x$ for some $x\in G$.
As $126=[G:P]\nmid240$ for $P\cong5^{1+2}\rtimes C_8$, we have $A\leq P^y$ for some $y\in G$.
Consequently, $|A|\mid\gcd(240,2520,1000)=40$.
We choose $T\leq G$ of order $7$.
Because $A$ is ordinary $G$-permutable, we choose $x\in G$ such that $R=AT^x$ is a subgroup.
We have $|R|=7|A|\leq280$.
As only maximal subgroups of type $A_7$ contain an element of order $7$, we choose $H\cong A_7$ with $R\leq H$.
Because $|R|<|H|$, we obtain from Lemma \ref{lem:unitary-three-subgroups} that $R\leq L_0$ for some $L_0\cong\PSL_2(7)$.
Hence, $|R|\in\{7,21,168\}$ and $|A|\in\{1,3,24\}$.
As $|A|\mid40$, we conclude that $A=1$.
\end{proof}

\subsection{\texorpdfstring{The groups $\PSU_r(q)$ with prime $r\geq5$}{The groups PSUr(q) with prime r at least 5}}
\begin{theorem}\label{thm:unitary-prime}
Let $r$ be an odd prime and let $q$ be a prime power.
Suppose that $G=\PSU_r(q)$ is simple.
Then, $G$ has no proper nontrivial ordinary $G$-permutable subgroup.
\end{theorem}
\begin{proof}
For ease of reading, we divide the proof of the theorem into the following steps.

For $r=3$, we use Theorem \ref{thm:u3}.
Now, assume that $r\geq5$.
Let $q=p^f$ with $p$ prime, and let $d=(r,q+1)$.
Let $A<G$ be nontrivial and ordinary $G$-permutable.
Firstly, assume that $(r,q)\ne(5,2)$.
We consider $\PSU_5(2)$ at the end and use the following classifications.
\begin{itemize}
\item[(1)] Since $r$ is prime, we have $XY\ne G$ for all $X,Y<G$ by \cite[Theorem 3.5]{LiXia}.
\item[(2)] We take a unitary Singer torus $T$ with $|T|=(q^r+1)/((q+1)d)$.
Let $N=N_G(T)=T\rtimes C_r$.
By \cite{Bereczky}, we have $H=N$ for every maximal subgroup $H<G$ with $T\leq H$.
The only exceptional pair with prime $r\geq5$ is $(r,q)=(5,2)$.
\end{itemize}

We take $P_1=\operatorname{Stab}_G(L)$ for a singular line $L\leq\F_{q^2}^r$.
By Lemma \ref{lem:max}, we have $|A|\mid|N|$ and $|A|\mid|P_1|$.
We have $|P_1|=q^{2r-3}(q^2-1)|\SU_{r-2}(q)|/d$.
The integer $\Phi_{2r}(q)=(q^r+1)/(q+1)$ is odd.
If $r\nmid q+1$, then $r\nmid\Phi_{2r}(q)$.
If $r\mid q+1$, then $v_r(\Phi_{2r}(q))=1$.
Since $d=(r,q+1)$, the quotient $|T|=\Phi_{2r}(q)/d$ is odd and prime to $r$. Any prime $\ell\mid |T|$ satisfies
$\ell\nmid q+1$ and $\ord_\ell(-q)=r$. Such a prime
divides none of the factors $q^i-(-1)^i$ with $i<r$, nor $q^2-1$.
Therefore, $(|T|,|P_1|)=1$ and $(|N|,|P_1|)\mid r$.
As $A\ne1$, we have $|A|=r$.

Since $(|T|,r)=1$, every subgroup of $N=T\rtimes C_r$ of order $r$ equals $C_r^g$ for some $g\in N$.
Let $V=\F_{q^{2r}}$ as an $\F_{q^2}$-vector space, so $\dim_{\F_{q^2}}V=r$.
We take $h(x,y)=\Tr_{\F_{q^{2r}}/\F_{q^2}}(xy^{q^r})$.
For $a,b\in\F_{q^2}$, $h(ax,by)=ab^q h(x,y)$ and $h(y,x)=h(x,y)^q$.
Let $w:V\to V$ satisfy $w(x)=x^{q^2}$.
By the trace formula, we have $h(wx,wy)=h(x,y)$ for every $x,y\in V$.
On a normal basis, $w$ acts as an $r$-cycle.
Hence, $\det w=1$ and $|w|=r$.
For $p\ne r$, every $r$th root is an eigenvalue of $w$ with multiplicity $1$.
For $p=r$, the lift $w$ has $1$ Jordan block of size $r$.
For every central scalar $\zeta$, the eigenvalue multiplicities of $\zeta w$ equal those of $w$.
By Lemma \ref{lem:max}, $A$ lies in a conjugate of $N$.
Schur--Zassenhaus then yields $g\in G$ such that $A^g=\gen{\overline w}$.
We replace $A$ by $A^g$.

\setlist[itemize]{leftmargin=4em,labelsep=.5em,topsep=.2\baselineskip,partopsep=0pt,parsep=0pt,itemsep=0pt}
\begin{itemize}
\item[{\bf Step 1.} ] If $(r,q)\ne(5,2)$ and $p\ne r$ and $q\not\equiv1\pmod r$, then $A=1$.
\end{itemize}
Let $\pi:\SU_r(q)\to G$ be the quotient map.
Next, we choose a unitary transvection $u\in\SU_r(q)$.
Let $t=\pi(u)$ and $U=\gen t$.
Then, $|U|=p$.
Since $A$ is ordinary $G$-permutable, we have a subgroup $K=AU^x$ with $|K|=rp$ for some $x\in G$.
If $p<r$, then $A\lhd K$.
Therefore, $t^x\in N_G(A)$.
By Lemma \ref{lem:transvection}, we have a contradiction.
For $p>r$, we have $U^x\lhd K$.
As $q\not\equiv1\pmod r$, we have $r\nmid p-1$.
Hence, $A$ centralizes $U^x$ in $G$.
We choose a lift $\widetilde x\in\SU_r(q)$ of $x$.
Let $u_x=u^{\widetilde x}$.
For a lift $a\in\SU_r(q)$ of a generator of $A$, we have $a u_x a^{-1}=\zeta u_x$ for a scalar $\zeta$.
Because $a u_x a^{-1}$ is unipotent, $\zeta=1$.
As $a$ is regular semisimple, $C_{\GL_r(q^2)}(a)$ is a torus.
Hence, $u_x=I$, contrary to the choice of a transvection.
The condition $r\mid q+1$ is included in $q\not\equiv1\pmod r$.

\begin{itemize}
\item[{\bf Step 2.} ] If $p\ne r$ and $q\equiv1\pmod r$, then $A=1$.
\end{itemize}
Now, suppose that $p\ne r$ and $q\equiv1\pmod r$.
Then, $d=1$, so $G=\SU_r(q)$.
Let $m=q+1$.
Since $q\equiv1\pmod r$ and $q\geq2$, we have $m\geq r+2$.
Let $\mu_m=\{z\in\F_{q^2}^{\times}:z^m=1\}$.
In an orthonormal basis, we take the abelian subgroup
$$
B=\{\diag(z,z,z_3,\ldots,z_r): z,z_i\in\mu_m,\ z^2z_3\cdots z_r=1\}.
$$
By the diagonal constraint, we have $|B|=m^{r-2}$, prime to $r$.
If $K=AB^x$ were a subgroup, then $B^x$ would have index $r$.
We replace $K,A$ by $K^{x^{-1}},A^{x^{-1}}$ with $x^{-1}\in G$.
Then, $K=AB$.
Let $C=\Core_K(B)$.
By Lemma \ref{lem:abelian-index}, we have $[B:C]\leq r-1$.

Let $V=\F_{q^2}^r$.
Let $V=W_1\perp\cdots\perp W_k$ be the decomposition into common eigenspaces of $C$.
Since the first $2$ coordinates have equal $C$-characters, we have $k\leq r-1$.
As $C\unlhd K$, the group $A$ permutes $W_1,\ldots,W_k$.
Because $|A|=r>k$, we have $aW_i=W_i$ for every $a\in A$ and every $i$.
As $C$ acts by scalars on each $W_i$, we have $A\leq C_G(C)$.

For a generator $a$ of $A$, the nontrivial eigenvalues pair as
$\lambda,\lambda^{-q}=\lambda^{-1}$ on each nondegenerate invariant
block. Every $r$th root is an eigenvalue of $a$ with multiplicity $1$.
Hence, the block containing the eigenvalue $1$ is the only block of odd dimension.
All other blocks have even dimension.
Therefore, $k\leq(r+1)/2$.

Let $s_i=\dim_{\F_{q^2}}W_i$ for $1\leq i\leq k$.
Let
$$
D=\{\diag(z_1I_{s_1},\ldots,z_kI_{s_k}):
z_i\in\mu_m,\ \prod_{i=1}^kz_i^{s_i}=1\}.
$$
By the $C$-weight space decomposition and the determinant constraint, we have $C\leq D$ and $|D|=m^{k-1}(m,s_1,\ldots,s_k)$.
Since $(m,r)=1$ and $\sum_{i=1}^k s_i=r$, we have $(m,s_1,\ldots,s_k)=1$.
Hence, $|D|=m^{k-1}$.
Consequently,
$$
[B:C]\geq m^{r-2}/m^{k-1}\geq m^{(r-3)/2}>r-1.
$$
This contradicts $[B:C]\leq r-1$ from Lemma \ref{lem:abelian-index}.

\begin{itemize}
\item[{\bf Step 3.} ] If $p=r$, then $A=1$.
\end{itemize}
Now, suppose that $p=r$.
Then, $d=1$, so $G=\SU_r(q)$.
A generator of $A$ is regular unipotent.
By Lemma \ref{lem:regular-unipotent}, we have a contradiction.

\begin{itemize}
\item[{\bf Step 4.} ] If $G=\PSU_5(2)$, then $A=1$.
\end{itemize}
For $G=\PSU_5(2)$, we have $XY\ne G$ for all $X,Y<G$.
The maximal subgroups of $\PSU_5(2)$
include $L\cong\PSL_2(11)$ of order $660$ and $P_1$ of order $82944$
\cite{Atlas1985}. Therefore, $|A|\mid(660,82944)=12$.
We take a Singer torus $T\leq G$ of order $11$.
By \cite{Bereczky}, we have $H\cong L$ for every maximal subgroup $H<G$ with $T\leq H$.
Since $A$ is ordinary $G$-permutable, we have a subgroup $AT^x$ for some $x\in G$.
We have $|AT^x|=11|A|\leq132<660<|G|$.
As $AT^x<G$, we choose a maximal subgroup $H<G$ with $AT^x\leq H$.
Because $T\leq H^{x^{-1}}$, we have $H\cong L$.
Hence, $AT^x<H$.
In $\PSL_2(11)$, every proper subgroup containing a subgroup of order
$11$ lies in $C_{11}\rtimes C_5$, by Dickson's maximal subgroup list for $\PSL_2(11)$ \cite{Dickson1901}. Hence,
$|A|\mid5$.
As $|A|\mid12$, we conclude that $A=1$.
\qedhere
\end{proof}

\subsection{\texorpdfstring{$p$-subgroups of $\PSU_n(q)$ in characteristic $p$}{p-subgroups of PSUn(q) in characteristic p}}

\begin{lemma}\label{lem:unitary-nondegenerate-product}
Let $q=p^f$ with $p$ prime.
Let $(V,h)$ be a nondegenerate Hermitian space over $\F_{q^2}$.
Let $B\leq\SU(V)$ be an abelian $p'$-subgroup such that every
$B$-invariant subspace is nondegenerate. If $P\leq\SU(V)$ is a
$p$-subgroup and $PB$ is a subgroup, then $B$ is normal in $PB$.
\end{lemma}
\begin{proof}
Since $P$ and $B$ are nilpotent, $K=PB$ is soluble by the Kegel--Wielandt theorem \cite{Kegel1961}.
Let $U=O_p(K)$ and $V^U=\{v\in V:uv=v\text{ for every }u\in U\}$.
If $U\ne1$, then $V^U\ne0$ in characteristic $p$.
As $U\lhd K$, the subspace $V^U$ is $B$-invariant and nondegenerate.
The complement $(V^U)^\perp$ is $U$-invariant.
Suppose that $(V^U)^\perp\ne0$.
We have a nonzero $v\in(V^U)^\perp$ with $uv=v$ for every $u\in U$.
Then, $v\in(V^U)^\perp\cap V^U=0$, a contradiction.
Therefore, $V^U=V$.
Because $U\leq\SU(V)$ acts faithfully, we have $U=1$.

Since $O_p(K)=1$, the Fitting subgroup $F(K)$ is a $p'$-group.
Since $F(K)\unlhd K$ and $B$ is a Hall $p'$-subgroup, we have $F(K)\leq B$.
As $B$ is abelian, we have $B\leq C_K(F(K))$.
Because $K$ is soluble, we obtain from \cite{Isaacs2008} that $C_K(F(K))\leq F(K)$.
Consequently, $B\leq C_K(F(K))\leq F(K)\leq B$.
Therefore, $B=F(K)$ is normal in $K$.
\end{proof}

\begin{proposition}\label{prop:unitary-p-subgroups}
Let $q=p^f$ with $p$ prime.
Suppose that $G=\PSU_n(q)$ is nonabelian simple and $n\geq3$.
Then, no nontrivial $p$-subgroup is ordinary $G$-permutable.
\end{proposition}
\begin{proof}
For ease of reading, we divide the proof of the proposition into the following steps.

The case $n=3$ follows from Theorem \ref{thm:u3}.
Suppose that $n\geq4$, and let $A$ be a nontrivial ordinary
$G$-permutable $p$-subgroup.
Let $\widetilde G=\SU_n(q)$ and $Z=Z(\widetilde G)$.
Let $\pi:\widetilde G\to G$ be the quotient map.
Since $|Z|=(n,q+1)$ is prime to $p$, we have $\pi^{-1}(A)=ZP$ for a $p$-subgroup $P\leq\widetilde G$.
The restriction $\pi|_P:P\to A$ is an isomorphism.

\begin{itemize}[leftmargin=4em,labelsep=.5em,topsep=.2\baselineskip,partopsep=0pt,parsep=0pt,itemsep=0pt]
\item[{\bf Step 1.} ] If $(n,p)\ne(4,3)$, then $A=1$.
\end{itemize}
We choose a maximal torus $T\leq\widetilde G$ with odd partition $(n_1,\ldots,n_t)$.
Next, let $V=\F_{q^2}^n=V_1\perp\cdots\perp V_t$ be the decomposition for $T$.
We have $\dim_{\F_{q^2}}V_i=n_i$.
If $t\geq2$, then each $V_i$ is an irreducible nondegenerate $T$-module.
Let $\pi_i:T\to\GU(V_i)$ satisfy $\pi_i(g)=g|_{V_i}$.
For $j\ne i$, we adjust the determinant on $V_j$ to realize every torus multiplier on $V_i$.
Therefore, $\pi_i(T)$ is cyclic of order $q^{n_i}+1$.
The summands $V_i$ are mutually nonisomorphic as $T$-modules.
When $n_i=n_j=1$, the other torus factors separate the characters on $V_i$ and $V_j$.
Since $T$ has order prime to $p$ and the summands are mutually nonisomorphic, every $T$-invariant subspace is a sum of nondegenerate summands.

Suppose that $n\geq5$ is odd.
For the partition $(n)$, we have $|T|=(q^n+1)/(q+1)$.
A constituent of smaller degree would have degree $e\mid n$, with $e\leq n/3$.
The eigenvalue order on such a constituent divides $q^{2e}-1$, whereas
$\frac{q^n+1}{q+1}>q^{2n/3}-1$.
Therefore, the torus for $(n)$ acts irreducibly.
For each torus $T$ chosen below, we use Lemma \ref{lem:unitary-nondegenerate-product}.

For a partition with $m_i$ parts of size $i$, we have
$|N_{\widetilde G}(T)/T|=\prod_i i^{m_i}m_i!$
by \cite[Lemma 4.1 and Theorem 4.10]{Baykalov2024}.
We choose the partition in the following table.
$$
\begin{array}{|c|c|}
\hline
\text{condition}&\text{partition}\\ \hline
n\text{ odd},\ p\nmid n&(n)\\
n\text{ odd},\ p\mid n&(n-2,1,1)\\
n\text{ even},\ p\nmid n-1&(n-1,1)\\
n\text{ even},\ p\mid n-1,\ p\geq5&(n-3,3)\\
n\text{ even},\ p=3,\ 3\mid n-1,\ n\geq10&(n-5,5).\\ \hline
\end{array}
$$
We next claim that each quotient order is prime to $p$.
In the fourth row, the partition for $n=6$ is $(3,3)$.
Then, we have $p=5$ and $|N_{\widetilde G}(T)/T|=18$.
In the fifth row, we have $|N_{\widetilde G}(T)/T|=50$ for $n=10$.
All other rows have distinct larger parts, so the claim
follows immediately from the displayed formula.

Since $A$ is ordinary $G$-permutable, we choose $x\in G$ such that $A(T/Z)^x$ is a subgroup.
We choose a lift $\widetilde x\in\widetilde G$ of $x$.
As $Z\leq T$, the inverse image of $A(T/Z)^x$ is $PT^{\widetilde x}$.
By Lemma \ref{lem:unitary-nondegenerate-product}, we have $P\leq N_{\widetilde G}(T^{\widetilde x})$.
Both $|T|$ and $|N_{\widetilde G}(T)/T|$ are prime to $p$.
Because $P\ne1$, the $2$ $p'$-orders contradict $P\leq N_{\widetilde G}(T^{\widetilde x})$.

For $(n,p)\ne(4,3)$, we have proved that $A$ cannot be nontrivial.
The remaining cases are $n=4$ with $p=3$.
Now, we take $q=3^f$.
\begin{itemize}[leftmargin=4em,labelsep=.5em,topsep=.2\baselineskip,partopsep=0pt,parsep=0pt,itemsep=0pt]
\item[{\bf Step 2.} ] If $(n,p)=(4,3)$, then $|P|=3$ and a generator of $P$ has Jordan type $(3,1)$.
\end{itemize}
For the partition $(3,1)$, we take a torus $T$ with $|T|=q^3+1$ and $|N_{\widetilde G}(T)/T|=3$.
Since $A$ is ordinary $G$-permutable, choose $x\in G$ such that $A(T/Z)^x$ is a subgroup for this torus $T$.
Choose a lift $\widetilde x\in\widetilde G$ of $x$.
By Lemma \ref{lem:unitary-nondegenerate-product}, we have $P\leq N_{\widetilde G}(T^{\widetilde x})$.
Since $3\nmid|T|$, we have $|P|=3$.
Let $V=U\perp L$ be the decomposition for $T^{\widetilde x}$.
We have $\dim_{\F_{q^2}}U=3$ and $\dim_{\F_{q^2}}L=1$.
Now, let $u$ generate $P$.
The element $u$ cyclically permutes the $3$ eigenlines on $U$ and fixes $L$ pointwise.
As $u^3=1$ in characteristic $3$, the Jordan type of $u$ is $(3,1)$.

We choose now a cyclic torus $S$ of type $(4)$ and order
$(q-1)(q^2+1)$ in $\widetilde G$.
Since $A$ is ordinary $G$-permutable, we choose $y\in G$ such that $A(S/Z)^y$ is a subgroup.
We choose a lift $\widetilde y\in\widetilde G$ of $y$.
The inverse image of $A(S/Z)^y$ is $K=PS^{\widetilde y}$.
As $|P|=3$ and $3\nmid|S|$, all Sylow subgroups of $K$ are cyclic.
By \cite[Section 2.1]{DietrichLow2020}, we have a cyclic derived subgroup $D=K'$ and a cyclic quotient $K/D$.

\begin{itemize}[leftmargin=4em,labelsep=.5em,topsep=.2\baselineskip,partopsep=0pt,parsep=0pt,itemsep=0pt]
\item[{\bf Step 3.} ] If $(n,p)=(4,3)$, then $A=1$.
\end{itemize}
Firstly, suppose that $3\mid|D|$. Since $D$ is cyclic and normal in $K$, its unique subgroup of order $3$ is normal in $K$.
This subgroup equals $P$, so $P\unlhd K$.
Let $b=(q^2+1)/2$.
Because $q$ is odd, the integer $b$ is odd and greater than $1$.
Let $B_b\leq S^{\widetilde y}$ have order $b$.
As $|\Aut(P)|=2$ and $|B_b|$ is odd, we have $B_b\leq C_{\widetilde G}(u)$.

Every semisimple unitary element $h$ centralizing a unipotent
element of Jordan type $(3,1)$ has order dividing $q+1$.
To prove this, we decompose the module into the eigenspaces of $h$
over the algebraic closure. The unique Jordan block of size $3$
lies in an eigenspace of dimension at least $3$.
Since $h$ is unitary, the eigenvalues $\lambda$ and $\lambda^{-q}$ have equal multiplicities.
As the eigenspace containing the Jordan block of size $3$ has dimension at least $3$, we have $\lambda^{-q}=\lambda$.
Hence, $\lambda^{q+1}=1$.
Any other eigenspace has dimension $1$.
Its eigenvalue $\nu$ satisfies $\nu^{-q}=\nu$, so $\nu^{q+1}=1$.
Thus, $h^{q+1}=1$.
A generator of $B_b$ is semisimple and centralizes $u$.
Therefore, its order $b>1$ would divide $q+1$, contrary to $(b,q+1)=1$.

Suppose that $3\nmid|D|$. Since $D$ is normal and
$S^{\widetilde y}$ is a Hall $3'$-subgroup, we have
$D\leq S^{\widetilde y}$. We claim that every element of order
$3$ normalizing a cyclic subgroup $C\leq S^{\widetilde y}$ centralizes $C$.

For a generator $c$ of $C$, the eigenvalues are $\lambda,\lambda^{-q},\lambda^{q^2},\lambda^{-q^3}$ with $\lambda\in\F_{q^4}^*$.
Under $\alpha\mapsto\alpha^{-q}$, the orbit length belongs to $\{1,2,4\}$.
For lengths $1,2$, an element of order $3$ fixes each eigenspace and centralizes $c$.
For length $4$, let $\overline F$ be an algebraic closure of $\F_{q^2}$.
On $\overline V=\overline F\otimes_{\F_{q^2}}V$, let $\sigma(a\otimes v)=a^{q^2}\otimes v$.
The map $\sigma$ permutes the $4$ eigenlines of $c$ in $2$ transpositions.
Moreover, every $\F_{q^2}$-linear matrix normalizing $C$ commutes with $\sigma$ on $\overline V$.
Since the centralizer of this permutation in $S_4$ has order $8$, an element of order $3$ fixes all $4$ eigenlines.
Therefore, the element centralizes $c$.

Consequently, $P$ centralizes $D$. The abelian group
$S^{\widetilde y}$ also centralizes $D$, so $D\leq Z(K)$.
Since $K/D$ is cyclic and $D\leq Z(K)$, the group $K$ is abelian.
Since $K$ is abelian, $P$ centralizes $S^{\widetilde y}$.
A generator of $S$ has $4$ distinct eigenvalues, so $C_{\widetilde G}(S)=S$.
Hence, $P\leq S^{\widetilde y}$, a $3'$-group.
As $P\ne1$, we have a contradiction.
\end{proof}

\subsection{\texorpdfstring{The groups $\PSU_{2^k}(q)$}{The groups PSU(2 to the power k)(q)}}

For a finite-dimensional $FH$-module $V$, let $\overline F$ be an algebraic closure of $F$.
The action on $\overline F\otimes_F V$ is $h(a\otimes v)=a\otimes hv$.
We call $V$ absolutely irreducible if $\overline F\otimes_F V$ is irreducible.
Choose an $H$-invariant series $0=V_0<V_1<\cdots<V_s=V$ whose quotients $V_i/V_{i-1}$ are irreducible.
We call these quotients the composition factors of $V$.

\begin{lemma}\label{lem:odd-power-two-modules}
Let $H$ be a group of odd order, and let $E$ be a finite field.
Let $V$ be an irreducible $EH$-module with $\dim_E V=2^a$ for some integer $a\geq0$.
Let $\rho:H\to\GL_E(V)$ be the representation.
Then, $\rho(H)$ is abelian.
\end{lemma}
\begin{proof}
The group $H$ is soluble by the Feit--Thompson theorem
\cite{FeitThompson1963}. Let $p$ be the characteristic of $E$.
By the Fong--Swan theorem \cite{Fong1961,Isaacs2006}, every irreducible Brauer character
of $H$ lifts to an ordinary irreducible character.
Every lifted character has odd degree.

Let $\overline E$ be an algebraic closure of $E$, and let $\overline V=\overline E\otimes_E V$.
For $h\in H$, the action on $\overline V$ is $h(a\otimes v)=a\otimes\rho(h)v$.
The absolutely irreducible constituents of $\overline V$ form a Galois orbit.
Each constituent has multiplicity $1$, and their common degree divides $\dim_E V$.
Since the common degree is odd and divides $2^a$, every constituent has degree $1$.
Thus, $\rho(H)$ is abelian.
\end{proof}

\begin{proposition}\label{prop:unitary-power-two-odd}
Let $n=2^k$ with $k\geq2$, and let $q$ be a prime power.
Let $G=\PSU_n(q)$.
Then, no nontrivial subgroup of odd order is ordinary $G$-permutable.
\end{proposition}
\begin{proof}
Let $q=p^f$ with $p$ prime.
Suppose that $A$ is an ordinary $G$-permutable subgroup of odd order.
Let $\widetilde G=\SU_n(q)$ and let $\pi:\widetilde G\to G$ be the quotient map.
Since $n$ is a power of $2$, the central kernel $Z(\widetilde G)$ is a $2$-group.
Let $J\leq G$ have odd order.
By Schur--Zassenhaus, we choose a complement $\widehat J$ to $Z(\widetilde G)$ in $\pi^{-1}(J)$.
Any $2$ complements are conjugate by an element of $Z(\widetilde G)$.
Since $Z(\widetilde G)$ is central, $\widehat J$ is unique.
The restriction $\pi|_{\widehat J}:\widehat J\to J$ is an isomorphism.
For $B\leq G$ of odd order, every subgroup $AB^x$ with $x\in G$ also has odd order.
We identify each subgroup of odd order with its unique lift $\widehat J$.
If $AB^x$ is a subgroup, then its unique lift contains the lifts of $A$ and $B^x$.
By the product order formula, the product of those lifts equals the unique lift of $AB^x$.
For $x,y\in G$, we choose $\widetilde x,\widetilde y\in\widetilde G$ with $\pi(\widetilde x)=x$ and $\pi(\widetilde y)=y$.
For the subgroups $S,B\leq\widetilde G$ used below, we use $S^x=\widetilde x^{-1}S\widetilde x$ and $B^y=\widetilde y^{-1}B\widetilde y$.
As $\ker\pi=Z(\widetilde G)$, the choices of $\widetilde x,\widetilde y$ do not affect $S^x,B^y$.

Let $d=n/2$.
Next, we take a cyclic torus $T$ of type $(n)$ with $|T|=M=(q^n-1)/(q+1)$.
Let $S$ be the Hall subgroup of $T$ of odd order.
Let $a=|S|$.
Then, $(q^d+1)_{2'}>1$ and $(q^d+1)_{2'}\mid a$.
The notation $m_{2'}$ denotes the largest odd divisor of $m$.
Since $d$ is a power of $2$, we have $\operatorname{ord}_{(q^d+1)_{2'}}(q^2)=d$.
We also have $\operatorname{ord}_{(q^d+1)_{2'}}(-q)=n$.
Next, we choose a generator $s\in S$.
Let $\lambda\in\F_{q^n}^{\,*}$ be an eigenvalue of $s$ with $|\lambda|=a$.
Let $\mathcal O_0=\{\lambda^{q^{2j}}:0\leq j<d\}$ and $\mathcal O_1=\{\alpha^{-q}:\alpha\in\mathcal O_0\}$.
As $(q^d+1)_{2'}\mid a$ and $a\mid M$, we have $\ord_a(q^2)=d$ and $\ord_a(-q)=n$.
Hence, the sets $\mathcal O_0,\mathcal O_1$ are disjoint orbits under $\alpha\mapsto\alpha^{q^2}$.
In particular, $|\mathcal O_0|=|\mathcal O_1|=d$.
Thus, $S$ has $2$ nonisomorphic irreducible $\F_{q^2}$-modules of dimension $d$.
As unitary duality $\alpha\mapsto\alpha^{-q}$ sends $\mathcal O_i$ to $\mathcal O_{1-i}$ for $i=0,1$, both modules are totally isotropic.

Since $A$ is ordinary $G$-permutable, we have a subgroup $K_1=AS^x$ for some $x\in G$.
We use $V=\F_{q^2}^n$ with the natural Hermitian form $h:V\times V\to\F_{q^2}$.
Suppose that $V$ is irreducible under $K_1$.
We obtain from Lemma \ref{lem:odd-power-two-modules} that $K_1$ is abelian.
Then, every element of $K_1$ would centralize $S^x$ and stabilize both irreducible $S^x$-submodules of $V$.
This contradicts irreducibility.
Hence, we choose a nonzero proper $K_1$-invariant subspace $W$.
As $S^x\leq K_1$, the space $W$ is one of the $2$ irreducible $S^x$-submodules.
Thus, $W$ is totally isotropic and $\dim_{\F_{q^2}}W=d$.
Let $\rho_W:K_1\to\GL(W)$ satisfy $\rho_W(k)=k|_W$.
Because $\rho_W(S^x)$ is irreducible, $\rho_W(K_1)$ is irreducible.
By Lemma \ref{lem:odd-power-two-modules}, we obtain that $\rho_W(K_1)$ is abelian.
Hence, $\rho_W(K_1)$ centralizes $\rho_W(S^x)$, whose commuting algebra is $\F_{q^n}$.
We identify $W$ with the $\F_{q^2}$-vector space $\F_{q^n}$.
Thus, every $g\in\rho_W(K_1)$ acts by $v\mapsto\lambda v$ for some $\lambda\in\F_{q^n}^{\,*}$.
For $k\in K_1$, let $g=\rho_W(k)$.
As $h(kv,kw)=h(v,w)$ for $v,w\in V$, the diagonal blocks of $k$ are $g$ and $(g^{-1})^{q\mathsf T}$ in dual bases.
Because the $2$ diagonal blocks have combined determinant $1$, we have $(\det g)^{q-1}=1$.
We also have $\det g=\lambda^{(q^n-1)/(q^2-1)}$.
Substituting the displayed determinant in $(\det g)^{q-1}=1$, we obtain $\lambda^M=1$.
Next, in the dual bases above, every element of $\ker\rho_W$ has matrix $u_C=\begin{pmatrix}I_d&C\\0&I_d\end{pmatrix}$.
We have $C\in M_d(\F_{q^2})$ and $C+C^{q\mathsf T}=0$.
As $u_Cu_D=u_{C+D}$, the group $\ker\rho_W$ is a $p$-group.

Let $c\ne1$ be a $p'$-element of $A$.
Since $\ker\rho_W$ is a $p$-group, the restriction $\rho_W(c)=c|_W$ has order $|c|$.
The eigenvalues of $c$ are
$\lambda,\lambda^{-q},\ldots,\lambda^{(-q)^{n-1}}$,
where $\lambda$ has order $|c|$.
Since $|\lambda|=|c|>1$ and $c$ is a $p'$-element, no eigenvalue equals $1$.
Moreover, $\lambda^M=1$ and $|c|$ is odd, so $|c|\mid M_{2'}$.
As $M\equiv-n\pmod{q+1}$ and $n$ is a power of $2$, we have $(M_{2'},q+1)=1$.

Let $r=n-1$.
Next, we choose the Hall subgroup $B$ of odd order in a torus of type $(r,1)$.
Let $V=U\perp L$ be the torus decomposition with $\dim_{\F_{q^2}}U=r$ and $\dim_{\F_{q^2}}L=1$.
The order $b=|B|=(q^r+1)_{2'}$ satisfies $b\geq(q^r+1)/(q+1)$.
For $r\geq7$, the right side exceeds $q^{2r/3}-1$.
Since the eigenvalue degree divides $r$, any proper degree is at most $r/3$.
Hence, $B$ acts irreducibly on $U$ for $r\geq7$.
For $r=3$, we have $b>q+1$ and $(q^3+1,q^2-1)=q+1$.
Hence, $B$ also acts irreducibly on $U$ for $r=3$.
The $B$-modules $U$ and $L$ are nondegenerate and nonisomorphic.

Since $A$ is ordinary $G$-permutable, we have a subgroup $K_2=AB^y$ for some $y\in G$.
Let $U_y=\widetilde y^{-1}U$ and $L_y=\widetilde y^{-1}L$.
Since $B^y\leq K_2$, every $K_2$-invariant subspace is $B^y$-invariant and nondegenerate.
Let $X\leq V$ be any $K_2$-invariant subspace.
As $K_2\leq\SU(V)$, we have $gX^\perp=X^\perp$ for every $g\in K_2$.
Hence, $V$ is completely reducible under $K_2$, even when $p\mid|A|$.
Suppose that $V$ is irreducible under $K_2$.
We obtain from Lemma \ref{lem:odd-power-two-modules} that $K_2$ is abelian.
Then, every element of $K_2$ would centralize $B^y$ and stabilize its unique line $L_y$.
This contradicts irreducibility.
Therefore, $V$ is reducible under $K_2$.
The only nonzero proper $B^y$-submodules are $U_y$ and $L_y$.
We choose a nonzero proper $K_2$-invariant subspace.
Because $U_y^\perp=L_y$, we have $gL_y=L_y$ for every $g\in K_2$.

Let $c\in A$ be a $p'$-element.
The order of the scalar $c|_{L_y}$ divides $q+1$.
From $|c|\mid M_{2'}$ and $(M_{2'},q+1)=1$, we obtain $(|c|,q+1)=1$.
Therefore, $c|_{L_y}=I$.
As every nontrivial $p'$-element of $A$ has no eigenvalue $1$, we conclude that $c=1$.
By Cauchy's theorem, $A$ is a $p$-group.
By Proposition \ref{prop:unitary-p-subgroups}, we conclude that $A=1$.
\end{proof}

\begin{proposition}\label{prop:unitary-power-two-cyclic-sylow}
Let $n=2^k$ with $k\geq2$, and let $q$ be a power of $2$.
If $A\leq G=\PSU_n(q)$ is nontrivial and ordinary
$G$-permutable, then its Sylow $2$-subgroups are noncyclic.
\end{proposition}
\begin{proof}
Since $(n,q+1)=1$, we have $G=\SU_n(q)$.
By Proposition \ref{prop:unitary-power-two-odd}, $A$ has
even order. Suppose that the Sylow $2$-subgroups of $A$ are cyclic.
Let $T$ be a cyclic torus of type $(n)$, with odd order $(q^n-1)/(q+1)$.
As $A$ is ordinary $G$-permutable, we have a subgroup $K_1=AT^x$ for some $x\in G$.
Because $|T|$ is odd, the Sylow $2$-subgroups of $K_1$ are cyclic.

For a cyclic Sylow $2$-subgroup $P$ of $K_1$, the quotient
$N_{K_1}(P)/C_{K_1}(P)$ has both odd order and $2$-power order.
Thus, $P\leq Z(N_{K_1}(P))$.
By Burnside's theorem on normal complements \cite{Isaacs2008}, we have a normal Hall subgroup $H\unlhd K_1$ of odd order.
Then, $T^x\leq H$.

The $2$ irreducible $T^x$-submodules have dimension $n/2$.
Hence, every composition factor of the natural $\F_{q^2}K_1$-module has dimension $n/2$ or $n$.
Let $V=\F_{q^2}^n$ and $V_{\mathrm{alg}}=\overline{\F}_{q^2}\otimes_{\F_{q^2}}V$.
For $g\in K_1$, the action on $V_{\mathrm{alg}}$ is $g(a\otimes v)=a\otimes gv$.
Let $X$ be an irreducible constituent of the $\overline{\F}_{q^2}K_1$-module $V_{\mathrm{alg}}$.
By the preceding $T^x$-module decomposition, the irreducible constituents over $\F_{q^2}$ have dimensions $n/2$ or $n$.
Their absolutely irreducible constituents have degrees dividing those dimensions.
Hence, $\dim X$ is a power of $2$.
By Clifford's theorem \cite{Isaacs2006}, every irreducible constituent of $X\!\downarrow_H$
has degree dividing $\dim X$.
The notation $X\!\downarrow_H$ denotes the restriction of the $K_1$-representation on $X$ to $H$.
Since $H$ has odd order and the characteristic is $2$, every constituent degree is odd.
The degree also divides a power of $2$.
Hence, every constituent has degree $1$.
Since $|H|$ is odd and the characteristic is $2$, the natural $H$-module is completely reducible.
Its constituents all have degree $1$, so the faithful image of $H$ is abelian.
Consequently,
$T^x\leq H\leq C_G(T^x)=T^x$.
Therefore, $H=T^x\unlhd K_1$.
Hence, $A\leq N_G(T^x)$.

The torus normalizer quotients are
$N_G(T)/T\cong C_n, N_G(R)/R\cong C_{n-1}$,
where $R$ is a cyclic torus of type $(n-1,1)$
\cite[Lemma 4.1 and Theorem 4.10]{Baykalov2024}.
An involution $t\in A$ has image of order $2$ in $N_G(T^x)/T^x\cong C_n$.
Hence, $t$ permutes the $n$ eigenlines in $n/2$ transpositions.
Since $t^2=1$ in characteristic $2$, the Jordan type of $t$ is $(2^{n/2})$.
Therefore,
$\Fix(t)=\operatorname{im}(t-1)$.
Let $h:V\times V\to\F_{q^2}$ be the Hermitian form.
As $t^2=I$, we have $(t-I)^2=0$ in characteristic $2$.
Because $h(tv,tw)=h(v,w)$ and $t^{-1}=t$, we obtain $h((t-I)v,w)=h(v,(t-I)w)$.
Thus, $h((t-I)v,(t-I)w)=h(v,(t-I)^2w)=0$ for $v,w\in V$.
Therefore, $\operatorname{im}(t-I)$ is totally isotropic.
In particular, $t$ fixes no nondegenerate line pointwise.

Since $A$ is ordinary $G$-permutable, we also have a subgroup $K_2=AR^y$ for some $y\in G$.
The torus $R^y$ has irreducible nondegenerate summands
of dimensions $n-1$ and $1$.
Every $R^y$-invariant subspace is nondegenerate, so the natural $K_2$-module is completely reducible.
Let $L_y$ be the nondegenerate summand of dimension $1$ in the $R^y$-module $V$.
If $V$ is reducible under $K_2$, then we have $gL_y=L_y$ for every $g\in K_2$.
As $q+1$ is odd, the scalar action of $t$ on $L_y$ is trivial.
This contradicts $\Fix(t)$ being totally isotropic.

Suppose that the natural $K_2$-module $V=\F_{q^2}^n$ is irreducible.
We have $\dim_{\F_{q^2}}V=n=2^k$.
Since the Sylow $2$-subgroups of $K_2$ are cyclic, we have a normal Hall subgroup $H_2$ of odd order by Burnside's theorem.
We use $V_{\mathrm{alg}}=\overline{\F}_{q^2}\otimes_{\F_{q^2}}V$ with $g(a\otimes v)=a\otimes gv$ for $g\in K_2$.
The absolutely irreducible constituents of $V_{\mathrm{alg}}$ have degrees dividing the power of $2$ given by $\dim_{\F_{q^2}}V$.
By Clifford's theorem and the odd order of $H_2$, every irreducible constituent of $V_{\mathrm{alg}}\!\downarrow_{H_2}$ has degree $1$.
As $|H_2|$ is odd, the faithful $H_2$-module is completely reducible.
Hence, $H_2$ is abelian.
Because $R^y\leq H_2$ and $C_G(R^y)=R^y$, we have $H_2=R^y$.
Therefore, $K_2\leq N_G(R^y)$.
But $|R|=q^{n-1}+1$ and $|N_G(R)/R|=n-1$ are odd.
This contradicts the even order of $A$.
\end{proof}

\begin{theorem}\label{thm:unitary-four-even}
Let $q=2^f$, where $f\geq1$. Then, $G=\PSU_4(q)$ has no
proper nontrivial ordinary $G$-permutable subgroup.
\end{theorem}
\begin{proof}
The case $q=2$ follows from Proposition \ref{prop:finite-tables}.
Suppose that $q\geq4$ and that $A<G$ is ordinary
$G$-permutable. Since $(4,q+1)=1$, we have $G=\SU_4(q)$.

Let $M$ stabilize a decomposition into $4$ orthogonal
nondegenerate lines. Then,
$M\cong C_{q+1}^3\rtimes S_4$ and $|M|=24(q+1)^3$.
The soluble subgroup $M$ is maximal \cite[Table 8.10]{BHR}.
Suppose that $G=MY$ with $Y<G$.
By \cite[Proposition 4.1]{LiXia}, $Y$ is not soluble for $G=\PSU_4(q)$ with $q\geq4$.
By \cite[Theorem 1.1 and Tables 1.1--1.2]{LiXia}, a generic soluble factor $X<G$ in $G=XY$ has order dividing $2q^4(q-1)(q^2+1)$.
The odd part of this bound is coprime to $q+1$.
The exceptional soluble factor $X<G$ for $q=8$ has odd order, whereas $M$ has even order.
Therefore, $MY\ne G$ for every $Y<G$.

Since $A$ is ordinary $G$-permutable, we have a subgroup $AM^x$ for some $x\in G$.
As $MY\ne G$ for every $Y<G$, we have $AM^x<G$.
Because $M^x$ is maximal, we have $AM^x=M^x$ and $A\leq M^x$.

We take a cyclic torus $R$ of type $(3,1)$ with $|R|=q^3+1$.
Since $A$ is ordinary $G$-permutable, we have a subgroup $AR^y$ for some $y\in G$.
We have $|A||R|\leq24(q+1)^3(q^3+1)<|G|=q^6(q^4-1)(q^3+1)(q^2-1)$.
Therefore, $AR^y<G$.
We choose a primitive prime divisor $\ell$ of $2^{6f}-1$.
As $f\geq2$, the prime $\ell$ exists by Zsigmondy's theorem.
Then, $\ord_\ell(2)=6f$ and $\ord_\ell(q)=6$.
Consequently, $R$ contains an element of order $\ell$.
By \cite[Lemma 2.6]{GuralnickMalle2012}, every maximal
subgroup containing this element stabilizes a nondegenerate
line. None of the exceptional pairs in the lemma has even $q\geq4$.
Therefore, $AR^y$ stabilizes a nondegenerate line $L$.
Because $A\leq AR^y$, we have $aL=L$ for every $a\in A$.

Let $P$ be a Sylow $2$-subgroup of $A$.
Next, let $V=L_1\perp\cdots\perp L_4$ be the line decomposition stabilized by $M^x$.
We have $\dim_{\F_{q^2}}L_i=1$ for $1\leq i\leq4$.
Let $\theta:M^x\to S_4$ satisfy $gL_i=L_{\theta(g)(i)}$.
Since $\ker\theta$ is the diagonal subgroup of odd order, $\theta|_P:P\to S_4$ is injective.
Every $2$-subgroup of $S_4$ of order at least $4$ contains a double transposition.
If a $2$-subgroup $Q\leq S_4$ contains a $4$-cycle $c$, then $c^2$ is a double transposition.
Otherwise, all nonidentity elements of $Q$ are involutions, so $Q$ is elementary abelian.
Suppose that none is a double transposition.
$2$ distinct commuting transpositions have disjoint supports.
Their product is a double transposition, a contradiction.

Suppose that $|P|\geq4$.
We choose $t\in P$ such that $\theta(t)$ is a double transposition.
Since $\theta|_P$ is injective, we have $t^2=I$ and $t\ne I$.
We number the lines so that $tL_1=L_2$ and $tL_3=L_4$.
As $t^2=I$, we have $tL_2=L_1$ and $tL_4=L_3$.
Hence, the Jordan type of $t$ in characteristic $2$ is $(2,2)$.
Therefore, $\Fix(t)=\operatorname{im}(t-I)$.
Let $h:V\times V\to\F_{q^2}$ be the Hermitian form on $V=\F_{q^2}^4$.
Because $t^2=I$ and $\operatorname{char}\F_{q^2}=2$, we have $(t-I)^2=0$.
As $h(tv,tw)=h(v,w)$ and $t^{-1}=t$, we obtain $h((t-I)v,w)=h(v,(t-I)w)$ for $v,w\in V$.
Thus, $h((t-I)v,(t-I)w)=h(v,(t-I)^2w)=0$ for $v,w\in V$.
Because $\Fix(t)=\operatorname{im}(t-I)$, the subspace $\Fix(t)$ is totally isotropic.
But $t$ stabilizes the nondegenerate line $L$.
The scalar $t|_L$ has order dividing the odd number $q+1$.
Hence, $t|_L=I$.
This places the nondegenerate line $L$ in the totally isotropic space $\Fix(t)$, a contradiction.

Since $|P|\leq2$, we conclude that $A=1$ by Propositions \ref{prop:unitary-power-two-odd} and \ref{prop:unitary-power-two-cyclic-sylow}.
\end{proof}

\begin{theorem}\label{thm:unitary-four}
For every prime power $q\geq2$, the group $G=\PSU_4(q)$
has no proper nontrivial ordinary $G$-permutable subgroup.
\end{theorem}
\begin{proof}
For even $q$, we use Theorem \ref{thm:unitary-four-even}.
For $q=3$, we use Proposition \ref{prop:finite-tables}.
Let $q=p^f\geq5$ be odd, and let $d=(4,q+1)$.
Suppose that $A<G$ is nontrivial and ordinary $G$-permutable.
By Proposition \ref{prop:unitary-power-two-odd}, $|A|$ is even.

Let $M$ be the monomial maximal subgroup.
The group $M$ is soluble and has order $24(q+1)^3/d$
\cite[Table 8.10]{BHR}.
Suppose that $G=MY$ with $Y<G$.
By \cite[Proposition 4.1]{LiXia}, $Y$ is not soluble.
By \cite[Theorem 1.1 and Tables 1.1--1.2]{LiXia}, the generic soluble factor $M$ would have order dividing
$\frac{2q^4(q-1)(q^2+1)}{d}$.
If $q+1$ has an odd prime divisor $r$, then $r\mid|M|$ but $r\nmid2q^4(q-1)(q^2+1)/d$.
If $q+1$ is a power of $2$, then we have $(2q^4(q-1)(q^2+1)/d)_2=2$.
But $|M|_2>2$.
For odd $q\geq5$, the cited list has no exceptional equality $G=XY$ with $X,Y<G$ and $X$ soluble.
Therefore, $MY\ne G$ for every $Y<G$.
By Lemma \ref{lem:max}, we have $A\leq M^x$ for some $x\in G$.
In particular, $A$ is soluble.

We take a cyclic torus $R$ of type $(3,1)$ with $|R|=(q^3+1)/d$.
Let $B$ be the Hall subgroup of $R$ of odd order.
Since $A$ is ordinary $G$-permutable, we have a subgroup $K_2=AB^y$ for some $y\in G$.
As $|A||B|\leq\frac{24(q+1)^3(q^3+1)}{d^2}<|G|$, we have $K_2<G$.
A primitive prime divisor of $p^{6f}-1$ belongs to $B$.
Because $K_2<G$, we choose a maximal subgroup $H<G$ with $K_2\leq H$.
By \cite[Lemma 2.6]{GuralnickMalle2012}, we have $2$ possibilities for $H$.
Either $H$ stabilizes a nondegenerate line $L\leq V=\F_{q^2}^4$, or $q=5$ and $H\cong A_7$.
The group $A_7$ has no element of order $9$.
As $B^y$ is cyclic of order $63$ for $q=5$, we cannot have both $q=5$ and $H\cong A_7$.
Thus, $K_2\leq H\leq N_1=\Stab_G(L)$.

Let $\rho_L:\GU(L^\perp)\to\PGU(L^\perp)$ be the quotient map.
For $\overline g\in\Stab_G(L)$, choose a lift $g\in\SU_4(q)$.
Let $\pi_L:\Stab_G(L)\to\PGU(L^\perp)$ satisfy $\pi_L(\overline g)=\rho_L(g|_{L^\perp})$.
Since $2$ lifts differ by $\lambda I$ with $\lambda\in\F_{q^2}^*$ and $\lambda^{q+1}=1$, the value of $\pi_L(\overline g)$ is independent of the lift.
The kernel is $\ker\pi_L=\{\overline g:g|_{L^\perp}=\lambda I_{L^\perp}\text{ for some }\lambda\in\F_{q^2}^*\}$.
Let $J=\rho_L^{-1}(\pi_L(K_2))\leq\GU(L^\perp)\cong\GU_3(q)$.
From the kernel formula and the torus decomposition, we obtain $|\ker\pi_L|=(q+1)/d$ and $\ker\pi_L\leq R^y$.
Moreover, $\pi_L(B^y)$ is a Singer torus of $\PGU(L^\perp)$ with order $q^2-q+1$.
Since $B^y\leq K_2$, the group $J$ contains $F=\rho_L^{-1}(\pi_L(B^y))$.
By the preceding Singer torus calculation, $F$ is cyclic of order $q^3+1$.
Since both quotient kernels and $B^y$ have order prime to $p$, we have $|J|_p=|A|_p\leq3_p<q^3$.
By Lemma \ref{lem:unitary-singer-small-p-part}, we have $J\leq N_{\GU_3(q)}(F)$.
Since $\pi_L(B^y)\leq\pi_L(R^y)$ and both groups have order $q^2-q+1$, we have $\pi_L(B^y)=\pi_L(R^y)$.
Because $J$ normalizes $F$, $\pi_L(K_2)$ normalizes $\rho_L(F)=\pi_L(R^y)$.
As $\ker\pi_L\leq R^y$, we have $\pi_L^{-1}(\pi_L(R^y))=R^y$.
Therefore, $A\leq K_2\leq N_G(R^y)$.
By \cite[Lemma 4.1 and Theorem 4.10]{Baykalov2024}, we have $[N_G(R):R]=3$.
As $R^y$ is cyclic of odd index in $N_G(R^y)$, every Sylow $2$-subgroup of $A$ lies in $R^y$ and is cyclic.
We have $v_2(|R|)=\max\{v_2(q+1)-2,0\}$.
Because $|A|$ is even, we conclude that $q\equiv7\pmod8$.
Since no power of $3$ is $7$ modulo $8$, we have $p\ne3$.
In particular, $p\nmid|M|$.

We take a torus $T$ of type $(4)$ in $G$.
Let $S$ be the Hall subgroup of $T$ of odd order.
Since $A$ is ordinary $G$-permutable, we have a subgroup $K_1=AS^z$ for some $z\in G$.
As the Sylow $2$-subgroups of $K_1$ are cyclic, we have a normal Hall subgroup $H$ of odd order by Burnside's theorem.
Let $\pi:\SU_4(q)\to G$ be the quotient map.
We take the unique lift $\widetilde H\leq\SU_4(q)$ of odd order.
Since $H\unlhd K_1$ and its lift of odd order is unique, $\widetilde H$ is normal in the inverse image $\widetilde K_1$ of $K_1$.
The groups $K_1,H,\widetilde K_1,\widetilde H$ have order coprime to $p$.

Let $\widetilde S\leq\SU_4(q)$ be the unique lift of $S$ of odd order.
Next, choose $\widetilde z\in\SU_4(q)$ with $\pi(\widetilde z)=z$.
We use $\widetilde S^z=\widetilde z^{-1}\widetilde S\widetilde z$.
Since $\ker\pi=Z(\SU_4(q))$, the choice of $\widetilde z$ does not affect $\widetilde S^z$.
Let $V=\F_{q^2}^4=U_1\oplus U_2$ be the decomposition into the $2$ nonisomorphic irreducible $\widetilde S^z$-modules.
We have $\dim_{\F_{q^2}}U_i=2$ for $i=1,2$.
Hence, every irreducible $\F_{q^2}\widetilde K_1$-constituent
has dimension $2$ or $4$.
Let $V_{\mathrm{alg}}=\overline{\F}_{q^2}\otimes_{\F_{q^2}}V$.
For $g\in\widetilde K_1$, the action on $V_{\mathrm{alg}}$ is $g(a\otimes v)=a\otimes gv$.
Let $X$ be an irreducible constituent of the $\overline{\F}_{q^2}\widetilde K_1$-module $V_{\mathrm{alg}}$.
By scalar extension over the finite field $\F_{q^2}$, the degree $\dim X$ divides $2$ or $4$.
By Clifford's theorem, we obtain that each irreducible constituent of $X\!\downarrow_{\widetilde H}$ has degree dividing $\dim X$.
The notation $X\!\downarrow_{\widetilde H}$ denotes the restriction of the $\widetilde K_1$-representation on $X$ to $\widetilde H$.
As $\widetilde H$ has odd order, the constituent degrees are odd and hence equal $1$.
Because $p\nmid|\widetilde H|$, the faithful $\widetilde H$-module is completely reducible.
Hence, $\widetilde H$ is abelian.
As $\widetilde S^z\leq\widetilde H$, we have $\widetilde H\leq C_{\SU_4(q)}(\widetilde S^z)$.
The centralizer is the inverse image of $T^z$, of order $(q-1)(q^2+1)$.
Its Hall subgroup of odd order is $\widetilde S^z$.
Hence, $\widetilde H=\widetilde S^z$ and $A\leq N_G(T^z)$.

We claim that $A\cap T^z=1$.
To prove this, let $a\in A\cap T^z$, and choose a lift $\widetilde a\in\SU_4(q)$ with multiplier $\lambda$.
The eigenvalues of $\widetilde a$ are $\lambda^{(-q)^j}$, $0\leq j\leq3$.
Since $aL=L$ and $L$ is nondegenerate, an eigenvalue of $\widetilde a$ has order dividing $q+1$.
Thus, the order of $\lambda$ divides $((q-1)(q^2+1),q+1)=d$.
Hence, $\widetilde a$ is scalar, proving $A\cap T^z=1$.
Let $\nu_z:N_G(T^z)\to N_G(T^z)/T^z\cong C_4$ be the quotient map.
As $\ker(\nu_z|_A)=A\cap T^z=1$, the homomorphism $\nu_z|_A$ is injective.

We take an involution $t\in A$.
Then, we choose a lift $\widetilde t\in\SU_4(q)$ of $t$.
Let $E_1,\ldots,E_4$ be the eigenlines of the inverse image of $T^z$ on $\overline{\F}_{q^2}^{\,4}$.
The image of $t$ in $N_G(T^z)/T^z\cong C_4$ has order $2$.
Hence, we number the lines so that $\widetilde tE_1=E_2$ and $\widetilde tE_3=E_4$.
Since $\widetilde t^{\,2}$ is scalar, we have $\widetilde tE_2=E_1$ and $\widetilde tE_4=E_3$.
Let $\zeta\in\F_{q^2}^*$ satisfy $\widetilde t^2=\zeta I$.
On each of $E_1\oplus E_2$ and $E_3\oplus E_4$, the characteristic polynomial of $\widetilde t$ is $X^2-\zeta$.
As $q$ is odd, the eigenvalue multiplicities of $\widetilde t$ are $(2,2)$.
Because $t$ lies in the cyclic $2$-part of $R^y$, we choose a $2$-element $r_t\in\pi^{-1}(R^y)$ with $\pi(r_t)=t$.
The eigenvalues of $r_t$ are $\mu,\mu,\mu,\mu^{-3}$ with $\mu^{q+1}=1$.
As $t\ne1$ projectively, $r_t$ has eigenvalue multiplicities $(3,1)$.
Because $r_t$ and $\widetilde t$ lift $t$, we have $r_t=\eta\widetilde t$ for some $\eta\in\F_{q^2}^*$ with $\eta^4=\eta^{q+1}=1$.
Scalar multiplication does not change eigenvalue multiplicities.
The multiplicities $(2,2)$ and $(3,1)$ are incompatible, so we have a contradiction.
\end{proof}

\subsection{\texorpdfstring{The groups $\PSU_n(q)$ with odd $n$}{The groups PSUn(q) with odd n}}

\begin{lemma}\label{lem:unitary-order-bound}
Let $n\geq8$ and let $q\geq2$ be a prime power. Then,
$n!(q+1)^{n-1}<q^{n(n-1)/2}$.
\end{lemma}
\begin{proof}
For $q=2$, let $R_n=2^{n(n-1)/2}/(n!3^{n-1})$.
Then, we have $R_8>1$ and $R_{n+1}/R_n=2^n/(3(n+1))>1$ for $n\geq8$.
For fixed $n$, the function $q^{n/2}/(q+1)$ is increasing.
Since this function is increasing, the inequality for $q=2$ implies the inequality for every $q\geq2$.
\end{proof}

\begin{lemma}\label{lem:unitary-singer-small-p-part}
\label{lem:special-unitary-singer-small-p-part}
Let $m\geq3$ be odd and let $q=p^f$ with $p$ prime.
Let $X=\GU_m(q)$ or $\SU_m(q)$.
Let $J\leq X$ and let $F\leq J$ be a Singer subgroup.
Suppose that $|F|=q^m+1$ for $X=\GU_m(q)$ and $|F|=(q^m+1)/(q+1)$ for $X=\SU_m(q)$.
Suppose that $(m,q)\notin\{(3,2),(3,3),(5,2)\}$.
For $X=\SU_m(q)$, also suppose that $(m,q)\ne(3,5)$.
If $|J|_p<q^m$, then $J\leq N_X(F)$.
\end{lemma}
\begin{proof}
We induct on $m$.
For $X=\GU_m(q)$, we have $|F|=q^m+1$.
For $X=\SU_m(q)$, we have $|F|=(q^m+1)/(q+1)$.
Since $|J|_p<q^m\leq|X|_p=q^{m(m-1)/2}$, we have $J<X$.
By Bereczky's theorem \cite{Bereczky}, we have $J\leq E$.
See also \cite[Theorem 2.1.7]{Brookfield2014}.
Next, let $Q=q^s$, where $s$ is an odd prime and $bs=m$.
We use $V=(\F_{Q^2})^b$ as an $\F_{q^2}$-vector space.
Then, $E=(\GU_b(Q)\rtimes\operatorname{Gal}(\F_{Q^2}/\F_{q^2}))\cap X$.

We choose a primitive prime divisor $\ell$ of $q^{2m}-1$.
By Zsigmondy's theorem, we have $\ell>2m$.
Let $F_\ell$ be the Sylow $\ell$-subgroup of $F$.
Since $\ell>s$, we have $F_\ell\leq E\cap\GU_b(Q)$.
Let $E_F=C_{\operatorname{End}_{\F_{q^2}}(V)}(F_\ell)$.
Since $F_\ell$ acts irreducibly on $V$, we have $E_F\cong\F_{q^{2m}}$.
Both $F$ and the scalar maps from $\F_{Q^2}$ lie in $E_F$.
As $E_F$ is commutative, we have $F\leq E\cap\GU_b(Q)$.
If $b=1$, then $E=N_X(F)$.

Suppose that $b\geq3$.
Let $Y=\GU_b(Q)$ when $X=\GU_m(q)$.
Let $Y=\SU_b(Q)$ when $X=\SU_m(q)$.
Then, $J_0=J\cap Y\unlhd J$ and $F_0=F\cap Y\leq J_0$.
For $X=\GU_m(q)$, we have $|F_0|=q^m+1$.
For $X=\SU_m(q)$, we have $|F_0|=(q^m+1)/(Q+1)$.
Hence, $F_0$ is a Singer subgroup of $Y$.
Also, $\ell\nmid Q+1$, so $F_\ell\leq F_0$.
Moreover, $|J_0|_p<q^m=Q^b$.
Since $Q\geq8$, the pair $(b,Q)$ satisfies the lemma's hypotheses.
By induction, we have $J_0\leq N_Y(F_0)$.
The quotient $N_Y(F_0)/F_0$ has order $b<\ell$.
Hence, $F_\ell$ is the unique Sylow $\ell$-subgroup of $J_0$.
Thus, $F_\ell$ is characteristic in $J_0$ and normal in $J$.
Since $F_\ell\unlhd J$, the group $J$ normalizes $C_X(F_\ell)=F$.
Hence, $J\leq N_X(F)$.
\end{proof}

\begin{theorem}\label{thm:unitary-odd-dimension}
Let $n\geq3$ be odd and let $q$ be a prime power.
Suppose that $G=\PSU_n(q)$ is simple.
Then, $G$ has no proper nontrivial ordinary
$G$-permutable subgroup.
\end{theorem}
\begin{proof}
For prime $n$, we use Theorem \ref{thm:unitary-prime}.
Suppose that $n\geq9$ is composite and that $A<G$
is nontrivial and ordinary $G$-permutable.
Let $q=p^f$ with $p$ prime, and let $d=(n,q+1)$.

The monomial subgroup $M$ of type $\GU_1(q)\wr S_n$
is maximal \cite[Proposition 2.2.9]{Uffelmann2017}.
By \cite[Theorem C]{LPS}, we have $MY\ne G$ for every $Y<G$.
See also \cite[Theorem 2.15 and Table A.3]{LiXia}.
By Lemma \ref{lem:max}, we have $A\leq M^g$ for some $g\in G$.
Hence, $|A|_p\leq(n!)_p<q^n$.
The strict inequality follows from $v_p(n!)<n/(p-1)\leq n$.

Let $T$ be the projective Singer subgroup of order
$(q^n+1)/((q+1)d)$.
By Lemma \ref{lem:unitary-order-bound}, we have $|A||T|<|G|$.
Since $A$ is ordinary $G$-permutable, we have a proper subgroup $K=AT^x$ for some $x\in G$.
Let $\pi:\SU_n(q)\to G$ be the quotient map and let $\widetilde K=\pi^{-1}(K)$.
Then, $\widetilde K$ contains a Singer subgroup of order $(q^n+1)/(q+1)$ and satisfies $|\widetilde K|_p=|A|_p$.
By Lemma \ref{lem:special-unitary-singer-small-p-part}, we have $A\leq N_G(T^x)$.

Let $N_1<G$ be the stabilizer of a nondegenerate line.
For odd $n\geq9$, the subgroup $N_1$ is maximal and $N_1Y\ne G$ for every $Y<G$
\cite[Theorem 2.15 and Table A.3]{LiXia}.
By Lemma \ref{lem:max}, $A$ stabilizes a nondegenerate line.
An element of the Singer torus stabilizing such a line
has multiplier $\lambda\in\F_{q^2}$ with $\lambda^{q+1}=1$.
The determinant of the corresponding lift is $\lambda^n$, so the lift is central in $\SU_n(q)$.
Thus, $A\cap T^x=1$.
Let $\nu_x:N_G(T^x)\to N_G(T^x)/T^x\cong C_n$ be the quotient map.
Since $\ker(\nu_x|_A)=A\cap T^x=1$, the restriction $\nu_x|_A$ is injective.
Therefore, $A$ is cyclic and $|A|\mid n$.

We choose a primitive prime divisor $\ell$ of $q^{2(n-2)}-1$.
Then, $\ell>2(n-2)>n$.
We choose a nondegenerate $U\leq V=\F_{q^2}^n$ with $\dim_{\F_{q^2}}U=n-2$.
Then, we take $\widetilde B\leq\GU_n(q)$ of order $\ell$ acting irreducibly on $U$ and fixing $U^\perp$ pointwise.
Since $\ell\nmid q+1$, we have $\widetilde B\leq\SU_n(q)$.
Let $B\leq G$ be the projective image of $\widetilde B$.
As $A$ is ordinary $G$-permutable, we have a subgroup $K=AB^y$ for some $y\in G$.
We have $|K|=\ell|A|$.
The number of Sylow $\ell$-subgroups of $K$ divides $|A|<\ell$ and is congruent to $1$ modulo $\ell$.
Hence, $B^y\lhd K$ and $A\leq N_G(B^y)$.
The quotient $N_G(B^y)/C_G(B^y)$ has order dividing $n-2$.
Let $\eta:A\to N_G(B^y)/C_G(B^y)$ satisfy $\eta(a)=aC_G(B^y)$.
By the homomorphism $\eta$, the order $|\eta(A)|$ divides both $|A|$ and $n-2$.
Since $|A|\mid n$ and $(n,n-2)=1$, we have $\eta(A)=1$.
Therefore, $A$ centralizes $B^y$.
We choose a lift $\widetilde y\in\SU_n(q)$ of $y$.
As $\ell\nmid|Z(\SU_n(q))|$, the inverse image of $B^y$ has a unique subgroup $\widetilde B^{\widetilde y}$ of order $\ell$.
The restriction $\pi|_{\widetilde B^{\widetilde y}}:\widetilde B^{\widetilde y}\to B^y$ is an isomorphism.
Every lift of $A$ normalizes $\widetilde B^{\widetilde y}$ and fixes its projective image pointwise.
Since this restriction is injective, every lift of $A$ centralizes $\widetilde B^{\widetilde y}$.

We choose a prime $r\mid|A|$ and an element $a\in A$
of order $r$.
Since $\nu_x|_A$ is injective, the image of $a$ in $N_G(T^x)/T^x$ has order $r$.
We choose a lift $\widetilde a\in\SU_n(q)$.
Next, let $\overline F$ be an algebraic closure of $\F_{q^2}$.
The inverse image of $T^x$ has $n$ eigenlines on $\overline F\otimes_{\F_{q^2}}V$.
Because the image of $a$ has order $r$, the lift $\widetilde a$ permutes these eigenlines in $n/r$ cycles of length $r$.
Along each cycle, $\widetilde a$ sends a weight vector to a nonzero scalar multiple of the next vector.
Since $\widetilde a^r$ is scalar, every geometric eigenspace has dimension at most $n/r$.
The bound also holds in characteristic $r$.

Let $U_y=\widetilde y^{-1}U$.
The lift $\widetilde a$ centralizes $\widetilde B^{\widetilde y}$.
Because $\widetilde B^{\widetilde y}$ acts irreducibly on $U_y$, its commuting algebra is $\F_{q^{2(n-2)}}$.
Thus, $\widetilde a|_{U_y}$ is multiplication by some $\lambda\in\F_{q^{2(n-2)}}^*$.
Since $\widetilde a^r=\zeta I$ for some $\zeta\in\F_{q^2}^*$, we have $\lambda^r=\zeta$.
The degree $e=[\F_{q^2}(\lambda):\F_{q^2}]$ divides $n-2$.
As the field conjugates of $\lambda$ are roots of $X^r-\zeta$, we have $e\leq r$.
Because $r\mid n$ and $r$ is odd, $r\nmid n-2$.
Hence, $e\leq r-1$.
Since $\widetilde a|_{U_y}$ is multiplication by $\lambda$ of degree $e$, some geometric eigenspace of $\widetilde a|_{U_y}$ has dimension $(n-2)/e\geq(n-2)/(r-1)$.
As $n$ is odd and composite, we have $n\geq3r$.
Therefore, $(n-2)/(r-1)>n/r$.
The geometric eigenspace dimension exceeds $n/r$, a contradiction.
We conclude that $A=1$.
\end{proof}

\subsection{\texorpdfstring{The groups $\PSU_6(q)$}{The groups PSU6(q)}}

\begin{theorem}\label{thm:unitary-six}
For every prime power $q$, the group $G=\PSU_6(q)$ has no
proper nontrivial ordinary $G$-permutable subgroup.
\end{theorem}
\begin{proof}
Suppose that $1<A<G$ is ordinary $G$-permutable.
Firstly, assume that $q\geq3$.
Let $d=(6,q+1)$, $F=\F_{q^2}$ and $E=\F_{q^6}$.
Next, we use $V=E^2$ as an $F$-vector space, with $\dim_F V=6$.
On $V$, we use the Hermitian form $h_0(v,w)=\Tr_{E/F}(v_1w_1^{q^3}+v_2w_2^{q^3})$.
With this form, let $M$ be the projective image of
$(\GU_2(q^3)\rtimes\operatorname{Gal}(E/F))\cap\SU(V,h_0)$.
The subgroup $M$ is maximal \cite[Table 8.26]{BHR}.
The order of $M$ is
$|M|=\frac{3q^3(q^3+1)(q^6-1)}{(q+1)d}$.
By \cite[Theorem C]{LPS}, we have $MY\ne G$ for every $Y<G$.
See also \cite[Theorem 2.15 and Table A.3]{LiXia}.
By Lemma \ref{lem:max}, choose $g\in G$ with $A^g\leq M$ and replace $A$ by $A^g$.

Let $R$ be a torus of type $(5,1)$, of order $(q^5+1)/d$.
Then,
$\frac{|G|}{|M||R|} =\frac{q^{12}(q^2-1)(q^4-1)(q+1)d}{3}>1$.
Since $A$ is ordinary $G$-permutable, we have a subgroup $K=AR^x$ for some $x\in G$.
As $|A||R|<|G|$, we have $K<G$.
Let $q=p^f$ with $p$ prime.
We choose a primitive prime divisor $\ell$ of $p^{10f}-1$.
Since $\ord_\ell(q)=10$ and $\ell\nmid d$, the torus $R^x\leq K$ contains an element of order $\ell$.
By \cite[Lemma 2.6]{GuralnickMalle2012}, $K$ stabilizes
a nondegenerate line $L$.

Let $\rho_L:\GU(L^\perp)\to\PGU(L^\perp)$ be the quotient map.
For $\overline g\in\Stab_G(L)$, choose a lift $g\in\SU_6(q)$.
Let $\pi_L:\Stab_G(L)\to\PGU(L^\perp)$ satisfy $\pi_L(\overline g)=\rho_L(g|_{L^\perp})$.
Since $2$ lifts differ by $\lambda I$ with $\lambda\in\F_{q^2}^*$ and $\lambda^{q+1}=1$, the value of $\pi_L(\overline g)$ is independent of the lift.
The kernel is $\ker\pi_L=\{\overline g:g|_{L^\perp}=\lambda I_{L^\perp}\text{ for some }\lambda\in\F_{q^2}^*\}$.
Let $J=\rho_L^{-1}(\pi_L(K))\leq\GU(L^\perp)\cong\GU_5(q)$.
Since $R^x\leq K$, the group $J$ contains $T_J=\rho_L^{-1}(\pi_L(R^x))$.
This is a Singer subgroup of order $q^5+1$.
The kernel $\ker\pi_L$ has order $(q+1)/d$ and lies in $R^x$.
Consequently,
$|J|_p=|A|_p\leq3_pq^3<q^5$.
As $q>2$, we have $J\leq N_{\GU_5(q)}(T_J)$ by Lemma \ref{lem:unitary-singer-small-p-part}.
Since $J$ normalizes $T_J$, $\pi_L(K)$ normalizes $\rho_L(T_J)=\pi_L(R^x)$.
As $\ker\pi_L\leq R^x$, we have $\pi_L^{-1}(\pi_L(R^x))=R^x$.
Therefore, $A\leq N_G(R^x)$.
For $\overline g\in A$, choose a lift $g\in\SU_6(q)$.
Let $\sigma_g\in\operatorname{Gal}(E/F)$ satisfy $g(av)=\sigma_g(a)g(v)$ for all $a\in E$ and $v\in E^2$.
Because $2$ lifts differ by $\lambda I_6$ with $\lambda\in F^*$, the map $\sigma_g$ depends only on $\overline g$.
The homomorphism $\chi:A\to\operatorname{Gal}(E/F)$ sends $\overline g$ to $\sigma_g$.
Let $A_0=\ker\chi$. Then, $A_0\lhd A$ and $[A:A_0]\mid3$.
We choose $0\ne v\in L$.
Then, $L=Fv$.
Every lift $g$ of an element of $A_0$ satisfies $g(v)=\alpha v$ for some $\alpha\in F^*$.
By $E$-linearity, $g|_{Ev}=\alpha I$.
Hence, $g$ has an eigenspace of $F$-dimension at least $3$.

Let $V_5=L^\perp$.
Since $L$ is a nondegenerate line in $V$, we have $\dim_F V_5=5$.
The inverse image of $N_G(R^x)$ has quotient $C_5$ over the inverse image of $R^x$.
Consequently, a lift of an element in $N_G(R^x)\setminus R^x$ acts on $V_5$ by a $5$-cycle on the torus weight lines.
Each weight vector maps to a nonzero scalar multiple of the next vector.
Thus, every geometric eigenspace on $V_5$ has dimension at most $1$.
The bound also holds in characteristic $5$.
Hence, every geometric eigenspace on $V$ has dimension at most $2$.
A lift of an element of $A_0$ has an eigenspace of dimension at least $3$.
Hence, no such element lies outside $R^x$, and $A_0\leq R^x$.

For a lift $g$ of an element of $A_0$, the restriction $g|_{V_5}$ is multiplication by $\lambda\in\F_{q^{10}}$.
The degree $[F(\lambda):F]$ is $1$ or $5$.
If $[F(\lambda):F]=5$, then every geometric eigenspace of $g$ on $V$ would have dimension at most $2$.
This contradicts the eigenspace of dimension at least $3$ obtained from $E$-linearity.
Thus, $\lambda\in F$.
The eigenvalues are then $\lambda$ with multiplicity $5$
and $\lambda^{-5}$ on $L$.
Since $L\leq Ev$ and $\dim_F(Ev\cap L^\perp)\geq2$, the scalar $\alpha$ equals both $\lambda^{-5}$ and $\lambda$.
Therefore, $\lambda=\lambda^{-5}$.
Every lift of $A_0$ is scalar, so $A_0=1$.

If $A\ne1$, then $|A|=3$.
The quotient $N_G(R^x)/R^x$ has order $5$, so $A\leq R^x$.
Since $A\leq R^x$ and $R^x$ has order prime to $p$, we have $p\ne3$.
We choose a lift $g\in\SU_6(q)$ of a generator of $A$.
Let $\mu\in F$ be its eigenvalue on $L$.
Then, let $h=\mu^{-1}g$.
Since $g$ lifts an element of order $3$ in $G$, the matrix $g^3$ is scalar.
Its scalar value is $\mu^3$, as seen on $L$.
Therefore, $h|_L=I$ and $h^3=I$.
We identify $V_5$ with $\F_{q^{10}}$ as an $F$-vector space.
Then, $h|_{V_5}$ is multiplication by an element $\nu\in\F_{q^{10}}^*$ with $|\nu|=3$.
As $[F(\nu):F]$ divides both $5$ and $2$, we have $\nu\in F$.
Thus, $h$ has the nontrivial eigenvalue $\nu$ with multiplicity $5$.
Because $A_0=1$, we have $\chi(\overline g)\ne1$.
We choose an algebraic closure $\overline F$ containing $E$.
Next, let $I=\operatorname{Gal}(E/F)$.
Let $\overline V=\overline F\otimes_F V$.
Then, the operator $\overline h$ on $\overline V$ satisfies $\overline h(a\otimes v)=a\otimes h(v)$.
Let $\Phi:\overline V\to\bigoplus_{\sigma\in I}\overline F^2$ be the $\overline F$-linear isomorphism with
$$
 \Phi\bigl(a\otimes(v_1,v_2)\bigr)
 =\bigl((a\sigma(v_1),a\sigma(v_2))\bigr)_{\sigma\in I}.
$$
As $\chi(\overline g)$ has order $3$, the operator $\Phi\overline h\Phi^{-1}$ cyclically permutes the $3$ summands.
Each summand has dimension $2$ over $\overline F$.
Because $h^3=1$, the characteristic polynomial is $(X^3-1)^2$.
Every eigenvalue has multiplicity $2$.
The eigenvalue $\nu$ cannot have both multiplicity $5$ and multiplicity $2$.
Therefore, $A=1$ for $q\geq3$.

For $q=2$, let $M_2$ be a maximal subgroup of type
$\GU_3(2)\wr S_2$ \cite[Table 8.26]{BHR}.
The projective order is $|M_2|=93312=2^7\cdot3^6$.
Since $\GU_3(2)$ is soluble, so is $M_2$.
If $G=M_2Y$ with $Y<G$, then $Y$ is not soluble by \cite[Proposition 4.1]{LiXia}.
The bound for a soluble factor $X<G$ in $G=XY$ is $|X|\mid2^9\cdot7\cdot3=10752$ by \cite[Table 1.1]{LiXia}.
As $|M_2|_3=3^6>(10752)_3=3$, we have $M_2Y\ne G$ for every $Y<G$.
Thus, $A\leq M_2^g$ for some $g\in G$.

We take a torus $R$ of type $(5,1)$ with $|R|=11$.
Since $A$ is ordinary $G$-permutable, we choose $x\in G$ such that $AR^x$ is a subgroup.
As $|A|$ divides $93312$, we have $|AR^x|\leq11\cdot93312<|G|$.
The prime divisors of $|AR^x|$ lie in $\{2,3,11\}$.
By \cite[Lemma 2.6]{GuralnickMalle2012}, we choose a maximal subgroup $M<G$ with $AR^x\leq M$.
We have $M\cong\PSU_5(2)$ or $M\cong M_{22}$.
Because $5,7\mid|M|$ and neither prime divides $|AR^x|$, we have $AR^x<M$.
We choose a maximal subgroup $L<M$ with $AR^x\leq L$.
As $R^x\leq L$ and $|R|=11$, we have $L\cong\PSL_2(11)$
\cite{Bereczky,Atlas1985}.
Hence,
$|A|\mid(93312,660)=12, |AR^x|\leq132<660$.
By Dickson's subgroup classification \cite{Dickson1901}, we have $AR^x\leq C_{11}\rtimes C_5$.
Hence, $|A|\mid(93312,55)=1$.
\end{proof}

\subsection{\texorpdfstring{The groups $\PSU_n(q)$ with even $n$}{The groups PSUn(q) with even n}}

We call a group $H$ metabelian if $H''=(H')'=1$.

\begin{theorem}\label{thm:unitary-even-dimension}
Let $n\geq4$ be even and let $q$ be any prime power.
Then, $G=\PSU_n(q)$ has no proper nontrivial ordinary
$G$-permutable subgroup.
\end{theorem}
\begin{proof}
For ease of reading, we divide the proof of the theorem into the following steps.

\setlist[itemize]{leftmargin=4em,labelsep=.5em,topsep=.2\baselineskip,partopsep=0pt,parsep=0pt,itemsep=0pt}
\begin{itemize}
\item[{\bf Step 1.} ] For $n\geq8$, we have $|A|_p<q^{n-1}$.
\end{itemize}
For $n=4,6$, we use Theorems \ref{thm:unitary-four} and \ref{thm:unitary-six}.
Suppose that $n\geq8$ and that $A<G$ is nontrivial and
ordinary $G$-permutable.
Let $q=p^f$ with $p$ prime.
Let $m=n-1$ and $d=(n,q+1)$.

The monomial subgroup $M$ of type $\GU_1(q)\wr S_n$
is maximal \cite[Proposition 2.2.9]{Uffelmann2017}.
By \cite[Theorem C]{LPS}, we have $MY\ne G$ for every $Y<G$.
See also \cite[Theorem 2.15 and Table A.3]{LiXia}.
By Lemma \ref{lem:max}, we have $A\leq M^g$ for some $g\in G$.
Hence, $|A|_p\leq(n!)_p$.
For odd $p$, we have $v_p(n!)<n/(p-1)\leq n/2<m$.
For $q\geq4$ even, $v_2(n!)<n\leq fm$.
Let $s_2(n)$ count the nonzero binary digits of $n$.
For $q=2$, we have $v_2(n!)=n-s_2(n)$.
For $s_2(n)>1$, we have $v_2(n!)<m$.
We consider $q=2$, $n=2^k$ after finding an invariant line.

Let $R$ be a torus of type $(m,1)$, of order $(q^m+1)/d$.
By Lemma \ref{lem:unitary-order-bound}, we have $|A||R|<|G|$.
Since $A$ is ordinary $G$-permutable, we have a proper subgroup $K=AR^x$ for some $x\in G$.
The torus $R^x$ contains an element whose order is a primitive prime divisor of $p^{2mf}-1$.
Since $K<G$, \cite[Lemma 2.6]{GuralnickMalle2012} implies that $K$ stabilizes a nondegenerate line $L\leq V=\F_{q^2}^n$.
The line $L$ is the unique line stabilized by $R^x$.

Suppose that $q=2$ and $n=2^k$.
Then, $G=\SU_n(2)$.
We choose $g_0\in G$ with $A\leq M^{g_0}$.
Then, we choose an orthonormal basis on which $M^{g_0}$ acts monomially.
Let $P$ be a Sylow $2$-subgroup of $A$.
By Schur--Zassenhaus, we choose $g\in M^{g_0}$ such that $P^g$ consists of permutation matrices in the chosen basis.
Replace $(A,K,L,P)$ by $(A^g,K^g,g^{-1}L,P^g)$, and replace $x$ by $xg$.
Let $h:V\times V\to\F_4$ be the Hermitian form.
Since $q+1$ is odd, every $u\in P$ fixes $L$ pointwise.
We choose $v=\sum_{i=1}^n v_i e_i\in L$ with $h(v,v)\ne0$.
Since $P$ consists of permutation matrices and fixes $v$, the coordinates $v_i$ are constant on each $P$-orbit on $\{1,\ldots,n\}$.
As the basis is orthonormal, we have $h(v,v)=\sum_{i=1}^n v_i^3$.
Every nonsingleton $P$-orbit has even size and contributes zero to this sum.
Since $h(v,v)\ne0$, at least one $P$-orbit has size $1$.
Thus, $P$ fixes an index in $\{1,\ldots,n\}$.
The faithful permutation action embeds $P$ in $S_{n-1}$.
Therefore, $|A|_2=|P|\leq((n-1)!)_2<2^{n-1}=q^m$.
Consequently, $|A|_p<q^m$ for all $n\geq8$ and all $q$.

\begin{itemize}
\item[{\bf Step 2.} ] $\widetilde A$ is abelian with cyclic Sylow $2$-subgroups.
\end{itemize}
By Step~1, we have $A\leq K\leq\Stab_G(L)$.
Let $\rho_L:\GU(L^\perp)\to\PGU(L^\perp)$ be the quotient map.
For $\overline g\in\Stab_G(L)$, choose a lift $g\in\SU_n(q)$.
Let $\pi_L:\Stab_G(L)\to\PGU(L^\perp)$ satisfy $\pi_L(\overline g)=\rho_L(g|_{L^\perp})$.
Since $2$ lifts differ by $\lambda I$ with $\lambda\in\F_{q^2}^*$ and $\lambda^{q+1}=1$, the value of $\pi_L(\overline g)$ is independent of the lift.
The kernel is $\ker\pi_L=\{\overline g:g|_{L^\perp}=\lambda I_{L^\perp}\text{ for some }\lambda\in\F_{q^2}^*\}$.
From the kernel formula and the torus decomposition, the kernel $\ker\pi_L$ has order $(q+1)/d$ and lies in $R^x$.
Let $J=\rho_L^{-1}(\pi_L(K))\leq\GU(L^\perp)\cong\GU_m(q)$.
Since $R^x\leq K$, the group $J$ contains the Singer subgroup $F=\rho_L^{-1}(\pi_L(R^x))$ of order $q^m+1$.
As both quotient kernels and $R^x$ have order prime to $p$, we have $|J|_p=|A|_p<q^m$.
By Lemma \ref{lem:unitary-singer-small-p-part}, $J$ normalizes $F$.
Hence, $\pi_L(K)$ normalizes $\pi_L(R^x)$.
Because $\ker\pi_L\leq R^x$, we have $\pi_L^{-1}(\pi_L(R^x))=R^x$.
Therefore, $A\leq N_G(R^x)$.
Since $|R|$ is prime to $p$ and $[N_G(R):R]=m$, we have $|A|_p\leq m_p$.

We take a nondegenerate $W_0\leq V$ with $\dim_{\F_{q^2}}W_0=2$.
The subgroup $N_2=\operatorname{Stab}_G(W_0)$ is maximal \cite[Theorem 4.3]{King2005}.
By \cite[Theorem 2.15 and Table A.3]{LiXia}, we have $N_2Y\ne G$ for every $Y<G$.
By Lemma \ref{lem:max}, we have $A\leq N_2^z$ for some $z\in G$.
We choose a lift $\widetilde z\in\SU_n(q)$ of $z$.
Let $W=\widetilde z^{-1}W_0$.
Then, $aW=W$ for every lift $a$ of an element of $A$.
Let $\widetilde G=\SU_n(q)$.
Next, let $\widetilde A$ and $\widetilde R$ be the inverse images of $A$ and $R^x$ under $\widetilde G\to G$.
Let $\widetilde C=\widetilde A\cap\widetilde R$.

The torus $\widetilde R$ is cyclic of order $q^m+1$.
Now, let $U=L^\perp$.
We have $\dim_{\F_{q^2}}U=m$.
For $c\in\widetilde R$, $c|_U$ is multiplication by $\lambda\in\F_{q^{2m}}$.
The degree $e=[\F_{q^2}(\lambda):\F_{q^2}]$ divides the odd integer $m$.
For $e>1$, every $c$-invariant $\F_{q^2}$-subspace of $V$ has dimension $0$ or $1$ modulo $e$.
Since $e\geq3$ and $\dim_{\F_{q^2}}W=2$, we have $cW\ne W$.
Every element of $\widetilde C$ stabilizes $W$.
Therefore, every element of $\widetilde C$ has $e=1$.
The matrix is $\diag(\lambda I_m,\lambda^{-m})$ with $\lambda^{q+1}=1$.
The group $N_{\widetilde G}(\widetilde R)$ stabilizes $U$ and $L$.
Since each displayed matrix is scalar on both summands, it centralizes $N_{\widetilde G}(\widetilde R)$.
As $N_{\widetilde G}(\widetilde R)/\widetilde R\cong C_m$, the quotient $\widetilde A/\widetilde C$ is cyclic.
Because $\widetilde C$ is central in $\widetilde A$ and $\widetilde A/\widetilde C$ is cyclic, $\widetilde A$ is abelian.
Because $|N_{\widetilde G}(\widetilde R)/\widetilde R|=m$ is odd, every Sylow $2$-subgroup of $\widetilde A$ lies in $\widetilde R$.
As $\widetilde R$ is cyclic, those Sylow subgroups are cyclic.

\begin{itemize}
\item[{\bf Step 3.} ] $\widetilde K$ is metabelian and $O_p(\widetilde K)=1$.
\end{itemize}
Next, we take a cyclic torus $T\leq\widetilde G$ of type $(n)$ with $|T|=(q^n-1)/(q+1)$.
Let $S$ be the Hall subgroup of $T$ of odd order.
Let $\pi:\widetilde G\to G$ be the quotient map.
Since $A$ is ordinary $G$-permutable, choose $y\in G$ with $A\pi(S)^y$ a subgroup.
Choose $\widetilde y\in\widetilde G$ with $\pi(\widetilde y)=y$.
Then, we use $S^y=\widetilde y^{-1}S\widetilde y$ and $T^y=\widetilde y^{-1}T\widetilde y$.
Since $\ker\pi=Z(\widetilde G)$, the choice of $\widetilde y$ does not affect $S^y,T^y$.
As $Z(\widetilde G)\leq\widetilde A$, the inverse image of $A\pi(S)^y$ is $\widetilde K=\widetilde A S^y$.
Because $\widetilde A$ and $S^y$ are abelian, we obtain from It\^o's theorem \cite{Ito1955} that $\widetilde K$ is metabelian.
As $|S|$ is odd, a Sylow $2$-subgroup of $\widetilde A$ is also Sylow in $\widetilde K$.
Hence, the Sylow $2$-subgroups of $\widetilde K$ are cyclic.
By Burnside's theorem on normal complements, we have a normal Hall subgroup $H\unlhd\widetilde K$ of odd order with $S^y\leq H$.

Let $P_0=O_p(\widetilde K)$.
Since $S^y$ and $Z(\widetilde G)$ have order prime to $p$, we have $|\widetilde K|_p=|A|_p$.
By the bound from Step~2, we have $|A|_p\leq m_p$.
Therefore, $|P_0|\leq m_p\leq m$.
We show that $P_0=1$.
If $4\mid n$, then choose a primitive prime divisor $\ell$
of $p^{fn}-1$.
If $n\equiv2\pmod4$, then choose a primitive prime divisor $\ell$ of $p^{fn/2}-1$.
For $n\geq8$, neither chosen exponent has a Zsigmondy exception.
For the first choice, we have $\ell>fn\geq n$.
For the second choice with odd $f$, the integer $fn/2$ is odd and divides $\ell-1$.
Since $\ell-1$ is even, we have $\ell-1\geq fn$.
For the second choice with even $f$, we have $\ell>fn/2\geq n$.
Thus, $\ell>n$ in both choices.
Moreover,
$\ord_\ell(q^2)=n/2, \ord_\ell(-q)=n$.
Since $\ell$ is odd and $\ell\nmid q+1$, we have $\ell\mid|S|$.
Therefore, $S^y$ contains an element $z$ of order $\ell$.
Let $\lambda\in\F_{q^n}^{\,*}$ be an eigenvalue of $z$ with $|\lambda|=\ell$.
Let $\mathcal O_0=\{\lambda^{q^{2j}}:0\leq j<n/2\}$ and $\mathcal O_1=\{\alpha^{-q}:\alpha\in\mathcal O_0\}$.
As $\ord_\ell(-q)=n$, we have $\mathcal O_0\cap\mathcal O_1=\varnothing$.
The sets $\mathcal O_0,\mathcal O_1$ are the $2$ eigenvalue orbits under $\alpha\mapsto\alpha^{q^2}$.
We have $|\mathcal O_0|=|\mathcal O_1|=n/2$.
Thus, $z$ has $2$ nonisomorphic irreducible constituents of dimension $n/2$.
Unitary duality $\alpha\mapsto\alpha^{-q}$ sends $\mathcal O_i$ to $\mathcal O_{1-i}$ for $i=0,1$.
The two disjoint orbits contain $n$ distinct eigenvalues.
Thus, $z$ is regular semisimple and $C_{\widetilde G}(z)=T^y$, a $p'$-torus.
The permutation action on $P_0$ induces an injective homomorphism $\Aut(P_0)\to S_{|P_0|}$.
Since $\ell>|P_0|$, we have $\ell\nmid|\Aut(P_0)|$.
Therefore, $z$ centralizes $P_0$ and $P_0\leq C_{\widetilde G}(z)=T^y$.
As $T^y$ is a $p'$-group, we have $P_0=1$.

\begin{itemize}
\item[{\bf Step 4.} ] $\widetilde K$ normalizes $T^y$.
\end{itemize}
Let $D=H'$, where $H\unlhd\widetilde K$ is the normal Hall subgroup of odd order from Step~3.
Since $\widetilde K''=1$, we have $D'=1$.
Since $H'$ is characteristic in $H$, we have $D\unlhd\widetilde K$.
As $D$ is abelian, its Sylow $p$-subgroup is characteristic in $D$ and hence normal in $\widetilde K$.
Because $O_p(\widetilde K)=1$, that Sylow subgroup is trivial.
Thus, $D$ is a $p'$-group.
Let $\overline F$ be an algebraic closure of $\F_{q^2}$.
Let $\overline V=\overline F\otimes_{\F_{q^2}}V$ with $d(a\otimes v)=a\otimes dv$ for $d\in D$.
Because $D$ is abelian of order prime to $p$, its elements are simultaneously diagonalizable on $\overline V$.
Thus, $D$ has at most $n$ distinct common eigenspaces.
Because $z$ normalizes $D$, the element $z$ permutes the $D$-weight spaces.
Since $z$ has order $\ell>n$, its permutation of the $D$-weight spaces is trivial.
Because each element of $D$ is scalar on every $D$-weight space, $z$ centralizes $D$.
Therefore, $D\leq C_{\widetilde G}(z)=T^y$.
Because $|D|$ is odd, we also have $D\leq S^y$.

As $H/H'$ is abelian and $H'\leq S^y$, the subgroup $S^y$ is normal in $H$.
Therefore, $H$ normalizes its centralizer $C_{\widetilde G}(S^y)=T^y$.
Hence, $H\leq N_{\widetilde G}(T^y)$.
The quotient $N_{\widetilde G}(T^y)/T^y$ is $C_n$.
Because $\ell>n$, every $\ell$-subgroup of $H$ lies in the cyclic torus $T^y$.
Thus, $H$ has a unique Sylow $\ell$-subgroup $S_\ell$, and $S_\ell\leq S^y$.
Hence, $S_\ell$ is characteristic in $H$ and normal in $\widetilde K$.
Since $S_\ell\unlhd\widetilde K$, the group $\widetilde K$ normalizes $C_{\widetilde G}(S_\ell)=T^y$.
Hence, $\widetilde K\leq N_{\widetilde G}(T^y)$.
Let $\overline T_y=\pi(T^y)\leq G$.
Thus, $A\leq N_G(\overline T_y)$.

An element of $T^y$ stabilizing the nondegenerate line $L$ has multiplier $\lambda\in\F_{q^2}$ with $\lambda^{q+1}=1$.
Hence, the element is $\lambda I_V$.
Since $\lambda I_V\in Z(\widetilde G)$, we have $A\cap\overline T_y=1$.
Reducing $(q^n-1)/(q+1)$ modulo $q+1$, we also obtain that $\gcd((q^n-1)/(q+1),q+1)=d$.
Let $\nu_y:N_G(\overline T_y)\to N_G(\overline T_y)/\overline T_y\cong C_n$ be the quotient map.
As $\ker(\nu_y|_A)=1$, the subgroup $A$ is cyclic and $|A|\mid n$.
Let $\eta:A\to N_G(R^x)/R^x\cong C_m$ satisfy $\eta(a)=aR^x$.
The order $|\eta(A)|$ divides both $|A|$ and $m$.
Since $|A|\mid n$ and $(n,m)=1$, we have $\eta(A)=1$.
Hence, $A\leq R^x$.

We choose an element $a\in A$ of prime order $r$.
By Step~2, a lift $\widetilde a\in\widetilde C$ is scalar on $U=L^\perp$ and on $L$.
Since $a\ne1$ projectively, the two scalars are distinct.
Thus, $\widetilde a$ has $2$ distinct eigenvalues with multiplicities $(n-1,1)$.
Since $\nu_y|_A$ is injective, the element $\nu_y(a)$ has order $r$.
The lift $\widetilde a$ permutes the torus weight lines in $n/r$ cycles of length $r$.
On each cycle, an eigenvector is determined by its coordinate on one weight line.
Hence, every geometric eigenspace of $\widetilde a$ has dimension at most $n/r$.
As $n-1>n/r$, we have a contradiction.
\end{proof}

\begin{corollary}\label{cor:all-unitary}
Let $n\geq3$ and let $q$ be a prime power.
Suppose that $G=\PSU_n(q)$ is nonabelian simple.
Then, $G$ has no proper nontrivial ordinary $G$-permutable subgroup.
\end{corollary}
\begin{proof}
For odd $n\geq3$, we use Theorem \ref{thm:unitary-odd-dimension}.
For even $n\geq4$, we use Theorem \ref{thm:unitary-even-dimension}.
\end{proof}

\section{Symplectic groups}\label{sec:symplectic}

Let $q$ be a prime power, $m\geq1$ and $V=\F_q^{2m}$.
A bilinear form $\beta:V\times V\to\F_q$ is alternating if $\beta(v,v)=0$ for every $v\in V$.
For $W\leq V$, let $W^\perp=\{v\in V:\beta(v,w)=0\text{ for every }w\in W\}$.
We assume $V^\perp=0$ and use
$\Sp(V,\beta)=\{g\in\GL_{\F_q}(V):\beta(gv,gw)=\beta(v,w)\text{ for all }v,w\in V\}$.
A symplectic basis $(e_1,\ldots,e_m,f_1,\ldots,f_m)$ satisfies $\beta(e_i,e_j)=\beta(f_i,f_j)=0$ and $\beta(e_i,f_j)=\delta_{ij}$.
A subspace $W$ is Lagrangian if $W=W^\perp$.
It is totally isotropic if $\beta|_{W\times W}=0$.
It is nondegenerate if $W\cap W^\perp=0$.
The notation $V=U\perp W$ means $V=U\oplus W$ and $\beta(U,W)=0$.
For even $q$, an alternating matrix $S$ satisfies $S^{\mathsf T}=S$ and $S_{ii}=0$ for every $i$.

A Singer torus in $\Sp_{2m}(q)$ acts irreducibly on $V$ and has order $q^m+1$.
In a field model $V=\F_{q^{2m}}$, it consists of $v\mapsto\lambda v$ with $\lambda^{q^m+1}=1$.
Its image in $\PSp_{2m}(q)$ has order $(q^m+1)/(2,q-1)$ and is also called a Singer torus.
Let $E=\F_{q^e}$ with $e,d\geq1$, and let $\sigma\in\operatorname{Gal}(E/\F_q)$.
A $\sigma$-semilinear map is a bijection $g:E^d\to E^d$ satisfying $g(av+bw)=\sigma(a)g(v)+\sigma(b)g(w)$ for $a,b\in E$ and $v,w\in E^d$.
We call $\sigma$ its field companion.
For the orthogonal groups in the auxiliary lemmas, we use the definitions at the beginning of Section~\ref{sec:orthogonal}.

\begin{lemma}\label{lem:symplectic-singer-involutions}
Let $q$ be a prime power and $m\geq1$.
Let $\pi:\Sp_{2m}(q)\to G=\PSp_{2m}(q)$ be the quotient map.
Let $T\leq G$ be a Singer torus with $|T|=(q^m+1)/(2,q-1)$.
Every involution in $N_G(T)$ has rank $m$ when $q$ is even.
When $q$ is odd, every involution $t\in N_G(T)$ has a lift
$h\in\Sp_{2m}(q)$ with $\pi(h)=t$ and $h^2=-I_{2m}$.
Suppose that $A\leq N_G(T)$ stabilizes a line $\Lambda\leq\F_q^{2m}$.
Then, $A$ is cyclic with $|A|\mid2m$.
When $q\equiv3\pmod4$, such a subgroup $A$ has odd order.
\end{lemma}
\begin{proof}
Let $E=\F_{q^{2m}}$ and $F=\F_{q^m}$.
We choose $\mu\ne0$ with $\mu^{q^m}=-\mu$ and use the nondegenerate alternating form
$\beta(x,y)=\Tr_{E/\F_q}(\mu xy^{q^m})$.
For $\lambda\in E^*$ with $\lambda^{q^m+1}=1$, let $M_\lambda(x)=\lambda x$.
Together, these multipliers form the inverse image of $T$ under $\pi$ and generate the field algebra $E$.
Because each element normalizing this multiplier group normalizes $E$, the elements of $N_G(T)$ have semilinear lifts.
For each $0\leq i<2m$, let $\sigma_i(x)=x^{q^i}$.
Since $(\mu^{q^i}/\mu)^{q^m}=\mu^{q^i}/\mu$, the ratio $\mu^{q^i}/\mu$ lies in $F^*$.
As the norm $N_{E/F}:E^*\to F^*$ is surjective, choose $c_i$ with $c_i^{q^m+1}=\mu^{q^i}/\mu$.
Then, the map $h_i(x)=c_i x^{q^i}$ satisfies
$\beta(h_ix,h_iy)=\Tr_{E/\F_q}((\mu xy^{q^m})^{q^i})=\beta(x,y)$.
Consequently, $h_i$ lies in $\Sp_{2m}(q)$.
Conjugation on the field algebra $E$ induces $\theta:N_G(T)\to\operatorname{Gal}(E/\F_q)$ with $\theta(\pi(h_i))=\sigma_i$.
Because a multiplier $M_\lambda$ preserves $\beta$ exactly when $\lambda^{q^m+1}=1$, we have $\ker\theta=T$.
As every $\sigma_i$ lies in the image of $\theta$, we obtain $N_G(T)/T\cong C_{2m}$.

Let $h\in\Sp_{2m}(q)$ lift an involution $\pi(h)\in N_G(T)$.
As $h^2\in\ker\pi$, we have $h^2=\pm I_E$.
Consequently, the field automorphism of $h$ is trivial or $x\mapsto x^{q^m}$.
If $h$ is a multiplier and $q$ is odd, then $h^2=I_E$ would imply $h=\pm I_E$.
Because $\pi(h)\ne1$, we obtain $h^2=-I_E$ in this case.
For even $q$, the order $|T|=q^m+1$ is odd, so an involution cannot lie in $T$.
If the field automorphism of $h$ is $x\mapsto x^{q^m}$, then $h(x)=cx^{q^m}$ for some $c\in E^*$.
Substituting $h(x)=cx^{q^m}$ in $\beta(hx,hy)=\beta(x,y)$ yields $c^{q^m+1}=-1$.
Hence, $h^2=-I_E$ in this case as well.
When $q$ is even, we have $c^{q^m+1}=1$.
By Hilbert's Theorem 90, choose $d\in E^*$ with $c=d/d^{q^m}$.
Then, $h(dy)=dy$ if and only if $y\in F$, so $\ker(h-I_E)=dF$ has dimension $m$ over $\F_q$.
By rank--nullity, we obtain $\operatorname{rank}(h-I_E)=m$.

Now, suppose that $A$ stabilizes the line $\Lambda$.
For $a\in A\cap T$, choose $\lambda\in E^*$ with $a=\pi(M_\lambda)$ and $\lambda^{q^m+1}=1$.
As $a\in A$, we have $M_\lambda(\Lambda)=\Lambda$.
For $0\ne v\in\Lambda$, the equality $\lambda v=cv$ holds for some $c\in\F_q^*$.
Therefore, $\lambda=c\in\F_q^*$ and $\lambda^2=\lambda^{q^m+1}=1$.
It follows that $M_\lambda=\lambda I_E\in\ker\pi$ and $a=1$.
Since $a\in A\cap T$ was arbitrary, we have $A\cap T=1$.
Now, consider the homomorphism $\eta:A\to N_G(T)/T$ with $\eta(a)=aT$.
Because $\ker\eta=A\cap T=1$, the map $\eta$ is injective.
Thus, $A$ is cyclic with $|A|\mid2m$.
Finally, suppose that $q\equiv3\pmod4$ and $A$ contains an involution $t=\pi(h)$.
As $h\Lambda=\Lambda$, we have $hv=cv$ for some $0\ne v\in\Lambda$ and $c\in\F_q^*$.
But $h^2=-I_E$ would imply $c^2=-1$, which is impossible for $q\equiv3\pmod4$.
Consequently, $A$ contains no involution. By Cauchy's theorem, $|A|$ is odd.
\end{proof}

\begin{lemma}\label{lem:symplectic-unitary-sylow}
Let $F\leq E$ be finite fields of characteristic $p$.
For $b\geq1$, a Sylow $p$-subgroup of $\Sp_{2b}(E)$ acts indecomposably on $E^{2b}$ over $F$.
For $E=\F_{Q^2}$ and $b\geq3$, a Sylow $p$-subgroup of $\SU_b(Q)$ acts indecomposably on $E^b$ over $F$.
\end{lemma}
\begin{proof}
Let $d=2b$ for $\Sp_{2b}(E)$ and $d=b$ for $\SU_b(Q)$.
We use the upper-triangular Sylow subgroup $P$ on $V=E^d$ with basis $e_1,\ldots,e_d$.
For this subgroup, the common fixed space is $V^P=Y=Ee_1$.
Moreover, a nonidentity highest-root element $u\in P$ satisfies
$\operatorname{im}(u-I_V)=Y$ and $\ker(u-I_V)=H=Ee_1+\cdots+Ee_{d-1}$.
For every $X\in C_{\operatorname{End}_F(V)}(P)$, we have $X(Y)\leq Y$ and $X(H)\leq H$.

Suppose that $d\geq3$.
For each $a\in E$, choose the adjacent root element $u_a\in P$ satisfying
$u_a(e_2x)=e_2x+e_1ax$ for $x\in E$.
Let $X_Y=X|_Y:Y\to Y$.
Let $D:E\to E$ take $x$ to the second coordinate of $X(e_2x)$.
Since $X(e_2x)\in H$, its last coordinate is zero. Thus, the linked root terms in $Xu_a=u_aX$ vanish.
We obtain that $X_Y(ax e_1)=aD(x)e_1$.
Taking $a=1$, we have $X_Y(xe_1)=D(x)e_1$. Thus, $X_Y(axe_1)=aX_Y(xe_1)$ for $a,x\in E$.
Since $Y=Ee_1$, we have $X_Y=\lambda I_Y$ for some $\lambda\in E$.
For $\Sp_2(E)$, we have $X(I+aE_{12})=(I+aE_{12})X$ for every $a\in E$.
Hence, $X_Y=\lambda I_Y$ for some $\lambda\in E$.

If $X$ is idempotent, then $X|_Y=0$ or $I_Y$.
Every nonzero module for a finite $p$-group in characteristic $p$ has a nonzero fixed vector.
Since $\operatorname{im}X$ and $\ker X$ are $P$-modules, each nonzero space contains a nonzero $P$-fixed vector.
Suppose that $X_Y=0$ and $\operatorname{im}X\ne0$.
Since $V^P=Y$, we choose $0\ne v\in\operatorname{im}X\cap Y$.
Then, $Xv=v$ by $X^2=X$, whereas $Xv=0$ by $X_Y=0$.
Suppose that $X_Y=I_Y$ and $\ker X\ne0$.
We choose $0\ne v\in\ker X\cap Y$. Then, $Xv=0$ and $Xv=v$, a contradiction.
Therefore, $X=0$ or $I_V$. We conclude that $P$ acts indecomposably over $F$.
\end{proof}

\begin{lemma}\label{lem:orthogonal-minus-two-subgroup}
Let $Q$ be a power of a prime $p$ and $a\geq2$.
Let $F\leq E=\F_Q$.
The group $\Omega^-_{2a}(Q)$ contains a $p$-subgroup acting
indecomposably on its natural module over $F$.
\end{lemma}
\begin{proof}
Let $r=a-1$ and $L=\F_{Q^2}$.
On $V=E^r\oplus L\oplus E^r$, use
$\mathcal Q(x,z,y)=x^{\mathsf T}y+z\overline z, \overline z=z^Q$.
Next, let $U_r=\{g\in\GL_r(E):g_{ii}=1,\ g_{ij}=0\text{ for }i>j\}$.
We use $\{\diag(g,I_L,g^{-\mathsf T}):g\in U_r\}$ on $V$.
For $1\leq i<j\leq r$ and $b\in E$, use the maps
$$
 x_i\mapsto x_i+b y_j,\  x_j\mapsto x_j-b y_i.
$$
For $1\leq i\leq r$ and $c\in L$, use
$$
 s_i(c):z\mapsto z+c y_i,\  x_i\mapsto x_i-\Tr_{L/E}(\overline c z)-c\overline c y_i.
$$
Each map fixes the coordinates not displayed.
For each displayed map $g$, we have $\mathcal Q(gv)=\mathcal Q(v)$ for $v\in V$.
The maps lie in an upper unitriangular group.
Let $P$ be generated by the displayed root maps and $\{\diag(g,I_L,g^{-\mathsf T}):g\in U_r\}$.
For $p=2$, each generator $g$ has even $\operatorname{rank}(g-I_V)$.
By the Dickson homomorphism \cite[Proposition 1.6.11]{BHR}, we have $P\leq\Omega^-_{2a}(Q)$.
For odd $p$, the generators have determinant $1$.
The spinor homomorphism $\operatorname{SO}^-_{2a}(Q)\to C_2$
is trivial on the $p$-group $P$, so again $P\leq\Omega^-_{2a}(Q)$.
For a common fixed vector $(x,z,y)$, the equations $s_i(c)(x,z,y)=(x,z,y)$ for all $i,c$ imply that $y=z=0$.
From the upper-unitriangular subgroup of $\GL_r(E)$, we have $V^P=Ee_1$.

Let $X\in C_{\operatorname{End}_F(V)}(P)$ satisfy $X^2=X$.
Let $Y=Ee_1$ and $X_Y=X|_Y:Y\to Y$.
For $r\geq2$, choose a long root element $u$ with $\operatorname{im}(u-I)=Ee_1+Ee_2$.
Since $Xu=uX$, we have $X(Ee_1+Ee_2)\leq Ee_1+Ee_2$.
Let $D:E\to E$ take $t$ to the $e_2$-coordinate of $X(te_2)$.
For $b\in E$, let $u_b=\diag(I_r+bE_{12},I_L,(I_r+bE_{12})^{-\mathsf T})\in P$.
Since $X(Ee_1+Ee_2)\leq Ee_1+Ee_2$ and $Xu_b=u_bX$, we have $X_Y(bt e_1)=bD(t)e_1$ for $b,t\in E$.
Taking $b=1$, we obtain that $X_Y(bt e_1)=bX_Y(te_1)$. Thus, $X_Y$ is multiplication by an element of $E$.

If $r=1$, then let $H=Y\oplus L$.
Since $H=\sum_{c\in L}\operatorname{im}(s_1(c)-I_V)$, we have $X(H)\leq H$.
We identify $H/Y$ with $L$ by $(te_1,z,0)+Y\mapsto z$, for $t\in E$ and $z\in L$.
Under this identification, let $T:L\to L$ send $z$ to the $L$-coordinate of $X(0,z,0)$.
For $X|_Y$, let $f:E\to E$ be the $F$-linear map satisfying $X(te_1)=f(t)e_1$ for $t\in E$.
As $Xs_1(c)=s_1(c)X$ on $H$, we have
$$
 f\bigl(\Tr_{L/E}(\overline c z)\bigr)=\Tr_{L/E}(\overline c Tz).
$$
For $b\in E$, replacing $c$ by $bc$ in the displayed equation yields that
$$
 f\bigl(b\Tr_{L/E}(\overline c z)\bigr)
 =b\Tr_{L/E}(\overline c Tz)
 =bf\bigl(\Tr_{L/E}(\overline c z)\bigr).
$$
Because the trace pairing takes every value in $E$, we have $f(bt)=bf(t)$ for $b,t\in E$.
Thus, $X_Y=\lambda I_Y$ for some $\lambda\in E$.
As $X^2=X$, we have $X_Y^2=X_Y$.
Hence, $X_Y=0$ or $I_Y$.
Every nonzero image or kernel of $X$ has a nonzero $P$-fixed vector.
By the argument in Lemma \ref{lem:symplectic-unitary-sylow}, we have $X=0$ or $I_V$.
Therefore, $P$ acts indecomposably over $F$.
\end{proof}

\begin{lemma}\label{lem:symplectic-odd-singer}
Let $q=p^f$ be odd, $m\geq2$ and $(m,q)\ne(2,3)$.
Let $S\leq J<\Sp_{2m}(q)$, where $S$ is a Singer torus of order $q^m+1$.
Suppose that every $p$-subgroup $P\leq J$ stabilizes some decomposition $\F_q^{2m}=V_1\oplus V_2$ with $\dim_{\F_q}V_1=2$.
Then, $J\leq N_{\Sp_{2m}(q)}(S)$.
\end{lemma}
\begin{proof}
We choose a primitive prime divisor $\ell$ of $q^{2m}-1$.
By Zsigmondy's theorem, we have $\ell>2m$.
Let $S_\ell$ be the Sylow $\ell$-subgroup of $S$.
The subgroup $S_\ell$ acts irreducibly, with $C_{\GL_{2m}(q)}(S_\ell)=\F_{q^{2m}}^*$ and
$C_{\Sp_{2m}(q)}(S_\ell)=S$.

By Bereczky's theorem \cite{Bereczky}, we choose a maximal subgroup $M<\Sp_{2m}(q)$ with $J\leq M$.
We take the normal subgroup $L\unlhd M$ from that classification.
By Bereczky's classification, we have $L\cong\Sp_{2m/s}(q^s)$ for a prime $s\mid m$ or $L\cong\GU_m(q)$.
In the first case, $L$ consists of $\F_{q^s}$-linear maps.
In the unitary case, $m$ is odd.
For both subgroup types, we refer to \cite[Table 2.3]{Brookfield2014}.
Since $\ell\nmid[M:L]$, we have $S_\ell\leq L$.
Let $F_0$ be the scalar field of $L$.
As $F_0$ centralizes $S_\ell$, we have $F_0\leq\F_{q^{2m}}$.
Every element of $S$ commutes with $F_0$, so $S\leq L$.

Let $J_0=J\cap L\unlhd J$. Then, $S\leq J_0$.
If $J_0=L$, then $J_0$ contains a natural Sylow $p$-subgroup
of the symplectic group or special unitary core of $L$.
Lemma \ref{lem:symplectic-unitary-sylow} contradicts the assumed
direct decomposition.
If $J_0<L$, then we choose a maximal subgroup $M_1<L$ with $J_0\leq M_1$.
We repeat the preceding argument with $J_0\leq M_1<L$.
At each repetition, we increase the accumulated field degree and decrease the relative natural dimension.
Throughout these intersections, $S$ is the relative Singer torus of order $q^m+1$.
Since the next scalar field centralizes $S_\ell$, every element of $S$ is linear over that field.
All the intersections are normal in their predecessors.

When the relative symplectic dimension is $2$, the group is $\SL_2(q^m)$.
We have $q^m\geq25$.
By Dickson's classification \cite{Dickson1901}, every subgroup $H$ with $S\leq H<\SL_2(q^m)$ satisfies
$H\leq N_{\SL_2(q^m)}(S)$.
When the relative unitary dimension is $1$, the group equals $S$.
The only additional relative group in odd characteristic is $\GU_3(3)$.
It requires $(m,q)=(3,3)$.
In $\GU_3(3)$, we also consider a maximal subgroup $M_u\cong C_4\times\PSL_2(7)$ with $S\leq M_u$
\cite[Table 8.6]{BHR}.
The torus of order $28$ contains the central $C_4$.
Thus, the intersection is $C_4\times H$ for $H\leq\PSL_2(7)$.
We have $7\mid|H|$.
If $H$ is proper, then $H\leq C_7\rtimes C_3$ and $H$ normalizes $S$.
If $H=\PSL_2(7)$, then choose $s_7\in S$ of order $7$.
Choose $u\in C_7\rtimes C_3\leq H$ of order $3$.
Let $L_0,L_1,L_2\leq\overline{\F}_3^3$ be the $3$ eigenlines of $s_7$.
We choose their numbering so that $uL_i=L_{i+1\pmod3}$ for $i=0,1,2$.
From the weighted $3$-cycle, we have Jordan form $J_3$ for $u$ on $\F_9^3$.
On the underlying $\F_3$-space, we have $J_3\oplus J_3$.
There is no decomposition $\F_3^6=W_1\oplus W_2$ with $\dim W_1=2$ and $uW_i=W_i$ for $i=1,2$.

Now, we denote the successive normal intersections by $H_0=J,H_1,\ldots,H_d$.
By construction, we have $S\leq H_i\unlhd H_{i-1}$ for $1\leq i\leq d$ and $H_d\leq N_{\Sp_{2m}(q)}(S)$.
Since $[N_{\Sp_{2m}(q)}(S):S]\mid2m<\ell$, the subgroup $S_\ell$ is the unique Sylow $\ell$-subgroup of $H_d$.
As $C_{H_d}(S_\ell)=S$, both $S_\ell$ and $S$ are characteristic in $H_d$.
Suppose that $S_\ell$ and $S$ are characteristic in $H_i$, with $1\leq i\leq d$.
Because $H_i\unlhd H_{i-1}$, we have $H_{i-1}\leq N_{\Sp_{2m}(q)}(S)$.
Thus, $S_\ell$ is the unique Sylow $\ell$-subgroup of $H_{i-1}$.
As $C_{H_{i-1}}(S_\ell)=S$, the subgroup $S$ is characteristic in $H_{i-1}$.
Repeating the argument for $i=d,d-1,\ldots,1$, we obtain that $J\leq N_{\Sp_{2m}(q)}(S)$.
\end{proof}

\subsection{\texorpdfstring{The groups $\Sp_4(q)$ with even $q$}{The groups Sp4(q) with even q}}

\begin{lemma}\label{lem:symplectic-four-tori}
Let $q\geq4$ be even and $G=\Sp_4(q)$.
We use $V=\F_q^4$ with symplectic form $\beta$.
There are cyclic tori $T_-$ and $T_+$ with $|T_-|=q^2+1$ and
$|T_+|=q^2-1$. Their normalizers are
$N_G(T_-)=T_-\rtimes C_4, N_G(T_+)=T_+\rtimes (C_2\times C_2)$.
Every nonidentity element of $T_-$ is irreducible on the natural
module. For every involution $t\in N_G(T_-)$, we have $\operatorname{rank}(t-I_V)=2$.
Also, $\beta(v,tv)\ne0$ for some $v\in V$.
\end{lemma}
\begin{proof}
We use the model in Lemma \ref{lem:symplectic-singer-involutions}
with $E=\F_{q^4}$, $F=\F_{q^2}$ and $\mu=1$.
Let $\sigma:E\to E$ satisfy $\sigma(x)=x^q$.
Since $\beta(\sigma x,\sigma y)=\beta(x,y)$ and $|\sigma|=4$, we have $N_G(T_-)=T_-\rtimes\gen{\sigma}$.
As $(q^2+1,q^2-1)=1$, every nonidentity multiplier in $T_-$
generates $E$ over $\F_q$.
Each involution has rank $2$ and the form
$t:x\mapsto\lambda x^{q^2}$, $\lambda^{q^2+1}=1$. Moreover,
$\beta(x,tx)=\Tr_{E/\F_q}(\lambda^{-1}x^2)$.
The square map is bijective and the trace is nonzero.
Thus, the function $x\mapsto\beta(x,tx)$ does not vanish identically.

For $T_+$, take $F\oplus F$ with form
$\beta((x,y),(x',y'))=\Tr_{F/\F_q}(xy'+x'y)$.
We realize $T_+$ as $\{(x,y)\mapsto(\lambda x,\lambda^{-1}y):\lambda\in F^*\}$.
The normalizer $N_G(T_+)$ also contains the commuting involutions
$(x,y)\mapsto(x^q,y^q), (x,y)\mapsto(y,x)$.
For a generator of $T_+$, the eigenvalues $\lambda,\lambda^q,\lambda^{-1},\lambda^{-q}$ are distinct.
Thus, the $2$ irreducible $\F_qT_+$-modules are nonisomorphic.
For $g\in N_G(T_+)$, we have $g(F\oplus0)=F\oplus0$ or $0\oplus F$.
The maps induced by $g$ on both constituents have a common companion $\sigma\in\operatorname{Gal}(F/\F_q)$.
Thus, $N_G(T_+)=T_+\rtimes(C_2\times C_2)$.
\end{proof}

\begin{theorem}\label{thm:symplectic-four-even}
Let $q=2^f$ with $f\geq2$, and let $G=\Sp_4(q)$.
If $q\ne4$, then $G$ has no proper nontrivial ordinary
$G$-permutable subgroup.
Let $\beta$ be the symplectic form on $V=\F_4^4$.
For $G=\Sp_4(4)$, the subgroups are exactly $\gen t$ for involutions $t$ satisfying $\operatorname{rank}(t-I_V)=2$ and $\beta(v,tv)\ne0$ for some $v\in V$.
Their generating involutions have centralizer order $256$.
\end{theorem}
\begin{proof}
Let $V=\F_q^4$ with symplectic form $\beta$, and let $A<G$ be ordinary $G$-permutable.
We choose a prime $r\mid f$ and let $q_0=2^{f/r}$.
For this choice of $q_0$, the subfield subgroup $M=\Sp_4(q_0)$ is maximal
\cite[Table 8.14]{BHR}.
Moreover, for every $Y<G$, we have $MY\ne G$
\cite[Theorem 2.15 and Table A.2]{LiXia}.
The subfield rows with $q=4$ and $q=16$ in that table require
a proper almost simple extension of $G$.
Choose $g\in G$ with $A^g\leq M$, and replace $A$ by $A^g$.
Then, $|A|_2\leq q_0^4\leq q^2$.

We choose $T=T_-$ as in Lemma \ref{lem:symplectic-four-tori}.
Since $A$ is ordinary $G$-permutable, we choose $x\in G$ such that $K=AT^x$ is a subgroup.
As $|K|_2=|A|_2\leq q^2<q^4=|G|_2$, we have $K<G$.
Choose a maximal subgroup $M_K<G$ with $K\leq M_K$.
By Bereczky's theorem \cite{Bereczky}\cite[Theorem 2.1.7]{Brookfield2014}, $M_K$ has a normal subgroup $L\cong\SL_2(q^2)$ of index $2$.

The group $M_K$ normalizes $\SL_2(q^2)$ acting on $\F_{q^2}^2$, or $\Omega^-_4(q)\cong\SL_2(q^2)$.
Let $J=K\cap L$. Because $[M_K:L]=2$ and $|T^x|$ is odd, we have $T^x\leq J\unlhd K$.

Suppose that $J=L$. Then, $|K|_2\geq q^2$.
Since $q^2\leq|K|_2\leq q_0^4\leq q^2$, we have $r=2$ and $|K|_2=|A|_2=q^2$.
Since $|K|_2=|L|_2=q^2$, every Sylow $2$-subgroup of $K$ is elementary abelian.
Because $|A|_2=|K|_2=q_0^4=|M|_2$, a Sylow $2$-subgroup of $A$ is Sylow in both $K$ and $M=\Sp_4(q_0)$.
But Sylow $2$-subgroups of $M$ are nonabelian for $q_0\geq2$. We obtain a contradiction.
Thus, $J<L$.
Since $q^2\geq16$ and $|T^x|=q^2+1$, we obtain from Dickson's classification \cite{Dickson1901} that $J\leq N_L(T^x)$.
The group $N_L(T^x)$ is dihedral, with cyclic subgroup $T^x$ of odd order and index $2$.
As $T^x\leq J$, the elements of odd order in $J$ form $T^x$.
Thus, $T^x$ is characteristic in $J\unlhd K$, and we have $K\leq N_G(T^x)$.
Replace $A$ by $xAx^{-1}$, with $x\in G$ from $K=AT^x$.
Then, $A\leq N_G(T)$ and $|A|\leq4(q^2+1)$.

The point stabilizer $P_1$ has index $(q+1)(q^2+1)>|A|$.
Since $A$ is ordinary $G$-permutable, we choose $y\in G$ such that $AP_1^y$ is a subgroup.
As $|A|<[G:P_1]$, we have $|AP_1^y|<|G|$.
By maximality of $P_1^y$, we obtain that $A\leq P_1^y$.
Thus, $A$ stabilizes a line in $V$.
By Lemma \ref{lem:symplectic-singer-involutions}, we have $A\cap T=1$ and $A$ is cyclic with $|A|\mid4$.

Suppose that $|A|=4$, and choose $T_+$ from Lemma \ref{lem:symplectic-four-tori}.
We choose $z\in G$ such that $H=AT_+^z$ is a subgroup.
Since the Sylow $2$-subgroups of $H$ are cyclic, we use Burnside's normal-complement theorem.
We obtain a normal Hall subgroup of odd order, namely $T_+^z$.
Thus, $A\leq N_G(T_+^z)$.
But $N_G(T_+^z)$ has elementary abelian Sylow $2$-subgroups.
Consequently, $|A|\ne4$.
If $A\ne1$, then $A=\gen t$ has order $2$.
By Lemma \ref{lem:symplectic-four-tori}, $t$ has rank $2$ and $\beta(v,tv)\ne0$ for some $v\in V$.

Now, suppose that $q\geq8$. In a symplectic basis
$e_1,e_2,f_1,f_2$, we choose a generator $a\in\F_q^*$ and take
$h=\diag(a,a^2,a^{-1},a^{-2}), H=\gen h$.
Let $U=\{u_c:v\mapsto v+c\beta(v,e_1)e_1\mid c\in\F_q\}$
be the long root subgroup with centre line $\F_qe_1$.
Since $H$ normalizes $U$, we take the subgroup $B=UH$.
As $U=O_2(B)$, we have $g(\F_qe_1)=\F_qe_1$ for $g\in N_G(B)$.
By Schur--Zassenhaus, choose $u\in U$ with $H^{gu}=H$.
We replace $g$ by $gu$, so $g\in N_G(H)$.
For some integer $j$, we have $g^{-1}hg=h^j$.
On $\F_qe_1$, we have $a=a^j$.
Because $a$ generates $\F_q^*$, $g$ centralizes $h$.
The $4$ eigenvalues of $h$ are distinct.
Thus, $C_G(H)=D=\{\diag(b,c,b^{-1},c^{-1}):b,c\in\F_q^*\}$.
As $|D|=(q-1)^2$ is odd, every involution in $UD$ lies in $U$.
The torus $D$ normalizes $B$, so $N_G(B)=UD$.
Every involution in $N_G(B)=UD$ belongs to $U$ and has rank $1$.
Lemma \ref{lem:involution-normalizers} contradicts ordinary
$G$-permutability of $A=\gen t$.

For $q=4$, take an involution of rank $2$ $t$ with $\beta(v,tv)\ne0$ for some $v\in V$.
Let $L_t=\operatorname{im}(t-I_V)=\ker(t-I_V)$. As $\beta((t-I_V)v,(t-I_V)w)=0$ for $v,w\in V$, the subspace $L_t$ is Lagrangian.
On the quotient $V/L_t$, we use the symmetric nonsingular form $\gamma_t:V/L_t\times V/L_t\to\F_q$, where
$$
 \gamma_t(v+L_t,w+L_t)=\beta(v,(t-I_V)w).
$$
Since $\beta(v,(t-I_V)v)=\beta(v,tv)$, the form $\gamma_t$ is nonalternating.
Over the perfect field $\F_q$, the form $\gamma_t$ is congruent to $I_2$.
In a suitable symplectic basis, we have
$$
 t=\begin{pmatrix}I_2&I_2\\0&I_2\end{pmatrix}.
$$
For this matrix $t$, we compute
$$
 C_G(t)=\left\{\begin{pmatrix}P&PS\\0&P\end{pmatrix}:
 P^{\mathsf T}P=I_2,\ S^{\mathsf T}=S\right\}.
$$
There are $q^3$ choices for $S$ and $q$ choices for $P$.
The matrices satisfying $P^{\mathsf T}P=I_2$ are exactly
$$
 P=\begin{pmatrix}b&1+b\\1+b&b\end{pmatrix}, b\in\F_q.
$$
Therefore, $|C_G(t)|=q^4=256$ when $q=4$.
By Proposition \ref{prop:finite-tables}, $\gen t$ is ordinary $G$-permutable when $|C_G(t)|=256$.
\end{proof}

\subsection{\texorpdfstring{Involutions in $\Sp_{2m}(q)$ with even $q$}{Involutions in Sp2m(q) with even q}}

\begin{lemma}\label{lem:symplectic-even-nonsingular}
Let $q$ be even and $m\geq3$.
An involution of rank $m$ in $G=\Sp_{2m}(q)$ does not generate
an ordinary $G$-permutable subgroup.
\end{lemma}
\begin{proof}
Let $V=\F_q^{2m}$ with its symplectic form.
Firstly, let $k\geq3$ be odd. Let $L=\F_q^k$ and $W=L\oplus L^*$.
On $W$, we use the standard symplectic form.
Next, let $\mathcal A=\{S\in M_k(\F_q):S^{\mathsf T}=S,\ S_{ii}=0\text{ for }1\leq i\leq k\}$. We take
$$
 R=\left\{\begin{pmatrix}I&S\\0&I\end{pmatrix}:S\in\mathcal A\right\}, B_k=R\rtimes \GL_k(q).
$$
For the action of $\GL(L)$ on $W$, we use the injective homomorphism $\iota:\GL(L)\to\Sp(W)$
with $\iota(g)=\diag(g,g^{-\mathsf T})$.
Since $R=O_2(B_k)$ and $W^R=L$, every $g\in N_{\Sp(W)}(B_k)$ satisfies $gL=L$.
Let $R_S=\begin{pmatrix}I_k&S\\0&I_k\end{pmatrix}$ for each symmetric $S\in M_k(\F_q)$.
We have $\operatorname{Stab}_{\Sp(W)}(L)=\{R_S\iota(h):S^{\mathsf T}=S,\ h\in\GL_k(q)\}$.
For $g\in N_{\Sp(W)}(B_k)$, we choose $b\in B_k$ with $bg=R_S$.
Since $R_S\in N_{\Sp(W)}(B_k)$, we have
$R_S\iota(h)R_S^{-1}\iota(h)^{-1}=R_{S+hSh^{\mathsf T}}\in R$ for $h\in\GL_k(q)$.
Thus, $S+hSh^{\mathsf T}\in\mathcal A$ for every $h\in\GL_k(q)$.
Next, let $\mathcal S$ be the space of symmetric $k\times k$ matrices over $\F_q$.
On this space, let $\pi:\mathcal S\to\F_q^k$ send $S$ to $(S_{11},\ldots,S_{kk})^{\mathsf T}$.
We have $\ker\pi=\mathcal A$ and $\pi(hSh^{\mathsf T})=h^{(2)}\pi(S)$,
where $h^{(2)}=(h_{ij}^2)$.
Since $S+hSh^{\mathsf T}\in\ker\pi$, we have $h^{(2)}\pi(S)=\pi(S)$ for every $h\in\GL_k(q)$.
Because elementary transvections have no common nonzero fixed vector under $h^{(2)}$, we obtain that $\pi(S)=0$.
Consequently, $S\in\mathcal A$ and $N_{\Sp(W)}(B_k)=B_k$.

For $Q(x,y)=x^{\mathsf T}y$, we have $Q(gv)=Q(v)$ for $g\in B_k$ and $v\in W$.
For even $q$, the function
$D(g)=\operatorname{rank}(I-g)\pmod2$ on $\operatorname{O}(W,Q)$
is the Dickson homomorphism \cite[Proposition 1.6.11]{BHR}.
Since $S$ is alternating, $\operatorname{rank}(R_S-I_W)$ is even.
For $g\in\GL_k(q)$, the $2$ diagonal blocks of $\iota(g)-I_W$ have equal rank.
Thus, $D(R)=D(\iota(\GL_k(q)))=0$. As these groups generate $B_k$, we have $D(B_k)=0$.
An involution $t$ on $W$ has $\operatorname{rank}(t-I_W)\leq k$.
Since $k$ is odd and $D(t)=0$ for $t\in B_k$, every involution of $B_k$ has rank at most $k-1$.

If $m$ is odd, then use $B_m$ on $V$.
The normalizer $N_G(B_m)$ has no involution of rank $m$.
Suppose that $m$ is even.
We choose a nondegenerate subspace $W\leq V$ with $\dim_{\F_q}W=2(m-1)$. Let $Y=W^\perp$, so $V=W\perp Y$ and $\dim_{\F_q}Y=2$.
We use $B_{m-1}$ on $W$ and a nonsplit torus $C$ of order $q+1$ on $Y$.
A central element of $B_{m-1}$ has general linear component $aI_{m-1}$ for some $a\in\F_q^*$.
For an alternating matrix $S$, we have $\iota(aI_{m-1})R_S\iota(aI_{m-1})^{-1}=R_{a^2S}$.
Since a central element commutes with $R$, we have $a^2=1$.
In characteristic $2$, we have $a=1$.
No nonzero alternating matrix is fixed by every element of $\GL_{m-1}(q)$.
Therefore, $Z(B_{m-1})=1$.
Thus, the centre of $B=B_{m-1}\times C$ is precisely $C$.
As $C=Z(B)$ and $V^C=W$, every $g\in N_G(B)$ satisfies $gW=W$.
Because $Y=W^\perp$, we also have $gY=Y$.
Therefore, $N_G(B)=B_{m-1}\times N_{\Sp(Y)}(C)$.
The involutions in $N_G(B)$ have rank at most $(m-2)+1=m-1$.
By Lemma \ref{lem:involution-normalizers}, a rank-$m$ involution cannot generate an ordinary $G$-permutable subgroup.
\end{proof}

\subsection{\texorpdfstring{The groups $\Sp_{2m}(q)$ with even $q$ and $m\geq3$}{The groups Sp2m(q) with even q and m at least 3}}

Let $d_1,\ldots,d_t\geq1$ and $m=d_1+\cdots+d_t$.
A torus $T$ of positive type $(d_1,\ldots,d_t)$ acts on $V=V_1\perp\cdots\perp V_t$ with $\dim_{\F_q}V_i=2d_i$.
We have $T=\prod_iT_i$ with $|T_i|=q^{d_i}-1$.
On $V_i=U_i\oplus U_i'$, the spaces $U_i,U_i'$ are totally isotropic and have dimension $d_i$.
We use bases of $U_i,U_i'$ in which $\beta|_{U_i\times U_i'}$ has matrix $I_{d_i}$.
After identifying $U_i=\F_{q^{d_i}}$, $T_i$ acts by $M_\lambda:u\mapsto\lambda u$ on $U_i$ and by $(M_\lambda^{-1})^{\mathsf T}$ on $U_i'$ for $\lambda\in\F_{q^{d_i}}^*$.
For $j\ne i$, the group $T_i$ acts as $I_{V_j}$ on $V_j$.

\begin{theorem}\label{thm:symplectic-even}
Let $q$ be even and $m\geq3$.
Then, $\Sp_{2m}(q)$ has no proper nontrivial ordinary
$\Sp_{2m}(q)$-permutable subgroup.
\end{theorem}
\begin{proof}
For ease of reading, we divide the proof of the theorem into the following steps.

\setlist[itemize]{leftmargin=4em,labelsep=.5em,topsep=.2\baselineskip,partopsep=0pt,parsep=0pt,itemsep=0pt}
\begin{itemize}
\item[{\bf Step 1.} ] For every $2$-subgroup $P\leq K$, the $\F_qP$-module $V$ is decomposable.
\end{itemize}
For $G=\Sp_6(2)$, we use Proposition \ref{prop:finite-tables}.
Let $A<G=\Sp_{2m}(q)$ be ordinary $G$-permutable, with $(m,q)\ne(3,2)$.
Next, let $P_2$ stabilize a totally isotropic subspace $U\leq V=\F_q^{2m}$ with $\dim_{\F_q}U=2$.
The subgroup $P_2$ is maximal. For every $Y<G$, we have $P_2Y\ne G$ \cite[Theorem 2.15 and Table A.2]{LiXia}.
By Lemma \ref{lem:max}, we choose $z\in G$ with $A\leq P_2^z$.
Let $U_A=z^{-1}U$.
Then, $aU_A=U_A$ for every $a\in A$, and $AY\ne G$ for every $Y<G$.
Let $M$ stabilize $V_1\perp V_2$, with $\dim_{\F_q}V_1=2$.
The subgroup $M$ is maximal \cite[Theorem 4.3]{King2005}.
By Lemma \ref{lem:max}, we choose $y\in G$ with $A\leq M^y$.

Let $T$ be a Singer torus of order $q^m+1$.
Since $A$ is ordinary $G$-permutable, we choose $x\in G$ such that $K=AT^x$ is a subgroup.
As $AT^x\ne G$, we have $K<G$.
Because $|T|$ is odd, $|K|_2=|A|_2$.
Thus, a Sylow $2$-subgroup of $A$ is Sylow in $K$.
For every $2$-subgroup $P\leq K$, Sylow's theorem permits a choice of $g\in K$ with $P^g\leq A$.
Since $A\leq M^y$, the group $P$ stabilizes $g(y^{-1}V_1)\perp g(y^{-1}V_2)$.
No $2$-subgroup of $K$ acts indecomposably on $V$.

\begin{itemize}
\item[{\bf Step 2.} ] We have $K\cap H<H$ for each nonabelian relative group $H=\Sp_{2a}(Q)$, $\Omega^-_{2a}(Q)$ or $\GU_b(Q)$.
\end{itemize}
We use the normal base intersections from the proof of Lemma \ref{lem:symplectic-odd-singer}.
Every intersection contains $T^x$.
Next, by Zsigmondy's theorem, we choose a primitive divisor $\ell$ of $q^{2m}-1$ with $\ell>2m$.
Let $T_\ell$ be the Sylow $\ell$-subgroup of $T^x$.
Then, $C_{\GL_{2m}(q)}(T_\ell)=\F_{q^{2m}}^*$ and $C_G(T_\ell)=T^x$.
For each scalar field $F_0$ in the recursion, the normal subgroup of $F_0$-linear maps is symplectic, orthogonal of minus type, or unitary of odd dimension
\cite{Bereczky}. See \cite[Theorem 2.1.7 and Table 2.3]{Brookfield2014}.
At an orthogonal stage, let $D\leq\operatorname{O}^-_{2a}(Q)$ denote the intersection obtained at that step.
We have $Q=q^s$ and $as=m$ for integers $a,s\geq1$.
Next, we replace $D$ by $D\cap\Omega^-_{2a}(Q)\unlhd D$.
Since $|T^x|$ is odd, we have $T^x\leq D\cap\Omega^-_{2a}(Q)$.
At each step, let $F_0$ be the scalar field of the normal subgroup.
The field $F_0$ centralizes the subgroup of order $\ell$ and lies in $\F_{q^{2m}}$.
As $T^x$ centralizes $F_0$, its elements are $F_0$-linear.
Because $|T^x|$ is odd, its Dickson invariant is zero.
Thus, every normal intersection contains $T^x$.

By Lemmas \ref{lem:symplectic-unitary-sylow} and \ref{lem:orthogonal-minus-two-subgroup}, each relative classical group contains an indecomposable $2$-subgroup.
Thus, no intersection equals the relative classical group.
For $\GU_1(Q)$ and $\Omega^-_2(Q)$, the relative group is the torus itself.
At a relative $\SL_2(Q)$ stage, let $H$ denote the intersection.
Since $T^x\leq H<\SL_2(Q)$, we have $H\leq N_{\SL_2(Q)}(T^x)$ \cite{Dickson1901}.
A relative field of size $2$ requires the original field $\F_2$ and extension degree $1$.
A relative group $\GU_3(2)$ or $\Omega^-_6(2)$ would require $(m,q)=(3,2)$.
As $(m,q)\ne(3,2)$, neither group is a relative group in this proof.
$2$ additional exceptional branches require attention.

\begin{itemize}
\item[{\bf Step 3.} ] For the exceptional subgroups $L\cong\PSL_2(17)$ and $M=C_3\times\PSL_2(11)$, we have $K\cap L<L$ and $K\cap M<M$.
\end{itemize}
In $\Sp_8(2)$, we also consider a maximal subgroup $L\cong\PSL_2(17)$ containing $T^x$
\cite[Table 8.49]{BHR}.
A proper subgroup of $L$ containing $T^x=C_{17}$ lies in $C_{17}\rtimes C_8$
and normalizes $T^x$.
If $L\leq K$, then $L=(A\cap L)T^x$.
Thus, $|A\cap L|$ is divisible by $144$.
No proper maximal subgroup in Dickson's list has such an order,
so $A\cap L=L$.
Since $L$ contains the irreducible $C_{17}$, no subspace of $V$ with dimension $2$ over $\F_2$ is $L$-invariant.
But $L\leq A$ and $aU_A=U_A$ for every $a\in A$, a contradiction.

In $\GU_5(2)$, we also consider a maximal subgroup $M=C_3\times\PSL_2(11)$ containing $T^x$
\cite[Table 8.21]{BHR}.
The torus $T^x=C_{33}$ contains the central $C_3$.
The intersection has form $C_3\times H$.
If $H$ is proper, then $H\leq C_{11}\rtimes C_5$ and normalizes $T^x$.
Suppose that $M\leq K$.
Let $\pi_M:C_3\times\PSL_2(11)\to\PSL_2(11)$ send $(c,h)$ to $h$.
Let $H_0=\pi_M(A\cap M)$.
Since $M=(A\cap M)T^x$, we have $\pi_M(M)=H_0\pi_M(T^x)=H_0C_{11}=\PSL_2(11)$.
By Dickson's list, we have $H_0=\PSL_2(11)$ or $H_0\cong A_5$.
Since the first factor of $M=C_3\times\PSL_2(11)$ is abelian, we have $(A\cap M)'\leq1\times H_0$.
As $H_0'=H_0$ and $\pi_M((A\cap M)')=H_0$, we obtain that $(A\cap M)'=1\times H_0\leq A\cap M$.
If $H_0=\PSL_2(11)$, then the subgroup $1\times H_0\leq A$ contains the irreducible $C_{11}$ on $V=\F_2^{10}$.
But every element of $1\times H_0$ maps $U_A$ to $U_A$, a contradiction.

Suppose that $H_0=A_5$.
The restriction map $1\times H_0\to\GL(U_A)$ sends $h$ to $h|_{U_A}$.
Since $A_5$ has no nontrivial quotient in $\GL_2(2)\cong S_3$, we have $h|_{U_A}=I_{U_A}$ for every $h\in1\times H_0$.
Thus, $1\times H_0$ fixes the $\F_4$-span $\F_4U_A$ pointwise.
But $A_5$ cannot fix a singular $\F_4$-line.
A singular-line parabolic has a $2$-group radical and Levi
$\GL_1(4)\times\GU_3(2)$, whose order is not divisible by $5$.
If $\dim_{\F_4}V^{A_5}\geq2$, then $V^{A_5}$ contains a singular line.
Thus, $\dim_{\F_4}V^{A_5}\leq1$, and any fixed $\F_4$-line is nondegenerate.
Because $\F_4U_A\leq V^{A_5}$ and $U_A\ne0$, we have $\dim_{\F_4}(\F_4U_A)=1$.
As $\dim_{\F_2}U_A=2$, the underlying $\F_2$-space of $\F_4U_A$ equals $U_A$.
The symplectic form is the trace of the Hermitian form, so $U_A$ is symplectically nondegenerate.
This contradicts the total isotropy of $U_A$.

For $L\cong\PSL_2(17)<\Sp_8(2)$, we have $K\cap L<L$.
For $M=C_3\times\PSL_2(11)<\GU_5(2)$, we have $K\cap M<M$.
\begin{itemize}
\item[{\bf Step 4.} ] We have $A=1$.
\end{itemize}
Now, we denote the successive normal intersections by $H_0=K,H_1,\ldots,H_d$.
The Sylow $\ell$-subgroup $T_\ell$ of $T^x$ is the unique Sylow $\ell$-subgroup of $H_d$.
We have $C_{H_d}(T_\ell)=T^x$.
By the proof of Lemma \ref{lem:symplectic-odd-singer}, the Sylow $\ell$-subgroup $T_\ell$ of $T^x$ is normal in $K$.
Since $C_G(T_\ell)=T^x$, we have $K\leq N_G(T^x)$.
By Lemma \ref{lem:max} for the maximal point stabilizer $P_1$, $A$ stabilizes a line $\Lambda\leq V$.
From Lemma \ref{lem:symplectic-singer-involutions}, we obtain that $A$ is cyclic with $|A|\mid2m$.

If an odd prime $r$ divides $|A|$, then choose a primitive divisor of
$q^{2m-2}-1$.
The only missing case is $(m,q)=(4,2)$, where $|A|\mid8$
and $|A|$ has no odd prime divisor.
The Sylow-count and semisimple multiplicity arguments in Theorem
\ref{thm:symplectic-odd-reduction} contradict the existence of $r$.
Thus, $A$ is a cyclic $2$-group.
If $m$ is odd, then $|A|\leq2$.
For even $m$, use the torus $B$ of positive type $(m-1,1)$.
We have $|B|=(q^{m-1}-1)(q-1)$ odd.
Since the $2$ blocks have different dimensions, the Sylow $2$-subgroups of $N_G(B)$ are elementary abelian of order $4$.
For $q=2$, the plane torus is trivial.
On the fixed complementary plane, the group $\Sp_2(2)$ has Sylow $2$-subgroup of order $2$.
The other block contributes a Sylow $2$-subgroup of order $2$.
Thus, the Sylow $2$-subgroups of $N_G(B)$ are elementary abelian of order $4$ also for $q=2$.
As $A$ is ordinary $G$-permutable, we choose $y\in G$ such that $AB^y$ is a subgroup.
By Burnside's normal-complement theorem, $B^y\unlhd AB^y$.
Therefore, $A\leq N_G(B^y)$ and $|A|\leq2$.
If $|A|=2$, then the generator of $A$ has rank $m$ by Lemma
\ref{lem:symplectic-singer-involutions}.
By Lemma \ref{lem:symplectic-even-nonsingular}, we conclude that $A=1$.
\end{proof}

\begin{corollary}\label{cor:symplectic-even}
Let $G=\Sp_{2m}(q)$ be simple, with $q$ even and $m\geq2$.
Then, only $\Sp_4(4)$ has proper nontrivial ordinary $G$-permutable subgroups.
For $G=\Sp_4(4)$, the proper nontrivial ordinary $G$-permutable subgroups are exactly those of order $2$
whose generating involutions have centralizer order $256$.
\end{corollary}
\begin{proof}
Combine Theorems \ref{thm:symplectic-four-even} and
\ref{thm:symplectic-even}.
\end{proof}

\subsection{\texorpdfstring{Reduction for $\PSp_{2m}(q)$ with odd $q$}{Reduction for PSp2m(q) with odd q}}

\begin{lemma}\label{lem:symplectic-complex-five}
Let $q\equiv5\pmod{24}$ and $m\geq2$. Let $G=\PSp_{2m}(q)$.
If $t\in G$ has a lift $h\in\Sp_{2m}(q)$ with $h^2=-I$,
then $\gen t$ is not ordinary $G$-permutable.
\end{lemma}
\begin{proof}
In $\F_q$, the element $-1$ is a square and $2$ is a nonsquare.
Also, $q\equiv-1\pmod3$.
Let $V=\F_q^{2m}$ with its symplectic form.
Next, we choose $V=X\perp Y\perp Z$ with nondegenerate summands and $\dim_{\F_q}X=\dim_{\F_q}Y=2$.
On $X$, take matrices $i,j$ with
$i^2=j^2=-I, ij=-ji$.
For example, take $i=\diag(s,-s)$ with $s^2=-1$ and
$j=\begin{pmatrix}0&1\\-1&0\end{pmatrix}$.
Let $Q=\gen{i,j}\cong Q_8$ and
$u=\tfrac12(-I+i+j+ij), L=\gen{Q,u}\cong\SL_2(3)$.
The element $u$ has order $3$ and permutes the $3$ cyclic subgroups
of order $4$ in $Q$. Consequently, $L'=Q$ and $L/L'\cong C_3$.

Firstly, we show that $N_{\SL_2(q)}(L)=L$.
We have $C_{\GL_2(q)}(Q)=\{\lambda I_2:\lambda\in\F_q^*\}$.
To prove that the map $\Aut(L)\to\Aut(Q)$ sending $\alpha$ to $\alpha|_Q$ is injective, consider an automorphism fixing $Q$ pointwise.
It must send $u$ to $u$ or $-u$.
Since $-u$ has order $6$, the automorphism fixes $u$ and hence every element of $L=\gen{Q,u}$.
There are $6$ choices for the image of $i$, and each allows $4$ choices for the image of $j$.
Thus, $|\Aut(Q)|\leq24$.
The inner automorphisms of $L$ induce $12$ automorphisms of $Q$.
Let $w=i+j\in\GL_2(q)$.
Then, $wiw^{-1}=j$, $wjw^{-1}=i$ and $w(ij)w^{-1}=-ij$.
We use the expression
$$
 L=Q\ \cup\
 \left\{\frac{\varepsilon_0I+\varepsilon_1i+\varepsilon_2j+
       \varepsilon_3ij}{2}\ \middle|\
       \varepsilon_k\in\{1,-1\},\ 0\leq k\leq3\right\}.
$$
Under $g\mapsto wgw^{-1}$, the set $Q$ and the set of $16$ displayed elements each map to themselves.
Therefore, $wLw^{-1}=L$.
The maps $g\mapsto lgl^{-1}$ with $l\in L$ induce a cyclic group of order $3$ on $\{\gen i,\gen j,\gen{ij}\}$.
As $w$ induces a transposition on this set, no $l\in L$ satisfies $lgl^{-1}=wgw^{-1}$ for every $g\in Q$.
Thus, $L$ and $wL$ induce $24$ distinct automorphisms of $Q$.
Since $|\Aut(Q)|\leq24$ and $C_{\GL_2(q)}(Q)=\F_q^*$, we obtain that
$N_{\GL_2(q)}(L)=\F_q^*L\ \cup\ \F_q^*(i+j)L$.
Because $\det(i+j)=2$ is a nonsquare, no $g\in\F_q^*(i+j)L$ has $\det g=1$.
We have $\F_q^*L\cap\SL_2(q)=L$.

Realize $L/L'\cong C_3$ irreducibly on $Y$, obtaining
$\chi:L\to\Sp(Y)$ with kernel $L'$.
Since $q\equiv-1\pmod3$, the image $\chi(L)$ lies in a nonsplit torus of $\Sp(Y)$.
We take
$B=\{\diag(b,\chi(b),I_Z):b\in L\}\leq\Sp(V)$.
Let $\pi:\Sp(V)\to G$ be the quotient map and $\overline B=\pi(B)$.
As $B\cap\ker\pi=1$, the restriction $\pi|_B:B\to\overline B$ is an isomorphism.
For $g\in\Sp(V)$ with $\pi(g)\in N_G(\overline B)$, let
$\varphi_g(b)=(\pi|_B)^{-1}(\pi(g)^{-1}\pi(b)\pi(g))$ for $b\in B$.
There is a homomorphism $\delta_g:B\to\{I_V,-I_V\}$ such that
$$
 g^{-1}bg=\delta_g(b)\varphi_g(b)\ (b\in B).
$$
Because $B/B'\cong C_3$, we have $\delta_g=1$.
Thus, $g^{-1}Bg=B$.
We have $X=\sum_{b\in B'}(b-I_V)V$, $Y=\sum_{b\in B}(b-I_V)V^{B'}$ and $Z=V^B$.
Thus, $gX=X,\ gY=Y,\ gZ=Z$.
The restriction $g|_X$ lies in $N_{\Sp(X)}(L)=L$.
For every $l\in L$, we have $(g|_X)^{-1}l(g|_X)L'=lL'$.
Since $g^{-1}Bg=B$ and $\ker\chi=L'$, we obtain that
$(g|_Y)^{-1}\chi(l)(g|_Y)=\chi((g|_X)^{-1}l(g|_X))=\chi(l)$.
Thus, $g|_Y$ centralizes $\chi(L)$.
The centralizer $C_{\Sp(Y)}(\chi(L))$ is cyclic of order $q+1$.
As $q+1\equiv6\pmod{24}$, the centralizer has no element of order $4$.
Therefore, no $g\in\Sp(V)$ with $\pi(g)\in N_G(\overline B)$ satisfies $g^2=-I_V$.
By Lemma \ref{lem:involution-normalizers}, we conclude that $\gen t$ is not ordinary $G$-permutable.
\end{proof}

\begin{lemma}\label{lem:symplectic-complex-extension}
Let $q=p^f$ with odd $p$ and $f\geq2$.
Let $m\geq2$.
In $G=\PSp_{2m}(q)$, an involution with a lift of square $-I$
does not generate an ordinary $G$-permutable subgroup.
\end{lemma}
\begin{proof}
Firstly, we use the basis $(e_1,e_2,f_2,f_1)$ of $V_0=\F_q^4$.
Let $\beta$ be the symplectic form with $\beta(e_i,f_j)=\delta_{ij}$ and $\beta(e_i,e_j)=\beta(f_i,f_j)=0$.
The standard upper unipotent group $U$ consists of the matrices
$$
 M(x,y,z,w)=
 \begin{pmatrix}
 1&x&xy+z&w+xz\\0&1&y&z\\0&0&1&-x\\0&0&0&1
 \end{pmatrix}, x,y,z,w\in\F_q.
$$
For the product of $M(x,y,z,w)$ and $M(x',y',z',w')$, we have
$$
 x''=x+x',\  y''=y+y',\  z''=z+z'-yx',\  w''=w+w'-2zx'+yx'^2.
$$
By the product formula, the highest-root subgroup $W=\{M(0,0,0,w):w\in\F_q\}$ is central.
Next, we identify $U/W$ with $\F_q^3$ by $M(x,y,z,w)W\mapsto(x,y,z+xy/2)$.
Let $\pi_W:U\to U/W$ be the quotient map, and let $c=z+xy/2$.
The quotient law is
$(x,y,c)(x',y',c') =(x+x',y+y',c+c'+(xy'-yx')/2)$.
Thus, $S=\{(t,t,t^p):t\in\F_q\}$ is a subgroup of $U/W$.
We take $H=\pi_W^{-1}(S)$, namely
$$
 H=\{M(t,t,t^p-t^2/2,w):t,w\in\F_q\}.
$$
Since $t,w$ vary independently in $\F_q$, we have $|H|=q^2$.
The formula uses only division by $2$, so also holds in characteristic $3$.

For $h=M(t,t,t^p-t^2/2,w)\in H$ with $t\ne0$, we have
$$
 (h-I)_{12}=t,\ (h-I)_{23}=t,\ (h-I)_{34}=-t.
$$
Hence, $((h-I)^3)_{14}=-t^3\ne0$.
For $h\in H$ with $t\ne0$, we have
$$
 \ker(h-I)=\gen{e_1},\  \ker((h-I)^2)=\gen{e_1,e_2},\  \ker((h-I)^3)=\gen{e_1,e_2,f_2}.
$$
Elements with $t=0$ are not regular.
For $g\in N_{\Sp_4(q)}(H)$, the element $ghg^{-1}\in H$ is regular.
Its kernels equal the displayed subspaces. Hence, $g\ker((h-I)^j)=\ker((h-I)^j)$ for $j=1,2,3$.
Thus, $g$ lies in the standard Borel subgroup.
Suppose that $g\in N_{\Sp_4(q)}(H)$ satisfies $g^2=-I_{V_0}$.
Since $g$ lies in the standard Borel subgroup, we have $g=uT$ with $u\in U$.
We take $T=\diag(t_1,t_2,t_2^{-1},t_1^{-1})$ with $t_1,t_2\in\F_q^*$.
From $g^2=-I$, we have $t_1^2=t_2^2=-1$.
On $U/W$, the map $hW\mapsto ThT^{-1}W$ sends
$$
 (x,y,c)\longmapsto((t_1/t_2)x,t_2^2y,t_1t_2c).
$$
If $\pi_W(u)=(a,b,d)$, then the map $hW\mapsto uhu^{-1}W$ sends
$(x,y,c)$ to $(x,y,c+ay-bx)$.
As $gHg^{-1}=H$ and $\pi_W(H)$ satisfies $x=y$, we have $t_1/t_2=t_2^2=-1$.
Thus, $t_1t_2=1$.
Therefore, $hW\mapsto ghg^{-1}W$ sends
$(t,t,t^p)$ to $(-t,-t,t^p+\mu t)$ for some $\mu\in\F_q$.
Because $gHg^{-1}=H$, we would need $2t^p+\mu t=0$ for every $t\in\F_q$.
For $t=1$, we have $\mu=-2$.
It would follow that $t^p=t$ for every $t$, contrary to $f\geq2$.
Thus, no $g\in N_{\Sp(V_0)}(H)$ satisfies $g^2=-I_{V_0}$.

Let $V=\F_q^{2m}=V_0\perp C$, where $C=V_0^\perp$.
We use the injective homomorphism $\iota_C:H\to\Sp(V_0\perp C)$ with $\iota_C(h)=h\oplus I_C$.
Then, we identify $H$ with $\iota_C(H)$.
Since $V_0^H=\gen{e_1}$ and $H$ fixes $C$ pointwise, its fixed space is $F=\gen{e_1}\oplus C$, with radical $\gen{e_1}$.
Suppose that $g\in N_{\Sp(V)}(H)$ satisfies $g^2=-I_V$.
Since $|g|=4$ is prime to $p$, $F$ has a $g$-invariant complement $C'$ to $\gen{e_1}$.
The complement $C'$ is nondegenerate and fixed pointwise by $H$.
We choose $w\in C$ such that $C'=\{c+\beta(w,c)e_1:c\in C\}$.
For this $w$, let $a\in\Sp(V)$ fix $e_1,e_2,f_2$ and satisfy
$a(c)=c+\beta(w,c)e_1$ for $c\in C$ and $a(f_1)=f_1+w$.
From these formulas and the matrices of $H$, we obtain that $a$ commutes with $H$.
Because $a(C)=C'$ and $gC'=C'$, we have $a^{-1}ga(C)=C$.
As $V_0=C^\perp$, we obtain that $a^{-1}ga(V_0)=V_0$.
Thus, $(a^{-1}ga)|_{V_0}$ normalizes $H$ and has square $-I_{V_0}$.
This contradicts the argument on $V_0=\F_q^4$.

For the quotient map $\pi:\Sp(V_0\perp C)\to G$, let $\overline H=\pi(H)$.
Then, $\pi^{-1}(\overline H)=H\times\gen{-I}$.
The unique odd Hall subgroup $H$ of $\pi^{-1}(\overline H)$ is characteristic.
Every $g\in\Sp(V)$ with $\pi(g)\in N_G(\overline H)$ normalizes $H$.
By Lemma \ref{lem:involution-normalizers}, we conclude that no involution with a lift of square $-I_V$ generates an ordinary $G$-permutable subgroup.
\end{proof}

\begin{lemma}\label{lem:symplectic-complex-uniform}
Let $q$ be odd and $m\geq3$.
In $G=\PSp_{2m}(q)$, an involution with a lift of square $-I$
does not generate an ordinary $G$-permutable subgroup.
\end{lemma}
\begin{proof}
Let $V_0=\F_q^6$ with symplectic form $\beta$.
We use a symplectic basis $(e_1,e_2,e_3,f_3,f_2,f_1)$ and the upper unipotent subgroup $U\leq\Sp(V_0)$.
Then, we use the matrices $A,B,C$ for simple roots and $D,E$ for roots of height $2$.
$$
 A=E_{12}-E_{56}, B=E_{23}-E_{45}, C=E_{34}, D=E_{13}-E_{46}, E=E_{24}+E_{35}.
$$
For $R\in\{A,B,C,D,E\}$ and $t\in\F_q$, let $x_R(t)=I+tR$.
Then, $[A,B]=D$ and $[B,C]=E$.
Next, let $U_3$ be generated by the $4$ root groups of height at least $3$.
The subgroup $U_3$ is normal in $U$ and has order $q^4$.
Since the commutators of a root group of height $2$ with $U$ lie in $U_3$, both root groups of height $2$ are central modulo $U_3$.
For the quotient map $\pi_3:U\to U/U_3$, let
$Q=\pi_3^{-1}(\{x_D(t)x_E(t)U_3:t\in\F_q\})$.
Then, $Q$ is normal of order $q^5$.

In the root order $A,B,C,D,E$, the product coordinates satisfy
$d''=d+d'-a'b$ and $e''=e+e'-b'c$.
With the coordinate $z=d-e$, the preceding product formulas yield that multiplication in $U/Q$ is
$(a,b,c,z)(a',b',c',z')= (a+a',b+b',c+c',z+z'-a'b+b'c)$.
We have $Z(U/Q)=\{(a,0,-a,z):a,z\in\F_q\}$.
Now, let $S=\{(a,b,2b-a,a+b^2-ab):a,b\in\F_q\}$.
By the multiplication law, $S$ is abelian and $|S|=q^2$.
For the quotient map $\pi_Q:U\to U/Q$, we take $H=\pi_Q^{-1}(S)$.
We have $|H|=q^7$.
The central intersection of $S$ is the line
$S\cap Z(U/Q)=\{(a,0,-a,a):a\in\F_q\}$.
For $u\in U$ and $z\in S\cap Z(U/Q)$, $\pi_Q(u)^{-1}z\pi_Q(u)=z$.

With $a=b=1$, we choose $u\in H$ such that
$$
 (u-I)_{12}=(u-I)_{23}=(u-I)_{34}=1,\  (u-I)_{45}=(u-I)_{56}=-1.
$$
The element $u$ has $1$ Jordan block of length $6$, so $\Fix(H)=\gen{e_1}$.
On $\gen{e_1}^\perp/\gen{e_1}$, the induced operator of $u$ has $1$ block of length $4$.
On $\gen{e_1,e_2}^\perp/\gen{e_1,e_2}$, it has $1$ block of length $2$.
Since $H$ is upper unitriangular and each induced operator has a unique fixed line, the common $H$-fixed lines on the $2$ quotients are
$\gen{e_2+\gen{e_1}}$ and $\gen{e_3+\gen{e_1,e_2}}$.
Thus, every $g\in N_{\Sp_6(q)}(H)$ stabilizes
$\gen{e_1}<\gen{e_1,e_2}<\gen{e_1,e_2,e_3}$.
Let $T=\{\diag(t_1,t_2,t_3,t_3^{-1},t_2^{-1},t_1^{-1}):t_1,t_2,t_3\in\F_q^*\}$.
We have $g\in TU$.

Suppose that $g\in N_{\Sp_6(q)}(H)$ has square $-I$.
The first $3$ diagonal entries of $g$ are $i$ or $-i$, where $i^2=-1$.
Since $|g|$ is prime to the characteristic, choose $u_0\in U$ with $h=g^{u_0}\in T$.
Then, $h\in N_{\Sp_6(q)}(H^{u_0})$.
The diagonal matrix $h$ has the form
$$
 h=\diag(t_1,t_2,t_3,t_3^{-1},t_2^{-1},t_1^{-1}),\  t_j\in\F_q^*,\  t_j^2=-1\ (1\leq j\leq3).
$$
For conjugation by $h$, the signs on the simple root groups are
$$
 s_A=t_1/t_2,\  s_B=t_2/t_3,\  s_C=t_3^2=-1.
$$
Let $\pi_{\mathrm s}:U\to\F_q^3$ be the homomorphism
$$
 x_A(a)x_B(b)x_C(c)x_D(d)x_E(e)u_3\longmapsto(a,b,c),\ a,b,c,d,e\in\F_q,\ u_3\in U_3.
$$
By the equation $a-2b+c=0$ for $H$, we have $\pi_{\mathrm s}(H)=\{(a,b,c):a-2b+c=0\}$.
For $v\in U$, we have $\pi_{\mathrm s}(u_0^{-1}vu_0)=\pi_{\mathrm s}(v)$.
Thus, $\pi_{\mathrm s}(H^{u_0})=\pi_{\mathrm s}(H)$.
As $h$ normalizes $H^{u_0}$, its diagonal action on $\F_q^3$ fixes the plane $a-2b+c=0$.
Because all $3$ coefficients are nonzero, we have $s_A=s_B=s_C=-1$.
The signs on roots of height $2$ are $s_D=s_As_B=1$ and $s_E=s_Bs_C=1$.
Therefore, $hQh^{-1}=Q$.
Because $Q\unlhd U$ and conjugation by $u_0$ fixes $Z(U/Q)$ pointwise, we have
$\pi_Q(H^{u_0})\cap Z(U/Q)=S\cap Z(U/Q)$.
On $U/Q$, the map $vQ\mapsto hvh^{-1}Q$ sends $(a,b,c,z)$ to $(-a,-b,-c,z)$.
It sends $(1,0,-1,1)$ to $(\!-1,0,1,1)$, outside this central intersection.
But $h$ normalizes $H^{u_0}$, so it must map that intersection to itself. We obtain a contradiction.

For $m=3$, let $V=V_0$ and $C=0$.
For $m>3$, let $V=\F_q^{2m}=V_0\perp C$ with
$V_0=\gen{e_1,e_2,e_3,f_3,f_2,f_1}$.
Let $\iota_C:H\to\Sp(V)$ satisfy $\iota_C(h)=h\oplus I_C$.
We identify $H$ with $\iota_C(H)$.
Suppose that $g\in N_{\Sp(V)}(H)$ satisfies $g^2=-I_V$.
By the complement argument in Lemma \ref{lem:symplectic-complex-extension},
we choose $w\in C$ and $a\in C_{\Sp(V)}(H)$ as follows.
The map $a$ fixes $e_1,e_2,e_3,f_3,f_2$ and satisfies
$a(c)=c+\beta(w,c)e_1$ for $c\in C$ and $a(f_1)=f_1+w$.
We have $a^{-1}ga(V_0)=V_0$.
Thus, $(a^{-1}ga)|_{V_0}$ normalizes $H$ and has square $-I_{V_0}$.
We obtain a contradiction.
The subgroup $H$ is characteristic in $H\times\gen{-I}$.
For the quotient map $\pi:\Sp(V_0\perp C)\to G$, every $g$ with $\pi(g)\in N_G(\pi(H))$ normalizes $H$.
By Lemma \ref{lem:involution-normalizers}, we conclude that no involution with a lift of square $-I$ generates an ordinary $G$-permutable subgroup.
\end{proof}

\begin{theorem}\label{thm:symplectic-odd-reduction}
Let $q$ be odd, $m\geq2$ and $G=\PSp_{2m}(q)$.
Suppose that $1<A<G$ is ordinary $G$-permutable.
Then, $m=2$ and $|A|=2$.
Moreover, $q$ is prime and $q\equiv1,13,17\pmod{24}$.
The generator of $A$ has a lift in $\Sp_4(q)$ with square $-I$.
\end{theorem}
\begin{proof}
By Corollary \ref{cor:all-unitary}, $\PSp_4(3)\cong\PSU_4(2)$ has no proper nontrivial ordinary $\PSp_4(3)$-permutable subgroup.
Let $1<A<G$ be ordinary $G$-permutable and $(m,q)\ne(2,3)$.
Let $p=\operatorname{char}\F_q$, so $q=p^f$.
For $m\geq3$, the parabolic $P_2$ satisfies $P_2Y\ne G$ for every $Y<G$
\cite[Theorem 2.15 and Table A.2]{LiXia}.
By Lemma \ref{lem:max}, we choose $x\in G$ with $A\leq P_2^x$.
Hence, $AY\ne G$ for every $Y<G$.
Let $M$ stabilize $V_1\perp V_2$, with $\dim_{\F_q}V_1=2$.
The subgroup $M$ is maximal \cite[Theorem 4.3]{King2005}.
By Lemma \ref{lem:max}, we choose $y\in G$ with $A\leq M^y$.
For $m=2$, we use the maximal decomposition stabilizer
$M=(\Sp_2(q)\wr S_2)/\gen{-I}$
\cite[Theorem 4.7]{King2005}.
For every $Y<G$, we have $MY\ne G$.
By Lemma \ref{lem:max}, we obtain that $AY\ne G$ for every $Y<G$.
A Sylow $p$-subgroup of the inverse image of $A$ stabilizes $V_1\perp V_2$.
We have $\dim_{\F_q}V_1=2$ in both dimensions $2m\geq6$ and $2m=4$.
Since $p$ is odd, a $p$-group acts trivially on $\{V_1,V_2\}$ in dimension $4$.
Thus, every element of this Sylow subgroup fixes each $V_i$ setwise.

Let $T$ be a projective Singer torus.
Since $A$ is ordinary $G$-permutable, we choose $x\in G$ such that $K=AT^x$ is a subgroup.
As $AT^x\ne G$, we have $K<G$.
For the quotient map $\pi:\Sp_{2m}(q)\to G$, let
$\widehat A=\pi^{-1}(A)$ and $\widehat K=\pi^{-1}(K)$.
Because $T$ has order prime to $p$, we have
$|\widehat A|_p=|\widehat K|_p$.
Let $P$ be a Sylow $p$-subgroup of $\widehat A$.
Then, $P$ is Sylow in $\widehat K$.
For every $p$-subgroup $Q\leq\widehat K$, some $g\in\widehat K$ satisfies $Q^g\leq P$.
By Lemma \ref{lem:symplectic-odd-singer}, we have $A\leq N_G(T^x)$.
For the maximal point stabilizer, we obtain from Lemma \ref{lem:max} that $\widehat A$ stabilizes a line $\Lambda\leq\F_q^{2m}$.
From Lemma \ref{lem:symplectic-singer-involutions}, we conclude that $A$ is cyclic and $|A|\mid2m$.
Every involution $t\in A$ has a lift $h\in\Sp_{2m}(q)$ with $h^2=-I$.

Suppose that $m\geq3$.
We choose a primitive prime divisor $\ell$ of $q^{2m-2}-1$.
Then, $\ell\geq2m-1$ and $\ell\nmid|A|$.
We take $B=C_\ell$ on $U\perp W=\F_q^{2m}$, with $\dim_{\F_q}U=2m-2$ and $\dim_{\F_q}W=2$.
The spaces $U,W$ are nondegenerate, and $B$ fixes $W$ pointwise.
Since $A$ is ordinary $G$-permutable, we choose $y\in G$ such that $AB^y$ is a subgroup.
The Sylow $\ell$-count in $AB^y$ divides
$|A|\leq2m$ and is $1$ modulo $\ell$.
If the Sylow count exceeds $1$, then $\ell=2m-1$ and $|A|=2m=\ell+1$.
The $\ell+1$ Sylow $\ell$-subgroups have pairwise trivial intersection and together contain $\ell^2$ elements.
Thus, exactly $\ell$ elements lie outside their union. Since $\ell\nmid|A|$ and $|A|=\ell+1$, those elements form $A\setminus\{1\}$.
Consequently, $A$ is characteristic in $AB^y$ and in particular normal.
Since $\ell>\varphi(\ell+1)$, the subgroup $B^y$ centralizes $A$.
This contradicts the exceptional Sylow count.
Hence, $B^y$ is normal in every case.

Let $a\in A$ have odd prime order $r$.
Since $r\mid m$ and $|N_G(B)/C_G(B)|$ divides $2m-2$,
the element $a$ centralizes $B^y$.
Choose a lift $h\in\Sp_{2m}(q)$ of $a$ with $|h|=r$.
Next, we choose a lift $\widetilde y\in\Sp_{2m}(q)$ of $y$ and let $U_y=\widetilde y^{-1}U$.
If $r=p$, then $h|_{U_y}=I_{U_y}$ and $\operatorname{rank}(h-I)\leq1$.
As $h$ has a nontrivial field companion in the Singer normalizer, we have $2m/p$ Jordan blocks of length $p$.
This contradicts $\operatorname{rank}(h-I)\leq1$.
If $r\ne p$, then $h|_{U_y}$ is multiplication by
$\lambda$ with $\lambda^r=1$.
For $r<m$, an eigenvalue has multiplicity at least
$\frac{2m-2}{r-1}>\frac{2m}{r}$.
But the Singer projection is a product of cycles of length $r$,
so every eigenvalue of $h$ has multiplicity $2m/r$.
For prime $r=m$, we have $r\nmid q^{m-1}+1$ by Fermat's theorem.
Thus, $\lambda=1$. We again obtain a contradiction to the Singer multiplicity.
Thus, $A$ is a $2$-group.
Because $|A|\mid4$ for $m=2$, $A$ is also a $2$-group in that dimension.

If $q\equiv3\pmod4$, then no involution of $A$ stabilizes $\Lambda$ by Lemma \ref{lem:symplectic-singer-involutions}.
Thus, $A=1$.
Suppose that $q\equiv1\pmod4$ and $|A|\geq4$.
Then, $m$ is even, and the involution $a$ of $A$ belongs to $A^2$.
Next, we choose a primitive prime divisor $\ell$ of $q^{2m-2}-1$.
For this prime, take $B=C_\ell$ on $U\perp W=\F_q^{2m}$, with $\dim_{\F_q}U=2m-2$ and $\dim_{\F_q}W=2$.
The spaces $U,W$ are nondegenerate, and $B$ fixes $W$ pointwise.
For $m=2$, we have $q+1>2$. Since $q+1$ is not a power of $2$, the required subgroup $B$ exists.
As $A$ is ordinary $G$-permutable, we choose $y\in G$ such that $AB^y$ is a subgroup.
By the preceding Sylow count, we have $B^y\unlhd AB^y$.
The image of the cyclic $2$-group $A$ in $N_G(B^y)/C_G(B^y)$ has order at most $2$.
Since $a\in A^2$, its image is trivial. Thus, $a$ centralizes $B^y$.
A lift $h$ of $a$ satisfies $h^2=-I$.
Next, we choose a lift $\widetilde y\in\Sp_{2m}(q)$ of $y$. Let $U_y=\widetilde y^{-1}U$.
Then, $\dim_{\F_q}U_y=2m-2$.
Thus, $h|_{U_y}$ would have order $4$ in $\{M_\lambda:\lambda^{q^{m-1}+1}=1\}$.
But that torus has order $q^{m-1}+1\equiv2\pmod4$.
Therefore, $|A|=2$.
For $m\geq3$, we conclude that $A=1$ by Lemma \ref{lem:symplectic-complex-uniform}.
For $m=2$, we use Lemmas \ref{lem:symplectic-complex-five} and \ref{lem:symplectic-complex-extension}.
Thus, $q$ is prime and $q\equiv1,13,17\pmod{24}$.
\end{proof}

\begin{corollary}\label{cor:symplectic-all-higher-rank}
Every simple $\PSp_{2m}(q)$ with $m\geq3$ has no proper
nontrivial ordinary $\PSp_{2m}(q)$-permutable subgroup.
\end{corollary}
\begin{proof}
Combine Theorem \ref{thm:symplectic-odd-reduction} and
Corollary \ref{cor:symplectic-even}.
\end{proof}

\subsection{\texorpdfstring{The groups $\PSp_4(p)$ over prime fields $\F_p$}{The groups PSp4(p) over prime fields}}

Let $F$ be a finite field of odd characteristic, with algebraic closure $\overline F$.
Let $L$ be a finite group and let $V$ be a completely reducible $FL$-module.
For an $FL$-module $U$, the dual $U^*=\Hom_F(U,F)$ has the action $(lf)(u)=f(l^{-1}u)$.
We call $U$ self-dual if $U\cong U^*$ as $FL$-modules.
A symplectic constituent $U$ admits an $L$-invariant nondegenerate alternating form on $U$.
An orthogonal constituent $U$ admits an $L$-invariant nondegenerate symmetric bilinear form on $U$.
We call $V$ isotypic if $V\cong U^{\oplus r}$ for an irreducible $FL$-module $U$ and $r\geq1$.
For absolutely irreducible $U$, let $V_U$ be the sum of all $L$-submodules of $V$ isomorphic to $U$.
Its multiplicity space is $M=\Hom_{FL}(U,V)$.
Each $f\in M$ is $F$-linear and satisfies $f(lu)=lf(u)$ for $l\in L$ and $u\in U$.
Let $\rho_U:L\to\GL_F(U)$ satisfy $\rho_U(l)u=lu$.
The evaluation map $U\otimes_F M\to V_U$, $u\otimes f\mapsto f(u)$, is an $L$-module isomorphism.
Under this map, $l$ acts as $\rho_U(l)\otimes I_M$.
For abelian $N\leq L$, an $N$-weight is a homomorphism $\chi:N\to\overline F^*$ with $V_\chi\ne0$.
The weight space is $V_\chi=\{v\in\overline F\otimes_F V:nv=\chi(n)v\text{ for every }n\in N\}$.
For $a\in\overline F$ and $v\in V$, the action satisfies $n(a\otimes v)=a\otimes nv$.
For $a\geq3$, the generalized quaternion group is $Q_{2^a}=\langle x,y\mid x^{2^{a-1}}=1,\ y^2=x^{2^{a-2}},\ yxy^{-1}=x^{-1}\rangle$.

\begin{lemma}\label{lem:symplectic-complex-complement}
Let $q$ be odd, and let $B\leq\PSp_4(q)$ contain no involution
with a lift of square $-I$.
The inverse image of $B$ in $\Sp_4(q)$ is $\gen{-I}\times L$,
where $L\cong B$ has cyclic or generalized quaternion Sylow $2$-subgroups.
For every involution $s\in L$, we have $\dim_{\F_q}\ker(s-I)=\dim_{\F_q}\ker(s+I)=2$.
If $2\mid|L|$ and $s\in L$ is an involution, then
$sO_{2'}(L)\in Z(L/O_{2'}(L))$.
\end{lemma}
\begin{proof}
Let $z=-I$.
$2$ distinct commuting projective involutions have lifts $s,t$ with $s^2=t^2=I$.
If $s,t$ anticommute, then $(st)^2=z$. This contradicts the hypothesis on $B$.
If $s,t$ commute, then the joint eigenspaces are nondegenerate and even-dimensional.
An elementary abelian subgroup generated by involutions with square $I$ has rank at most $2$ in dimension $4$.
A subgroup of rank $2$ of this form contains $z$, so the projective image has rank at most $1$.
Thus, $B$ contains no subgroup isomorphic to $C_2\times C_2$.
Its Sylow $2$-subgroups are cyclic or generalized quaternion.

We split the central extension on a Sylow $2$-subgroup $P$ of $B$.
No lift of a nonidentity element of $P$ can have twice the order of its projective image.
A suitable power of such a lift would have square $z$, contrary to the hypothesis.
For cyclic $P$, we lift a generator without changing its order and obtain a complement to $\gen z$.
Suppose that $P$ is quaternion with $|P|=2^a$ and $a\geq3$.
We lift its usual generators to $x,y$.
Since the lifts have the orders of their images, we have $|x|=2^{a-1}$ and $|y|=4$.
Let $s^2=I$ lift the unique involution of $P$.
Then, $x^{2^{a-2}},y^2\in\{s,zs\}$.
Let $V_+=\ker(s-I)$ and $V_-=\ker(s+I)$.
Both spaces are nondegenerate, with $\dim_{\F_q}V_+=\dim_{\F_q}V_-=2$.
Suppose that the central powers $x^{2^{a-2}}$ and $y^2$ differ.
Choose $s$ so that $x^{2^{a-2}}=s$ and $y^2=zs$.
The relation $yxy^{-1}=x^{-1}z^\varepsilon$ holds for some
$\varepsilon\in\{0,1\}$.
As $(y|_{V_-})^2=I_{V_-}$, the restriction $y|_{V_-}$ is scalar.
Thus, $(x|_{V_-})^2=(-I_{V_-})^\varepsilon$.
Because $(x|_{V_-})^2=(-I_{V_-})^\varepsilon$ and $(x|_{V_-})^{2^{a-2}}=-I_{V_-}$, we have $a=3$ and $\varepsilon=1$.
As $(x|_{V_+})^2=I_{V_+}$, the restriction $x|_{V_+}$ is scalar.
From $yxy^{-1}=x^{-1}z$, we have $(x|_{V_+})^2=-I_{V_+}$. We obtain a contradiction.
Thus, we may choose $s$ with $x^{2^{a-2}}=y^2=s$.

The remaining relation is $yxy^{-1}=x^{-1}z^\varepsilon$.
If $\varepsilon=1$, then $y|_{V_+}$ is scalar.
Thus, $(x|_{V_+})^2=-I_{V_+}$.
From the relation on $V_-$, we have
$\operatorname{tr}(x|_{V_-})
=-\operatorname{tr}((x|_{V_-})^{-1})
=-\operatorname{tr}(x|_{V_-})$.
Since $\det(x|_{V_-})=1$, again $(x|_{V_-})^2=-I_{V_-}$.
Thus, $x^2=z$. This contradicts the hypothesis on $B$.
Thus, $\varepsilon=0$ and $\gen{x,y}$ is a quaternion complement.

We use trivial coefficients $\F_2$ in $H^2(B,\F_2)$ and $H^2(P,\F_2)$.
The restriction map $\operatorname{Res}_P^B:H^2(B,\F_2)\to H^2(P,\F_2)$ and corestriction map
$\operatorname{Cor}_P^B:H^2(P,\F_2)\to H^2(B,\F_2)$ satisfy
$\operatorname{Cor}_P^B\operatorname{Res}_P^B=[B:P]I_{H^2(B,\F_2)}$.
Since $[B:P]$ is odd, $\operatorname{Res}_P^B$ is injective.
The central extension restricts to a split extension on $P$, so its class maps to $0$ under $\operatorname{Res}_P^B$.
Therefore, its class in $H^2(B,\F_2)$ is $0$, and the central extension splits.
As $-I\notin L$, we have $\dim_{\F_q}\ker(s-I)=\dim_{\F_q}\ker(s+I)=2$ for each involution $s\in L$.
Let $s$ be the unique involution of a Sylow $2$-subgroup of $L$.
The involution $s$ is weakly closed in the chosen Sylow $2$-subgroup of $L$.
By Glauberman's $Z^*$-theorem \cite{Glauberman1966}, we have $sO_{2'}(L)\in Z(L/O_{2'}(L))$.
\end{proof}

\begin{lemma}\label{lem:symplectic-prime-parabolic}
Let $p\equiv1\pmod4$ be prime, and let $L\leq\Sp_4(p)$ with $-I\notin L$.
If $L$ stabilizes a totally isotropic $W\leq\F_p^4$ with $1\leq\dim_{\F_p}W\leq2$,
then $N_{\Sp_4(p)}(L)$ contains an element of square $-I$.
\end{lemma}
\begin{proof}
For ease of reading, we divide the proof of the lemma into the following steps.

We choose $i\in\F_p$ with $i^2=-1$.
\begin{itemize}[leftmargin=4em,labelsep=.5em,topsep=.2\baselineskip,partopsep=0pt,parsep=0pt,itemsep=0pt]
\item[{\bf Step 1.} ] If $L$ lies in a Borel subgroup, then $g^2=-I$ for some $g\in N_{\Sp_4(p)}(L)$.
\end{itemize}
Firstly, suppose that $L\leq TU$.
We use the diagonal torus $T=\{\diag(a,b,b^{-1},a^{-1}):a,b\in\F_p^*\}$ and upper-unitriangular radical $U$.
The Lie algebra of $U$ has basis
$$
 X=E_{12}-E_{34}, Y=E_{23}, Z=E_{13}+E_{24}, W=E_{14},
$$
and brackets $[X,Y]=Z$, $[X,Z]=2W$.
Since $U$ has class $3$ and $p\geq5$, Lazard's correspondence \cite{Lazard1954}
identifies the subgroups of $U$ with $\F_p$-Lie subalgebras.
By Schur--Zassenhaus, we choose $u\in U$ with $L^u=Q\rtimes J$.
We have $Q\leq U$ and $J\leq T$.
Then, we replace $L$ by $L^u$.
Let $\mathfrak h=\log Q$.
For the action of $J$, let $\chi_X,\chi_Y,\chi_Z,\chi_W$ be its characters on
the $4$ root spaces. From $[X,Y]=Z$ and $[X,Z]=2W$, we have
$\chi_Z=\chi_X\chi_Y$ and $\chi_W=\chi_X^2\chi_Y$.
Let $D_0=\diag(i,i,-i,-i)$ and $D_1=\diag(-i,i,-i,i)$. Then, $D_j^2=-I$ for $j=0,1$.
Conjugation by $D_0$ on $(X,Y,Z,W)$ has signs $(1,-1,-1,-1)$.
For $D_1$, these signs are $(-1,-1,1,-1)$.
Since $D_0,D_1$ and $J$ lie in the diagonal torus, both matrices centralize $J$.

Let $\pi_{\mathrm s}:\gen{X,Y,Z,W}\to\gen{X,Y}$ be
$\pi_{\mathrm s}(aX+bY+cZ+dW)=aX+bY$.
If $\dim_{\F_p}\pi_{\mathrm s}(\mathfrak h)=2$, then we obtain by taking brackets that $\mathfrak h=\gen{X,Y,Z,W}$.
Then, $D_1\mathfrak hD_1^{-1}=\mathfrak h$.
We now consider $\dim_{\F_p}\pi_{\mathrm s}(\mathfrak h)\leq1$.
If every vector in $\pi_{\mathrm s}(\mathfrak h)$ has $X$-coefficient zero, then
$\mathfrak h\leq\gen{Y,Z,W}$ and $D_0\mathfrak hD_0^{-1}=\mathfrak h$.
Otherwise, $\pi_{\mathrm s}(\mathfrak h)=\gen{v}$, where $v=aX+bY$ and $a\ne0$.
Let $\mathfrak h_0=\mathfrak h\cap\gen{Z,W}$.
Suppose that $z_0=cZ+dW\in\mathfrak h_0$ with $c\ne0$.
Choose $v_0\in\mathfrak h$ with $\pi_{\mathrm s}(v_0)=v$.
Since $[v_0,z_0]=2acW$, we have $W\in\mathfrak h_0$.
Thus, $\mathfrak h_0=\gen{Z,W}$ and $\mathfrak h=\gen{v,Z,W}$.
Consequently, $D_1\mathfrak hD_1^{-1}=\mathfrak h$.
If $\mathfrak h_0=\gen W$, then we choose $c\in\F_p$ with $\mathfrak h=\gen{v+cZ,W}$.
If $c\ne0$, then stability under $J$ requires $\chi_Y(j)=1$ for every $j\in J$.
For $u_Y=\exp((c/a)Y)\in C_U(J)$, the subalgebra $u_Y\mathfrak h u_Y^{-1}$ has no $cZ$ term.
If $\mathfrak h_0=0$, then we choose $c,d\in\F_p$ with $\mathfrak h=\gen{v+cZ+dW}$.
If $c\ne0$, then stability under $J$ again requires $\chi_Y=1$ on $J$.
Thus, $u_Y=\exp((c/a)Y)\in C_U(J)$ and
$u_Y(v+cZ+dW)u_Y^{-1}=v+dW$.
If $d\ne0$ after the $Y$-conjugation, then $\chi_Z=1$ on $J$.
For $u_Z=\exp((d/(2a))Z)\in C_U(J)$, the subalgebra $u_Z\mathfrak h u_Z^{-1}$ has no $dW$ term.
The $2$ conjugation formulas follow from $[Y,v]=-aZ$ and $[Z,v]=-2aW$.
Let $u\in C_U(J)$ be the product of the maps $u_Y,u_Z$ used above.
Then, $D_1\in N_{\Sp_4(p)}(uLu^{-1})$.
Therefore, $u^{-1}D_1u\in N_{\Sp_4(p)}(L)$ has square $-I$.
Thus, $N_{\Sp_4(p)}(L)$ contains an element of square $-I$ when $L$ lies in a Borel subgroup.

\begin{itemize}[leftmargin=4em,labelsep=.5em,topsep=.2\baselineskip,partopsep=0pt,parsep=0pt,itemsep=0pt]
\item[{\bf Step 2.} ] If $L$ stabilizes a totally isotropic subspace of $\F_p^4$ of dimension $2$, then $g^2=-I$ for some $g\in N_{\Sp_4(p)}(L)$.
\end{itemize}
We use the following prime-field observation in both parabolics.
If $J\leq\GL_2(p)$ is irreducible and $p\mid|J|$, then choose $t=I+N\in J$ of order $p$.
Then, $\gen{t}=\{I+aN:a\in\F_p\}$.
Since $J$ is irreducible, we choose $j\in J$ with $j\ker N\ne\ker N$.
In a basis of $\ker N$ and $j\ker N$, we have opposite root groups $\{I+aE_{12}:a\in\F_p\}$ and $\{I+aE_{21}:a\in\F_p\}$.
Therefore, $\SL_2(p)\leq J$.

For the Siegel parabolic, use a symplectic basis $(e_1,e_2,f_1,f_2)$.
Let $R_S=\begin{pmatrix}I_2&S\\0&I_2\end{pmatrix}$ and $R=\{R_S:S\in M_2(\F_p),\ S^{\mathsf T}=S\}$.
We have $P_2=R\rtimes\{\diag(g,g^{-\mathsf T}):g\in\GL_2(p)\}$.
For this semidirect product, let $\pi_2:P_2\to\GL_2(p)$ send $(r,g)$ to $g$.
Then, let $J=\pi_2(L)$.
If $J$ is reducible, then $L$ is in a Borel.
If $p\mid|J|$, then we have $\SL_2(p)\leq J$ by the preceding argument.
Thus, $-I_2\in J$.
Let $z=-I_4$.
Every lift of $-I_2$ in $L$ has form $rz$ with $r\in R$.
Since $(rz)^p=z$, we have $z\in L$. We obtain a contradiction.
Thus, $p\nmid|J|$. Let $Q=L\cap R$.
By Schur--Zassenhaus, we choose $r\in R$ with $L^r=Q\rtimes J$. We replace $L$ by $L^r$.
The matrix $D=\diag(iI_2,-iI_2)$ has square $-I$ and centralizes $\GL_2(p)$.
As $DR_SD^{-1}=R_{-S}$, every subgroup $Q\leq R$ satisfies $DQD^{-1}=Q$.
Thus, $D\in N_{\Sp_4(p)}(Q\rtimes J)$.

\begin{itemize}[leftmargin=4em,labelsep=.5em,topsep=.2\baselineskip,partopsep=0pt,parsep=0pt,itemsep=0pt]
\item[{\bf Step 3.} ] If $L$ stabilizes a subspace of $\F_p^4$ of dimension $1$, then $g^2=-I$ for some $g\in N_{\Sp_4(p)}(L)$.
\end{itemize}
For the line parabolic, use the basis $(e_1,e_2,f_2,f_1)$.
We have $P_1=R\rtimes K$ with $K=\{\diag(a,M,a^{-1}):a\in\F_p^*,\ M\in\SL_2(p)\}$.
The group $R$ is Heisenberg of order $p^3$, with centre $W=Z(R)$ of order $p$.
For its quotient, we identify $R/W$ with $\F_p^2$, on which $\SL_2(p)$ acts naturally.
Let $\pi_1:P_1\to K$ be $(r,(a,M))\mapsto(a,M)$, and
$\operatorname{pr}_2:K\to\SL_2(p)$ by $(a,M)\mapsto M$.
Let $J=\pi_1(L)$, $J_2=\operatorname{pr}_2(J)$ and $Q=L\cap R$.
If $J_2$ stabilizes a line $\Lambda\leq R/W$, then $L$ lies in a Borel subgroup of $\Sp_4(p)$.
Now, we suppose that $J_2$ acts irreducibly on $R/W$.
Since $QW/W$ is an $\F_p$-subspace, multiplication by any $a\in\F_p^*$ maps it to itself.
For each $M\in J_2$, choose $(a,M)\in J$. The action $aM$ maps $QW/W$ to itself, so $M(QW/W)=QW/W$.
By irreducibility of $J_2$, we have $QW/W=0$ or $R/W$.
If $QW/W=R/W$, then $[Q,Q]=W$ and $Q=R$.
If $QW/W=0$, then $Q\leq W$.
As $|W|=p$, we have $Q\in\{1,W,R\}$.

Suppose that $J_2$ is $p$-free.
By Schur--Zassenhaus, we choose $r\in R$ with $L^r=Q\rtimes J$. We replace $L$ by $L^r$.
If $M\in J_2$ has square $-I_2$, then
$D=\diag(i,M,-i)$ normalizes $J$ and each possible $Q$.
Choose a lift $(a,M)\in J$.
For $j\in J$, we have $DjD^{-1}=(a,M)j(a,M)^{-1}$.
If no such $M$ exists, then $J_2$ is odd or has central Sylow subgroup $C_2$.
If $J_2$ is even, then we obtain a normal odd Hall subgroup by Burnside transfer.
By Fong--Swan and odd order, an odd Hall subgroup of $J_2$ has only linear constituents on $\overline{\F}_p^2$.
Hence, the odd Hall subgroup is abelian.
Since the Sylow $2$-subgroup is central, $J_2$ is abelian and lies in a cyclic nonsplit torus of order $p+1$.
We identify $R/W$ with $E=\F_{p^2}$ as a vector space over $\F_p$.
Choose $c\in E^*$ with $N_{E/\F_p}(c)=-1$.
For this $c$, let $M:E\to E$ satisfy $M(v)=cv^p$.
Since $N_{E/\F_p}(c)=-1$, we have $M^2=-I_E$ and $M\in\Sp_2(p)$.
We have $J\leq\F_p^*\times C_{p+1}$.
As $(p-1)_{2'}$ and $(p+1)_{2'}$ are coprime, each Sylow $\ell$-subgroup of $J$ lies in a factor for every odd prime $\ell$.
Conjugation by $D=\diag(i,M,-i)$ sends $(a,b)\in\F_p^*\times C_{p+1}$ to $(a,b^{-1})$.
Thus, $D$ maps each odd Sylow subgroup of $J$ to itself.
Because $(p+1)_2=2$, every element $b\in C_{p+1}$ of $2$-power order satisfies $b^{-1}=b$.
Consequently, $D$ centralizes a Sylow $2$-subgroup of $J$, and $DJD^{-1}=J$.
As $D$ also normalizes every $Q\in\{1,W,R\}$, we obtain that $D\in N_{\Sp_4(p)}(L)$.

Finally, suppose that $p\mid|J_2|$.
Then, $J_2=\SL_2(p)$.
Since $\SL_2(p)$ is perfect, we have $J=C\times\SL_2(p)$ for $C\leq\F_p^*$.
The matrix $D=\diag(i,i,-i,-i)$ has square $-I$ and normalizes $J$ and every possible $Q$.
For $Q=R$, we have $L=R\rtimes J$.
For $Q\leq W$, choose an involution lifting the central element
$s=(1,-I_2)\in J$.
Choose $r\in R$ sending the lifted involution to $s$ by $g\mapsto r^{-1}gr$,
and replace $L$ by $L^r$.
For $l\in L$, choose $r_l\in R$ and $j_l\in J$ with $l=r_lj_l$.
Since $s$ is central in $J$, we have $sls^{-1}l^{-1}\in Q\leq W$.
On $R/W$, conjugation by $s$ sends $r_lW$ to $-r_lW$.
Thus, $0=-2(r_lW)$, and $r_l\in W$ because $p$ is odd.
Therefore, $L\leq W\rtimes J$.
If $Q=W$, then $L=W\rtimes J$.
If $Q=1$, then we choose a cocycle $f:J\to W$ with $L=\{(f(j),j):j\in J\}$.
The group $\SL_2(p)$ is perfect and acts trivially on $W$.
Thus, $f|_{\SL_2(p)}=0$.
By coprime averaging over $C$, we choose $w\in W$ with $L^w=J$.
Because $D$ normalizes $J$, the element $wDw^{-1}$ normalizes $L$ and has square $-I$.
\end{proof}

\begin{lemma}\label{lem:symplectic-prime-two-planes}
Let $p\equiv1,13$ or $17\pmod{24}$ be prime, and let
$L\leq\Sp(X_1)\times\Sp(X_2)$ act completely reducibly on $V=X_1\perp X_2$.
Suppose that $\dim_{\F_p}X_i=2$ for $i=1,2$.
Suppose that $(-I,-I)\notin L$.
Then, $N_{\Sp_4(p)}(L)$ contains an element whose square is $-I$.
\end{lemma}
\begin{proof}
Choose symplectic bases of $X_1,X_2$ to identify $\Sp(X_i)$ with $\SL_2(p)$ for $i=1,2$.
Let $\operatorname{pr}_i:\SL_2(p)\times\SL_2(p)\to\SL_2(p)$ send $(g_1,g_2)$ to $g_i$.
Then, let $H_i=\operatorname{pr}_i(L)$.
By Dickson's theorem \cite{Dickson1901,BHR}, each completely reducible $H_i$ is cyclic in a torus or binary dihedral or isomorphic to one of $2\cdot A_4,2\cdot S_4,2\cdot A_5,\SL_2(p)$.
A completely reducible subgroup of a Borel lies in a torus.
Since $p$ is prime, $\F_p$ has no proper subfield.
Every even projection has the unique involution $z=-I$.
Every noncyclic projection contains an element of square $z$,
and its order has $2$-part at least $4$.

By Goursat's lemma, we have epimorphisms $\varphi_i:H_i\to Q$.
Let $N_i=\ker\varphi_i$.
Then, $L=\{(h_1,h_2):\varphi_1(h_1)=\varphi_2(h_2)\}$.
By Dickson's list, an odd projection is cyclic.
We show that a projection is cyclic when both $H_1$ and $H_2$ have even order.
Suppose that both projections are even.
If both kernels had even order, then each would contain $z$. Then, $(z,z)\in L$, contrary to the hypothesis.
Suppose that both $|N_i|$ are odd.
For $Q_2\leq Q$ of order $2$, we have $|\varphi_i^{-1}(Q_2)|=2|N_i|$.
By Cauchy's theorem, we choose an involution $t_i\in\varphi_i^{-1}(Q_2)$.
Since $|N_i|$ is odd, $\varphi_i(t_i)$ generates $Q_2$.
Therefore, $Q$ has the unique involution $\varphi_i(z)$ when $|N_i|$ is odd.
Thus, $\varphi_1(z)=\varphi_2(z)$ and $(z,z)\in L$. We obtain a contradiction.

Relabel so that $N_1$ is even and $N_2$ is odd.
Then, $Q$ is even with a unique involution, and is a quotient of
$H_1/\gen z$.
If $H_1$ is noncyclic, then $H_1/\gen z$ is dihedral or belongs to $\{A_4,S_4,A_5,\PSL_2(p)\}$.
A dihedral quotient of order greater than $2$ has several involutions.
For $A_4,S_4,A_5$ and $\PSL_2(p)$, their normal subgroup lists show that the only even quotient with a unique involution is $C_2$.
Thus, $Q\cong C_2$.
Since $|N_2|$ is odd, a noncyclic $H_2$ would satisfy $|Q|_2=|H_2|_2\geq4$.
This contradicts $Q\cong C_2$.
Hence, a cyclic projection exists.

Firstly, let $C$ be a cyclic subgroup of a torus in $\SL_2(p)$.
In a split torus, choose $i\in\F_p$ with $i^2=-1$.
The matrix $j_C=\diag(i,-i)$ centralizes $C$ and has square $-I$.
Alternatively, in a nonsplit torus, we identify the plane with $E=\F_{p^2}$.
Next, we choose $c\in E^*$ with $c^{p+1}=-1$. Let $j_C:E\to E$ be $v\mapsto cv^p$.
As $c^{p+1}=-1$, we have $\det j_C=1$ and $j_C^2=-I$.
For $M_\lambda(v)=\lambda v$ with $\lambda^{p+1}=1$, we have $j_CM_\lambda j_C^{-1}=M_{\lambda^{-1}}$.
Scalar subgroups lie in a split torus.

If both $H_i$ are cyclic, then choose $j_i=j_{H_i}$ in the $i$th factor for $i=1,2$.
If both tori are split, then $\varphi_i(j_ihj_i^{-1})=\varphi_i(h)$ for $h\in H_i$.
If both tori are nonsplit, then $\varphi_i(j_ihj_i^{-1})=\varphi_i(h)^{-1}$ for $h\in H_i$.
For different types, $|Q|\mid(p-1,p+1)=2$.
Thus, every $a\in Q$ satisfies $a^{-1}=a$.
Therefore, $\diag(j_1,j_2)$ normalizes $L$ and has square $-I$.

Relabel so that $H=H_1$ is noncyclic and $C=H_2$ is cyclic.
Let $\varphi_H=\varphi_1$ and $\varphi_C=\varphi_2$.
The common quotient $Q$ is cyclic.
For binary dihedral $H$, we have $H/H'\cong C_4$ or $C_2^2$.
For $H=2\cdot A_4$, we have $H/H'\cong C_3$.
For $H=2\cdot S_4$, we have $H/H'\cong C_2$.
The groups $2\cdot A_5$ and $\SL_2(p)$ are perfect.
Thus, $Q$ is one of $1,C_2,C_3,C_4$.

If $|Q|\leq2$, then choose $h\in H$ with $h^2=-I$.
Since $Q$ is a cyclic quotient of $H$, conjugation by $h\in H$ induces the identity on $Q$.
The normalizer $j_C$ also acts trivially on $Q$.
Thus, $(h,j_C)$ has the required properties.
Suppose that $Q=C_4$. Then, $4\mid|C|$.
As $|C|$ divides $p-1$ or $p+1$ and $(p+1)_2=2$, we have $|C|\mid p-1$.
Thus, $C$ is split.
Choose $h\in H$ with $h^2=-I$ and the split normalizer $j_C$.
Both act trivially on $Q$, so $(h,j_C)$ normalizes $L$ and has square $-I$.

Finally, suppose that $Q=C_3$. Then, $H=2\cdot A_4$.
If $C$ is split, then choose an element of order $4$ in $H'$ and
the split normalizer $j_C$.
Both induce the identity on $Q$.
If $C$ is nonsplit, then $3\mid p+1$.
Since $3\mid p+1$ and $p\equiv1,13,17\pmod{24}$, we have $p\equiv17\pmod{24}$. Hence, $2$ is a square in $\F_p$.
As $H\cong\SL_2(3)$, we choose $u,v\in H'$ with $H'=\gen{u,v}\cong Q_8$.
We have $u^2=v^2=-I$ and $uv=-vu$.
For this pair, let $r=(-I+u+v+uv)/2$.
With $i=u$ and $j=v$, we use the displayed expression for $L$ in the proof of Lemma \ref{lem:symplectic-complex-five}.
Then, we obtain that $H=\gen{H',r}$ and $r^3=I$.
Let $w=u+v\in\GL_2(p)$.
Then, $w^2=-2I$ and $\det w=2$.
Also, $wuw^{-1}=v$, $wvw^{-1}=u$ and $wrw^{-1}=-vr^{-1}$.
Because $H=\gen{u,v,r}$, we obtain that $wHw^{-1}=H$.
As $-v\in H'$, we have $(wrw^{-1})H'=r^{-1}H'$.
Because $H/H'=\gen{rH'}\cong C_3$, we obtain that $(whw^{-1})H'=h^{-1}H'$ for every $h\in H$.
Now, we choose $\lambda\in\F_p$ with $\lambda^2=1/2$ and let $j_H=\lambda w$.
Then, $\det j_H=1$ and $j_H^2=-I$.
For $h\in H$, we have $\varphi_H(j_Hhj_H^{-1})=\varphi_H(h)^{-1}$.
For $c\in C$, the nonsplit map $j_C$ satisfies $\varphi_C(j_Ccj_C^{-1})=\varphi_C(c)^{-1}$.
Thus, $(j_H,j_C)$ normalizes $L$ and has square $-I$.
\end{proof}

\begin{lemma}\label{lem:symplectic-prime-semisimple}
Let $p\equiv1,13$ or $17\pmod{24}$ be prime, and let
$L\leq\Sp_4(p)$ act completely reducibly on its natural module.
Let $\pi:\Sp_4(p)\to\PSp_4(p)$ be the central quotient.
Suppose that $-I\notin L$.
Then, there exists $g\in\Sp_4(p)$ with $g^2=-I$ and $\pi(g)\in N_{\PSp_4(p)}(\pi(L))$.
\end{lemma}
\begin{proof}
For ease of reading, we divide the proof of the lemma into the following steps.

\setlist[itemize]{leftmargin=4em,labelsep=.5em,topsep=.2\baselineskip,partopsep=0pt,parsep=0pt,itemsep=0pt}
\begin{itemize}
\item[{\bf Step 1.} ] If $V$ is irreducible, then $O_{2'}(L)$ is abelian.
\end{itemize}
Let $V=\F_p^4$ with symplectic form $\beta$.
Let $\rho:L\to\GL(V)$ be the inclusion and $z=-I$.
Next, choose $i\in\F_p$ with $i^2=-1$.
For the quotient map $\pi:\Sp(V)\to\PSp(V)$, let $\overline L=\pi(L)$.
If $\overline L$ contained an involution with a lift of square $z$, then its lift $l\in L$ would satisfy $l^2=z$.
Since $z\notin L$, no such involution belongs to $\overline L$.
By Lemma \ref{lem:symplectic-complex-complement}, a Sylow $2$-subgroup of $L$ is cyclic or generalized quaternion.
If $2\mid|L|$, then choose the involution $s$ in a Sylow $2$-subgroup of $L$.
The eigenvalues $1,-1$ of $s$ each have multiplicity $2$.
Let $N=O_{2'}(L)$.
By Lemma \ref{lem:symplectic-complex-complement}, we have $sN\in Z(L/N)$.
For every $l\in L$, we choose $n\in N$ with $s^l=sn$.

Firstly, suppose that $V$ is irreducible over $\F_p$.
Let $E=\operatorname{End}_{\F_p[L]}(V)\cong\F_{p^e}$, where $e\in\{1,2,4\}$.
Then, $\dim_E V=4/e$.
Every absolutely irreducible representation of a group of odd order in characteristic $p$ has odd degree.
This follows from the odd-order theorem and Fong--Swan
\cite{Fong1961,Isaacs2006}.
For absolute degree $2$ or $4$, the constituent degrees of $N$ divide that degree by Clifford theory.
Since those degrees are odd, every constituent is linear.
As $N$ acts faithfully, we have $N$ abelian of order prime to $p$.
Let $V_{\rm alg}=\overline{\F}_p\otimes_{\F_p}V$.
The representation $\rho_{\rm alg}:L\to\GL_{\overline{\F}_p}(V_{\rm alg})$ satisfies $\rho_{\rm alg}(l)(a\otimes v)=a\otimes\rho(l)v$.
For $e=1$, the distinct $N$-weights form a transitive $L$-orbit with equal multiplicities by Clifford's theorem.

\begin{itemize}
\item[{\bf Step 2.} ] If $e=1$, then there is $g\in\Sp(V)$ with $g^2=z$ and $\pi(g)\in N_{\PSp(V)}(\overline L)$.
\end{itemize}
Firstly, suppose that $e=1$.
Since a group of odd order has no absolutely irreducible representation of degree $4$, we have $2\mid|L|$.
Thus, we choose the involution $s$ as above.
There are $1$, $2$ or $4$ weights of $N$.
If $N$ has $1$ weight, then $N\leq Z(\Sp(V))=\gen z$.
Note that $|N|$ is odd. Thus, $N=1$ and $s\in Z(L)$.
The $2$ eigenspaces of $s$ contradict absolute irreducibility.

With $2$ weights, symplectic duality pairs their spaces of dimension $2$.
Suppose that both weight spaces $W_1,W_2$ have $\F_p$-bases.
We let $g|_{W_1}=iI_{W_1}$ and $g|_{W_2}=-iI_{W_2}$.
Then, $g\in\Sp_4(p)$ and $g^2=z$.
If $lW_i=W_i$ for $i=1,2$, then $glg^{-1}=l$.
If $lW_1=W_2$ and $lW_2=W_1$, then $glg^{-1}=-l$.
Thus, $\pi(g)\in N_{\PSp(V)}(\overline L)$.

If the $2$ weights $\chi_1,\chi_2$ satisfy $\chi_1(n)^p=\chi_2(n)$ and $\chi_2(n)^p=\chi_1(n)$ for $n\in N$, then
$E_N=\F_p[\rho(N)]\cong\F_{p^2}$ acts by scalars on $V$.
We have $\dim_{E_N}V=2$.
For every $l\in L$, we have $\rho(l)E_N\rho(l)^{-1}=E_N$.
Suppose that $s$ is $E_N$-linear.
Since $N$ acts by $E_N$-scalars, we have $1=(sn)^2=n^2$ for $s^l=sn$ and $l\in L$.
As $|N|$ is odd, we would have $n=1$ and $s\in Z(L)$.
The $2$ eigenspaces of $s$ would then be nonzero proper $L$-submodules, contrary to irreducibility of $V$.
Thus, $s$ is not $E_N$-linear, and $s(av)=a^p s(v)$ for $a\in E_N$ and $v\in V$.
Let $L_0=L\cap\GL_{E_N}(V)$, so $[L:L_0]=2$.
Every involution of $L$ has form $s^l$ for some $l\in L$.
Because $s\notin L_0\unlhd L$, the group $L_0$ is odd.
It follows that $N=L_0$ and $L=N\rtimes\gen s$.
Choose $u\in E_N^*$ with $u^p=-u$.
For $v\in V$, the vectors $v+s(v)$ and $u(v-s(v))$ are fixed by $s$.
Because those vectors span $V$ over $E_N$, we choose an $E_N$-basis $(v_1,v_2)$ with $s(v_j)=v_j$ for $j=1,2$.
As $N=L_0$ acts by $E_N$-scalars, we have $l(E_Nv_j)=E_Nv_j$ for $l\in L$ and $j=1,2$.
Because $\dim_{\F_p}(E_Nv_1)=2<4$, we obtain a contradiction to irreducibility.

With $4$ weights, let $\mathcal W$ be the set of their lines in $V_{\rm alg}$.
Choose a bijection $\nu:\{1,2,3,4\}\to\mathcal W$.
Using this numbering, let $\sigma:L\to S_4$ satisfy $\sigma(l)(j)=\nu^{-1}(\rho_{\rm alg}(l)\nu(j))$.
For the dual weights, let $\tau\in S_4$ send each index to the index of its dual weight.
Then, $\tau$ is a fixed-point-free involution commuting with $X=\sigma(L)$.
Thus, $X\leq C_{S_4}(\tau)\cong D_8$.
Since $sN\in Z(L/N)$ and $\sigma(N)=1$, we have $\sigma(s)\in Z(X)$.
Suppose that $\sigma(s)=1$. Then, $s$ would commute with $N$.
For $s^l=sn$, we would have $1=(sn)^2=n^2$.
As $|N|$ is odd, we would have $n=1$ and $s\in Z(L)$.
The $2$ eigenspaces of $s$ contradict irreducibility. Therefore, $\sigma(s)\ne1$.
The transitive possibilities $C_4,C_2\times C_2,D_8$ all contain $\tau$.
Thus, we choose $h\in L$ inducing $\tau$.
Next, on a pair of dual weight lines, choose $v,w$ with $\beta(v,w)=1$.
As $h$ induces $\tau$, we have $hv=aw$ and $hw=bv$ for $a,b\in\overline{\F}_p^*$.
Because $\beta(hv,hw)=1$, we have $ab=-1$ and $h^2=-I$ on this pair.
Hence, $h^2=z$, contradicting $z\notin L$.
This calculation over the algebraic closure is an equality of the
original matrices over $\F_p$.

\begin{itemize}
\item[{\bf Step 3.} ] If $e=2$, then $L=L_0\rtimes\gen{s}$.
\end{itemize}
Now, suppose that $e=2$.
Let $a\mapsto a^*$ be the adjoint on $E=\F_{p^2}$, satisfying $\beta(av,w)=\beta(v,a^*w)$.
This field involution is the identity or $a\mapsto a^p$.
If $a^*=a$ for $a\in E$, then $\beta$ is the trace of a nondegenerate alternating form on $E^2$.
Thus, $L\leq\SL_2(E)$.
Every even subgroup contains its unique involution $z$.
An odd subgroup cannot have absolute degree $2$.
Thus, this case is impossible.

Suppose that $a^*=a^p$ for every $a\in E$.
We choose a Hermitian form $h:E^2\times E^2\to E$.
For $a,b\in E$, we have $h(av,bw)=ab^p h(v,w)$ and $h(w,v)=h(v,w)^p$.
Choose $\mu\in E^*$ with $\mu^p=-\mu$ such that
$\beta(v,w)=\Tr_{E/\F_p}(\mu h(v,w))$.
Then, $L\leq\GU(E^2,h)\cong\GU_2(p)$.
Moreover, $L$ is even.
Next, let $\overline E$ be an algebraic closure of $E$.
On $\overline E\otimes_E V$, we use the action $n(a\otimes v)=a\otimes nv$ for $n\in N$.
Since $N$ is abelian and $\dim_E V=2$, we have $1$ or $2$ distinct eigenvalue characters.
Suppose that an eigenvalue $\lambda$ has degree $2$ over $E$.
Its Galois orbit is
$\{\lambda,\lambda^{p^2}\}$.
The Hermitian dual $\lambda^{-p}$ belongs to $\{\lambda,\lambda^{p^2}\}$.
As $p\nmid|\lambda|$, either equality implies that $\lambda^{p+1}=1$.
Then, $\lambda\in E$, a contradiction.
Therefore, each weight line is spanned over $E$.
We diagonalize the commuting semisimple elements of $N$ simultaneously over $E$.

If $N$ has $1$ weight, then $N$ acts by $E$-scalars.
For $s^l=sn$, we have $1=(sn)^2=n^2$.
Since $|N|$ is odd, we have $n=1$ and $s\in Z(L)$. This contradicts absolute irreducibility over $E$.
With $2$ weights, let $W_1,W_2$ be their lines.
We have $sW_1=W_2$ and $sW_2=W_1$.
Otherwise, $s$ centralizes $N$. For $s^l=sn$, we then have $1=(sn)^2=n^2$. As $|N|$ is odd, we obtain that $n=1$ and $s\in Z(L)$. We obtain a contradiction.
Let $L_0=\{l\in L:lW_i=W_i\text{ for }i=1,2\}$.
Every involution of $L$ has form $s^l$ for some $l\in L$.
Because $s\notin L_0\unlhd L$, the group $L_0$ is odd.
Thus, $L=L_0\rtimes\gen s$.

\begin{itemize}
\item[{\bf Step 4.} ] If $e\in\{2,4\}$, then there is $g\in\Sp(V)$ with $g^2=z$ and $\pi(g)\in N_{\PSp(V)}(\overline L)$.
\end{itemize}
If $W_1,W_2$ are nondegenerate, then choose $v_j\in W_j$ with $h(v_i,v_j)=\delta_{ij}$.
For some $a\in E^*$ with $a^{p+1}=1$, the matrix of $s$ is $\begin{pmatrix}0&a^{-1}\\a&0\end{pmatrix}$.
Take $g=\begin{pmatrix}0&-a^{-1}\\a&0\end{pmatrix}$.
Then, $g\in\GU(E^2,h)$ and $g^2=z$.
For $l\in L_0$, we have $glg^{-1}=sls^{-1}$ and $gsg^{-1}=-s$.
Thus, $\pi(g)\in N_{\PSp(V)}(\overline L)$.

If $W_1,W_2$ are isotropic, then choose $0\ne v_1\in W_1$ and let $v_2=s(v_1)$.
Since $s^2=I_V$, its matrix in $(v_1,v_2)$ is $\begin{pmatrix}0&1\\1&0\end{pmatrix}$.
Let $h_{ij}=h(v_i,v_j)\in E$.
We have $h_{11}=h_{22}=0$ and $h_{21}=h_{12}^p$.
As $s$ is unitary, $h_{12}=h_{21}$. Thus, $h_{12}\in\F_p^*$.
Let $F:E\to E$ satisfy $F(a)=a^p$.
Let $F_V:V\to V$ satisfy $F_V(a_1v_1+a_2v_2)=F(a_1)v_1+F(a_2)v_2$.
Because $\mu^p=-\mu$, we have $\beta(F_Vv,F_Vw)=-\beta(v,w)$ for $v,w\in V$.
Choose $c\in E^*$ with $c^{p+1}=-1$.
Thus, $g(v)=cF_V(v)$ is symplectic and $g^2=z$.
As the entries of $s$ belong to $\F_p$, we have $gs=sg$.
For every $l\in L_0$, we have $h(lv_1,lv_2)=h(v_1,v_2)$.
Because $l(Ev_j)=Ev_j$ for $j=1,2$, its matrix in $(v_1,v_2)$ is $l=\diag(a,a^{-p})$ for some $a\in E^*$.
Thus, $glg^{-1}=sl^{-1}s\in L_0$.
As $gs=sg$ and $L=L_0\rtimes\gen s$, we have $g\in N_{\Sp(V)}(L)$.

Finally, suppose that $e=4$.
Since $\dim_E V=1$ and $L$ commutes with $E$, every element of $L$ acts by an $E$-scalar.
The symplectic scalars form a cyclic torus of order $p^2+1$, so $L$ is cyclic in that torus.
We identify $V$ with $\F_{p^4}$.
Choose $\mu\in\F_{p^4}^*$ with $\mu^{p^2}=-\mu$ and use
$$
 \beta(x,y)=\operatorname{Tr}_{\F_{p^4}/\F_p}(\mu xy^{p^2}).
$$
Next, we choose $c\in\F_{p^4}^*$ with $c^{p^2+1}=-1$. Let $M_\lambda(x)=\lambda x$ for $x\in V$ and $\lambda\in\F_{p^4}^*$.
For $g(x)=cx^{p^2}$, we have $\beta(gx,gy)=\beta(x,y)$ and $g^2=z$.
Since $gM_\lambda g^{-1}=M_{\lambda^{p^2}}$, the map $g$ normalizes every subgroup of $\{M_\lambda:\lambda^{p^2+1}=1\}$.
This completes the irreducible cases.

\begin{itemize}
\item[{\bf Step 5.} ] If $V$ is not isotypic, then $g^2=z$ for some $g\in N_{\Sp(V)}(L)$.
\end{itemize}
Now, consider $V=W_1\oplus\cdots\oplus W_r$ with $r\geq2$, where each $W_i$ is an irreducible $L$-submodule.
We show that no irreducible constituent $W$ has $\dim_{\F_p}W=3$.
If $W$ is not self-dual, then $V$ also has constituent $W^*$. Thus, $\dim_{\F_p}V\geq6$.
If $W$ is self-dual, then its odd dimension admits no nondegenerate alternating form.
Since the symplectic multiplicity of $W$ is even, we have $\dim_{\F_p}V\geq6$.

An irreducible constituent $W\leq V$ with $\dim_{\F_p}W=2$ and $W\not\cong W^*$ is paired with $W^*$.
Since $\dim_{\F_p}V=4$, these spaces are complementary $L$-invariant Lagrangians.
Thus, on $W\oplus W^*$, the map $g=iI_W\oplus(-iI_{W^*})$ centralizes $L$ and has square $z$.
For $W\not\cong W^*$ with $\dim_{\F_p}W=1$, the space $W\oplus W^*$ is a nondegenerate $L$-submodule.
Different self-dual constituents are orthogonal.
If $V$ is not isotypic, then we have $V=V_1\perp V_2$ with $\dim_{\F_p}V_1=\dim_{\F_p}V_2=2$. Each $l\in L$ satisfies $lV_i=V_i$ for $i=1,2$.
By Lemma \ref{lem:symplectic-prime-two-planes}, we have $g\in N_{\Sp(V)}(L)$ with $g^2=z$.

\begin{itemize}
\item[{\bf Step 6.} ] If $V$ is reducible and isotypic, then $g^2=z$ for some $g\in C_{\Sp(V)}(L)$.
\end{itemize}
If $V$ is isotypic, then consider its irreducible constituent.
A self-dual constituent of dimension $1$ is a sign character.
The multiplicity space of the sign character is symplectic.
Thus, $C_{\Sp(V)}(L)$ contains a matrix of square $z$.
Suppose that an absolutely irreducible symplectic constituent $U$ with $\dim_{\F_p}U=2$ has multiplicity $2$.
Let $\rho_U:L\to\Sp(U)=\SL_2(p)$ be the representation on $U$.
If $2\mid|\rho_U(L)|$, then $-I_U\in\rho_U(L)$ acts as $z$ on both copies. This contradicts $z\notin L$.
An odd image cannot have absolute degree $2$.
Thus, this multiplicity is impossible.
For an absolutely irreducible orthogonal constituent $U$ with $\dim_{\F_p}U=2$ and multiplicity $2$, the multiplicity space $M$ is symplectic with $\dim_{\F_p}M=2$.
Under $V\cong U\otimes_{\F_p}M$, the maps $I_U\otimes h$ with $h\in\Sp(M)$ centralize $L$.
In a symplectic basis of $M$, take $h=\diag(i,-i)$.
Then, $(I_U\otimes h)^2=z$.
For a self-dual constituent $U$ irreducible over $\F_p$ but not absolutely irreducible, we have $\dim_{\F_p}U=2$.
Let $E_U=\operatorname{End}_{\F_pL}(U)\cong\F_{p^2}$.
The adjoint automorphism of $E_U$ sends $a$ to $a^p$.
Thus, on the multiplicity space $M=E_U^2$, the centralizer of $L$ is $\GU_2(p)$.
Consequently, an element of order $4$ in $\SU_2(p)\leq\GU_2(p)$ acts on $V$ with square $z$.
Since every completely reducible $L$-module $V$ has one of the listed decompositions, we have proved the lemma.
\end{proof}

\begin{theorem}\label{thm:symplectic-four-odd-all}
Let $q$ be odd and $G=\PSp_4(q)$.
Then, $G$ has a proper nontrivial ordinary $G$-permutable subgroup
if and only if $q$ is prime and $q\equiv1,13,17\pmod{24}$.
In these cases, the subgroups are precisely those of order $2$
whose generating involutions have a lift of square $-I$.
\end{theorem}
\begin{proof}
By Theorem \ref{thm:symplectic-odd-reduction}, a proper nontrivial ordinary $G$-permutable subgroup has order $2$.
Its generating involution has a lift of square $-I$.
Moreover, $q$ is prime with $q\equiv1,13,17\pmod{24}$.

Let $p\equiv1,13,17\pmod{24}$ be prime. Let $G=\PSp_4(p)$.
For the natural representation, let $V=\F_p^4$ and let $\pi:\Sp(V)\to G$ be the quotient map.
Let $A=\gen{\pi(s)}$ with $s\in\Sp(V)$ and $s^2=-I_V$.
Next, choose $i\in\F_p$ with $i^2=-1$.
Let $V_+=\ker(s-iI_V)$ and $V_-=\ker(s+iI_V)$.
Since $s^2=-I_V$ and $s$ is symplectic, we have $V=V_+\oplus V_-$, with $\dim_{\F_p}V_+=\dim_{\F_p}V_-=2$.
Moreover, $\beta(sv,sw)=\beta(v,w)$ and $i^2=-1$ show that both eigenspaces are totally isotropic.
For $2$ such subgroups $A_1,A_2$, paired symplectic bases provide $g\in G$ with $A_1^g=A_2$.
Let $B\leq G$.
If $B$ contains such a subgroup $A_1$, then choose $x\in G$ with $A_1^x=A$.
Then, $A\leq B^x$, so $AB^x$ is a subgroup.

Suppose that $B$ contains no involution with a lift of square $-I$.
We obtain from Lemma \ref{lem:symplectic-complex-complement} a complement $L\cong B$ in $\pi^{-1}(B)$.
If the natural $L$-module is completely reducible, then we use Lemma \ref{lem:symplectic-prime-semisimple}.
We have an involution $t_B\in N_G(B)$ with a lift of square $-I$.
Suppose that $V=\F_p^4$ is not completely reducible under $L$.
For a proper degenerate $L$-submodule $W$, the radical $W\cap W^\perp$ is a nonzero totally isotropic $L$-submodule.
For a proper nondegenerate $L$-submodule $W$, we have $V=W\perp W^\perp$ and $\dim_{\F_p}W=2$.
Because $V$ is not completely reducible, at least $1$ of $W,W^\perp$ is reducible under $L$.
Thus, $L$ stabilizes an isotropic line in that plane.
Thus, $L$ stabilizes a totally isotropic subspace $W$ with $\dim_{\F_p}W=1$ or $2$.
By Lemma \ref{lem:symplectic-prime-parabolic}, we have an involution $t_B\in N_G(B)$ with a lift of square $-I$.
Choose $x\in G$ with $\gen{t_B^x}=A$.
Then, we have $A\leq N_G(B^x)$.
Thus, for every $B\leq G$, there is $x\in G$ with $AB^x=B^xA$.
\end{proof}

\section{Orthogonal groups}\label{sec:orthogonal}

Let $q$ be a prime power and let $V$ be a finite-dimensional $\F_q$-vector space.
Let $Q:V\to\F_q$ be a quadratic form.
We have $Q(av)=a^2Q(v)$, and its polar form $\beta_Q(v,w)=Q(v+w)-Q(v)-Q(w)$ is bilinear.
We use $W^\perp$ for the orthogonal complement with respect to $\beta_Q$.
For the spaces used below, $\beta_Q$ is nondegenerate, meaning $V^\perp=0$.
A subspace $W$ is totally singular if $Q|_W=0$, and totally isotropic if $\beta_Q|_{W\times W}=0$.
These conditions can differ in characteristic $2$.
We call $(V,Q)$ anisotropic if $Q(v)=0$ implies $v=0$.
For nondegenerate $V$ of dimension $2m$, the Witt index is $w=\max\{\dim W:Q|_W=0\}$.
Plus type means $w=m$ and minus type means $w=m-1$.
A hyperbolic plane has a basis $(e,f)$ with $Q(xe+yf)=xy$.

Let $\operatorname{O}(V,Q)=\{g\in\GL_{\F_q}(V):Q(gv)=Q(v)\text{ for all }v\in V\}$.
For odd $q$, let $\operatorname{SO}(V,Q)=\operatorname{O}(V,Q)\cap\SL_{\F_q}(V)$ and $\Omega(V,Q)=\operatorname{SO}(V,Q)\cap\ker\operatorname{sn}$.
The spinor norm $\operatorname{sn}:\operatorname{O}(V,Q)\to\F_q^*/(\F_q^*)^2$ satisfies $\operatorname{sn}(r_v)=Q(v)(\F_q^*)^2$ for the reflection $r_v(x)=x-\beta_Q(x,v)Q(v)^{-1}v$ with $Q(v)\ne0$.
For even $q$ and $\dim_{\F_q}V=2m\geq6$, we use $\Omega(V,Q)=\ker D$.
The Dickson homomorphism $D:\operatorname{O}(V,Q)\to\F_2$ satisfies $D(g)=\operatorname{rank}_{\F_q}(g-I_V)\pmod2$.
For even $q$ and $\dim_{\F_q}V=4$ of minus type, we also use $\Omega(V,Q)=\ker D$.

For $V=V_1\perp\cdots\perp V_t$ with $\dim V_i=2d_i$, signed type $+d_i$ or $-d_i$ specifies a summand of plus or minus type.
The cyclic torus factor on $V_i$ has order $q^{d_i}-1$ or $q^{d_i}+1$, respectively.
The torus in $\Omega(V,Q)$ is the intersection of the product of these factors with $\Omega(V,Q)$.
For $\POmega(V,Q)$, we then take its image under the scalar quotient map.
For minus type with $t=1$, we call the torus a Singer torus.

\subsection{\texorpdfstring{The groups $\POmega^-_{2m}(q)$}{The groups POmega minus 2m(q)}}\label{sec:orthogonal-minus}

\begin{theorem}\label{thm:orthogonal-minus-even}
Let $q$ be even and $m\geq4$.
Then, $\Omega^-_{2m}(q)$ has no proper nontrivial ordinary
$\Omega^-_{2m}(q)$-permutable subgroup.
\end{theorem}
\begin{proof}
For ease of reading, we divide the proof of the theorem into the following steps.

\setlist[itemize]{leftmargin=4em,labelsep=.5em,topsep=.2\baselineskip,partopsep=0pt,parsep=0pt,itemsep=0pt}
Let $G=\Omega^-_{2m}(q)$ and let $A<G$ be ordinary $G$-permutable.
We use the natural $G$-module $V=\F_q^{2m}$.
\begin{itemize}
\item[{\bf Step 1.} ] $A$ normalizes $T^x$.
\end{itemize}
Next, let $P_2$ stabilize a totally singular subspace $U\leq V$ with $\dim_{\F_q}U=2$.
The subgroup $P_2$ is maximal.
For every $Y<G$, we have $P_2Y\ne G$
\cite[Theorem 4.2]{King2005}\cite[Table A.5]{LiXia}.
Thus, $A\leq P_2^x$ for some $x\in G$.
Hence, $A$ stabilizes $x^{-1}U$ and $AY\ne G$ for every $Y<G$.
The anisotropic plane stabilizer $N_2^-$ is maximal
\cite[Theorem 4.3]{King2005}.
For characteristic $2$, the exception to maximality of $N_2^-$ has $\dim_{\F_q}V=6$.
By Lemma \ref{lem:max}, we choose $y\in G$ with $A\leq(N_2^-)^y$.
Thus, $A$ stabilizes $V=V_1\perp V_2$. We have $\dim_{\F_q}V_1=2$ and $\dim_{\F_q}V_2=2m-2$.

Next, we take a Singer torus $T$ of order $q^m+1$.
Since $A$ is ordinary $G$-permutable, we choose $x\in G$ such that $K=AT^x$ is a subgroup.
As $AT^x\ne G$, we have $K<G$.
Since $|T|$ is odd, we have $|K|_2=|A|_2$.
Thus, a Sylow $2$-subgroup of $A$ is Sylow in $K$.
For every $2$-subgroup $P\leq K$, Sylow's theorem permits a choice of $g\in K$ with $P^g\leq A$.
Because $A$ stabilizes $V_1\perp V_2$, the group $P$ stabilizes $gV_1\perp gV_2$.
We claim that $K\leq N_G(T^x)$.
To prove $K\leq N_G(T^x)$, we choose a primitive prime divisor $\ell$ of $q^{2m}-1$.
There is no Zsigmondy exception in this range, and $\ell>2m$.
For this primitive prime, let $T_\ell$ be the Sylow $\ell$-subgroup of $T^x$. We have
$C_{\GL_{2m}(q)}(T_\ell)=\F_{q^{2m}}^*, C_G(T_\ell)=T^x$.

By Bereczky's theorem \cite{Bereczky,Brookfield2014}, we use successive subgroups $\Omega^-_{2a}(Q)$ or $\GU_b(Q)$ with odd $b$.
At each step, intersect with the normal subgroup consisting of maps linear over $\F_Q$ or $\F_{Q^2}$.
For an orthogonal subgroup, also intersect with its normal $\Omega^-$ subgroup.
The intervening quotients have order at most $2m<\ell$.
At each step, let $F_0=\F_Q$ for an orthogonal group $\Omega^-_{2a}(Q)$ and $F_0=\F_{Q^2}$ for a unitary group $\GU_b(Q)$.
Since $F_0$ commutes with $T_\ell$, we have $F_0\leq\F_{q^{2m}}$.
As $T^x$ centralizes $F_0$, its elements are $F_0$-linear.
Because $|T^x|$ is odd, $T^x$ lies in the relative $\Omega^-$ subgroup.
Thus, every normal intersection contains $T^x$.

By Lemma \ref{lem:orthogonal-minus-two-subgroup}, each relative $\Omega^-_{2a}(Q)$ with $a\geq2$ contains an indecomposable $2$-subgroup.
From Lemma \ref{lem:symplectic-unitary-sylow}, each relative $\GU_b(Q)$ with odd $b\geq3$ contains an indecomposable $2$-subgroup.
But every $P\leq K$ of $2$-power order stabilizes $gV_1\perp gV_2$ for some $g\in K$.
Thus, an intersection cannot equal either relative group.
We continue with decreasing relative dimension.
The groups $\Omega^-_2(Q)$ and $\GU_1(Q)$ are tori.
For $\Omega^-_4(Q)\cong\SL_2(Q^2)$, we use Dickson's classification.
Every subgroup $H$ with $T^x\leq H<\Omega^-_4(Q)$ satisfies
$H\leq N_{\Omega^-_4(Q)}(T^x)$.

The relative $\Omega^-_6(2)$ and $\GU_3(2)$ exceptions would require
the original $m=3$ and $q=2$.
The relative fields have characteristic $2$, so none has order $3$.
For $\GU_5(2)$, the original pair is $(m,q)=(5,2)$.
By the proof of Theorem \ref{thm:symplectic-even}, neither $\PSL_2(11)$ nor $A_5$ stabilizes a totally singular $U\leq\F_2^{10}$ with $\dim_{\F_2}U=2$.
Let $H$ be the last normal intersection.
Since $H$ lies in the relative normalizer of $T^x$, the subgroup $T_\ell$ is characteristic in $H$.
By the proof of Lemma \ref{lem:symplectic-odd-singer}, we have $T_\ell\unlhd K$ through the successive normal intersections.
As $C_G(T_\ell)=T^x$ and $A\leq K$, we have $A\leq N_G(T^x)$.

\begin{itemize}
\item[{\bf Step 2.} ] $A$ is a cyclic $2$-group.
\end{itemize}
Replace $A$ by $xAx^{-1}$, with $x\in G$ from $K=AT^x$.
Then, $A\leq N_G(T)$.
We use $E=\F_{q^{2m}}$ with quadratic form
$Q(v)=\Tr_{\F_{q^m}/\F_q}(v^{q^m+1})$.
Since $Q$ has $q^{2m-1}-(q-1)q^{m-1}$ zeroes, we have $Q$ of minus type.
Let $\sigma:E\to E$ satisfy $\sigma(v)=v^q$. We have $|\sigma|=2m$ and $N_{\operatorname{O}(E,Q)}(T)=T\rtimes\gen\sigma$.
As $E^{\gen\sigma}=\F_q$, we have $\operatorname{rank}(\sigma-I_E)=2m-1$.
Thus, the Dickson invariant of $\sigma$ is $1$.
Therefore,
$N_G(T)=T\rtimes \gen{\sigma^2}, N_G(T)/T\cong C_m$.
By Lemma \ref{lem:max} for the maximal singular line stabilizer, $A$ stabilizes a line $\Lambda\leq V$.
For a Singer multiplier $M_\lambda$ stabilizing $\Lambda$, we have $\lambda\in\F_q^*$ and $\lambda^{q^m+1}=1$.
Because $q$ is even, $\lambda=1$.
Consequently, $A\cap T=1$ and $A$ is cyclic with $|A|\mid m$.

Suppose that an odd prime $r$ divides $|A|$.
Let $q=2^f$. We choose a primitive divisor $h$ of
$2^{f(2m-2)}-1$.
For $(f,m)=(1,4)$, we have $|A|\mid4$. This contradicts $r\mid|A|$.
Otherwise, $h\geq2m-1>m$.
We choose a nondegenerate subspace $U\leq V$ of minus type with $\dim_{\F_q}U=2m-2$. Let $W=U^\perp$, so $V=U\perp W$ and $\dim_{\F_q}W=2$.
The space $U$ has minus type and $W$ is hyperbolic.
We take $B=C_h$ irreducible on $U$ and trivial on $W$.
Since $A$ is ordinary $G$-permutable, choose $x\in G$ with $AB^x\leq G$.
As $|A|<h$, we obtain by Sylow counting that $B^x\unlhd AB^x$.
Let $z\in A$ have order $r$.
Because $r\mid m$ and $|N_G(B^x)/C_G(B^x)|$ divides $2m-2$, we have $[z,B^x]=1$.
Let $U_x=x^{-1}U$. The restriction $z|_{U_x}$
is a multiplier in $C_{q^{m-1}+1}$.
For $r<m$, an eigenvalue of $z|_{U_x}$ has multiplicity at least
$(2m-2)/(r-1)>2m/r$.
For $r=m$, we have $z|_{U_x}=I_{U_x}$ by Fermat's theorem.
Either case contradicts the image of $z$ under
$N_G(T)\to N_G(T)/T=\gen{\sigma^2}T/T$:
the field companion of $z$ permutes the eigenlines in cycles of length $r$.
Thus, each eigenvalue of $z$ has multiplicity $2m/r$.
Consequently, $A$ is a cyclic $2$-group.
For odd $m$, we have $|A|\mid m$.
As $A$ is a $2$-group, we have $A=1$.

\begin{itemize}
\item[{\bf Step 3.} ] If $m$ is even, then $A=1$.
\end{itemize}
Let $m$ be even and let $d=m-1$.
For this even $m$, we use the odd torus $B$ of signed type $(-d,+1)$, of order
$(q^d+1)(q-1)$.
Moreover, the normalizer $N_{\operatorname{O}(V)}(B)$ has Sylow $2$-subgroup $\gen{b_1}\times\gen{b_2}\cong C_2^2$.
Let $B_1,B_2$ be the torus factors on the $2$ summands of dimensions $2d$ and $2$.
For $i=1,2$ and $s\in B_i$, we have $b_i^{-1}sb_i=s^{-1}$.
We have $\operatorname{rank}(b_1-I_V)=d$ and $\operatorname{rank}(b_2-I_V)=1$.
Let $\delta:N_{\operatorname{O}(V)}(B)\to\F_2$ be $\delta(g)=\operatorname{rank}(g-I_V)\pmod2$.
Since $d$ is odd, the preceding rank formulas yield that $\delta(b_1)=\delta(b_2)=1$.
Therefore, $N_G(B)=\ker\delta$ has Sylow $2$-subgroup $\gen{b_1b_2}\cong C_2$.
For $q=2$, the plane torus is trivial. The orthogonal group of its fixed hyperbolic plane is $\operatorname{O}_2^+(2)=C_2$.
The normalizer $N_{\operatorname{O}(V)}(B)$ still has Sylow $2$-subgroup $\gen{b_1}\times\gen{b_2}\cong C_2^2$.
Since $\delta(b_1)=\delta(b_2)=1$, its Dickson kernel $N_G(B)$ has Sylow $2$-subgroup $\gen{b_1b_2}\cong C_2$.
We choose $x\in G$ such that $AB^x$ is a subgroup.
By Burnside's normal-complement theorem, $B^x\unlhd AB^x$.
Hence, $A\leq N_G(B^x)$ and $|A|\leq2$.

Suppose that $A\ne1$, and let $t$ generate $A$.
Since $|A|\leq2$ and $A\leq N_G(T)$, we obtain from Lemma \ref{lem:symplectic-singer-involutions} that $\operatorname{rank}(t-I_V)=m$.
Let $k=m-1$. Then, $k\geq3$ is odd.
We use $B_k=R_{\mathcal A}\rtimes\GL_k(q)$ on $E=\F_q^{2k}$ with a form of plus type.
For its normal matrix subgroup, let $\mathcal A=\{S\in M_k(\F_q):S^{\mathsf T}=S,\ S_{ii}=0\text{ for }1\leq i\leq k\}$. We take
$R_{\mathcal A}=\left\{\begin{pmatrix}I&S\\0&I\end{pmatrix}:S\in\mathcal A\right\}$.
By Lemma \ref{lem:symplectic-even-nonsingular}, we have $N_{\Sp_{2k}(q)}(B_k)=B_k$.
Let $D=E^\perp$ be anisotropic with $\dim_{\F_q}D=2$. We take $C=C_{q+1}$ on $D$.
Then, we have $B_k\times C\leq G$ and $Z(B_k\times C)=C$.
As $C$ is characteristic, every $g\in N_{\Sp(V)}(B_k\times C)$ satisfies $gE=E$ and $gD=D$.
By $N_{\Sp(E)}(B_k)=B_k$, we obtain that
$N_{\Sp(V)}(B_k\times C)=B_k\times D_{2(q+1)}$.
Every element of $B_k$ has Dickson invariant zero, while the
plane reflections have invariant $1$.
Therefore,
$N_G(B_k\times C)=B_k\times C_{q+1}$.
The involutions in $N_G(B_k\times C)$ have rank at most $k-1=m-2$.
No involution in $N_G(B_k\times C)$ has rank $m$.
By Lemma \ref{lem:involution-normalizers}, we conclude that $A=1$.
\end{proof}

\begin{theorem}\label{thm:orthogonal-minus-odd}
Let $q$ be odd and $m\geq4$.
Then, $\POmega^-_{2m}(q)$ has no proper nontrivial ordinary
$\POmega^-_{2m}(q)$-permutable subgroup.
\end{theorem}
\begin{proof}
For ease of reading, we divide the proof of the theorem into the following steps.

\setlist[itemize]{leftmargin=4em,labelsep=.5em,topsep=.2\baselineskip,partopsep=0pt,parsep=0pt,itemsep=0pt}
Let $G=\POmega^-_{2m}(q)$ and $p=\operatorname{char}\F_q$, so $q=p^f$.
Let $A<G$ be ordinary $G$-permutable.
\begin{itemize}
\item[{\bf Step 1.} ] $A$ normalizes $T^x$.
\end{itemize}
Next, let $P_2$ stabilize a totally singular $U\leq V=\F_q^{2m}$ with $\dim_{\F_q}U=2$.
The subgroup $P_2$ is maximal.
For every $Y<G$, we have $P_2Y\ne G$
\cite[Table A.5]{LiXia}\cite[Theorem 4.2]{King2005}.
By Lemma \ref{lem:max}, we obtain that $AY\ne G$ for every $Y<G$.
For the maximal anisotropic-plane stabilizer $M$,
some $x\in G$ satisfies $A^x\leq M$ \cite[Theorem 4.3]{King2005}.
We replace $A$ by $A^x$.
Let $\pi:\widetilde G=\Omega^-_{2m}(q)\to G$ be the quotient map and let $\widetilde A=\pi^{-1}(A)$.
Then, $\widetilde A$ stabilizes $V=V_1\perp V_2$, with $\dim_{\F_q}V_1=2$ and $\dim_{\F_q}V_2=2m-2$.
The central kernel has order at most $2$, prime to the characteristic $p$.

Let $S\leq\widetilde G$ be a Singer torus of order $(q^m+1)/2$.
For its projective image, let $T=\pi(S)$.
Moreover, $\ker\pi\leq S$.
Since $A$ is ordinary $G$-permutable, we choose $x\in G$ such that $K=AT^x$ is a subgroup.
Next, choose $\widetilde x\in\widetilde G$ with $\pi(\widetilde x)=x$.
We use $S^x=\widetilde x^{-1}S\widetilde x$.
Since $\ker\pi\leq Z(\widetilde G)$, the choice of $\widetilde x$ does not affect $S^x$.
The inverse image of $K$ is $\widetilde K=\widetilde A S^x$.
We have $\widetilde K<\widetilde G$ and $|\widetilde K|_p=|\widetilde A|_p$.
For each $p$-subgroup $P\leq\widetilde K$, some $g\in\widetilde K$ satisfies $P^g\leq\widetilde A$.
Thus, $P$ stabilizes $gV_1\perp gV_2$.

We use the normal base intersections from the proof of Theorem \ref{thm:orthogonal-minus-even}.
For odd $q$, we make the following changes to the relative tori and scalar fields.
At each relative $\Omega^-_{2a}(Q)$ stage, the intersection contains the Singer torus of order $(Q^a+1)/2=(q^m+1)/2$.
Thus, this torus has index $2$ in $\{M_\lambda:\lambda^{Q^a+1}=1\}$, where $M_\lambda(v)=\lambda v$.
Next, choose a primitive divisor $\ell>2m$ of $q^{2m}-1$.
The scalar field $F_0$ at each step lies in the commutative centralizer $\F_{q^{2m}}$ of the subgroup of order $\ell$.
Since $S^x$ centralizes $F_0$, its elements are $F_0$-linear.
At a relative unitary stage $J\leq\GU_b(Q)$ with odd $b$, let $Z=Z(\GU_b(Q))$.
The centre $Z$ has order $Q+1$ and lies in the cyclic torus of order $Q^b+1$.
Since $b$ is odd, we have $v_2(Q^b+1)=v_2(Q+1)$ and $|S^x\cap Z|=(Q+1)/2$.
Thus, $|S^xZ|=Q^b+1$.
We use Bereczky's classification \cite{Bereczky} for this Singer torus \cite[Theorem 2.1.7]{Brookfield2014}.
Because $JZ/J$ is a $p'$-group, every $p$-subgroup of $JZ$ lies in $J$.
Thus, each such subgroup stabilizes $gV_1\perp gV_2$ for some $g\in\widetilde K$.

By Lemmas \ref{lem:orthogonal-minus-two-subgroup} and \ref{lem:symplectic-unitary-sylow}, each relative group contains an indecomposable $p$-subgroup.
Thus, no intersection equals the relative orthogonal or unitary group.
The small odd-field Singer exceptions have relative field $3$
and dimension at most $6$, requiring original $m\leq3$.
Because $m\geq4$, no relative field and dimension satisfy those exceptional pairs.
At a relative $\Omega^-_4(Q)\cong\PSL_2(Q^2)$ stage, let $H$ denote the intersection.
The subgroup $H$ contains the nonsplit torus of order $(Q^2+1)/2$.
Since $Q\geq q^2\geq9$, the nonsplit torus has order at least $41$.
Dickson's list leaves only the nonsplit torus normalizer.
By the proof of Lemma \ref{lem:symplectic-odd-singer}, the Sylow $\ell$-subgroup $S_\ell$ of $S^x$ satisfies $S_\ell\unlhd\widetilde K$.
As $C_{\widetilde G}(S_\ell)=S^x$, we have $\widetilde A\leq N_{\widetilde G}(S^x)$ and $A\leq N_G(T^x)$.

\begin{itemize}
\item[{\bf Step 2.} ] $A$ is cyclic and $|A|\leq2$.
\end{itemize}
Replace $A$ by $xAx^{-1}$, with $x\in G$ from $K=AT^x$.
Then, $A\leq N_G(T)$.
Let $Q=q^m$. We identify $V$ with $\F_{Q^2}$ and use the quadratic form $v\mapsto\Tr_{\F_Q/\F_q}(v^{Q+1})$.
For $c\in\F_{Q^2}^*$, let $M_c(v)=cv$. We take $S_0=\{M_c:c^{Q+1}=1\}\leq\operatorname{SO}^-_{2m}(q)$.
Then, we have $|S_0|=Q+1$.
Choose a generator $M_\lambda$ of $S_0$.
Next, let $\sigma:V\to V$ satisfy $\sigma(v)=v^q$.
In a normal basis of $\F_{Q^2}$ over $\F_q$, the matrix of $\sigma$ is a cycle of length $2m$.
Thus, $\det\sigma=-1$.
Therefore, $N_{\operatorname{SO}^-_{2m}(q)}(S_0)/S_0\cong C_m$.
Let $\chi_q$ be the quadratic character of $\F_q^\times$.
If $F$ is a Gram matrix, then $\chi_q(\det F)=-\chi_q((-1)^m)$.
Let $v=N_{\F_{Q^2}/\F_Q}(1-\lambda)$.
Then, $v=-\lambda^{-1}(1-\lambda)^2$ and $v^{(Q-1)/2}=(-1)^{(Q-1)/2}$.
Let $\operatorname{sn}:\operatorname{SO}(V)\to\F_q^*/(\F_q^*)^2$ be the spinor norm.
By \cite[Proposition 1.6.11(ii)]{BHR}, we have $\operatorname{sn}(M_\lambda)=N_{\F_{Q^2}/\F_q}(1-\lambda)\det F\cdot(\F_q^*)^2$.
From the preceding norm and discriminant formulas, we obtain that
$$
 \chi_q(N_{\F_{Q^2}/\F_q}(1-\lambda)\det F)
 =(-1)^{(Q-1)/2}\chi_q(\det F)
 =\chi_q((-1)^m)\chi_q(\det F)=-1.
$$
Thus, $S=S_0^2$ has index $2$ in $S_0$.
For each $g\in N_{\operatorname{SO}(V)}(S_0)$, choose $h\in S_0$ with $\operatorname{sn}(hg)=1$.
Then, $hg\in N_{\widetilde G}(S)$ induces the field automorphism induced by $g$.
Since $C_{\operatorname{SO}(V)}(S)=S_0$, every element normalizing $S$ also normalizes $S_0$.
Consequently,
$N_{\widetilde G}(S)/S\cong N_G(T)/T\cong C_m$.

By Lemma \ref{lem:max} for the maximal singular line stabilizer, $\widetilde A$ stabilizes a line $\Lambda\leq V$.
A Singer multiplier stabilizing $\Lambda$ is $\pm I$ and has trivial projective image.
Hence, $A\cap T=1$, and $A$ is cyclic with $|A|\mid m$.
We take a primitive divisor $h\geq2m-1>m$ of $q^{2m-2}-1$.
Next, we choose a nondegenerate subspace $U\leq V$ of minus type with $\dim_{\F_q}U=2m-2$. Let $W=U^\perp$, so $V=U\perp W$ and $\dim_{\F_q}W=2$.
The space $U$ has minus type and $W$ is hyperbolic.
Take $B=C_h$ irreducible on $U$ and trivial on $W$.
Since $A$ is ordinary $G$-permutable, we choose $x\in G$ such that $AB^x$ is a subgroup.
By Sylow counting, $B^x\unlhd AB^x$.
Let $a\in A$ have odd prime order $r$.
Since $r\mid m$ and $|N_G(B^x)/C_G(B^x)|$ divides $2m-2$, we have $[a,B^x]=1$.
If $r=p$, then $C_G(B^x)$ has order prime to $p$ and cannot contain $a$. We obtain a contradiction.
For $r\ne p$, every eigenvalue of the Singer companion of order $r$ has multiplicity $2m/r$ by Theorem \ref{thm:orthogonal-minus-even}.
In the centralizer torus of $B^x$, an eigenvalue has multiplicity greater than $2m/r$. We obtain a contradiction.
Thus, $A$ is a cyclic $2$-group.
For odd $m$, we have $|A|\mid m$.
As $A$ is a $2$-group, we have $A=1$.

Let $m$ be even.
Since $q^m+1\equiv2\pmod4$, we have $-I_V\notin\widetilde G$ and $Z(\widetilde G)=1$.
An involution $t\in N_{\widetilde G}(S)$ has form $v\mapsto c v^{q^m}$,
with $c^{q^m+1}=1$.
By Hilbert 90, choose $a\in\F_{q^{2m}}^*$ with $c=a^{1-q^m}$.
For the multiplier $h:v\mapsto av$ in $\GL_{2m}(q)$, we have $t^h:v\mapsto v^{q^m}$.
Both eigenspaces of $t$ have dimension $m$.
If $|A|\geq4$, then the involution $t\in A$ lies in $A^2$.
Since $|N_G(B^x)/C_G(B^x)|$ divides $2m-2$, the image of $A$ in this quotient has order at most $2$.
Thus, the image of $t\in A^2$ is trivial, and $t$ centralizes $B^x$.
For $x\in G$ with $AB^x\leq G$, choose $\widetilde x\in\widetilde G$ satisfying $\pi(\widetilde x)=x$.
Let $U_x=\widetilde x^{-1}U$, so $\dim_{\F_q}U_x=2m-2$.
Since $t$ centralizes the irreducible subgroup $B^x$ on $U_x$ and $t^2=I_V$, we have $t|_{U_x}=I_{U_x}$ or $-I_{U_x}$.
Thus, an eigenspace of $t$ has dimension at least $2m-2>m$.
This contradicts the dimension $m$ of both eigenspaces. Therefore, $|A|\leq2$.

\begin{itemize}
\item[{\bf Step 3.} ] If $m$ is even, then $A=1$.
\end{itemize}
Let $k=m-1$. Then, $k\geq3$ is odd.
We choose $V=E\perp D$ with $E=U\oplus U^*$ of plus type and $D$ anisotropic.
For the dimensions, we require $\dim_{\F_q}U=k$ and $\dim_{\F_q}D=2$.
On $E=U\oplus U^*$, let
$$
 W=\left\{\begin{pmatrix}I_k&S\\0&I_k\end{pmatrix}:
 S\in M_k(\F_q),\ S^{\mathsf T}=-S\right\}.
$$
Next, identify $U$ with $\F_{q^k}$ over $\F_q$.
For this field structure, let $H=\{M_{\lambda^2}:\lambda\in\F_{q^k}^*\}$, where $M_\lambda(u)=\lambda u$.
To act on $E=U\oplus U^*$, we use $\iota:\GL(U)\to\operatorname{SO}(E)$ with $\iota(g)=\diag(g,g^{-\mathsf T})$.
Then, we have $|H|=(q^k-1)/2$.
Since $H$ contains an element of primitive prime order for $q^k-1$, we have $H$ irreducible on $U$.
For $g\in\GL(U)$, the spinor norm of $\iota(g)$ is $\det(g)\cdot(\F_q^*)^2$.
As $\det(M_{\lambda^2})=N_{\F_{q^k}/\F_q}(\lambda)^2$, every element of $\iota(H)$ has spinor norm $1$.
Thus, $B=W\rtimes\iota(H)$ lies in $\widetilde G$ and fixes $D$ pointwise.
We have $O_p(B)=W$ and $\sum_{w\in W}\operatorname{im}(w-I_V)=U$.
Because $H$ is irreducible on $U$, we have $V^B=D$.
Thus, every $g\in N_{\widetilde G}(B)$ stabilizes $E,D$ and $U$.
Let $g\in N_{\widetilde G}(B)$ be an involution.
For the action on $U$, let $\nu:\operatorname{Stab}_{\operatorname{O}(E)}(U)\to\GL(U)$ send $h$ to $h|_U$. Its kernel is $W$.
Next, let $\sigma:U\to U$ satisfy $\sigma(u)=u^q$ in $U=\F_{q^k}$.
Since $\nu(g|_E)$ normalizes $H$, it normalizes $C_{\operatorname{End}_{\F_q}(U)}(H)=\F_{q^k}$.
Thus,
$g|_E\in W\rtimes\iota(\{M_\lambda:\lambda\in\F_{q^k}^*\}\rtimes\gen\sigma)$.
Because $k$ is odd and $g^2=I_V$, the image of $g$ in $\gen\sigma\cong C_k$ is trivial.
Hence, $\nu(g|_E)=M_\lambda$ with $\lambda^2=1$, so $g|_E=w\iota(\pm I_U)$ for some $w\in W$.
As both scalars act trivially on $W$, we have $I_E=(g|_E)^2=w^2$.
Since $|W|$ is odd, $w=1$, and we obtain that $g|_E=\pm I_E$.
As $\det(g)=\det(g|_E)=1$, we have $g|_D\in\operatorname{SO}(D)$.
Because $g^2=I_V$ and $\operatorname{SO}(D)$ is cyclic, we have $g|_D=I_D$ or $-I_D$.
Thus, $\dim\ker(g+I_V)\in\{0,2,2m-2,2m\}$.
No listed negative eigenspace dimension equals $m$.
For every involution $t\in N_{\widetilde G}(S)$, we have $\dim_{\F_q}\ker(t+I_V)=m$.
By Lemma \ref{lem:involution-normalizers}, we conclude that $A=1$.
\end{proof}

\begin{corollary}\label{cor:orthogonal-minus-all}
Let $m\geq3$ and let $q$ be any prime power.
Then, $\POmega^-_{2m}(q)$ has no proper nontrivial ordinary
$\POmega^-_{2m}(q)$-permutable subgroup.
\end{corollary}
\begin{proof}
For $m\geq4$, we obtain the conclusion from Theorems \ref{thm:orthogonal-minus-even} and \ref{thm:orthogonal-minus-odd}.
For $m=3$, we use $\POmega^-_6(q)\cong\PSU_4(q)$ and Corollary \ref{cor:all-unitary}.
\end{proof}

\subsection{\texorpdfstring{The groups $\POmega^+_{2m}(q)$}{The groups POmega plus 2m(q)}}\label{sec:orthogonal-plus}

Let $q=p^f$ with $p$ prime and $f\geq1$.
For $d\geq1$ and a prime $\ell\ne p$, we call $\ell$ an absolute primitive prime divisor of $q^d-1$ if $\ord_\ell(p)=fd$.

\begin{lemma}\label{lem:orthogonal-plus-odd-subgroup}
Let $q$ be even and let $m\geq4$ be even. Let $G=\Omega^+_{2m}(q)$.
On $V=\F_q^{2m}$, $G$ contains a subgroup $H$ of odd order
such that $\operatorname{rank}(t-I_V)\leq m-1$
for every involution $t\in N_G(H)$.
\end{lemma}
\begin{proof}
We take $H=\gen{x,y\mid x^9=y^3=1,\ yxy^{-1}=x^4}$, of order $27$.
For its representation, let $E_0=\F_{64}$ and $D_0=\F_4$, viewed as vector spaces over $\F_2$.
On these spaces, we use $Q_E(v)=\Tr_{\F_8/\F_2}(v^9)$ and $Q_D(v)=v^3$.
Next, choose $\lambda\in E_0^*$ and $\eta\in D_0^*$ of orders $9$ and $3$.
Using these elements, let $\rho_0:H\to\operatorname{O}(E_0\perp D_0)$ satisfy
$\rho_0(x)(v,w)=(\lambda v,w)$ and
$\rho_0(y)(v,w)=(v^4,\eta w)$.
Suppose that $\rho_0(x^a y^b)=I$ with $0\leq a<9$ and $0\leq b<3$.
As $\rho_0(x^a y^b)|_{D_0}=M_{\eta^b}$, we have $\eta^b=1$.
Because $|\eta|=3$ and $0\leq b<3$, we obtain that $b=0$.
Since $|\rho_0(x)|=9$, we then have $a=0$.
Thus, $\rho_0$ is faithful.
Moreover, $E_0,D_0$ have minus type, so $E_0\perp D_0$ has plus type.
Next, let $E=E_0\otimes_{\F_2}\F_q$ and $D=D_0\otimes_{\F_2}\F_q$.
For $X=E_0,D_0$, we choose a basis $v_1,\ldots,v_d$ over $\F_2$.
Let $B_X(v,w)=Q_X(v+w)-Q_X(v)-Q_X(w)$, where $Q_{E_0}=Q_E$ and $Q_{D_0}=Q_D$.
For $a_1,\ldots,a_d\in\F_q$, we use the quadratic form on $X\otimes_{\F_2}\F_q$ satisfying
$$
 Q_{X,q}\left(\sum_{i=1}^d a_i(v_i\otimes1)\right)
 =\sum_{i=1}^d a_i^2Q_X(v_i)+\sum_{i<j}a_i a_jB_X(v_i,v_j).
$$
We choose $V=E\perp D\perp C$ with these forms on $E,D$.
Moreover, the space $C$ has plus type and $\dim_{\F_q}C=2m-8$.
Let $\rho:H\to\operatorname{O}(V)$ satisfy $\rho(h)=(\rho_0(h)\otimes I_{\F_q})\oplus I_C$.
Then, we identify $H$ with $\rho(H)$. Since $|H|$ is odd, $H\leq\Omega(V)=G$.

Let $W\leq V$ satisfy $\rho(h)W=W$ for $h\in H$, and let $J\leq H$.
Let $[W,J]=\sum_{j\in J}(\rho(j)-I)W$ and $W^J=\{w\in W:\rho(j)w=w\text{ for all }j\in J\}$.
From the characteristic centre $Z(H)=\gen{x^3}$, we have $[V,Z(H)]=E$ and $[V^{Z(H)},H]=D$.
We also have $V^H=C$.
Thus, every $g\in N_G(H)$ satisfies $gE=E,\ gD=D,\ gC=C$.
Let $\rho_D:H\to\operatorname{O}(D)$ send $h$ to $\rho(h)|_D$.
Since $\ker\rho_D=\gen{x}$, we have $g^{-1}\gen{x}g=\gen{x}$ for $g\in N_G(H)$.
Suppose that $g^{-1}yg\gen{x}=y^{-1}\gen{x}$.
Then, $g^{-1}xg=x^a$ and $g^{-1}yg=x^by^{-1}$, with $(a,9)=1$.
From $g^{-1}(yxy^{-1})g=g^{-1}x^4g$, we have $x^{7a}=x^{4a}$. This contradicts $(a,9)=1$.

Let $t\in N_G(H)$ be an involution.
Because the preceding computation forbids $t^{-1}yt\gen x=y^{-1}\gen x$, we have $t^{-1}yt\gen x=y\gen x$.
As $\ker\rho_D=\gen x$, the restriction $t|_D$ centralizes $\rho_D(H)=C_3$.
Since $C_{\operatorname{O}(D)}(\rho_D(H))$ is a torus of odd order,
$t|_D=I_D$.
Also, $\operatorname{rank}(t|_E-I_E)\leq3$ and
$\operatorname{rank}(t|_C-I_C)\leq m-4$.
Therefore, $\operatorname{rank}(t-I_V)\leq m-1$.
\end{proof}

\begin{theorem}\label{thm:orthogonal-plus}
Let $G=\POmega^+_{2m}(q)$ with $m\geq4$.
Suppose that $m$ is odd or $m\geq6$ with $(m,q)\ne(6,2)$ or $m=4$ with $q\geq4$.
Then, $G$ has no proper nontrivial ordinary $G$-permutable subgroup.
\end{theorem}
\begin{proof}
For ease of reading, we divide the proof of the theorem into the following steps.

\setlist[itemize]{leftmargin=4em,labelsep=.5em,topsep=.2\baselineskip,partopsep=0pt,parsep=0pt,itemsep=0pt}
\begin{itemize}
\item[{\bf Step 1.} ] $K=\widetilde A R$ is a proper subgroup of $\Omega^+_{2m}(q)$.
\end{itemize}
Let $p=\operatorname{char}\F_q$ and $k=m-1$.
Let $V=\F_q^{2m}$ with quadratic form $\mathcal Q$ and polar form $\beta(v,w)=\mathcal Q(v+w)-\mathcal Q(v)-\mathcal Q(w)$.
Next, let $\pi:\Omega^+_{2m}(q)\to G$ be the quotient map.
Let $A<G$ be ordinary $G$-permutable and let $\widetilde A=\pi^{-1}(A)$.
Now, let $P_2$ be the maximal stabilizer of a totally singular $X\leq V$ with $\dim_{\F_q}X=2$.
We have $P_2Y\ne G$ for every $Y<G$
\cite[Tables A.6--A.7]{LiXia}\cite[Theorem 4.2]{King2005}.
By Lemma \ref{lem:max}, we obtain that $AY\ne G$ for every $Y<G$.
Moreover, for each maximal subgroup $M<G$, we obtain from Lemma \ref{lem:max} an element $x\in G$ with $A\leq M^x$.

For odd $q$, choose the maximal stabilizer of $V=V_1\perp V_2$ with $\dim_{\F_q}V_1=3$.
For even $q$ and $m\geq5$, choose such a stabilizer with $V_1$ of minus type and $\dim_{\F_q}V_1=4$ \cite[Theorem 4.3]{King2005}.
For $m=4$ and even $q$, use the maximal stabilizer of $\{V_1,V_2\}$ \cite[Theorem 4.7]{King2005}.
Both $V_i$ have minus type and $\dim_{\F_q}V_i=4$.
Let $M$ be the chosen decomposition stabilizer.
By Lemma \ref{lem:max}, we choose $y\in G$ with $A\leq M^y$.
We choose a lift $\widetilde y\in\Omega^+_{2m}(q)$ of $y$ and replace $V_i$ by $\widetilde y^{-1}V_i$ for $i=1,2$.
Then, $\widetilde A$ stabilizes the chosen decomposition.
For $m=4$ and even $q$, each $g\in\widetilde A$ satisfies
$(gV_1,gV_2)=(V_1,V_2)$ or $(V_2,V_1)$.
If $p\nmid|B|$ and $K=AB^x$ is a subgroup, then each $p$-subgroup $P\leq K$ has $P^g\leq A$ for some $g\in K$.
Choose $\widetilde g\in\Omega^+_{2m}(q)$ with $\pi(\widetilde g)=g$. Then, $\pi^{-1}(P)$ stabilizes $\widetilde gV_1\perp\widetilde gV_2$.

We choose a nondegenerate subspace $E\leq V$ with $\dim_{\F_q}E=2k$. Let $D=E^\perp$, so $V=E\perp D$ and $\dim_{\F_q}D=2$.
The space $E$ has minus type and $D$ is anisotropic.
On these spaces, choose Singer tori $T_E\leq\operatorname{SO}(E)$ and $T_D\leq\operatorname{SO}(D)$ with $|T_E|=q^k+1$ and $|T_D|=q+1$.
For their product, let $R_0=\{h_E\oplus h_D:h_E\in T_E,\ h_D\in T_D\}$ and $R=R_0\cap\Omega^+_{2m}(q)$.
If $q$ is odd, then $[R_0:R]=2$. If $q$ is even, then $R=R_0$.
Let $\pi_E:R\to\operatorname{O}(E)$ be $g\mapsto g|_E$.
For odd $q$, the spinor norm is surjective on $T_D$. Thus, every $h_E\in T_E$ has $h_D\in T_D$ with $h_E\oplus h_D\in R$.
We identify $E$ with $\F_{q^{2k}}$.
Then, $\pi_E(R)=\{M_\lambda:\lambda^{q^k+1}=1\}$ with $M_\lambda(v)=\lambda v$.
We have $|\pi_E(R)|=q^k+1$.
Let $\overline R$ be the image of $R$ in $G$.
Since $A$ is ordinary $G$-permutable, we choose $x\in G$ such that $A\overline R^x$ is a subgroup.
As $A\overline R^x\ne G$, we have $A\overline R^x<G$.
Choose a lift $\widetilde x\in\Omega^+_{2m}(q)$ of $x$.
The inverse image of $A\overline R^x$ under $\pi$ is $K=\widetilde A R^{\widetilde x}<\Omega^+_{2m}(q)$.
Then, we replace $K,\widetilde A,A$ by
$\widetilde xK\widetilde x^{-1}$, $\widetilde x\widetilde A\widetilde x^{-1}$ and $xAx^{-1}$.
Also, replace $V_i$ by $\widetilde xV_i$ for $i=1,2$.
Then, $K=\widetilde A R$, and $\widetilde A$ stabilizes the transported decomposition.

\begin{itemize}
\item[{\bf Step 2.} ] The only nonzero proper $R$-submodules of $V$ are $E$ and $D$.
\end{itemize}
Firstly, we show that $gE=E$ and $gD=D$ for $g\in K$.
We choose $c_0\in R$ in the primitive-divisor class prescribed in
\cite[Lemma 2.9 and Table 4]{GuralnickMalle2012}.
Let $\mathcal E=\{(4,2),\allowbreak(4,3),\allowbreak(4,4),\allowbreak(4,5),\allowbreak(5,2),\allowbreak(6,2),\allowbreak(7,2),\allowbreak(10,2),\allowbreak(11,2)\}$.
For $(m,q)\notin\mathcal E$, $c_0$ has order greater than $4k+1$ and an absolute primitive prime divisor.
The hypotheses omit $(4,2),(4,3),(6,2)$.
For $(m,q)\in\mathcal E$ under the hypotheses, we choose $\ell_0$ from
$$
 (m,q,\ell_0)\in\{(4,4,13),(4,5,7),(5,2,17),(7,2,13),(10,2,19),(11,2,41)\}.
$$
We have $\ell_0\mid|c_0|$ and $\operatorname{ord}_{\ell_0}(p)=2fk$, where $q=p^f$.
For $(m,q)\notin\mathcal E$, choose an absolute primitive divisor $\ell_0\mid|c_0|$.
We take $C_0=\gen{c_0^{|c_0|/\ell_0}}$.
Since $\operatorname{ord}_{\ell_0}(q)=2k$, $C_0$ acts irreducibly on $E$ and trivially on $D$.
Thus, $D$ is the unique $C_0$-submodule of $V$ with $\dim_{\F_q}D=2$.
The image $\{g|_D:g\in R\}$ is irreducible on $D$.
Thus, the only nonzero proper $R$-submodules of $V$ are $E$ and $D$.
Let $H\leq\Omega^+_{2m}(q)$ with $R\leq H$, and suppose that $H$ acts reducibly on $V$.
As $H$ consists of isometries, we have $gE=E$ and $gD=D$ for every $g\in H$.

Suppose that $m,q$ are odd.
Suppose that $R\leq N_{\Omega^+_{2m}(q)}(\Omega_m(q^2))$, where $\Omega_m(q^2)$ acts linearly over $\F_{q^2}$.
Since $\ell_0$ is odd and the normalizer quotient has order $2$, we have $C_0\leq\Omega_m(q^2)$.
Since $C_0$ acts irreducibly on $E$, its ring of commuting $\F_q$-linear maps is a field.
This field contains $\F_{q^2}$ and $\{r|_E:r\in R\}$, so $r|_E$ commutes with $\F_{q^2}$ for every $r\in R$.
For such $r$, choose $\tau_r\in\operatorname{Gal}(\F_{q^2}/\F_q)$ with $r(av)=\tau_r(a)r(v)$ for $a\in\F_{q^2}$ and $v\in V$.
For $0\ne e\in E$, the restriction gives $a\,r(e)=r(ae)=\tau_r(a)r(e)$.
Since $r(e)\ne0$, we obtain that $\tau_r(a)=a$ for all $a\in\F_{q^2}$.
Thus, $r$ is $\F_{q^2}$-linear on $V$.
Consequently, $r|_D=I_D$ or $-I_D$, as $\dim_{\F_{q^2}}D=1$.
As $|\{r|_D:r\in R\}|=q+1>2$, we obtain a contradiction.
Thus, $R\nleq N_{\Omega^+_{2m}(q)}(\Omega_m(q^2))$.
At $(m,q)=(7,2)$, the torus contains an element of order $195$.
The extra groups $A_{16}$, $\PSL_2(13)$ and $G_2(3).2$ do not.
A permutation of order $195$ requires at least $13+5+3$ points,
so $A_{16}$ has no element of order $195$.
Neither $|\PSL_2(13)|$ nor $|G_2(3).2|$ is divisible by $5$.
At $(m,q)=(10,2)$, the torus $R$ contains an element of order $513$.
The $3$-parts of $|\PSL_2(19)|$ and $|J_1|$ are smaller than $27$, so neither group contains $R$.

\begin{itemize}
\item[{\bf Step 3.} ] $K$ stabilizes $E$ and $D$.
\end{itemize}
For even $m$, we also consider $K\leq M_U=N_{\Omega^+_{2m}(q)}(\SU_m(q))$.
Let $J_1=K\cap\SU_m(q)$, so $J_1\unlhd K$.
The subgroup $J_1$ contains the torus of type $(k,1)$ and order $q^k+1$.
For ordered summands, a Sylow $p$-subgroup of $\SU_m(q)$ is indecomposable by Lemma \ref{lem:symplectic-unitary-sylow}.
It cannot stabilize the displayed decomposition.
For the unordered pair $\{V_1,V_2\}$, both summands have dimension $4$ over $\F_q$.
Since $q\geq4$, we have $|\SU_4(q)|_2=q^6>8q^4$.
The decomposition stabilizer has $2$-part at most $8q^4$.
Therefore, $J_1<\SU_m(q)$.
For $(m,q)\ne(4,5)$, $J_1$ stabilizes a nondegenerate line over $\F_{q^2}$
\cite[Lemma 2.6]{GuralnickMalle2012}.
Under the hypotheses, $(m,q)\notin\{(6,2),(4,2),(4,3)\}$.
At $(4,5)$, the extra $A_7$ has no element of order $63$.
As $C_0$ acts irreducibly on the complementary $\F_{q^2}$-space of dimension $k$, the unitary line is unique.
Because $J_1\unlhd K$, $K$ stabilizes that line.
The underlying $\F_q$-space is $D$, with $\dim_{\F_q}D=2$.

For $m=4$, we compare $\overline R$ with the listed maximal subgroups $M<G$ of spin or adjoint unitary type.
Let $V=\F_q^8$.
For the spin type, we use the spin representation $\rho_s:\operatorname{Spin}_7(q)\to\GL(V)$.
For the unitary type, we use the adjoint representation $\rho_u:\PSU_3(q)\to\GL(V)$.
The unitary type requires $q\equiv2\pmod3$ \cite[Table 8.50]{BHR}.
Suppose that $\overline R\leq M$, and let $\overline C_0=\pi(C_0)$.
For the spin type, we have $|C_M(\overline C_0)|\leq q^3+1$.
For odd $q\geq5$, we have $|\overline R|=(q^3+1)(q+1)/4>q^3+1$.
For even $q\geq4$, we have $|\overline R|=(q^3+1)(q+1)>q^3+1$.
For the adjoint unitary type, we have $|C_M(\overline C_0)|\leq2(q^2-q+1)$.
Thus, $\overline R\nleq M$ for either type.
At $q=5$, the other listed groups have no element of odd order $63$.
For $A_{10}$, this follows by cycle lengths, and for $\Sz(8)$
by its order.
In $\Omega^+_8(2)$, an element $t$ of order $7$ has $2$ reciprocal irreducible constituents of dimension $3$ over $\F_2$.
Its fixed space is hyperbolic of dimension $2$.
Thus, $9\nmid|C_{\Omega^+_8(2)}(t)|$.
For a central quotient $\rho:\widetilde H\to H$ with $|\ker\rho|=2$, we have $|\rho(h)|=|h|$ whenever $h\in\widetilde H$ has odd order.
Therefore, $gE=E$ and $gD=D$ for $g\in K$.

\begin{itemize}
\item[{\bf Step 4.} ] We have $J\cap H<H$ for each nonabelian relative group $H=\Omega^-_{2a}(Q)$ or $\SU_b(Q)$.
\end{itemize}
For odd $q$, let
$K_0=\{g\in K:\det(g|_E)=\det(g|_D)=1\}$.
For even $q$ and $X=E,D$, let $D_X:\operatorname{O}(X)\to\F_2$ send $h$ to $\operatorname{rank}(h-I_X)\pmod2$.
We use $K_0=\{g\in K:D_E(g|_E)=D_D(g|_D)=0\}$.
Then, $K_0\unlhd K$ and $[K:K_0]\leq2$.
Let $\pi_E:K_0\to\operatorname{O}(E)$ be $g\mapsto g|_E$,
and let $J=\pi_E(K_0)$.
We have $\pi_E(R)\leq J$ and $|\pi_E(R)|=q^k+1$.
Moreover, $\ker\pi_E\leq\{I_E\}\times\operatorname{SO}(D)$ is central and has order prime to $p$.
By Bereczky's theorem \cite{Bereczky,Brookfield2014}, we use successive relative orthogonal groups $\Omega^-_{2a}(Q)$ and unitary groups $\GU_b(Q)$ with odd $b$.
At each step, let $H_{\rm rel}\leq\operatorname{O}(E)$ be the chosen relative subgroup containing $\pi_E(C_0)$.
For $H_{\rm rel}$, let $F_0$ be its scalar field.
Next, let $H_{\rm lin}\unlhd H_{\rm rel}$ be its subgroup of $F_0$-linear maps.
Since $[H_{\rm rel}:H_{\rm lin}]\leq2k<\ell_0$, the image of $\pi_E(C_0)$ in $H_{\rm rel}/H_{\rm lin}$ is trivial.
Hence, $h|_E(av)=a(h|_E)(v)$ for $h\in C_0$, $a\in F_0$ and $v\in E$.
As $C_0$ acts irreducibly on $E$, we have $C_{\operatorname{End}_{\F_q}(E)}(\{h|_E:h\in C_0\})\cong\F_{q^{2k}}$.
It contains $\pi_E(R)$ and $F_0$.
Thus, every element of $\pi_E(R)$ is $F_0$-linear.

By Lemmas \ref{lem:orthogonal-minus-two-subgroup} and \ref{lem:symplectic-unitary-sylow}, each relative orthogonal or special unitary group contains an indecomposable $p$-subgroup on $E$, trivial on $D$.
Thus, the indecomposable summands have dimensions $2k,1,1$.
By Krull--Schmidt, no direct summand has the required dimension $3$ or $4$.
For the unordered pair, choose $V=V_1\perp V_2$ with $\dim_{\F_q}V_i=4$ for $i=1,2$.
Let $P$ be a Sylow $p$-subgroup of the relative group.
Let $\sigma:P\to S_2$ satisfy $\sigma(g)(i)=j$ when $gV_i=V_j$, and let $P_0=\ker\sigma$.
If some $g\in P$ satisfies $gV_1=V_2$, then $g^2\in P_0$.
The map $v\mapsto v+gv$ is an isomorphism $V_1^{P_0}\to V^P$.
Thus, $\dim_{\F_q}V_i^{P_0}=\dim_{\F_q}V^P$ for $i=1,2$.
For a Sylow subgroup of the relative $\SU_3(q)$, we have $\dim_{\F_q}V_i^{P_0}=4$.
Thus, $P_0=1$. We obtain a contradiction.
For a Sylow subgroup of the relative $\Omega^-_6(q)$, we have $\dim_{\F_q}V_i^{P_0}=3$ for $i=1,2$.
For $i=1,2$, let $P_i=\{g|_{V_i}:g\in P_0\}\leq\operatorname{O}(V_i)$.
Each $P_i$ fixes a hyperplane pointwise.
For each $i$, let $Y_i\leq V_i$ be a hyperplane fixed pointwise by $P_i$ and choose $0\ne v_i\in Y_i^\perp$.
As each $g\in P_i$ fixes $Y_i$ pointwise and is an isometry, it has the form $g(w)=w+a\beta(w,v_i)v_i$ for some $a\in\F_q$.
Since $g$ is an isometry, $a+a^2\mathcal Q(v_i)=0$.
This equation has at most $2$ solutions, so $|P_i|\leq2$.
Thus, $|P|\leq8$. This contradicts $|P|=q^6$.

\begin{itemize}
\item[{\bf Step 5.} ] $R$ is normal in $K$ and $N_G(\overline R)/\overline R\cong C_{2k}$.
\end{itemize}
We check the relative small exceptions under the hypotheses on $(m,q)$.
For odd $m$, the integer $k$ is even.
At a relative unitary stage of odd dimension $b\geq3$, the field has order $Q=q^{k/b}$.
Since $k$ is even and $b$ is odd, $k/b$ is even.
Thus, $Q\geq q^2$ and $Q\notin\{2,3\}$.
For a relative $\Omega^-_4(Q)$ stage, we have $Q\geq q^2$.
By Dickson's theorem, the proper intersection lies in the nonsplit torus normalizer.
For even $m\geq6$, a relative $\SU_5(2)$ stage requires $(m,q)=(6,2)$.
The hypotheses omit that pair.
For $m=4$, both relative exceptions of dimension $6$ have $q<4$.

Let $H$ be the last normal intersection.
Since $H$ lies in the relative normalizer of $\pi_E(R)$, the Sylow $\ell_0$-subgroup of $\pi_E(R)$ is the unique Sylow $\ell_0$-subgroup of $H$.
By the proof of Lemma \ref{lem:symplectic-odd-singer}, the Sylow $\ell_0$-subgroup is characteristic in each successive normal intersection.
Let $R_{\ell_0}$ be the Sylow $\ell_0$-subgroup of $R$.
By the preceding argument, $J$ has a unique Sylow $\ell_0$-subgroup $J_{\ell_0}$.
As $\ker\pi_E$ is central and $\ell_0\nmid|\ker\pi_E|$, the group $\pi_E^{-1}(J_{\ell_0})$ has a unique Sylow $\ell_0$-subgroup.
This subgroup is $R_{\ell_0}$, and every Sylow $\ell_0$-subgroup of $K_0$ lies in $\pi_E^{-1}(J_{\ell_0})$.
Thus, $R_{\ell_0}$ is characteristic in $K_0$.
Since $K_0\unlhd K$, we have $R_{\ell_0}\unlhd K$.
Moreover, $C_{\Omega^+_{2m}(q)}(R_{\ell_0})=R$.
Because $R\leq K$, we have $R\unlhd K$ and $A\leq N_G(\overline R)$.
In the field model $E=\F_{q^{2k}}$, the companion maps are
$v\mapsto v^{q^j}$.
For odd $q$, the determinant of $v\mapsto v^{q^j}$ on $E$ is $(-1)^j$.
For even $q$, its Dickson invariant is $j\pmod2$.
The companion map on $D$ must have determinant $(-1)^j$ or Dickson invariant $j\pmod2$, respectively.
For odd $q$, we multiply by $h_D\in T_D$ so that the map on $V$ has spinor norm $1$.
Therefore, $N_G(\overline R)/\overline R\cong C_{2k}$.

\begin{itemize}
\item[{\bf Step 6.} ] $A$ is cyclic and $|A|\mid2(m-1)$.
\end{itemize}
By Lemma \ref{lem:max} for $P_1$, $\widetilde A$ stabilizes a singular line $\Lambda\leq V$.
Since $D$ is anisotropic, we choose $v=e+d\in\Lambda\setminus\{0\}$ with $0\ne e\in E$ and $d\in D$.
For $g\in\widetilde A\cap R$, let $\lambda\in\F_{q^{2k}}^*$ be its multiplier on $E$.
As $g\Lambda=\Lambda$, we have $gv=cv$ for some $c\in\F_q^*$.
Taking the projection on $E$, we obtain that $\lambda e=ce$, so $\lambda=c$.
Because $\lambda^{q^k+1}=1$, we have $c^2=1$.
Thus, $g|_E=I_E$ or $-I_E$.
By Lemma \ref{lem:max} for a maximal isotropic parabolic, $\widetilde A$ stabilizes a totally singular $W\leq V$ with $\dim_{\F_q}W=m$.
Suppose that $g\in\widetilde A\cap R$ has $g|_E=cI_E$ and $g|_D\ne cI_D$.
Since $g$ is semisimple and $c$ is not an eigenvalue of $g|_D$, we would have $W=(W\cap E)\oplus(W\cap D)$.
As $D$ is anisotropic and $W$ is totally singular, we would obtain that $W\cap D=0$ and $W\leq E$.
This contradicts $\dim_{\F_q}W=m$, because $E$ has Witt index $m-2$.
Thus, every element of $\widetilde A\cap R$ is a scalar on $V$.
Hence,
$A\cap\overline R=1, A\cong C_a, a\mid2(m-1)$.

Since $(m,q)\ne(6,2)$, we choose a primitive prime $\ell$ of $q^m-1$ by Zsigmondy's theorem.
We take $B=C_\ell$ in the corresponding linear torus of $G$.
As $\ell\geq m+1$, we have $\ell\nmid|A|$.
Because $A$ is ordinary $G$-permutable, choose $x\in G$ with $L=AB^x\leq G$.
As $A$ and $B^x$ are abelian, $L$ is metabelian by It\^o's theorem.
If $\ell\mid|L'|$, then $B^x$ is the unique Sylow $\ell$-subgroup of $L'$
and is normal in $L$.
Otherwise, $AL'$ is an $\ell$-free subgroup containing the Hall
subgroup $A$, so $AL'=A$.
Then, $A\unlhd L$.
Because $\ell\nmid|\Aut(A)|$, the subgroup $B^x$ centralizes $A$.
Thus, $B^x\unlhd L$.

\begin{itemize}
\item[{\bf Step 7.} ] $A$ is a cyclic $2$-group.
\end{itemize}
Let $r$ be an odd prime dividing $|A|$, and choose $c\in A$ of order $r$.
Since $r\mid m-1$ and $(r,2m)=1$, the element $c$ centralizes $B^x$.
Choose a lift $\widetilde c\in\Omega^+_{2m}(q)$ of $c$ with $|\widetilde c|=r$.
Choose $\widetilde B_x\leq\Omega^+_{2m}(q)$ of order $\ell$ with projective image $B^x$.
For $b\in\widetilde B_x$, the commutator $[\widetilde c,b]$ is a central element of $2$-power order.
Since $|\widetilde c|=r$ and this commutator is central, we have $[\widetilde c,b]^r=[\widetilde c^r,b]=1$.
As $r$ is odd, we obtain that $[\widetilde c,b]=1$.
For odd $m$, $\widetilde B_x$ has distinct reciprocal constituents $U_1,U_2$ with $\dim_{\F_q}U_i=m$.
We have $C_{\operatorname{SO}(V)}(\widetilde B_x)\cong\F_{q^m}^*$.
Because $p\nmid q^m-1$, we have $r\ne p$.
A nonidentity multiplier has degree $e\mid m$ and $e\mid r-1$,
so $e$ is odd and $e\leq(r-1)/2$.
Each eigenvalue $\lambda$ of $\widetilde c|_D$ belongs to $\F_{q^2}^*$ and satisfies $\lambda^{q+1}=1$.
If $e=1$, then a nontrivial eigenvalue $\lambda\in\F_q^*$ has odd order and $\lambda^{q+1}=\lambda^2\ne1$.
If $e\geq3$, then $e$ is odd, so $\lambda\notin\F_{q^2}$.
Thus, no nontrivial eigenvalue of $\widetilde c|_{U_i}$ is an eigenvalue of $\widetilde c|_D$.
Consequently, its eigenspace on $V$ lies in $E$ and has dimension at least
$\frac m e\geq\frac{2m}{r-1}>\frac{2(m-1)}r$.
But the nontrivial Singer companion of $\widetilde c$ has eigenvalue multiplicity $2(m-1)/r$ on $E$.

For even $m$, we have $C_{\operatorname{SO}(V)}(\widetilde B_x)\cong\GU_2(q^{m/2})$.
We identify $V$ with $\F_{q^m}^2$, on which this group acts linearly.
The possible dimensions of $\ker(\widetilde c-I_V)$ over $\F_q$ are $0,m,2m$.
For the Singer companion, we have $\dim_{\F_q}\ker(\widetilde c-I_V)=2(m-1)/r+d$ with $d\in\{0,2\}$.
For $m\geq6$, we have $0<\dim_{\F_q}\ker(\widetilde c-I_V)<m$.
For $m=4$, the only prime is $r=3$.
If $p\ne3$, then fixed dimension $4$ requires $3\mid q^2+1$. We obtain a contradiction.
If $p=3$, then the unitary unipotent has Jordan type $2^4$.
The companion has type $3^2 1^2$, so the types differ.
Thus, $A$ is a cyclic $2$-group.

\begin{itemize}
\item[{\bf Step 8.} ] We have $A=1$.
\end{itemize}
For odd $q$, choose $H\leq G$ from Lemma \ref{lem:split-root-normalizer} with $|N_G(H)|$ odd.
Since $A$ is ordinary $G$-permutable, we choose $y\in G$ such that $AH^y$ is a subgroup.
As $A$ is a cyclic $2$-group and $|H|$ is odd, we obtain from Burnside's theorem that $H^y\unlhd AH^y$.
Thus, $A\leq N_G(H^y)$ and $A=1$.
For even $q$ and odd $m$, the normalizer of $B^x$ is odd.
The torus in $N_G(B^x)$ has order $q^m-1$, and its automizer has order dividing $m$.
A map sending the $2$ reciprocal summands to each other has Dickson invariant $m\pmod2=1$.
Thus, this map lies outside $\Omega^+_{2m}(q)$.
Because $B^x\unlhd L$, we have $A\leq N_G(B^x)$.
The group $A$ is a $2$-group and $N_G(B^x)$ is odd, so $A=1$.

Finally, suppose that $q$ and $m$ are even and $A\ne1$.
Let $t$ be the involution of $A$.
The field companion of $t|_E$ is $v\mapsto v^{q^k}$.
For $s\in T_D$, we have $(t|_D)s(t|_D)^{-1}=s^{-1}$.
By Hilbert 90, we have $\operatorname{rank}(t|_E-I_E)=k$ and $\operatorname{rank}(t|_D-I_D)=1$.
Thus, $\operatorname{rank}(t-I_V)=m$.
Take $H$ from Lemma \ref{lem:orthogonal-plus-odd-subgroup}.
Since $A$ is ordinary $G$-permutable, we choose $x\in G$ such that $AH^x$ is a subgroup.
Its Sylow $2$-subgroup is cyclic.
By Burnside's theorem, $H^x\unlhd AH^x$.
Thus, $t\in N_G(H^x)$ and $\operatorname{rank}(t-I_V)\leq m-1$. This contradicts $\operatorname{rank}(t-I_V)=m$.
\end{proof}

\begin{theorem}\label{thm:orthogonal-plus-small}
For $G=\POmega^+_8(3)$ or $\Omega^+_{12}(2)$, there is no
proper nontrivial ordinary $G$-permutable subgroup.
\end{theorem}
\begin{proof}
For ease of reading, we divide the proof of the theorem into the following steps.

\setlist[itemize]{leftmargin=4em,labelsep=.5em,topsep=.2\baselineskip,partopsep=0pt,parsep=0pt,itemsep=0pt}
\begin{itemize}
\item[{\bf Step 1.} ] If $K\leq N_G(\GU_6(2))$, then $J<\SU_6(2)$.
\end{itemize}
Firstly, let $G=\Omega^+_{12}(2)$ and let $A<G$ be ordinary $G$-permutable.
We use $P_2$ and the stabilizer of $V_1\perp V_2$ from Theorem \ref{thm:orthogonal-plus}.
For $V=\F_2^{12}$, we require $\dim_{\F_2}V_1=4$ and $\dim_{\F_2}V_2=8$.
Take $R=C_{33}\times C_3$.
Choose $u\in G$ such that $K=AR^u$ is a subgroup.
Replace $A,K$ by $uAu^{-1},uKu^{-1}$, so $K=AR$.
For each $2$-subgroup $P\leq K$, we choose $V=W_1\perp W_2$ with $PW_i=W_i$.
We have $\dim_{\F_2}W_1=4$ and $\dim_{\F_2}W_2=8$.
Now, we consider the maximal subgroup $M_U=N_G(\GU_6(2))<G$ with $K\leq M_U$.
Let $K_0=K\cap\GU_6(2)$ and $J=K\cap\SU_6(2)$.
We have $K_0,J\unlhd K$ and $[K:K_0]\leq2$.
Since $\det(R)=C_3$, the map $\det:K_0\to\{a\in\F_4^*:a^3=1\}$ is surjective.
We have $Z(\SU_6(2))\leq J$ and $C_{33}\leq J$.
By Lemma \ref{lem:symplectic-unitary-sylow}, a Sylow $2$-subgroup of $\SU_6(2)$ is indecomposable on $V$.
Since every $2$-subgroup of $K$ stabilizes a decomposition $W_1\perp W_2$, we have $J<\SU_6(2)$.
Thus, the additional possibility in \cite[Lemma 2.6]{GuralnickMalle2012} is
$J\leq M_E=3.M_{22}<\SU_6(2)$.
The central quotient sends $M_E$ onto $M_{22}$.

\begin{itemize}
\item[{\bf Step 2.} ] In the $3.M_{22}$ and $\GU_5(2)$ cases, $K$ normalizes $R$.
\end{itemize}
If $J/C_3=M_{22}$, then $J=3.M_{22}$.
The natural representation of $3.M_{22}$ of dimension $6$ is absolutely irreducible.
Thus, $C_{\GU_6(2)}(3.M_{22})=\{\lambda I_6:\lambda\in\F_4^*,\ \lambda^3=1\}$.
Since $\det(\lambda I_6)=\lambda^6=1$, every element of this centralizer has determinant $1$.
Let $Z_0=Z(3.M_{22})$.
Let $\psi:\Aut(3.M_{22})\to\Aut(M_{22})$ satisfy $\psi(\alpha)(gZ_0)=\alpha(g)Z_0$.
As $3.M_{22}$ is perfect, $\psi$ is injective.
We have $|\Out(3.M_{22})|\leq2$ \cite{Atlas1985}.
Thus, $\det(N_{\GU_6(2)}(3.M_{22}))=1$. This contradicts $\det(R)=C_3$.
If $J/C_3<M_{22}$, then we choose $L\cong\PSL_2(11)$ with $J/C_3\leq L$.
Suppose that $J/C_3<L$.
Since every proper subgroup of $L$ with order divisible by $11$ lies in $C_{11}\rtimes C_5$, the group $J/C_3$ has a unique Sylow $11$-subgroup.
The central kernel $C_3$ then yields a unique Sylow $11$-subgroup $S$ of $J$.
As $J\unlhd K$ and $[K:J]\mid6$, we have $S\unlhd K$, and $S$ is Sylow in $K$.
Because $C_G(S)=R$, we obtain that $K\leq N_G(R)$.

Suppose that $J/C_3=L$.
Since the Schur multiplier of $L$ has order $2$, $J=C_3\times L$ \cite{Atlas1985}.
The characteristic derived subgroup $L$ is normal in $K$.
The nontrivial absolutely irreducible $L$-modules in characteristic $2$ have dimensions $5,5,10,12,12$ \cite[p.~7]{BrauerAtlas1995}.
Therefore, $L$ has no simple module of dimension $6$ over $\F_4$.
The $\F_4L$-module $\F_4^6$ has composition factors of dimensions $5$ and $1$.
A nonsplit extension of these factors has different socle and top dimensions.
As $\F_4^6$ is Hermitian self-dual, we have a direct sum of the $2$ factors.
The fixed line is unique and nondegenerate.
For a singular-line parabolic $P<\GU_6(2)$, we have $11\nmid|P|$.
Thus, $K$ stabilizes the unique $L$-fixed unitary line and its perpendicular complement.
Over $\F_2$, we choose $V=E\perp D$ with $E$ of minus type and $D$ anisotropic.
We have $\dim_{\F_2}E=10$ and $\dim_{\F_2}D=2$.

For the decomposition $V=E\perp D$ stabilized by $K$, we have $\dim_{\F_2}E=10$ and $\dim_{\F_2}D=2$.
Let
$$
 K_0=\{g\in K:g|_E\text{ and }g|_D\text{ are }\F_4\text{-linear}\}\unlhd K.
$$
For the relative exception in $\GU_5(2)$, let $L=\PSL_2(11)$ and $M=C_3\times L$.
By the relative subgroup classification, we have $\{g|_E:g\in K_0\}\leq M$.
Let $\pi_L:K_0\to L$ be the composite of $g\mapsto g|_E$
and $M\to L$, $(c,h)\mapsto h$.
If $\pi_L(K_0)<L$, then $\pi_L(K_0)\leq C_{11}\rtimes C_5$, so this image has a unique Sylow $11$-subgroup.
As $\ker\pi_L\leq C_3^2$ is central, $K_0$ has a unique Sylow $11$-subgroup $S$.
Since $K_0\unlhd K$ and $[K:K_0]\leq2$, we have $S\unlhd K$.
Because $C_G(S)=R$, we obtain that $K\leq N_G(R)$.
Suppose that $\pi_L(K_0)=L$.
As $[K:K_0]\leq2$ and $|\ker\pi_L|\mid9$, we have
$|K|\mid2\cdot9\cdot660=11880$.
Let $A_0=A\cap K_0$. Because $R\leq K_0$, we have $K_0=A_0R$.
As $\pi_L(R)=C_{11}$, the index of $\pi_L(A_0)$ in $L$ is $1$ or $11$.
Thus, $\pi_L(A_0)$ is $L$ or $A_5$.
The central kernel is a subgroup of $C_3^2$.
As the Schur multiplier of $A_5$ has order $2$, the inverse image in $A_0$
of an $A_5$ subgroup of $\pi_L(A_0)$ contains a subgroup isomorphic to $A_5$.

Next, we choose $V=(U\oplus U^*)\perp W$ with $\dim_{\F_2}U=5$ and $\dim_{\F_2}W=2$.
On this decomposition, take $B=C_{31}$ acting irreducibly on $U,U^*$ and trivially on the hyperbolic space $W$.
Since $A$ is ordinary $G$-permutable, we choose $x\in G$ such that $AB^x$ is a subgroup.
As $31\nmid11880$ and $|A|\mid11880$, the Sylow $31$-count in $AB^x$ divides $11880$ and is $1$ modulo $31$.
Only the divisor $1$ has this property, so $B^x$ is normal.
Since $\Aut(B^x)$ is abelian and $A_5$ is perfect, the subgroup $A_5\leq A$ centralizes $B^x$.
For $U_x=x^{-1}U$ and $W_x=x^{-1}W$, the centralizer on $U_x\oplus U_x^*$ is $\GL_1(32)$.
Moreover, $\operatorname{SO}(W_x)=\operatorname{SO}_2^+(2)=1$.
Thus, $A_5\leq C_G(B^x)=C_{31}$, a contradiction.

\begin{itemize}
\item[{\bf Step 3.} ] If $G=\Omega^+_{12}(2)$ and $K\leq N_G(R)$, then $A=1$.
\end{itemize}
Suppose that $K\leq N_G(R)$.
By the argument for $P_1$ and the maximal isotropic parabolic in Theorem \ref{thm:orthogonal-plus}, we have $A\cap R=1$.
Thus, $A$ is cyclic with $|A|\mid10$.
At $(m,q)=(6,2)$, Zsigmondy's exception at exponent $6$ requires a different comparison.
We choose $V=U\perp W$ with $U,W$ of minus type.
For $U,W$, we require $\dim_{\F_2}U=8$ and $\dim_{\F_2}W=4$.
Take $B=C_{17}\leq\Omega(U)$ irreducible on $U$ and trivial on $W$.
Since $A$ is ordinary $G$-permutable, choose $x\in G$ with $AB^x\leq G$.
The Sylow $17$-count divides $10$ and is $1$ modulo $17$.
Thus, $B^x\unlhd AB^x$.
Suppose that $5\mid|A|$, and choose $a\in A$ with $|a|=5$.
Since $|N_G(B^x)/C_G(B^x)|$ divides $8$, the element $a$ centralizes $B^x$.
For $U_x=x^{-1}U$ and $W_x=x^{-1}W$, we have
$C_G(B^x)=C_{17}\times\Omega(W_x)$.
Thus, $a|_{U_x}=I_{U_x}$.
In $N_G(R)$, the field companion of $a|_E$ has $2$ cycles of length $5$, so $\dim_{\F_2}\ker(a|_E-I_E)=2$.
As $|\operatorname{SO}(D)|=3$, we also have $a|_D=I_D$.
Therefore, $\dim_{\F_2}\ker(a-I_V)=4$.
This contradicts $a|_{U_x}=I_{U_x}$, because $\dim_{\F_2}U_x=8$.
Hence, $A$ is cyclic of $2$-power order.
Suppose that $A\ne1$, and let $t$ be its involution.
Take $H$ from Lemma \ref{lem:orthogonal-plus-odd-subgroup}.
As $A$ is ordinary $G$-permutable, we choose $x\in G$ such that $AH^x$ is a subgroup.
By Burnside's theorem, $H^x\unlhd AH^x$.
Thus, $t\in N_G(H^x)$.
The involution $t$ has rank $6$, but every involution in $N_G(H^x)$ has rank at most $5$.
We obtain a contradiction.

\begin{itemize}
\item[{\bf Step 4.} ] $M$ stabilizes $V=E\perp D$ or normalizes $\SU_4(3)$.
\end{itemize}
Now, let $G=\POmega^+_8(3)$ and let $A<G$ be ordinary $G$-permutable.
Let $\pi:\Omega^+_8(3)\to G$ be the quotient map and $\widetilde A=\pi^{-1}(A)$.
We use $P_2$ and the stabilizer of $V_1\perp V_2$ from Theorem \ref{thm:orthogonal-plus}. For every $Y<G$, we have $P_2Y\ne G$.
We have $\dim_{\F_3}V_1=3$ and $\dim_{\F_3}V_2=5$.
Since $\operatorname{SO}_3(3)$ and $\operatorname{SO}_5(3)$ are
of orders not divisible by $7$, we have $7\nmid|A|$.
The torus $R$ of signed type $(-3,-1)$ has order $56$
in $\Omega^+_8(3)$ and order $28$ projectively.
Let $\overline R=\pi(R)$.
Since $A$ is ordinary $G$-permutable, we choose $u\in G$ such that $K=A\overline R^u$ is a subgroup.
Then, we choose a lift $\widetilde u\in\Omega^+_8(3)$ of $u$. Let $\widetilde K=\pi^{-1}(K)$.
Then, we replace $A,K$ by $uAu^{-1},uKu^{-1}$.
Correspondingly, we replace $\widetilde A,\widetilde K$ by $\widetilde u\widetilde A\widetilde u^{-1},\widetilde u\widetilde K\widetilde u^{-1}$ and $V_i$ by $\widetilde uV_i$ for $i=1,2$.
Then, $\widetilde K=\widetilde A R<\Omega^+_8(3)$, and $\widetilde A V_i=V_i$ for $i=1,2$.
Since $|R|=56$, we have $|\widetilde K|_3=|\widetilde A|_3$.
For every $3$-subgroup $P\leq\widetilde K$, Sylow's theorem permits a choice of $g\in\widetilde K$ with $P^g\leq\widetilde A$.
Thus, $P$ stabilizes $gV_1\perp gV_2$, with $\dim_{\F_3}gV_1=3$ and $\dim_{\F_3}gV_2=5$.
By \cite[Lemma 2.9 and Table 4]{GuralnickMalle2012}, we choose a maximal subgroup $M<\Omega^+_8(3)$ with $\widetilde K\leq M$.
No line $L_0\leq V$ with $\dim_{\F_3}L_0=1$ satisfies $rL_0=L_0$ for every $r\in R$.
In $\operatorname{Spin}_7(3)$, a primitive element of order $7$ has centralizer
order $28$, and the extra $2.\Omega^+_8(2)$ has centralizer at most $28$.
For an element $s\in\Omega^+_8(2)$ of order $7$, the module $\F_2^8$ has reciprocal constituents of dimension $3$.
The fixed space $\ker(s-I)$ is hyperbolic of dimension $2$.
Its orthogonal centralizer has order $14$, doubled at most in $2.\Omega^+_8(2)$.
As $R\leq\widetilde K\leq M$, we have $M\not\cong\operatorname{Spin}_7(3)$ and $M\not\cong2.\Omega^+_8(2)$.
Thus, $M$ stabilizes a nondegenerate decomposition $V=E\perp D$ or normalizes $\SU_4(3)$.
In the decomposition alternative, we have $\dim_{\F_3}E=6$ and $\dim_{\F_3}D=2$.

\begin{itemize}
\item[{\bf Step 5.} ] If $M$ normalizes $\SU_4(3)$, then $|A|\mid192$ or $\widetilde K D=D$ for some nondegenerate $D\leq V$ with $\dim_{\F_3}D=2$.
\end{itemize}
In the unitary case, let $K_0=\widetilde K\cap\GU_4(3)$
and $J=\widetilde K\cap\SU_4(3)$.
We have $[\widetilde K:K_0]\leq2$ and $[K_0:J]\leq2$.
For $\det:\GU_4(3)\cap\Omega^+_8(3)\to\{a\in\F_9^*:a^4=1\}$,
we have image $\{1,-1\}=C_2$ and $\det(R)=C_2$.
Since $R\leq K_0$, we have $C_{28}\leq J$ and $Z(\SU_4(3))=C_4\leq J$.
By Lemma \ref{lem:symplectic-unitary-sylow}, a Sylow $3$-subgroup of $\SU_4(3)$ is indecomposable on $V$.
Since every $3$-subgroup of $\widetilde K$ stabilizes a decomposition of dimensions $3$ and $5$, we have $J<\SU_4(3)$.
Thus, the exceptional projective images of $J$ are $A_7$ and $\PSL_3(4)$.

If the projective image of $J$ is $A_7$, then its derived subgroup satisfies $J'\cong2.A_7$.
By \cite[Proposition 4.5.5(ii)]{BHR}, we have $N_{\operatorname{CGU}_4(3)}(J')=J'\{\lambda I_4:\lambda\in\F_9^*\}$.
Thus, we have $N_{\GU_4(3)}(J')=J'\{\lambda I_4:\lambda\in\F_9^*,\lambda^4=1\}$.
Since $J'$ is perfect and $\lambda^4=1$, every element of $N_{\GU_4(3)}(J')$ has determinant $1$. This contradicts $\det(R)=\{1,-1\}$.
If the projective image of $J$ is $\PSL_3(4)$, then the quotient $N_{\GU_4(3)}(J)/C_4$
has order at most $2|\PSL_3(4)|$
\cite[Proposition 4.5.6]{BHR}.
We have an image of $R$ of order $14$ in $N_{\GU_4(3)}(J)/C_4$.
As $K_0=A_0R$ with $A_0=\widetilde A\cap K_0$, the image of $A_0$ has index at most $14$ in a group containing $\PSL_3(4)$.
Every proper subgroup of $\PSL_3(4)$ has index at least $21$
\cite{Atlas1985}.
Thus, the image of $A_0$ contains $\PSL_3(4)$. This contradicts $7\nmid|A_0|$.

A proper exceptional image containing $7$ lies in $\PSL_3(2)$
or a smaller $C_7\rtimes C_3$ subgroup.
Thus, $|J|\mid4\cdot168$.
We have $[\widetilde K:K_0]\leq2$ and $[K_0:J]\leq2$. Since $|\ker\pi|=2$, we have $|K|\mid1344$.
As $7\nmid|A|$, we have $|A|\mid192$.
Suppose that $J$ stabilizes a nondegenerate unitary line $D_0$ over $\F_9$.
Let $C_7$ be the Sylow $7$-subgroup of $R$.
Because $C_7$ acts irreducibly on $D_0^\perp$ over $\F_9$, $D_0$ is the unique $C_7$-stable line.
As $C_7\leq J$ and $JD_0=D_0$, it follows that $D_0$ is the unique $J$-stable line.
For $g\in\widetilde K$, the relation $gJg^{-1}=J$ shows that $gD_0$ is $J$-stable.
Thus, $gD_0=D_0$ for every $g\in\widetilde K$.
The underlying $\F_3$-space $D$ has $\dim_{\F_3}D=2$.

\begin{itemize}
\item[{\bf Step 6.} ] We have $|A|\mid192$.
\end{itemize}
We choose $V=E\perp D$ with $\dim_{\F_3}E=6$ and $\dim_{\F_3}D=2$.
For this decomposition, let $K_0=\{g\in\widetilde K:\det(g|_E)=\det(g|_D)=1\}\unlhd\widetilde K$,
$A_0=\widetilde A\cap K_0$ and $J=\{g|_E:g\in K_0\}\leq\operatorname{SO}(E)$.
Since $R\leq K_0$ and $\widetilde K=\widetilde A R$, we have $K_0=A_0R$.
The group $J$ contains the Singer torus $C_{28}$ in $\operatorname{SO}^-_6(3)$.
For the action on $E$, let $\pi_E:K_0\to J$ send $g$ to $g|_E$.
Restriction to $D$ is injective on $\ker\pi_E$ and takes values in $\Omega(D)$.
Since $|\Omega(D)|\leq2$, we have $|\ker\pi_E|\leq2$.
Let $J_0=J\cap\Omega^-_6(3)$. Then, $[J:J_0]\leq2$.
Under $\Omega^-_6(3)=\SU_4(3)/\gen{-I}$, the inverse image of $J_0$ contains $C_{28}$.
Suppose that $J_0=\Omega^-_6(3)$.
By Lemma \ref{lem:orthogonal-minus-two-subgroup}, choose a $3$-subgroup $P\leq J_0$ indecomposable on the $\F_3$-space $E$.
Since $|\ker\pi_E|\leq2$, a Sylow $3$-subgroup $\widehat P$ of $\pi_E^{-1}(P)$ maps isomorphically onto $P$.
Moreover, $\widehat P$ acts trivially on $D$, because $|\operatorname{SO}(D)|$ is $2$ or $4$.
Thus, the indecomposable $\F_3\widehat P$-summands of $V=E\perp D$ have dimensions $6,1,1$.
By Krull--Schmidt, no direct summand has dimension $3$.
This contradicts $\widehat P\leq\widetilde K$ and the decomposition property for its $3$-subgroups.
Therefore, $J_0<\Omega^-_6(3)$.

If the projective image of $J_0$ is $A_7$, then $J_0'\cong A_7$.
In this case, $E\cong\{(x_1,\ldots,x_7)\in\F_3^7:\sum_{i=1}^7x_i=0\}$ as an $\F_3J_0'$-module.
The group $J_0'$ acts by permuting the $7$ coordinates.
Moreover, the $\F_3J_0'$-module $E$ is absolutely irreducible \cite[Tables 4.2 and 8.34]{BHR}.
Thus, $C_{\operatorname{O}(E)}(J_0')=\{I_E,-I_E\}$.
Since $C_{\Aut(A_7)}(C_7)=C_7$, an element of order $7$ has centralizer
order at most $14$ in $N_{\operatorname{O}(E)}(J_0')$.
Thus, $N_{\operatorname{O}(E)}(J_0')$ cannot contain the Singer torus $C_{28}$.
If the projective image equals $\PSL_3(4)$, then the orthogonal normalizer
adds at most a factor $2$ \cite[Table 8.34]{BHR}.
Modulo the scalar $C_2$ of the orthogonal normalizer, the Singer image has order $14$.
As every proper subgroup of $\PSL_3(4)$ has index at least $21$, the image of $A_0$ contains $\PSL_3(4)$.
This contradicts $7\nmid|A_0|$.
For proper exceptional images, we have $|J_0|\mid2\cdot168$.
For the unitary line stabilizer, identify $E$ with $\F_9^3$ over $\F_3$.
We also consider $C_{28}\leq M_3=C_4\times\PSL_3(2)<\GU_3(3)$
\cite[Tables 8.6--8.7]{BHR}\cite{Bereczky}.
A Sylow $3$-subgroup of $\SU_3(3)$ is indecomposable on $E$.
On $V=E\perp D$, it has no direct summand of dimension $3$ over $\F_3$.
From the plane kernel and component indices, we have $|A|\mid192$ in both cases.

\begin{itemize}
\item[{\bf Step 7.} ] If $G=\POmega^+_8(3)$, then $A=1$.
\end{itemize}
We choose $V=(U\oplus U^*)\perp W$ with $\dim_{\F_3}U=3$ and $\dim_{\F_3}W=2$.
Take $B=C_{13}$ irreducible on $U,U^*$ and trivial on the hyperbolic space $W$.
Since $A$ is ordinary $G$-permutable, choose $x\in G$ with $AB^x\leq G$.
As only the divisor $1$ of $192$ is $1$ modulo $13$, we have $B^x\unlhd AB^x$.
We have $C_G(B^x)=C_{13}$, and $N_G(B^x)/C_G(B^x)$ is cyclic of order dividing $6$.
Since $|A|\mid192$, we have $A\cap C_G(B^x)=1$.
Thus, the quotient map is injective on $A$, so $A$ is cyclic of order dividing $6$.
Suppose that $3\mid|A|$, and choose $a\in A$ with $|a|=3$.
Choose $\widetilde x\in\Omega^+_8(3)$ with $\pi(\widetilde x)=x$.
Let $\widetilde a\in\Omega^+_8(3)$ be the unique lift of $a$ with $|\widetilde a|=3$.
From the action on $\widetilde x^{-1}U$, $\widetilde x^{-1}U^*$ and $\widetilde x^{-1}W$, we obtain Jordan type $3^2 1^2$ for $\widetilde a$.
We take $C_5$ as in Theorem \ref{thm:orthogonal-plus} and choose $y\in G$ such that $AC_5^y$ is a subgroup.
By It\^o's theorem and the normality argument in Theorem \ref{thm:orthogonal-plus}, we have $C_5^y\unlhd AC_5^y$.
Because $|N_G(C_5^y)/C_G(C_5^y)|$ divides $4$, we have $a\in C_G(C_5^y)$.
Let $\widetilde C_5\leq\Omega^+_8(3)$ satisfy $\pi(\widetilde C_5)=C_5^y$ and $|\widetilde C_5|=5$.
As $|\ker\pi|=2$, we have $\widetilde a\in C_{\Omega^+_8(3)}(\widetilde C_5)$.
We have $C_{\operatorname{SO}_8(3)}(\widetilde C_5)\cong\GU_2(9)$.
A nonidentity unipotent in this group has Jordan type $2^4$ on $V=\F_3^8$.
Because $2^4\ne3^2 1^2$, we have $3\nmid|A|$.
Now, $|A|\mid2$.
By Lemma \ref{lem:split-root-normalizer}, choose an odd-order subgroup $H\leq G$ with $|N_G(H)|$ odd.
By Lemma \ref{lem:involution-normalizers}, we conclude that $A=1$.
\end{proof}

\begin{corollary}\label{cor:orthogonal-plus-all}
For every $m\geq4$ and every prime power $q$, the group
$\POmega^+_{2m}(q)$ has no proper nontrivial ordinary
$\POmega^+_{2m}(q)$-permutable subgroup.
\end{corollary}
\begin{proof}
We obtain the conclusion from Theorems \ref{thm:orthogonal-plus} and \ref{thm:orthogonal-plus-small}.
For $\Omega^+_8(2)$, we use Proposition \ref{prop:finite-tables}.
\end{proof}

\subsection{\texorpdfstring{The groups $\Omega_{2m+1}(q)$ with odd $q$}{The groups Omega(2m+1)(q) with odd q}}\label{sec:orthogonal-odd}

Let $q$ be an odd prime power and let $m\geq3$.
A Singer torus $T\leq\Omega_{2m+1}(q)$ fixes a nondegenerate line $L\leq V=\F_q^{2m+1}$ pointwise.
We have $V=E\perp L$, where $E=L^\perp$ has dimension $2m$ and minus type.
The group $T$ acts irreducibly on $E$ and has order $(q^m+1)/2$.

\begin{lemma}\label{lem:orthogonal-odd-singer}
Let $G=\Omega_{2m+1}(q)$ with $m\geq3$ and $q$ odd.
If $A<G$ is ordinary $G$-permutable, then $A$ lies in a
Singer normalizer $N_G(T)$ with $|T|=(q^m+1)/2$.
\end{lemma}
\begin{proof}
For ease of reading, we divide the proof of the lemma into the following steps.

\setlist[itemize]{leftmargin=4em,labelsep=.5em,topsep=.2\baselineskip,partopsep=0pt,parsep=0pt,itemsep=0pt}
\begin{itemize}
\item[{\bf Step 1.} ] If $K\leq N_1^-$ and $(m,q)\ne(3,3)$, then $T^x$ is normal in $K$.
\end{itemize}
Let $p=\operatorname{char}\F_q$, so $q=p^f$.
Let $P_2$ stabilize a totally singular $X\leq V=\F_q^{2m+1}$ with $\dim_{\F_q}X=2$.
The subgroup $P_2$ is maximal.
For every $Y<G$, we have $P_2Y\ne G$
\cite[Table A.4]{LiXia}\cite[Theorem 4.2]{King2005}.
By Lemma \ref{lem:max}, we obtain that $AY\ne G$ for every $Y<G$.
For each maximal subgroup $M<G$, we choose $x\in G$ with $A\leq M^x$ by Lemma \ref{lem:max}.
Choose $M$ to stabilize an anisotropic space of dimension $2$ \cite[Theorem 4.3]{King2005}.
Then, $A$ stabilizes $V=U\perp W$, with $\dim_{\F_q}U=2$ and $\dim_{\F_q}W=2m-1$.
Take a Singer torus $T$ of order $(q^m+1)/2$.
We choose $x\in G$ such that $K=AT^x$ is a subgroup.
Since $AT^x\ne G$, we have $K<G$.
Let $L=V^{T^x}$ and $E=L^\perp$. Then, $V=E\perp L$ and $\dim_{\F_q}E=2m$.
The group $T^x$ acts irreducibly on $E$, and $L$ is a nondegenerate line.
As $|K|_p=|A|_p$, every $p$-subgroup $P\leq K$ satisfies $P^g\leq A$ for some $g\in K$.
Thus, $P$ stabilizes $gU\perp gW$.

We choose a maximal subgroup $M_7<G$ with $K\leq M_7$.
Let $N_1^- =\operatorname{Stab}_G(L)$.
For $m\geq4$, the only maximal subgroup $M<G$ with $T^x\leq M$ is
$N_1^-$
\cite[Lemma 2.7]{GuralnickMalle2012}.
Thus, $M_7=N_1^-$ when $m\geq4$.
For $(m,q)\notin\{(3,3),(3,5)\}$, the primitive-part class in \cite[Lemma 2.7]{GuralnickMalle2012} exists and lies in $T$.
Suppose that $K\leq N_1^-$. Let $J=K\cap\Omega^-(E)\unlhd K$.
Then, $T^x\leq J$.
For $J$, we use the normal base intersections from the proof of Theorem \ref{thm:orthogonal-minus-odd}.
At each relative $\Omega^-_{2a}(Q)$ stage, the intersection contains the Singer torus of order $(Q^a+1)/2=(q^m+1)/2$.
At each unitary stage $\GU_b(Q)$ with odd $b$, adjoin $Z(\GU_b(Q))$.
Since $v_2(Q^b+1)=v_2(Q+1)$, we then have the unitary Singer torus of order $Q^b+1$.
As $|Z(\GU_b(Q))|=Q+1$ and $p\nmid Q+1$, adjoining $Z(\GU_b(Q))$ does not change the Sylow $p$-subgroups.
By Lemmas \ref{lem:orthogonal-minus-two-subgroup} and \ref{lem:symplectic-unitary-sylow}, each relative classical group contains an indecomposable $p$-subgroup on $E$.
On $V$, its indecomposable summands have dimensions $2m$ and $1$ over $\F_q$.
Thus, there is no direct summand of dimension $2$.
The small odd-field exceptions require original $(m,q)=(3,3)$
\cite[Theorem 2.1.7]{Brookfield2014}.
When $(m,q)\ne(3,3)$ and $K\leq N_1^-$, choose a primitive prime divisor $\ell$ of $q^{2m}-1$.
Let $T_\ell$ be the Sylow $\ell$-subgroup of $T^x$.
By the proof of Theorem \ref{thm:orthogonal-minus-odd}, $T_\ell$ is characteristic in $J\unlhd K$.
Thus, $T^x=C_G(T_\ell)\unlhd K$.

\begin{itemize}
\item[{\bf Step 2.} ] If $m=3$ and $q\geq5$, then $T^x$ is normal in $K$.
\end{itemize}
For $m=3$, we use the maximal subgroup list in
\cite[Tables 8.39--8.40]{BHR}.
A primitive $\Phi_6$ element in $G_2(q)$ has cyclic centralizer
of order $q^2-q+1$.
This is smaller than $|T|=(q^3+1)/2$, so $G_2(q)$ cannot contain $T$.
We check the primitive-element centralizer in the unitary normalizer in characteristic $3$ from \cite[Lemma 2.7]{GuralnickMalle2012}.
On the irreducible adjoint module $V=\F_q^7$ of $\PSU_3(q)$,
the scalar centralizer in $\operatorname{SO}_7(q)$ is trivial.
Now, let $q=3^f$. Next, we choose an absolute primitive divisor $\ell$
of $3^{6f}-1$ in the displayed primitive class.
Since $q=3^f$ and $(3,q+1)=1$, the diagonal outer quotient of $\PSU_3(q)$ is trivial.
Next, let $c$ be a generator of the subgroup of order $\ell$ in the primitive torus.
For this element, the graph and field automorphisms normalizing $\gen c$ induce maps $c\mapsto c^{3^i}$ with $0\leq i<2f$, modulo the inner Singer Weyl subgroup.
Moreover, an inner Singer Weyl generator sends $c$ to $c^{-q}$, where $-q\equiv3^{4f}\pmod\ell$.
The inner Singer Weyl subgroup has exponents $0,2f,4f$ modulo $6f$,
so only $i=0$ among $0\leq i<2f$ centralizes the primitive element modulo that subgroup.
Thus, the primitive-element centralizer in $\Aut(\PSU_3(q))$ has order $q^2-q+1$
and cannot contain $T$.
Let $J_u\leq G$ be a unitary subgroup from the maximal list acting reducibly on $V$.
For this $J_u$, let $C_\ell\leq J_u$ be its primitive subgroup and $L_u=V^{C_\ell}$.
Since $\ord_\ell(q)=6$ and $\dim_{\F_q}V=7$, we have $\dim_{\F_q}L_u=1$.
Since $J_u$ is reducible and $C_\ell$ is irreducible on $L_u^\perp$, a nonzero proper $J_u$-submodule is $L_u$ or $L_u^\perp$.
As $J_u$ consists of isometries, we have $gL_u=L_u$ for $g\in J_u$.
Thus, $J_u\leq N_1^-$.
For a Ree normalizer, the primitive-element centralizer is cyclic of order $q\pm\sqrt{3q}+1\leq q^2-q+1$.
Moreover, the torus quotient has order dividing $6$.
For its field quotient, the exponents satisfy $0\leq i<f$.
Since $\ord_\ell(3)=6f$, we have $\ord_\ell(3^i)=6f/(6f,i)>6$ for $0<i<f$.
Thus, no nonidentity field automorphism can centralize the primitive element after multiplication by a torus representative.
At $q=3$, let $R(3)=\PSL_2(8)\rtimes C_3$.
For $s\in R(3)$ with $|s|=7$, we have $C_{R(3)}(s)\cong C_7$.
Thus, no Ree normalizer contains $T$
\cite[Tables 8.42--8.43]{BHR}.
For $q\geq7$, the remaining case is $K\leq N_1^-$.
At $q=5$, we also consider the possibility that the maximal subgroup $M_7<G$ with $K\leq M_7$ has one of the types
$$
 A_8,\  C_2^6\rtimes A_7,\  \Sp_6(2),\  \PSU_3(3)\rtimes C_2
$$
\cite[Table 2]{GuralnickMalle2012}.
Neither $A_8$ nor $C_2^6\rtimes A_7$ has an element of odd order divisible by $9$.
As $\operatorname{ord}_9(2)=6$, every $s\in\Sp_6(2)$ with $|s|=9$ acts irreducibly on $\F_2^6$.
Thus, $|C_{\Sp_6(2)}(s)|=9$, so this centralizer has no element of order $7$.
For a regular $s\in\PSU_3(3)$ with $|s|=7$, we have $|C_{\PSU_3(3)}(s)|=7$.
Consequently, its centralizer in $\PSU_3(3)\rtimes C_2$ has order dividing $14$.
Since $|T^x|=63$, these facts show that none of the listed types contains $T^x$.
Therefore, $K\leq N_1^-$ at $(m,q)=(3,5)$.
Let $Q$ be the Sylow $7$-subgroup of $T^x$.
By the proof of Theorem \ref{thm:orthogonal-minus-odd}, $Q$ is characteristic in $J\unlhd K$.
Thus, $T^x=C_G(Q)\unlhd K$.

\begin{itemize}
\item[{\bf Step 3.} ] If $(m,q)=(3,3)$ and $M_7\cong S_9$, then $A\leq N_G(T^x)$.
\end{itemize}
Now, we consider $(m,q)=(3,3)$.
By $A\leq P_2^y$ for some $y\in G$, we know that $A$ is a $\{2,3\}$-group.
Thus, $5\nmid|K|$.
In $C_2^6\rtimes A_7$, a $7$-cycle fixes no nonidentity sign vector
$(\varepsilon_1,\ldots,\varepsilon_7)$ with $\prod_i\varepsilon_i=1$.
Since its centralizer in $A_7$ is $C_7$, its centralizer in $C_2^6\rtimes A_7$ is also $C_7$.
For $c\in\Sp_6(2)$ with $|c|=7$, the module $\F_2^6$ has $2$ reciprocal irreducible constituents of dimension $3$.
Thus, $C_{\Sp_6(2)}(c)\cong C_7$.
As neither centralizer has an involution, neither subgroup type contains $T^x=C_{14}$.
Therefore, $M_7=N_1^-$ or $M_7\cong S_9$.

Suppose that $M_7\cong S_9$.
We identify $M_7$ with $S_9$, so $K\leq S_9$.
An element of $T^x$ of order $14$ has cycle type $(7)(2)$,
so $K$ contains a transposition.
A primitive subgroup of $S_9$ containing a transposition equals $S_9$, contrary to $5\nmid|K|$.
A transitive imprimitive group of degree $9$ has blocks of size $3$,
which cannot accommodate an element of order $7$.
Thus, $K$ is intransitive.
Let $\alpha,\beta$ be the $2$ points outside the $7$-cycle.
Since $(\alpha\ \beta)\in K$, a $K$-orbit contains both points or neither point.
Thus, the $K$-orbit containing the $7$-cycle cannot have size $8$.
It follows that $K\leq S_7\times S_2$.
For this subgroup of $S_7\times S_2$, let $\pi_7:K\to S_7$ send $(g,h)$ to $g$.
Since $\pi_7(K)$ contains a $7$-cycle, its action of prime degree $7$ is transitive and primitive.
If $\pi_7(K)$ contained a single $3$-cycle, then we would have $A_7\leq\pi_7(K)$ by Jordan's theorem \cite{JonesCycles}.
This contradicts $5\nmid|K|$.
Thus, every element of order $3$ in $\pi_7(K)$ has type $3^2 1$.
As the kernel of $\pi_7$ has order at most $2$, we identify each such $a$ with its unique lift $(a,1)\in K$ of order $3$.
Next, choose $y\in S_7\times\{1\}\leq M_7$ with $a^y$ a generator of a $C_3$ complement in $C_7\rtimes C_3$.
As $\ord_7(3)=6$, $C_7$ has $6$ distinct nontrivial eigenlines and a fixed line on $\overline{\F}_3^7$.
The complement permutes the $6$ eigenlines in $2$ $3$-cycles.
Each weighted cycle has Jordan block $J_3$ in characteristic $3$.
Thus, each element $a$ of order $3$ has Jordan form $J_3+J_3+J_1$ on $V=\F_3^7$.
It has no invariant direct summand of dimension $2$.
This contradicts the decomposition $gU\perp gW$ stabilized by $\gen a$.
Hence, $A$ is a $2$-group.
Because $|K|$ has only the prime divisors $2$ and $7$, we have $K$ soluble by Burnside's theorem.
As $\pi_7(K)$ is soluble and primitive, we have $\pi_7(K)\leq\AGL_1(7)$.
Because $3\nmid|K|$, we have $\pi_7(K)\leq C_7\rtimes C_2$.
The Sylow $7$-subgroup is normal in $K$.
Let $S$ be the normal Sylow $7$-subgroup of $K$.
As $C_G(S)=T^x$, we have $A\leq N_G(T^x)$.

\begin{itemize}
\item[{\bf Step 4.} ] If $(m,q)=(3,3)$ and $K\leq N_1^-$, then $A\leq N_G(T^x)$.
\end{itemize}
Finally, suppose that $(m,q)=(3,3)$ and $K\leq N_1^-$.
Let $J=K\cap\Omega^-_6(3)\unlhd K$.
By the indecomposability argument in Step 1, we have $J<\Omega^-_6(3)$.
As $5\nmid|K|$, we have $5\nmid|J|$.
Since $[K:J]\leq2$, every element of $A$ of order $3$ lies in $J$.
Since $J<\Omega^-_6(3)$, we choose a maximal subgroup $M_6<\Omega^-_6(3)$ with $J\leq M_6$.
As $T^x=C_{14}\leq J$, $M_6$ is isomorphic to one of $\SU_3(3)\times C_2,2\times A_7,2.\PSL_3(4)$
\cite[Tables 8.33--8.34]{BHR}.
For the unitary type, identify $E$ with $\F_9^3$ over $\F_3$.
We identify $M_6$ with $\{g\in\GU_3(3):\det g\in\{1,-1\}\}=\SU_3(3)\times C_2$.
Let $C_4=Z(\GU_3(3))$ and $T_U=T^xC_4$, so $|T_U|=28$.
We have $JC_4<\GU_3(3)$.
Otherwise, $JC_4$ would contain an indecomposable Sylow $3$-subgroup of $\SU_3(3)$ by Lemma \ref{lem:symplectic-unitary-sylow}.
Every proper subgroup $H<\GU_3(3)$ with $T_U\leq H$ lies in $N_{\GU_3(3)}(T_U)$ or $C_4\times\PSL_3(2)$
\cite{Bereczky,Brookfield2014}.
For $M_6\cong C_2\times A_7$, we identify $M_6$ with $C_2\times A_7$.
Let $\pi_A:C_2\times A_7\to A_7$ send $(c,h)$ to $h$.
Because $5\nmid|\pi_A(J)|$, the image has no single $3$-cycle by Jordan's theorem.
For $M_6\cong2.\PSL_3(4)$, let $\pi_L:2.\PSL_3(4)\to\PSL_3(4)$ be the central quotient.
Then,
we have $5\nmid|\pi_L(J)|$ and $\pi_L(J)<\PSL_3(4)$.
Thus, $\pi_L(J)$ lies in
$\PSL_3(2)$ \cite{Atlas1985}.
For every element $a\in J$ of order $3$, we choose a conjugating element $y$.
For the unitary type, we choose $y\in\GU_3(3)$ acting as $I_L$ on $L$.
For $M_6$ of type $C_2\times A_7$ or $2.\PSL_3(4)$, we choose $y\in M_6$.
In each case, $a^y$ lies in a $C_3$ complement of $C_7\rtimes C_3$.
Thus, $a$ has Jordan form $J_3+J_3+J_1$ on $V=\F_3^7$.
As $\gen a$ cannot stabilize a direct summand of dimension $2$, $3\nmid|A|$.
Because the only prime divisors of $|K|$ are $2$ and $7$, $K$ is soluble.
If $J\leq N_{\GU_3(3)}(T_U)$, then its Sylow $7$-subgroup lies in the cyclic torus and is unique.
For $J\leq C_4\times\PSL_3(2)$, let $\pi_2:C_4\times\PSL_3(2)\to\PSL_3(2)$ send $(c,h)$ to $h$.
In this case, let $J_{\rm im}=\pi_2(J)$.
For $M_6\cong C_2\times A_7$, let $J_{\rm im}=\pi_A(J)$.
For $M_6\cong2.\PSL_3(4)$, let $J_{\rm im}=\pi_L(J)$.
Because $K$ is soluble, $J_{\rm im}$ is soluble in each of the $3$ cases.
As $7\mid|J_{\rm im}|$, its action of degree $7$ is transitive and hence primitive.
By the preceding argument for degree $7$, $J_{\rm im}$ has a unique Sylow $7$-subgroup.
The central kernels have order prime to $7$, so $J$ also has a unique Sylow $7$-subgroup.
Let $Q$ be the unique Sylow $7$-subgroup of $J$, with $C_G(Q)=T^x$.
Because $Q$ is characteristic in $J\unlhd K$ and $[K:J]\leq2$, $Q$ is a normal Sylow $7$-subgroup of $K$.
Thus, $A\leq N_G(T^x)$.
\end{proof}

\begin{theorem}\label{thm:orthogonal-odd}
Let $q$ be odd and $m\geq3$.
Then, $\Omega_{2m+1}(q)$ has no proper nontrivial ordinary
$\Omega_{2m+1}(q)$-permutable subgroup.
\end{theorem}
\begin{proof}
Let $G=\Omega_{2m+1}(q)$ and $p=\operatorname{char}\F_q$.
Let $A<G$ be ordinary $G$-permutable.
By Lemma \ref{lem:orthogonal-odd-singer}, we have $A\leq N_G(T)$ for a Singer torus $T$.
We choose $V=E\perp L$ with $\dim_{\F_q}E=2m$ and $\dim_{\F_q}L=1$.
The torus $T$ acts irreducibly on $E$ and fixes $L$ pointwise.
Next, on $E=\F_{q^{2m}}$, let $M_\lambda(v)=\lambda v$ for $\lambda\in E^*$. We take $S_0=\{M_\lambda:\lambda^{q^m+1}=1\}$.
Then, we have $S_0\leq\operatorname{SO}(E)$ and $|S_0|=q^m+1$.
Let $\iota:S_0\to\operatorname{SO}(V)$ send $h$ to $h\oplus I_L$.
By the spinor calculation in Theorem \ref{thm:orthogonal-minus-odd}, we have $T=\iota(S_0^2)$.
Let $F:E\to E$ satisfy $F(v)=v^q$.
Since $\det F=-1$, the map $F\oplus(-I_L)$ has determinant $1$.
Choose $h\in S_0$ so that $\iota(h)(F\oplus(-I_L))$ has trivial spinor norm.
The map $\theta:N_G(T)\to\operatorname{Gal}(\F_{q^{2m}}/\F_q)$
takes $g\in N_G(T)$ to the field automorphism in the semilinear restriction $g|_E$.
By the preceding construction, we have $g_0\in N_G(T)$ with $\theta(g_0)(v)=v^q$ for $v\in E$.
Hence, $\theta$ is surjective and $\ker\theta=T$.
Therefore, $N_G(T)/T=C_{2m}$.
By Lemma \ref{lem:max} for $P_1$, $A$ stabilizes a singular line $\Lambda\leq V$.
Since $L$ is nondegenerate, we choose $v=e+l\in\Lambda\setminus\{0\}$ with $0\ne e\in E$ and $l\in L$.
For $a\in A\cap T$, let $M_\lambda=a|_E$.
As $a\Lambda=\Lambda$, we have $av=cv$ for some $c\in\F_q^*$.
Taking the projection on $E$, we obtain that $\lambda e=ce$, so $\lambda=c$.
Because $\lambda^{q^m+1}=1$, we have $c^2=1$.
Thus, $a|_E=I_E$ or $-I_E$, and $a|_L=I_L$.
Consequently, $|A\cap T|\leq2$.
The possible involution of $A\cap T$ is $t=(-I_E)\oplus I_L\in Z(N_G(T))$.
As $A\cap T\leq Z(N_G(T))$ and $A/(A\cap T)$ is cyclic, $A$ is abelian.
The odd Sylow subgroups of $A$ are cyclic of orders dividing $m$,
and the Sylow $2$-subgroup has rank at most $2$.

We take a primitive prime divisor $\ell$ of $q^{2m-2}-1$.
Next, we choose a nondegenerate subspace $U\leq V$ of minus type with $\dim_{\F_q}U=2m-2$. Let $W=U^\perp$, so $V=U\perp W$ and $\dim_{\F_q}W=3$.
Take $B=C_\ell$ irreducible on $U$ and trivial on $W$.
Since $\ell\geq2m-1>m$ and $\ell\geq5$, we have $\ell\nmid|A|$.
We choose $x\in G$ such that $H=AB^x$ is a subgroup.
As $A$ and $B^x$ are abelian, we conclude from It\^o's theorem that $H'$ is abelian.
If $\ell\mid|H'|$, then the Sylow $\ell$-subgroup $B^x$ of the abelian group $H'$ is characteristic.
Thus, $B^x\unlhd H$.
Otherwise, $AH'$ is $\ell$-free and contains $A$, so $AH'=A$ and $A\unlhd H$.
The odd Sylow subgroups of $A$ are cyclic with primes at most $m$, so their automorphism groups have no prime divisor greater than $m$.
For the Sylow $2$-subgroup of rank at most $2$, the automorphism group has no prime divisor other than $2$ or $3$.
Since $\ell>m$ and $\ell\geq5$, we obtain that $\ell\nmid|\Aut(A)|$.
Thus, $B^x$ centralizes $A$, and $B^x\unlhd H$.

Let $a\in A$ have odd prime order $r$.
Since $|A\cap T|\leq2$, we have $\theta(a)\ne1$.
Let $\alpha:N_G(B^x)\to\Aut(B^x)$ satisfy $\alpha(g)(b)=gbg^{-1}$.
As $|N_G(B^x)/C_G(B^x)|$ divides $2m-2$, the order of $\alpha(a)$ divides both $r$ and $2m-2$.
Since $r\mid m$ is odd, $a$ centralizes $B^x$.
Let $U_x=x^{-1}U$, so $B^x$ is irreducible on $U_x$ and $\dim_{\F_q}U_x=2m-2$.
If $r=p$, then $a|_{U_x}=I_{U_x}$.
As $\dim_{\F_q}U_x^\perp=3$ and $a|_{U_x^\perp}$ is unipotent, we have $\operatorname{rank}(a-I_V)\leq2$.
Because $\theta(a)\ne1$, the operator $a|_E$ has $2m/p$ Jordan blocks of length $p$.
Thus, $\operatorname{rank}(a|_E-I_E)=2m(p-1)/p\geq4$.
This is impossible.

Suppose that $r\ne p$.
Since $\theta(a)\ne1$, every nonidentity eigenvalue of $a$ has multiplicity $2m/r$.
Moreover, the eigenvalue $1$ has multiplicity $2m/r+1$.
As $a$ centralizes $B^x$, the restriction $a|_{U_x}$ is a multiplier of degree at most $r-1$.
If $a|_{U_x}\ne I_{U_x}$ and $r<m$, then a nonidentity eigenvalue has multiplicity at least $(2m-2)/(r-1)>2m/r$.
We have a contradiction.
If $r=m$ is prime, then we have $r\nmid q^{m-1}+1$ by Fermat's theorem.
Thus, $a|_{U_x}=I_{U_x}$.
If $a|_{U_x}=I_{U_x}$, then we have $\dim_{\F_q}\ker(a-I_V)\geq2m-2>2m/r+1$.
This contradicts the preceding multiplicity of eigenvalue $1$.
Thus, $A$ is a $2$-group.

If $T$ is odd, then $A\cap T=1$ and $A$ is cyclic.
By Corollary \ref{cor:split-odd-involutions}, we conclude that $A=1$.
If $T$ is even, then $m$ is odd and $q\equiv3\pmod4$.
The Sylow $2$-subgroup of $\theta(A)$ has order at most $2$, so $|A|\leq4$.
If $A$ is cyclic, then we conclude that $A=1$ by Corollary \ref{cor:split-odd-involutions}.
If $A$ is isomorphic to $C_2\times C_2$, then the torus involution $t\in A\cap T$ has fixed dimension $1$.
For every other involution $s\in A$,
$s|_E(v)=cv^{q^m}$ for $v\in E$ and $c\in\F_{q^{2m}}^*$ with $c c^{q^m}=1$.
By Hilbert 90, we have $\dim_{\F_q}\ker(s|_E-I_E)=\dim_{\F_q}\ker(s|_E+I_E)=m$.
Since $\det(s|_E)=(-1)^m=-1$ and $\det s=1$, we have $s|_L=-I_L$.
Thus, $\dim\ker(s-I_V)=m\ne1=\dim\ker(t-I_V)$.
As conjugation in $G$ does not change fixed-space dimensions, no $g\in G$ satisfies $t^g=s$.
By Lemma \ref{lem:split-root-normalizer}, we choose a $p$-subgroup $H\leq G$ with $|N_G(H)|$ odd.
However, by Lemma \ref{lem:odd-normalizer-fusion}, we would have $t^g=s$ for some $g\in G$. We obtain a contradiction.
Hence, $A=1$.
\end{proof}

\section{Alternating groups}\label{sec:alternating}

Let $\Omega$ be a finite set and $H\leq S_{\Omega}$.
A nonempty subset $\Delta\subseteq\Omega$ is a block if every $h\in H$ satisfies $h\Delta=\Delta$ or $h\Delta\cap\Delta=\varnothing$.
If $H$ is transitive, then $\{h\Delta:h\in H\}$ is a block system on $\Omega$.
A transitive action is primitive if every block has size $1$ or $|\Omega|$.
Otherwise, the transitive action is imprimitive.
For an $H$-invariant subset $\Delta$, let $\pi_\Delta:H\to S_{\Delta}$ satisfy $\pi_\Delta(h)=h|_\Delta$ and let $H^\Delta=\pi_\Delta(H)$.

For a prime power $q$, let $\AGL_1(q)=\{u\mapsto au+b:a\in\F_q^*,\ b\in\F_q\}$.
For prime $p$ and $c\geq1$, let $T=\{u\mapsto u+b:b\in\F_p^c\}$ and $\operatorname{AGL}_c(p)=\{u\mapsto Mu+b:M\in\GL_c(p),\ b\in\F_p^c\}$.
We call $H$ affine if there is a bijection $\xi:\Omega\to\F_p^c$ with $T\leq\xi H\xi^{-1}\leq\operatorname{AGL}_c(p)$ for some $p,c$.
For a prime $p$, let $\mu_p:\AGL_1(p)\to\F_p^*$ satisfy $\mu_p(u\mapsto au+b)=a$.
Its kernel is $\{u\mapsto u+b:b\in\F_p\}$ and has order $p$.
For $|\Delta|=p$ and a bijection $\xi:\Delta\to\F_p$, let $c_\xi:S_{\Delta}\to S_{\F_p}$ satisfy $c_\xi(h)=\xi\circ h\circ\xi^{-1}$.
For $g\in S_{\Omega}$, let $\operatorname{supp}(g)=\{\omega\in\Omega:g(\omega)\ne\omega\}$.
Let $\operatorname{supp}(H)=\bigcup_{h\in H}\operatorname{supp}(h)$ and $\Fix(H)=\Omega\setminus\operatorname{supp}(H)$.
In a composition series $1=K_0\lhd K_1\lhd\cdots\lhd K_s=K$, every quotient $K_i/K_{i-1}$ is simple.
We call these quotients the composition factors of $K$.

\begin{lemma}\label{lem:alternating-affine-restriction}
Let $H\leq S_{\Omega}$ and let $\Delta\subseteq\Omega$ satisfy $h\Delta=\Delta$ for every $h\in H$.
Suppose that $|\Delta|=p$ is prime and $p\nmid|H^\Delta|$.
Let $\xi_0:\Delta\to\F_p$ be a bijection with $c_{\xi_0}(H^\Delta)\leq\AGL_1(p)$.
There exist $v\in\F_p$ and a bijection $\xi:\Delta\to\F_p$ satisfying $\xi(\delta)=\xi_0(\delta)-v$ for every $\delta\in\Delta$.
The homomorphism $\theta=\mu_p\circ c_\xi\circ\pi_\Delta:H\to\F_p^*$ satisfies $\ker\theta=\ker\pi_\Delta$.
For $D=\theta(H)$, the integer $d=|D|=|H^\Delta|$ divides $p-1$.
Moreover, $c_\xi(H^\Delta)=\{u\mapsto cu\mid c\in D\}$.
For $d=1$, every point of $\Delta$ is fixed by $H$.
For $d>1$, $\xi^{-1}(0)$ is the unique $H$-fixed point in $\Delta$.
For $u\in\F_p^*$, the $H$-orbit of $\xi^{-1}(u)$ is $\xi^{-1}(uD)$ and has size $d$.
If $X\subseteq\Omega$ satisfies $hX=X$ for every $h\in H$ and $|X|\geq|\Omega\setminus\Delta|+2$, then $d\leq|X|$.
\end{lemma}
\begin{proof}
Since $\ker\mu_p$ has order $p$ and $p\nmid|H^\Delta|$, we have $\ker\mu_p\cap c_{\xi_0}(H^\Delta)=1$.
Thus, $c_{\xi_0}(H^\Delta)$ is cyclic and has order $m=|H^\Delta|$ dividing $p-1$.
For $m>1$, we choose $h\in H^\Delta$ with $H^\Delta=\gen h$.
We have $c_{\xi_0}(h)(u)=au+b$ with $a\ne1$.
As $av+b=v$ for $v=b/(1-a)$, the bijection $\xi(\delta)=\xi_0(\delta)-v$ satisfies $c_\xi(h)(u)=au$.
For $m=1$, we take $v=0$ and $\xi=\xi_0$.
The coordinate change $u\mapsto u-v$ sends $u\mapsto au+b$ to $u\mapsto au+b+(a-1)v$.
Consequently, $c_\xi(H^\Delta)\leq\AGL_1(p)$ and $\mu_p$ is injective on $c_\xi(H^\Delta)$.
We obtain that $\ker\theta=\ker\pi_\Delta$ and $|D|=m=d$.
Because $c_\xi(H^\Delta)$ consists of the maps $u\mapsto cu$ with $c\in D$, the fixed points and orbits have the required forms.

Suppose that $hX=X$ for every $h\in H$ and $|X|\geq|\Omega\setminus\Delta|+2$.
Then, $|X\cap\Delta|\geq2$.
We choose $u\in\F_p^*$ with $\xi^{-1}(u)\in X\cap\Delta$.
Since $hX=X$ for every $h\in H$, we have $\xi^{-1}(uD)\subseteq X$.
Therefore, $d\leq|X|$.
\end{proof}

\begin{proposition}\label{prop:alternating-involutions}
For every $n\geq7$, the group $A_n$ has no ordinary
$A_n$-permutable subgroup of order $2$.
For $n\geq14$, $A_n$ has no nontrivial ordinary
$A_n$-permutable cyclic $2$-subgroup.
\end{proposition}
\begin{proof}
For ease of reading, we divide the proof of the proposition into the following steps.

\begin{itemize}[leftmargin=4em,labelsep=.5em,topsep=.2\baselineskip,partopsep=0pt,parsep=0pt,itemsep=0pt]
\item[{\bf Step 1.} ] If $c\in A_n$ is an odd cycle of length $n$ or $n-1$, then every involution in $N_{A_n}(\gen c)$ fixes at most $n/3+1$ points.
\end{itemize}
Let $L=n$ when $n$ is odd, and $L=n-1$ when $n$ is even.
Let $C\leq A_n$ be generated by an $L$-cycle.
For this cyclic subgroup, we have $N_{S_n}(C)=C_L\rtimes\Aut(C_L)$,
with the point outside the cyclic orbit fixed when $n$ is even.
An involution $t\in N_{S_n}(C)$ fixes a proper divisor of $L$ points
on the cyclic orbit.
For each prime $p\mid L$, let $p^a$ denote the $p$-part of $L$.
On the factor $C_{p^a}\leq C_L$, the multiplier of $t$ is $1$ or $-1$.
When the multiplier is $1$, the component contributes its order to the number of fixed points.
When the multiplier is $-1$, the component contributes $1$ point.
Since $L$ is odd, every proper divisor of $L$ is at most $L/3$.
Thus, every involution in $N_{A_n}(C)$ has at most $n/3+1$ fixed points.
For an involution in $A_n$, the number $f$ of fixed points satisfies $f\equiv n\pmod4$.
By Lemma \ref{lem:involution-normalizers}, a generator of an ordinary $A_n$-permutable subgroup of order $2$ must satisfy both bounds.

We construct subgroups $B\leq A_n$ for which no involution in $N_{A_n}(B)$ has the cycle type of an involution in $N_{A_n}(C)$.
For this construction, let $H=C_7\rtimes C_3=\gen{a,b}$ act affinely on $7$ points.
We take $b^{-1}ab=a^2$.
We claim that every automorphism of $H$ acts trivially on $H/H'\cong C_3$.
To prove the claim, let $\varphi\in\Aut(H)$.
We choose $1\leq u\leq6$, $0\leq v\leq6$ and $e\in\{1,2\}$ with $\varphi(a)=a^u$ and $\varphi(b)=a^vb^e$.
Since $\varphi(b)^{-1}\varphi(a)\varphi(b)=\varphi(a)^2$, we have $2^e\equiv2\pmod7$.
Therefore, $e=1$.

We choose $r$ copies of $H$ on disjoint orbits of length $7$.
Let $\chi:H\to\F_3$ satisfy $\chi(a)=0$ and $\chi(b)=1$.
For $1\leq i\leq r$, let $\chi_i:H^r\to\F_3$ satisfy $\chi_i(h_1,\ldots,h_r)=\chi(h_i)$.
For $\Lambda\subseteq\operatorname{Hom}(H^r,\F_3)\setminus\{0\}$ with distinct kernels,
add $1$ orbit $\F_3$ for each $\lambda\in\Lambda$.
On the orbit indexed by $\lambda$, an element $h\in H^r$ sends $t$ to $t+\lambda(h)$.
Let $\rho:H^r\to A_n$ be the permutation homomorphism with $n-7r-3|\Lambda|$ additional fixed points.
Since each copy of $H$ acts faithfully on its $7$ points, we have $\ker\rho=1$.
Let $B=\rho(H^r)$.
Since $g\in N_{S_n}(B)$ maps $B$-orbits to $B$-orbits of equal length, $g$ permutes the orbits of length $7$.
Hence, $g$ also permutes the factors of $H^r$.
For the induced automorphism, let $\varphi_g:H^r\to H^r$ satisfy $\varphi_g(h)=\rho^{-1}(g\rho(h)g^{-1})$ for $h\in H^r$.
Let $\sigma\in S_r$ be the permutation of the $r$ factors induced by $\varphi_g$.
For $1\leq i\leq r$, we have $\chi_{\sigma(i)}\circ\varphi_g=\chi_i$.
Each $\lambda\in\Lambda$ satisfies $\ker(-\lambda)=\ker\lambda$.
Suppose that the kernels in $\Lambda$ have trivial stabilizer under permutations of the $r$ coordinates.
Then, $\sigma=1$ for every $g\in N_{S_n}(B)$.
By $\chi_{\sigma(i)}\circ\varphi_g=\chi_i$, we have $\chi_i\circ\varphi_g=\chi_i$ for $1\leq i\leq r$.
Since each $\lambda\in\Lambda$ is an $\F_3$-linear combination of the $\chi_i$, we have $\lambda\circ\varphi_g=\lambda$.
Because the kernels are distinct, $g$ maps each orbit indexed by $\lambda\in\Lambda$ to itself.
On that orbit, $g$ commutes with $u\mapsto u+1$.
The centralizer of this $3$-cycle has order $3$.
Consequently, every involution in $N_{S_n}(B)$ fixes all $3$ points on each of these orbits.
Let $t\in N_{S_n}(B)$ be an involution.
On each orbit of length $7$, $t$ normalizes $C_7$ and fixes at least $1$ point.
Thus, $t$ fixes at least $r+3|\Lambda|$ points.

\begin{itemize}[leftmargin=4em,labelsep=.5em,topsep=.2\baselineskip,partopsep=0pt,parsep=0pt,itemsep=0pt]
\item[{\bf Step 2.} ] For $n\geq20$, $A_n$ has no ordinary $A_n$-permutable subgroup of order $2$.
\end{itemize}
Suppose that $n\geq30$.
Let $n=10r+3k+\ell$ with $r=\lfloor n/10\rfloor\geq3$.
Then, $0\leq k\leq3$ and $0\leq\ell\leq2$.
We use $\Lambda_0=\{\chi_1\}\cup\{\chi_i+\chi_{i+1}\mid1\leq i<r\}$.
Then, we add the first $k$ characters in the list
$\chi_1+\chi_2+\chi_3,\chi_1+\chi_2-\chi_3,\chi_1-\chi_2+\chi_3$.
For $\lambda=\sum_i c_i\chi_i$, let $I(\lambda)=\{i:c_i\ne0\}$.
Since $I(\chi_1)=\{1\}$ is the only singleton among the sets $I(\lambda)$, a coordinate permutation stabilizing the kernels in $\Lambda$ fixes $1$.
The sets $I(\chi_i+\chi_{i+1})=\{i,i+1\}$ form the path $1,2,\ldots,r$.
As an automorphism of the path $1,2,\ldots,r$ fixing $1$ fixes every vertex, the coordinate stabilizer is trivial.
Because $|\Lambda|=r+k$, the action uses $10r+3k$ points.
We let $B$ fix the $\ell$ unused points.
Every involution normalizing $B$ fixes at least $4r+3k$ points.
As $r\geq3$ and $\ell\leq2$, we have $3(4r+3k)-n-3=2r+6k-\ell-3>0$.
This contradicts $f\leq n/3+1$ for involutions in $N_{A_n}(C)$.

For $20\leq n\leq22$, we take $r=2$ and $\Lambda=\{\chi_1,\chi_1+\chi_2\}$.
We fix the unused points.
By the bound $r+3|\Lambda|$, every involution in $N_{S_n}(B)$ fixes at least $8$ points.
For $n=20,21$, we have $n/3+1<8$, contrary to the bound from $N_{A_n}(C)$.
For $n=22$, we have $f\leq8$ and $f\equiv2\pmod4$.
Thus, $f\leq6$.
For $23\leq n\leq29$, we also add $\chi_1-\chi_2$ to $\Lambda$.
Then, every involution normalizing $B$ fixes at least $11$ points, greater than $n/3+1$.

\begin{itemize}[leftmargin=4em,labelsep=.5em,topsep=.2\baselineskip,partopsep=0pt,parsep=0pt,itemsep=0pt]
\item[{\bf Step 3.} ] For $7\leq n\leq19$, $A_n$ has no ordinary $A_n$-permutable subgroup of order $2$.
\end{itemize}
For $14\leq n\leq18$, let $m=n-10$ when $n$ is odd and $m=n-11$ when $n$ is even.
Let $\rho_m:H\times C_m\to A_n$ have disjoint orbits $\Omega_7,\Xi_3,\Gamma_m$ of lengths $7$, $3$ and $m$.
On $\Omega_7$, the factor $H$ has its affine action.
In addition, $H$ acts on $\Xi_3$ through $H/H'\cong C_3$.
On $\Gamma_m$, the factor $H$ acts trivially and $C_m$ acts regularly.
For the action of $C_m$ on the other points, we take every point of $\Omega_7\cup\Xi_3$ to be fixed.
For even $n$, we add $1$ point fixed by $\rho_m(H\times C_m)$.
Let $B=\rho_m(H\times C_m)$.
Every element of $N_{S_n}(B)$ maps $\Xi_3$ to itself.
For $n=14$, the centre $Z(B)$ fixes $\Xi_3$ pointwise.
Since $Z(B)$ is transitive on the other orbit of length $3$, no element of $N_{S_n}(B)$ sends either orbit to the other.
When $n=17,18$, the image $B^{\Omega_7}$ is nonabelian and $B^{\Gamma_7}$ is cyclic.
Hence, every element of $N_{S_n}(B)$ maps each orbit of length $7$ to itself.
Every normalizing involution fixes at least $5$ points for odd $n$
and $6$ for even $n$.
For $n=14,15,16,17,18$, the fixed-point bounds from $N_{A_n}(C)$ allow only $f=2,3,4,1,2$ in that order.
For $n=19$, the cyclic normalizer has no even involution.

For $n\in\{7,8,11,12\}$, we add $1$ to the proper divisors of $L$ when $n$ is even.
None of the possible values of $f$ satisfies $f\equiv n\pmod4$.
For $n=9,10,13$, the fixed-point bounds from $N_{A_n}(C)$ allow only $f=1,2,1$ in that order.
We use $B=\PSL_3(2)$ on $7$ points and fix the rest.
The unique $B$-orbit of length $7$ on triples is the Fano plane.
Since $N_{S_7}(B)$ permutes the $B$-orbits on triples of equal length, every element of $N_{S_7}(B)$ maps the Fano plane to itself.
As the automorphism group of the Fano plane is $B$, we have $N_{S_7}(B)=B$.
An involution in $B$ moves $4$ points.
Together with the permutations of the fixed points, an even
normalizing involution has at most $2$ transpositions for $n=9,10$
and $4$ for $n=13$.
Thus, neither possible number of transpositions agrees with the required cycle type.
Consequently, we obtain from Lemma \ref{lem:involution-normalizers} that $A_n$ has no ordinary $A_n$-permutable subgroup of order $2$.

\begin{itemize}[leftmargin=4em,labelsep=.5em,topsep=.2\baselineskip,partopsep=0pt,parsep=0pt,itemsep=0pt]
\item[{\bf Step 4.} ] For $n\geq14$, $A_n$ has no nontrivial ordinary $A_n$-permutable cyclic $2$-subgroup.
\end{itemize}
Suppose that $n\geq14$ and $A$ is a nontrivial cyclic $2$-subgroup that is ordinary $A_n$-permutable.
For $n=19$, we choose $y\in A_n$ such that $AC^y$ is a subgroup.
By Burnside's normal-complement theorem, we have $C^y\unlhd AC^y$.
The involution of $A$ then lies in $N_{A_n}(C^y)$, a contradiction.
For $n\ne19$, the constructed subgroup $B$ has odd order.
We choose $x,y\in A_n$ such that $AB^x$ and $AC^y$ are subgroups.
Both products have cyclic Sylow $2$-subgroups.
By Burnside's normal-complement theorem, we have $B^x\unlhd AB^x$ and $C^y\unlhd AC^y$.
Let $t$ be the unique involution of $A$.
Then, $t\in N_{A_n}(B^x)\cap N_{A_n}(C^y)$.
The fixed-point bounds for those $2$ normalizers are incompatible.
Therefore, no nontrivial cyclic $2$-subgroup $A$ is ordinary $A_n$-permutable.
\end{proof}

\begin{proposition}\label{prop:alternating-twenty-four}
The group $A_{24}$ has no proper nontrivial ordinary
$A_{24}$-permutable subgroup.
\end{proposition}
\begin{proof}
Let $A<G=A_{24}$ be ordinary $G$-permutable.
Let $M=(S_6\times S_{18})\cap G$, a maximal subgroup.
By \cite[Theorem D]{LPS}, we have $MY\ne G$ for every $Y<G$.
By Lemma \ref{lem:max}, we have $A\leq M^x$ for some $x\in G$.
Thus, $AY\ne G$ for every $Y<G$.
For each maximal subgroup $M<G$, there exists $x_M\in G$ with $A\leq M^{x_M}$.
We use the imprimitive subgroup $(S_3\wr S_8)\cap G$.
The order of $A$ divides $(3!)^8\,8!$, so $11$ and $23$ do not divide $|A|$.
Let $B$ be generated by a $23$-cycle.
Since $A$ is ordinary $G$-permutable, we choose $x\in G$ such that $K=AB^x$ is a subgroup.
As $AY\ne G$ for every $Y<G$, we have $K<G$.
We have $11\nmid|K|$.
The $K$-orbit containing the $23$-cycle has size $23$ or $24$.

Let $\Delta$ be the $K$-orbit containing $\operatorname{supp}(B^x)$.
Suppose that $|\Delta|=23$.
Then, $K$ fixes the point outside $\Delta$ and acts primitively on $\Delta$.
By \cite[Theorem 1.2]{JonesCycles}, $K^\Delta$ is affine or projective or $K^\Delta\cong M_{23}$ or $K^\Delta$ contains the alternating group on $\Delta$.
Since $11\mid|M_{23}|$ and $11\mid|A_{23}|$, neither $K^\Delta\cong M_{23}$ nor $A_{23}\leq K^\Delta$ is possible.
The projective equation $(q^a-1)/(q-1)=23$ has no solution with $q$ a prime power and $a\geq2$.
For $a\geq5$, the degree is at least $31$.
For $a=4$ and $q=2$, the degree is $15$.
For $a=4$ and $q\geq3$, the degree is at least $40$.
For $a=3$, we would have $q(q+1)=22$.
For $a=2$, we would have $q=22$.
As $22$ is not a prime power, this equality is impossible.
Therefore, we choose a bijection $\xi:\Delta\to\F_{23}$ with $c_\xi(K^\Delta)\leq\AGL_1(23)$.
Because $K$ fixes the outside point, every element of $K^\Delta$ is even.
Thus, $|K^\Delta|$ divides $23\cdot11$.
As $11\nmid|K|$ and $B^x\leq K$, we have $K=B^x$ and $A=1$.

Suppose that $|\Delta|=24$.

A nontrivial block system in degree $24$ has at most $12$ blocks, each of size at most $12$.
Since $|B^x|=23$, each block would be fixed setwise by $B^x$.
No block could contain the support of its $23$-cycle.
Therefore, $K$ is primitive.
By Jones's list, either $K$ is affine or $A_{24}\leq K$ or $K\in\{\PSL_2(23),\PGL_2(23),M_{24}\}$.
An affine degree is a power of a prime.
As $24$ is not a prime power, $K$ is not affine.
The orders of $\PSL_2(23)$, $\PGL_2(23)$ and $M_{24}$ are divisible by $11$.
A group containing $A_{24}$ contradicts $K<G$.
Therefore, $K$ cannot be transitive on all $24$ points.
\end{proof}

\begin{theorem}\label{thm:alternating-prime-plus-one}
Let $p\geq7$ be prime.
Then, $A_{p+1}$ has no proper nontrivial ordinary
$A_{p+1}$-permutable subgroup.
\end{theorem}
\begin{proof}
By Proposition \ref{prop:finite-tables}, we have the conclusion for $p=7,11$.
By Proposition \ref{prop:alternating-twenty-four}, we have the conclusion for $p=23$.
Assume that $p\geq13$ and $p\ne23$.
Let $n=p+1$ and $G=A_n$, and let $1<A<G$ be ordinary $G$-permutable.
To use Lemma \ref{lem:max}, we choose the maximal set stabilizer $M=(S_6\times S_{n-6})\cap G$.
By \cite[Theorem D]{LPS}, we have $MY\ne G$ for every $Y<G$.
Hence, $A\leq M^x$ for some $x\in G$, and $AY\ne G$ for every $Y<G$.
In particular, $p\nmid|A|$.

We choose a bijection $\eta:\{1,\ldots,n\}\to\mathbf P^1(\F_p)$.
Let $H=\eta^{-1}\PSL_2(p)\eta\leq G$, using the natural action on $\mathbf P^1(\F_p)$.
By \cite[proof of Corollary 3.2]{BurnessLiebeckShalevDepth}, $H$ is maximal in $G$.
Since $A$ is ordinary $G$-permutable, we choose $y\in G$ with $AH^y$ a subgroup.
As $AH^y\ne G$ and $H$ is maximal, we have $A\leq H^y$.
By Lemma \ref{lem:max} and $AY\ne G$ for every $Y<G$, we have a point fixed by $A$.
We replace $A$ by $A^{y^{-1}}\leq H$.
Because $A$ fixes a point, $A$ lies in a Borel subgroup $P=C_p\rtimes C_{(p-1)/2}$ of $H$.
As $p\nmid|A|$, we choose $z\in P$ and a split torus $T\leq P$ with $|T|=(p-1)/2$ and $A^z\leq T$.
Thus, $A$ is cyclic of order $d$.
As $A\leq T^{z^{-1}}$, every nonidentity element of $A$ fixes precisely the $2$ common eigenlines of $T^{z^{-1}}$ in the action on $\mathbf P^1(\F_p)$.
Thus, every point outside those $2$ fixed points has trivial stabilizer in $A$.
Consequently, all other $A$-orbits have length $d$.

Let $M_3=\operatorname{Stab}_G(\{1,2,3\})$.
By Lemma \ref{lem:max}, we choose $h\in G$ with $A\leq M_3^h$.
Then, $\Gamma=h^{-1}\{1,2,3\}$ satisfies $a\Gamma=\Gamma$ for every $a\in A$, and $|\Gamma|=3$.
Since $A$ has only $2$ fixed points and every other orbit has length $d$, we have $d\leq3$.
If $d=2$, then Proposition \ref{prop:alternating-involutions} contradicts ordinary permutability of $A$.
Suppose that $d=3$. Let $B$ be generated by a $5$-cycle.
As $A$ is ordinary $G$-permutable, we choose $x\in G$ with $AB^x$ a subgroup.
The subgroup has order $15$, so it is cyclic.
Let $a$ generate $A$ and let $b$ generate $B^x$.
Because $ab=ba$ and $|a|=3$ is prime to $|b|=5$, we have $a|_{\operatorname{supp}(b)}=1$.
This contradicts the $2$ fixed points of $a$.
Thus, $A=1$.
\end{proof}

\begin{lemma}\label{lem:alternating-prime-support}
Let $p\geq13$ be prime and $1\leq k\leq8$.
Let $A<A_{p+k}$ be ordinary $A_{p+k}$-permutable.
Then, $p\nmid|A|$. For each maximal subgroup $M<A_{p+k}$, there exists $y\in A_{p+k}$ with $A\leq M^y$.
The subgroup $A$ fixes a point and has an invariant subset of each size at most $(p+k)/2$.
Let $C=\gen c$, where $c$ is a $p$-cycle in $A_{p+k}$.
If $AC^x$ is a subgroup for $x\in A_{p+k}$,
then $A\operatorname{supp}(c^x)=\operatorname{supp}(c^x)$.
\end{lemma}
\begin{proof}
Let $n=p+k$.
Let $M=(S_6\times S_{n-6})\cap A_n$, a maximal subgroup.
By \cite[Theorem D]{LPS}, we have $MY\ne A_n$ for every $Y<A_n$.
By Lemma \ref{lem:max}, we have $A\leq M^x$ for some $x\in A_n$.
Thus, $AY\ne A_n$ for every $Y<A_n$.
For each maximal subgroup $M<A_n$, there exists $y\in A_n$ with $A\leq M^y$.
The stabilizer of $\{1,\ldots,k+1\}$ has order $(k+1)!(p-1)!/2$,
so $p\nmid|A|$.
By the point and set stabilizers, we have an $A$-fixed point.
We also have an $A$-invariant subset of each size below $n/2$.
When $n$ is even, we choose $y\in A_n$ with $A\leq((S_{n/2}\wr C_2)\cap A_n)^y$.
Let $\Omega_1,\Omega_2$ be its blocks, labelled so that an $A$-fixed point $\alpha$ lies in $\Omega_1$.
Since $g\alpha=\alpha$ for every $g\in A$, we have $g\Omega_i=\Omega_i$ for $i=1,2$.
Thus, $A$ has an invariant subset of size $n/2$.
Consequently, every $A$-orbit has size at most $\lceil n/2\rceil$.

Let $\Delta$ be the orbit of $K=AC^x$ containing $\operatorname{supp}(c^x)$.
If $|\Delta|>p$, then choose a $C^x$-fixed point $\alpha\in\Delta$.
Since $K=AC^x$ and $C^x$ fixes $\alpha$, we have $\Delta=\{a(\alpha):a\in A\}$.
As $p\geq13$ and $k\leq8$, we have $|\Delta|>p>\lceil(p+k)/2\rceil$.
This contradicts the $A$-orbit bound.
Hence, $|\Delta|=p$ and $A\Delta=\Delta=\operatorname{supp}(c^x)$.
\end{proof}

\begin{theorem}\label{thm:alternating-prime-plus-two}
Let $p\geq7$ be prime.
Then, $A_{p+2}$ has no proper nontrivial ordinary
$A_{p+2}$-permutable subgroup.
\end{theorem}
\begin{proof}
For $p=7,11$, we use Proposition \ref{prop:finite-tables}.
Assume that $p\geq13$, and let $A<G=A_{p+2}$ be ordinary $G$-permutable.
By Lemma \ref{lem:alternating-prime-support}, we have $p\nmid|A|$.
We also have $A$-invariant sets of sizes $3$ and $4$.
Let $C_p$ be generated by a $p$-cycle.
Since $A$ is ordinary $G$-permutable, we choose $x\in G$ such that $K=AC_p^x$ is a subgroup.
The support $\Delta$ of $C_p^x$ is a $K$-orbit of size $p$.
Every element of $\ker\pi_\Delta$ is an even permutation on the $2$ points outside $\Delta$.
Thus, $\ker\pi_\Delta=1$ and $\pi_\Delta:K\to S_{\Delta}$ is injective.
The primitive group $K^\Delta$ has subgroup $A^\Delta$ of index $p$.
Let $L\lhd K^\Delta$ with $p\mid|L|$. As $p\nmid|A^\Delta|$, we have $L\nleq A^\Delta$.
The $L$-orbits on $K^\Delta/A^\Delta$ have equal length dividing $p$.
Hence, $L$ is transitive on $K^\Delta/A^\Delta$ and $[L:L\cap A^\Delta]=p$.

By \cite[Theorem 1.2(1)]{JonesCycles}, $K^\Delta$ is affine or $L=\Soc(K^\Delta)$ is nonabelian simple.
If $K^\Delta$ is not affine, then $L\cong A_p$ or $L$ is projective or Mathieu.
For $L\cong A_p$ or $L\cong M_{23}$, we have $[L:L\cap A^\Delta]=p$.
The subgroup $L\cap A^\Delta$ is a natural point stabilizer \cite[Theorem 1.1]{JonesSezer}.
The natural point stabilizer has orbits of lengths $1$ and $p-1$.
With only $2$ outside points, $A$ cannot have an invariant subset of size $4$.
For a projective group, let $L=\PSL_b(q)$, where
$p=(q^b-1)/(q-1)$.
The subgroup $L\cap A^\Delta$ stabilizes a point or a hyperplane
\cite[Theorem 1.1 and the following paragraph]{JonesSezer}.
The orbits of a projective point stabilizer have lengths $1$ and $p-1$.
With only $2$ outside points, $A$ cannot have an invariant subset of size $4$.
For $b\geq3$, a hyperplane stabilizer has orbit lengths $(q^{b-1}-1)/(q-1)$ and $q^{b-1}$.
Both lengths are at least $4$.
The $2$ outside points cannot form an invariant subset of size $3$.
For $b=2$, the point and hyperplane stabilizers coincide.

Thus, $K^\Delta$ is affine.
Since $p\nmid|A|$, we obtain from Lemma \ref{lem:alternating-affine-restriction} a bijection $\xi:\Delta\to\F_p$ for the action of $A$ on $\Delta$.
Let $\theta=\mu_p\circ c_\xi\circ(\pi_\Delta|_A):A\to\F_p^*$, $D=\theta(A)$ and $d=|D|$.
We have $\ker\theta=\ker(\pi_\Delta|_A)$ and $A^\Delta\cong D\cong C_d$ with $d\mid p-1$.
As $\pi_\Delta|_A$ is injective, we have $\ker\theta=1$ and $A\cong C_d$.
We choose an $A$-invariant set $X_0$ with $|X_0|=4$.
Because $|\{1,\ldots,p+2\}\setminus\Delta|=2$, we obtain from Lemma \ref{lem:alternating-affine-restriction} that $d\leq4$.
For $d=2,4$, Proposition \ref{prop:alternating-involutions} contradicts ordinary permutability of $A$.
If $d=3$, then $A$ fixes the $2$ outside points.
A generator of $A$ has exactly $3$ fixed points.
Let $B$ be generated by a $5$-cycle.
As $A$ is ordinary $G$-permutable, we choose $x\in G$ with $AB^x$ a subgroup.
The subgroup has order $15$, so it is cyclic.
Let $a$ generate $A$ and let $b$ generate $B^x$.
Because $ab=ba$ and $|a|=3$ is prime to $|b|=5$, we have $a|_{\operatorname{supp}(b)}=1$.
This contradicts the $3$ fixed points of $a$.
Therefore, $A=1$.
\end{proof}

\begin{theorem}\label{thm:alternating-prime-plus-three}
Let $p\geq7$ be prime.
Then, $A_{p+3}$ has no proper nontrivial ordinary
$A_{p+3}$-permutable subgroup.
\end{theorem}
\begin{proof}
For $p=7$, we use Proposition \ref{prop:finite-tables}.
Since $A_{14}=A_{13+1}$, we use Theorem \ref{thm:alternating-prime-plus-one} for $p=11$.
Assume that $p\geq13$, and let $A<G=A_{p+3}$ be ordinary $G$-permutable.
By Lemma \ref{lem:alternating-prime-support}, we have $p\nmid|A|$.
We also have $A$-invariant sets of every size at most $(p+3)/2$.
For $M=(S_2\wr S_{(p+3)/2})\cap G$, there exists $y\in G$ with $A\leq M^y$.
A $3$-cycle cannot act within a block of size $2$.
Moving $3$ such blocks moves at least $6$ points.
Therefore, $A$ contains no single $3$-cycle.

Let $C_p$ be generated by a $p$-cycle.
Since $A$ is ordinary $G$-permutable, we choose $x\in G$ with $K=AC_p^x$ a subgroup.
By Lemma \ref{lem:alternating-prime-support}, $\Delta=\operatorname{supp}(C_p^x)$ is a $K$-orbit of size $p$.
The subgroup $A^\Delta$ has index $p$ in the primitive group $K^\Delta$.
By \cite[Theorem 1.2(1)]{JonesCycles}, $K^\Delta$ is affine or $L=\Soc(K^\Delta)$ is nonabelian simple.
Suppose that $L=\Soc(K^\Delta)$ is nonabelian simple.
As $L\unlhd K^\Delta$ and $p\nmid|A^\Delta|$, the $L$-orbits on $K^\Delta/A^\Delta$ have length $p$. Thus, $[L:H]=p$ for $H=L\cap A^\Delta$.
For $L\cong A_p$ or $L\cong M_{23}$, we obtain from \cite[Theorem 1.1]{JonesSezer} that $H$ is a natural point stabilizer.
For projective $L$, $H$ is a point or hyperplane stabilizer.
The point stabilizers have orbit lengths $1$ and $p-1$.
With only $3$ outside points, these orbits cannot form an $A$-invariant subset of size $5$.
Let $q$ be a prime power and $b\geq3$.
A hyperplane stabilizer in $\PSL_b(q)$ has orbit lengths $h=(q^{b-1}-1)/(q-1)$ and $q^{b-1}$.
If $h\geq5$, then neither orbit can meet an invariant subset of size $4$.
The $3$ outside points cannot form such a subset.
The only smaller possibility with $b\geq3$ and $p\geq13$ is $(b,q)=(3,3)$.
Then, $p=13$ and the orbit lengths are $4$ and $9$.
With only $3$ outside points, $A$ cannot have an invariant subset of size $8$.
For $b=2$, the point and hyperplane stabilizers coincide.

Thus, $K^\Delta$ is affine.
Since $p\nmid|A|$, we obtain from Lemma \ref{lem:alternating-affine-restriction} a bijection $\xi:\Delta\to\F_p$ for the action of $A$ on $\Delta$.
Let $\theta=\mu_p\circ c_\xi\circ(\pi_\Delta|_A):A\to\F_p^*$, $D=\theta(A)$ and $d=|D|$.
By Lemma \ref{lem:alternating-affine-restriction}, we have $\ker\theta=\ker(\pi_\Delta|_A)$ and $A^\Delta\cong D\cong C_d$ with $d\mid p-1$.
Let $A_0=\ker\theta=\ker(\pi_\Delta|_A)$.
As $A_0=\ker(\pi_\Delta|_A)$, every element of $A_0$ fixes $\Delta$ pointwise and acts evenly on the $3$ outside points.
Consequently, every nonidentity element of $A_0$ is a single $3$-cycle.
Since $A$ contains no such element, we have $A_0=1$ and $A\cong C_d$.
We choose an $A$-invariant set $X_0$ with $|X_0|=5$.
Because $|\{1,\ldots,p+3\}\setminus\Delta|=3$, we obtain from Lemma \ref{lem:alternating-affine-restriction} that $d\leq5$.
For $d=2,4$, Proposition \ref{prop:alternating-involutions} contradicts ordinary permutability of $A$.

If $d=3$, then an invariant subset of size $2$ requires an $A$-fixed outside point.
Since $A$ has order $3$, it fixes all $3$ outside points.
A generator then has exactly $4$ fixed points.
Let $B$ be generated by a $5$-cycle.
As $A$ is ordinary $G$-permutable, we choose $x\in G$ with $AB^x$ a subgroup.
The subgroup has order $15$, so it is cyclic.
Hence, a generator of $A$ centralizes $B^x$.
On the $5$ points moved by $B^x$, the centralizer of $B^x$ has order $5$.
Because the generator has order $3$, it fixes those $5$ points.
We have a contradiction.
If $d=5$, then the outside action is again trivial.
Let $B$ be generated by the product of $2$ disjoint $3$-cycles.
Because $A$ is ordinary $G$-permutable, we choose $x\in G$ with $AB^x$ a subgroup.
The subgroup has order $15$, so it is cyclic.
The centralizer of the $2$ cycles on their support of size $6$ has order $18$.
Thus, a generator of $A$ of order $5$ fixes $6$ points.
This contradicts the $4$ fixed points of a generator of $A$.
Therefore, $A=1$.
\end{proof}

\begin{lemma}\label{lem:alternating-small-products}
Let $K$ be a finite group.
If $|K|$ divides $12d\cdot13\cdot17$ with $1\leq d\leq6$, then
the Sylow $13$- and $17$-subgroups of $K$ are normal.
If $|K|$ divides $4d\cdot11\cdot13$ with $1\leq d\leq7$, then
the Sylow $11$- and $13$-subgroups of $K$ are normal.
\end{lemma}
\begin{proof}
The order formulas for finite simple groups show that the only possible
nonabelian composition factor in the first case is $A_5$.
In contrast, if $|K|\mid4d\cdot11\cdot13$ with $1\leq d\leq7$, then $K$ has no nonabelian composition factor.
We can also compare each possible simple order with the divisibility bounds using \cite{BreuerSimpleOrders}.
For this comparison, we have $12d\cdot13\cdot17\leq15912$ for $d\leq6$ and $4d\cdot11\cdot13\leq4004$ for $d\leq7$.
In the first case, $\PSL_2(16)$ requires divisibility by $2^4\cdot3\cdot5$.
Similarly, $\PSL_2(17)$ requires divisibility by $2^4\cdot3^2$.
Since neither number divides an allowed $12d$, neither group can be a composition factor.
For every other nonabelian simple group $S$ of order at most $15912$, we have $|S|\nmid12d\cdot13\cdot17$ for $1\leq d\leq6$.

Let $\ell$ be one of the primes in the lemma.
If $\ell\nmid|K|$, then its Sylow $\ell$-subgroup is trivial.
Suppose that $\ell\mid|K|$.
We induct on $|K|$ using a minimal normal subgroup $N$.
Since $\ell\mid|K|$ and $\ell^2\nmid|K|$, we have $v_\ell(|K|)=1$.
If $N$ is an $\ell$-group, then $N$ is a normal Sylow $\ell$-subgroup of $K$.
Otherwise, an elementary abelian $N$ has rank at most $4$ in characteristic $2$.
The rank is at most $2$ in characteristic $3$ and at most $1$ in every other characteristic.
Neither $\GL_4(2)$ nor $\GL_2(3)$ has order divisible by $\ell$.
For $N=C_r$ with $r$ prime, we have $\Aut(N)=\F_r^*$ and $\ell\nmid|\Aut(N)|$.
As $5^2\nmid|K|$, a nonabelian $N$ must be isomorphic to $A_5$.
The group $\Aut(N)\cong S_5$ has no element of order $\ell$.
Thus, a Sylow $\ell$-subgroup $P$ centralizes $N$.
By induction, $PN/N$ is normal in $K/N$.
The inverse image of $PN/N$ under $K\to K/N$ is $NP=N\times P$, in which $P$ is characteristic.
Hence, $P\lhd K$, as required.
\end{proof}

\begin{theorem}\label{thm:alternating-prime-plus-four-five}
Let $p\geq7$ be prime and $k\in\{4,5\}$.
Then, $A_{p+k}$ has no proper nontrivial ordinary
$A_{p+k}$-permutable subgroup.
\end{theorem}
\begin{proof}
For ease of reading, we divide the proof of the theorem into the following steps.

\setlist[itemize]{leftmargin=4em,labelsep=.5em,topsep=.2\baselineskip,partopsep=0pt,parsep=0pt,itemsep=0pt}
\begin{itemize}
\item[{\bf Step 1.} ] $A^\Delta$ is cyclic and $|A^\Delta|\leq k+2$.
\end{itemize}
For $p=7$, we use Proposition \ref{prop:finite-tables}.
For $p=11$, we use Theorems \ref{thm:alternating-prime-plus-two} and \ref{thm:alternating-prime-plus-three}.
For $n=15,16$, we take $p=13$.
Assume that $p\geq13$.
For these degrees, let $n=p+k$ and $G=A_n$.
Let $A<G$ be ordinary $G$-permutable.
By Lemma \ref{lem:alternating-prime-support}, we have $p\nmid|A|$.
The subgroup $A$ fixes a point and has invariant sets of each size at most $n/2$.
Let $C_p$ be generated by a $p$-cycle.
Since $A$ is ordinary $G$-permutable, we choose $x\in G$ with $K=AC_p^x$ a subgroup.
The support $\Delta$ of $C_p^x$ is a $K$-orbit of size $p$,
and $A^\Delta$ has index $p$ in $K^\Delta$.
As in the proof of Theorem \ref{thm:alternating-prime-plus-three}, we use \cite[Theorem 1.2(1)]{JonesCycles} and \cite[Theorem 1.1]{JonesSezer}.
Each alternating, Mathieu or projective point stabilizer has orbit lengths $1$ and $p-1$.
With only $k$ outside points, $A$ cannot have an invariant subset of size $k+2$.
Let $q$ be a prime power and $b\geq3$.
A hyperplane stabilizer in $\PSL_b(q)$ has orbit lengths $h=(q^{b-1}-1)/(q-1)$ and $q^{b-1}$.
When $h\geq k+2$, neither orbit can meet an $A$-invariant subset of size $k+1$.
The $k$ outside points cannot form such a subset.
The smaller cases are $p=13$, and also $p=31$ when $k=5$.
For $p=13$ and $n=17$, we use \cite{YG}.
For $n=18$, we use Theorem \ref{thm:alternating-prime-plus-one}.
For $(p,k)=(31,5)$, the projective orbits have sizes $6$ and $25$.
With only $5$ outside points, $A$ cannot have an invariant subset of size $12$.
Thus, $K^\Delta$ is affine.
As $p\nmid|A|$, we obtain from Lemma \ref{lem:alternating-affine-restriction} a bijection $\xi:\Delta\to\F_p$ for the action of $A$ on $\Delta$.
Let $\theta=\mu_p\circ c_\xi\circ(\pi_\Delta|_A):A\to\F_p^*$, $D=\theta(A)$ and $d=|D|$.
By Lemma \ref{lem:alternating-affine-restriction}, we have $\ker\theta=\ker(\pi_\Delta|_A)$ and $A^\Delta\cong D\cong C_d$ with $d\mid p-1$.
Let $A_0=\ker\theta=\ker(\pi_\Delta|_A)$.
As $A_0=\ker(\pi_\Delta|_A)$, every element of $A_0$ fixes $\Delta$ pointwise and acts evenly on the $k$ outside points.
Thus, $A_0$ is isomorphic to a subgroup of $A_k$.
We choose an $A$-invariant set $X_0$ with $|X_0|=k+2$.
Because $|\{1,\ldots,n\}\setminus\Delta|=k$, we obtain from Lemma \ref{lem:alternating-affine-restriction} that $d\leq k+2$.
\begin{itemize}
\item[{\bf Step 2.} ] The map $\pi_\Delta|_A:A\to A^\Delta$ is injective.
\end{itemize}
For $k=4$, we have $|A|\mid12d$.
A nontrivial $A_0$ contains a single $3$-cycle or an involution $u$ with $|\operatorname{supp}(u)|=4$.
For $k=5$, we choose $y\in G$ with $A\leq((S_2\wr S_{n/2})\cap G)^y$.
An element of order $3$ or $5$ moving the blocks of size $2$ moves at least $6$ or $10$ points.
Thus, $A$ contains no single $3$-cycle or $5$-cycle.
As $A_0$ acts evenly on $5$ points, $A_0$ is a $2$-group of order at most $4$.
Therefore, $|A|\mid4d$.
A nontrivial $A_0$ contains an involution $u$ with $|\operatorname{supp}(u)|=4$.

For $k=4$, let $n=13a+17b+j$ with $a,b\geq0$ and $0\leq j\leq2$.
For $k=5$, let $n=11a+13b+j$ with $a,b\geq0$ and $0\leq j\leq3$.
For the first prime length, we choose a product of $a$ disjoint cycles.
For the second prime length, we choose a product of $b$ disjoint cycles on the other points.
Let $B$ be generated by these $2$ permutations.
We let $B$ fix the $j$ unused points and omit a generator when its coefficient is zero.
Since $A$ is ordinary $G$-permutable, we choose $x\in G$ with $AB^x$ a subgroup.
By Lemma \ref{lem:alternating-small-products}, every Sylow subgroup of $B^x$ is normal in $AB^x$.
Suppose that $A_0\ne1$.
We choose $u\in A_0$ with $|u|\in\{2,3\}$ and $|\operatorname{supp}(u)|\in\{3,4\}$.
Since every Sylow subgroup of $B^x$ is normal in $AB^x$, $u$ normalizes $B^x$.
As $|\operatorname{supp}(u)|\leq4$, the element $u$ maps each orbit of prime length of $B^x$ to itself.
A nontrivial affine automorphism on an orbit moves at least $10$ points, so $u$ induces none.
For each orbit of prime length $\Gamma$ of $B^x$, we have $u|_\Gamma\in C_{S_{\Gamma}}((B^x)^\Gamma)$.
The centralizer $C_{S_{\Gamma}}((B^x)^\Gamma)$ has order $11,13$ or $17$, so $u|_\Gamma=1$.
Thus, $\operatorname{supp}(u)$ lies in the $j$ fixed points of $B^x$.
This contradicts $|\operatorname{supp}(u)|>j$.
Consequently, $A_0=1$.

The first degree representation holds for every $n\geq90$. For every residue modulo $13$, we have a representative $17b+j\leq102$.
We may take $0\leq b\leq6$ and $0\leq j\leq2$.
For the interval $90\leq n\leq101$, we verify the integer representations directly.
Among smaller degrees $p+4$ with $p$ prime, only $n\in\{21,23,33,63\}$ lack the representation $n=13a+17b+j$.
For the prime degree $23$, we use \cite{YG}.
For the other degrees, we use Theorem \ref{thm:alternating-prime-plus-two} with $21=19+2$, $33=31+2$ and $63=61+2$.
For the second representation, we use the residues $13b+j$ modulo $11$.
For every $n\geq55$, we have $n=11a+13b+j$ with $a\geq0$, $0\leq b\leq4$ and $0\leq j\leq3$.
The interval $44\leq n\leq54$ is checked directly.
Below $44$, only degree $18$ lacks the required representation
among the degrees $p+5$ with $p$ prime under consideration.
For $n=18$, we use Theorem \ref{thm:alternating-prime-plus-one}.

\begin{itemize}
\item[{\bf Step 3.} ] We have $|A|\notin\{2,3,4,6\}$.
\end{itemize}
Now, $A$ is cyclic of order $d\leq k+2\leq7$.
For $d=2,4$, Proposition \ref{prop:alternating-involutions} contradicts ordinary permutability of $A$.
For $d=3,6$, let $B$ be generated by the product of $2$ disjoint $5$-cycles.
Since $A$ is ordinary $G$-permutable, we choose $x\in G$ with $AB^x$ a subgroup.
If $|AB^x|=15$, then the Sylow $5$-subgroup is normal.
For $|AB^x|=30$, Sylow's theorem allows $1$ or $6$ Sylow $5$-subgroups.
Suppose that $|AB^x|=30$ and $AB^x$ has $6$ Sylow $5$-subgroups.
Their nonidentity elements number $24$, leaving exactly the $6$ elements of $A$.
Let $\Sigma$ be the set of the $24$ nonidentity elements in Sylow $5$-subgroups of $AB^x$.
As $g\Sigma g^{-1}=\Sigma$ for every $g\in AB^x$ and $A=AB^x\setminus\Sigma$, we have $A\unlhd AB^x$.
Let $P$ be a Sylow $5$-subgroup of $AB^x$.
Because $A\unlhd AB^x$ and $5\nmid|\Aut(A)|$, we have $[P,A]=1$.
As $AB^x=AP$, it follows that $P\unlhd AB^x$.
This contradicts the $6$ Sylow $5$-subgroups.
Therefore, $B^x\unlhd AB^x$.
We choose $c\in A$ with $|c|=3$.
As $B^x\unlhd AB^x$ and $3\nmid|\Aut(B^x)|=4$, we have $c\in C_G(B^x)$.
Let $B^x=\gen b$ and let $\Sigma=\operatorname{supp}(b)$.
Because $cb=bc$, we have $c\Sigma=\Sigma$.
We have $|\Sigma|=10$ and $|C_{S_{\Sigma}}(b|_\Sigma)|=50$.
As $3\nmid50$, we have $c|_\Sigma=\operatorname{id}_\Sigma$.
Because $c$ has at most $k+1\leq6$ fixed points, we obtain a contradiction.
\begin{itemize}
\item[{\bf Step 4.} ] We have $A=1$.
\end{itemize}
For $d=5$ and $k=4$, an element of order $5$ fixes the $4$ outside points.
For $k=5$, an $A$-invariant subset of size $2$ cannot meet an outside orbit of length $5$.
As $\Delta$ has only $1$ $A$-fixed point, such an outside orbit would leave no invariant subset of size $2$.
Thus, $A$ fixes all points outside $\Delta$.
Let $B$ be generated by the product of $3$ disjoint $3$-cycles.
Because $A$ is ordinary $G$-permutable, we choose $x\in G$ with $AB^x$ a subgroup.
The product has order $15$, so it is cyclic.
The centralizer of the $3$ cycles on their support of size $9$ has order $162$.
Thus, a generator of $A$ of order $5$ fixes $9$ points.
This contradicts the bound of at most $k+1\leq6$ fixed points for a generator of $A$.

Finally, $d=7$ requires $k=5$ and $p\geq29$.
The action has at least $4$ orbits of length $7$ and $6$ fixed points.
Let $B$ be generated by the product of $3$ disjoint $3$-cycles.
Since $A$ is ordinary $G$-permutable, we choose $x\in G$ with $AB^x$ a subgroup.
The product has order $21$, so $A$ is its normal Sylow $7$-subgroup.
Let $b$ generate $B^x$, so $|b|=3$ and $|\operatorname{supp}(b)|=9$.
Moving $3$ $A$-orbits of length $7$ would move $21$ points.
Thus, $b$ maps each $A$-orbit of length $7$ to itself.
If $bab^{-1}\ne a$ for a generator $a$ of $A$, then $b$ moves $6$ points on each such orbit.
As there are at least $4$ such orbits, $b$ would move at least $24$ points.
Therefore, $b$ centralizes $A$.
On each orbit of length $7$, $b$ lies in its cyclic centralizer of order $7$.
Because $|b|=3$, $b$ fixes each such orbit pointwise.
Then, $\operatorname{supp}(b)$ lies in the $6$ $A$-fixed points, contradicting its size $9$.
Thus, $d=1$ and $A=1$.
\end{proof}

\begin{lemma}\label{lem:alternating-finite-products}
Let $J$ be a finite group.
The Sylow $\ell$-subgroups of $J$ are normal if
$|J|\mid2520d\ell$, where $1\leq d\leq9$ and $\ell\in\{47,59\}$.
The Sylow $23$-subgroups of $J$ are normal if $|J|\mid20160d\cdot23\cdot29$,
where $1\leq d\leq10$.
Suppose that $1\leq d\leq10$ and $C,B_{23},B_{29}\leq J$.
Assume that $J=C(B_{23}\times B_{29})$ and $|C|\mid20160d$.
Assume also that $|B_{23}|=23$ and $|B_{29}|=29$.
Then, $B_{29}$ is normal in $J$.
\end{lemma}
\begin{proof}
When $|J|\mid2520d\ell$, the possible nonabelian simple composition factors are
$A_5,A_6,A_7,A_8,\allowbreak\PSL_2(7),\allowbreak\PSL_2(8),\allowbreak\PSL_3(4)$.
Their automorphism orders are
$120,\allowbreak1440,\allowbreak5040,\allowbreak40320,\allowbreak336,\allowbreak1512,\allowbreak241920$.
None is divisible by $47$ or $59$, and
$|S|^2\nmid2520d\ell$ for every displayed simple group $S$.
When $|J|\mid20160d\cdot23\cdot29$, the further possibilities are $A_9,\PSU_3(3),\PSp_4(3),\PSL_2(29)$.
Their automorphism orders in this order are $362880,12096,51840,24360$.
From the order formulas for finite simple groups, we obtain the displayed lists of $S$ by divisibility.
We can also use the list in \cite{BreuerSimpleOrders}.
In particular, we have the bounds $1338120$ when $|J|\mid2520d\ell$ and $134467200$ when $|J|\mid20160d\cdot23\cdot29$.
In the second case, a minimal normal subgroup with several nonabelian simple factors is $A_5^2$ or $\PSL_2(7)^2$.
We have $|\Aut(A_5^2)|=2\cdot120^2$ and $|\Aut(\PSL_2(7)^2)|=2\cdot336^2$.
For every displayed simple group $S$, we have
$|S|^3\nmid20160d\cdot23\cdot29$.

In the first case, an elementary abelian minimal normal subgroup $C_r^a$ satisfies $a\leq6$ for $r=2$.
In addition, we have $a\leq4$ for $r=3$ and $a\leq2$ for $r\in\{5,7\}$.
For the second case, only the bound at $r=2$ changes, to $a\leq9$.
For the displayed bounds on $a$, the orders $|\GL_a(r)|$ are prime to $47,59$
in the first case and to $23,29$ in the second.
For other primes $r$, the rank is at most $1$.
For the primes $\ell$ in the lemma, we have $v_\ell(|J|)\leq1$.
We also have $23\nmid28$ and $29\nmid22$.
A Sylow $\ell$-subgroup centralizes every minimal
normal subgroup of order prime to $\ell$.
The only possible nonabelian minimal normal subgroup whose order
is divisible by a prime $\ell$ in the lemma is $\PSL_2(29)$.

We induct on $|J|$ for $\ell\in\{47,59\}$.
For $\ell=23$ with $|J|\mid20160d\cdot23\cdot29$, we also induct on $|J|$.
If $\ell\nmid|J|$, then its Sylow $\ell$-subgroup is trivial.
Suppose that $\ell\mid|J|$, so $v_\ell(|J|)=1$.
We choose a minimal normal subgroup $N$.
If $N$ is an $\ell$-group, then $N$ is a normal Sylow $\ell$-subgroup of $J$.
Otherwise, a Sylow $\ell$-subgroup $P$ centralizes $N$.
By induction, we have $PN/N\unlhd J/N$.
The inverse image of $PN/N$ under $J\to J/N$ is $NP=N\times P$, in which $P$ is characteristic.
Hence, $P\lhd J$.

Suppose that $J=CB_{29}$ with $|C|\mid20160d$.
We choose a minimal normal subgroup $N$ of $J$.
If $N$ is a $29$-group, then $N=B_{29}$ and $B_{29}\unlhd J$.
Suppose that $29\nmid|N|$.
Then, $B_{29}$ centralizes $N$.
We have $J/N=(CN/N)(B_{29}N/N)$ and $|CN/N|\mid20160d$.
By induction, $B_{29}N/N\unlhd J/N$.
The inverse image of $B_{29}N/N$ under $J\to J/N$ is $N\times B_{29}$. Its Sylow $29$-subgroup $B_{29}$ is characteristic.
Therefore, $B_{29}\unlhd J$.
The only other possibility is $N\cong\PSL_2(29)$.
Since $v_{29}(|J|)=1$, we have $B_{29}\leq N$.
Thus, $J=CN$ and $[N:N\cap C]=[J:C]=29$.
By Dickson's classification \cite[Tables 8.1--8.2]{BHR}, every proper subgroup of $\PSL_2(29)$ has order at most $\max\{29\cdot28/2,30,60\}=406$.
A subgroup of index $29$ would have order $(29^2-1)/2=420>406$, a contradiction.
Thus, $B_{29}$ is normal in every such product.

In the mixed product, $B_{23}\lhd J$.
Let $\rho:J\to J/B_{23}$ be the quotient map.
By the preceding argument, we have $\rho(B_{29})\lhd J/B_{23}$.
The inverse image $\rho^{-1}(\rho(B_{29}))$ is $B_{23}\times B_{29}$ and is normal in $J$.
Since $B_{29}$ is characteristic in $B_{23}\times B_{29}$, we have $B_{29}\lhd J$.
\end{proof}

\begin{proposition}\label{prop:alternating-five-degrees}
For $n\in\{95,96,119,120,121\}$, the group $A_n$ has no proper
nontrivial ordinary $A_n$-permutable subgroup.
\end{proposition}
\begin{proof}
Let $p=89$ for $n=95,96$, and $p=113$ otherwise.
Let $n=p+k$, and let $A<G=A_n$ be ordinary $G$-permutable.
By Lemma \ref{lem:alternating-prime-support}, we have $p\nmid|A|$.
We also have $A$-invariant sets of each size at most $n/2$.
Let $C_p$ be generated by a $p$-cycle.
Since $A$ is ordinary $G$-permutable, we choose $x\in G$ with $K=AC_p^x$ a subgroup.
By Lemma \ref{lem:alternating-prime-support}, $\Delta=\operatorname{supp}(C_p^x)$ is a $K$-orbit of size $p$.
The primitive group $K^\Delta$ has the subgroup $A^\Delta$ of index $p$.
By \cite[Theorem 1.2(1)]{JonesCycles}, $K^\Delta$ is affine or $L=\Soc(K^\Delta)$ is nonabelian simple.
If $K^\Delta$ is not affine, then $L\cong A_p$ or $L$ is projective or Mathieu.
The Mathieu degrees are not $89$ or $113$.
For $L\cong A_p$, we obtain from \cite[Theorem 1.1]{JonesSezer} that $L\cap A^\Delta$ is a natural point stabilizer.
Its orbit of length $p-1$ contradicts the $A$-orbit bound.

Neither $89$ nor $113$ is a projective degree $(q^a-1)/(q-1)$ with $a\geq2$.
Since $p$ is prime, $a$ is prime.
From $2^a-1\leq113$, we have $a\in\{2,3,5\}$.
For $a=2$, we would need $q\in\{88,112\}$.
Neither number is a prime power.
For $a=3$, $q(q+1)$ cannot be $88$ or $112$.
For $a=5$ and $q=2$, the degree is $31$.
For $a=5$ and $q\geq3$, the degree is at least $121$.
Thus, $K^\Delta$ is affine.
As $p\nmid|A|$, we obtain from Lemma \ref{lem:alternating-affine-restriction} a bijection $\xi:\Delta\to\F_p$ for the action of $A$ on $\Delta$.
Let $\theta=\mu_p\circ c_\xi\circ(\pi_\Delta|_A):A\to\F_p^*$, $D=\theta(A)$ and $d=|D|$.
We have $\ker\theta=\ker(\pi_\Delta|_A)$ and $A^\Delta\cong D\cong C_d$ with $d\mid p-1$.
Next, we choose an $A$-invariant set $X_0$ with $|X_0|=k+2$.
Because $|\{1,\ldots,n\}\setminus\Delta|=k$, we obtain from Lemma \ref{lem:alternating-affine-restriction} that $d\leq k+2$.
Let $A_0=\ker\theta=\ker(\pi_\Delta|_A)$.
Every element of $A_0$ fixes $\Delta$ pointwise and acts evenly on the $k$ outside points.
Thus, $A_0$ is isomorphic to a subgroup of $A_k$ and $|A|\mid(k!/2)d$.
For $n\in\{95,96\}$, let $B$ be generated by the product of $2$ disjoint $47$-cycles.
For $n\in\{119,120\}$, let $B$ be generated by the product of $2$ disjoint $59$-cycles.
For $n=121=4\cdot23+29$, we take $B=C_{23}\times C_{29}$.
We generate the first factor by a product of $4$ disjoint $23$-cycles.
The second factor acts as a $29$-cycle on the other points.
As $A$ is ordinary $G$-permutable, we choose $x\in G$ with $AB^x$ a subgroup.
Each subgroup $B$ has at most $2$ fixed points.
For $k=6$, we have $|A|\mid360d\mid2520d$ and $d\leq8$.
For $k=7$, we have $|A|\mid2520d$.
For $k=8$, we have $|A|\mid20160d$.
By Lemma \ref{lem:alternating-finite-products}, we conclude that every Sylow subgroup of $B^x$ is normal in $AB^x$.

Suppose that $A_0\ne1$.
We choose $a\in A_0$ of prime order $r$.
Then, $r\leq7$ and $|\operatorname{supp}(a)|\leq k\leq8$.
Because every Sylow subgroup of $B^x$ is normal in $AB^x$, $a$ normalizes $B^x$.
Moving the prime-length orbits would move at least twice their length.
A nontrivial affine automorphism on such an orbit moves at least $1$ less than its length.
Since every such length is at least $23$, $a$ maps each orbit to itself and induces no nontrivial affine automorphism.
On each orbit of prime length $\Gamma$ of $B^x$,
we have $a|_\Gamma\in C_{S_{\Gamma}}((B^x)^\Gamma)$.
As $C_{S_{\Gamma}}((B^x)^\Gamma)$ has prime order greater than $7$, $a|_\Gamma=1$.
Thus, $|\operatorname{supp}(a)|\leq2$.
Because $a$ is even, we have $a=1$.
Therefore, $A_0=1$ and $A\cong C_d$.

For $p=89$, the divisors of $88$ with $d\leq9$ are $1,2,4,8$.
For $p=113$, the divisors of $112$ with $d\leq10$ are $1,2,4,7,8$.
Proposition \ref{prop:alternating-involutions} contradicts ordinary permutability of every nontrivial cyclic $2$-subgroup.
Now, let $A\cong C_7$ and $p=113$.
Let $B$ be generated by the product of $4$ disjoint $3$-cycles.
Since $A$ is ordinary $G$-permutable, we choose $x\in G$ with $K=AB^x$ a subgroup.
As $|K|=21$, we have $A\lhd K$. Let $b$ generate $B^x$.
Permuting $3$ $A$-orbits of length $7$ would move $21$ points.
Because $|\operatorname{supp}(b)|=12$, $b$ cannot permute those $3$ orbits.
If $bab^{-1}\ne a$ for a generator $a$ of $A$, then $b$ moves $6$ points on each of the $16$ orbits of length $7$ in $\Delta$.
This would move $96$ points.
As $|\operatorname{supp}(b)|=12$, $b$ centralizes $A$.
For each $A$-orbit $\Gamma$ of length $7$ with $b\Gamma=\Gamma$,
we have $b|_\Gamma\in C_{S_{\Gamma}}(A^\Gamma)\cong C_7$.
Because $b$ has order $3$, $b|_\Gamma=1$.
Thus, $|\operatorname{supp}(b)|\leq k+1\leq9$.
This contradicts $|\operatorname{supp}(b)|=12$.
Thus, $A=1$.
\end{proof}

\begin{corollary}\label{cor:alternating-near-prime}
For every $7\leq n\leq121$, the group $A_n$ has no proper
nontrivial ordinary $A_n$-permutable subgroup.
\end{corollary}
\begin{proof}
Firstly, for $7\leq n\leq13$, we use Proposition \ref{prop:finite-tables}.
For prime $n$, we use \cite{YG}.
Now, we consider composite $14\leq n\leq121$ with $n\notin\{95,96,119,120,121\}$.
Then, we have $n=p+k$ with $p$ prime and $1\leq k\leq5$.
For $k=1,2$, we use Theorems \ref{thm:alternating-prime-plus-one} and \ref{thm:alternating-prime-plus-two}.
For $k=3,4,5$, we use Theorems \ref{thm:alternating-prime-plus-three} and \ref{thm:alternating-prime-plus-four-five}.
Finally, for $n\in\{95,96,119,120,121\}$, we use Proposition \ref{prop:alternating-five-degrees}.
\end{proof}

\begin{lemma}\label{lem:small-support-stabilizer}
Let $K\leq S_{\Omega}$ be transitive with $|\Omega|=v$.
Let $\alpha\in\Omega$ and $J=K_\alpha$.
Suppose that the elements of $J$ moving at most $s$ points form
a subgroup $J_0$.
Then, $\Fix(J_0)$ is a block.
If $s<v/2$ and the support of $J_0$ has size at most $s$,
then $J_0=1$.
\end{lemma}
\begin{proof}
Let $g\in K$ with $J_0\leq J^g$.
Since every element of $J_0$ moves at most $s$ points, we have $J_0\leq J_0^g$.
As $|J_0|=|J_0^g|$, we have $J_0=J_0^g$.
For $g\in K$, the coset $Jg$ corresponds to the point $g^{-1}(\alpha)$.
That point is fixed by $J_0$ exactly when $J_0\leq J^g$.
By $J_0\leq J^g\Longrightarrow J_0=J_0^g$, those cosets are represented by $N_K(J_0)$.
Because conjugation by elements of $J$ does not change the number of moved points, we have $J_0\unlhd J$.
Thus, $J\leq N_K(J_0)$ and the fixed cosets form the block corresponding to $N_K(J_0)$.
A proper block has at most $v/2$ points.
As $|\operatorname{supp}(J_0)|\leq s<v/2$, we have $|\Fix(J_0)|>v/2$.
Hence, $\Fix(J_0)$ is the whole point set and $J_0=1$.
\end{proof}

\begin{theorem}\label{thm:alternating-cyclic-reduction}
Let $n\geq122$ and let $G=A_n$.
Let $p$ be the largest prime below $n$.
Then, $5n/6<p<n$. Let $k=n-p$.
Every proper ordinary $G$-permutable subgroup $A$ is cyclic of order $d\leq k+2$.
There exist an $A$-invariant set $\Delta$ with $|\Delta|=p$ and a bijection $\xi:\Delta\to\F_p$.
Let $D=\mu_p(c_\xi(A^\Delta))$.
We have $|D|=d$ and $c_\xi(A^\Delta)=\{u\mapsto cu\mid c\in D\}$.
For $d=1$, every point of $\Delta$ is fixed by $A$.
For $d>1$, $\xi^{-1}(0)$ is the unique $A$-fixed point in $\Delta$.
For $u\in\F_p^*$, the $A$-orbit of $\xi^{-1}(u)$ is $\xi^{-1}(uD)$ and has size $d$.
\end{theorem}
\begin{proof}
Let $A<G$ be ordinary $G$-permutable.
Let $M_9=(S_9\times S_{n-9})\cap G$, a maximal subgroup.
By \cite[Theorem D]{LPS}, we have $M_9Y\ne G$ for every $Y<G$.
By Lemma \ref{lem:max}, we have $A\leq M_9^z$ for some $z\in G$.
Thus, $AY\ne G$ for every $Y<G$.
For each maximal subgroup $M<G$, there exists $y\in G$ with $A\leq M^y$.
The subgroup $A$ fixes a point and has an invariant subset of every size below $n/2$.
For even $n$, we choose $y\in G$ such that $A\leq((S_{n/2}\wr C_2)\cap G)^y$.
Let $\Omega_1,\Omega_2$ be the corresponding blocks, with $|\Omega_1|=|\Omega_2|=n/2$.
We label the blocks so that an $A$-fixed point $\alpha$ lies in $\Omega_1$.
Since $g\alpha=\alpha$ for every $g\in A$, we have $g\Omega_1=\Omega_1$ and $g\Omega_2=\Omega_2$.
Thus, $A$ has an invariant subset of size $n/2$.
Every $A$-orbit has size at most $\lceil n/2\rceil$.

By Nagura's theorem \cite{Nagura1952}, we have $5n/6<p<n$.
Then, $k<n/6$, $p>5k$ and $p\geq103$.
The invariant $(k+1)$-set shows that $p\nmid|A|$.
Let $C_p$ be generated by a $p$-cycle.
Since $A$ is ordinary $G$-permutable, we choose $x\in G$ with $K=AC_p^x$ a subgroup.
As $AY\ne G$ for every $Y<G$, we have $K<G$.
Let $\Delta$ be the $K$-orbit containing the support of $C_p^x$.
If $|\Delta|>p$, then choose a $C_p^x$-fixed point $\alpha\in\Delta$.
Because $K=AC_p^x$ and $C_p^x$ fixes $\alpha$, we have $\Delta=\{a(\alpha):a\in A\}$.
But $|\Delta|>p>\lceil n/2\rceil$, contradicting the $A$-orbit bound.
Hence, $|\Delta|=p$.

The induced group $K^\Delta$ contains a $p$-cycle.
We have $p\nmid|A^\Delta|$ and $[K^\Delta:A^\Delta]=p$.
By \cite[Theorem 1.2]{JonesCycles}, $K^\Delta$ is affine or $L=\Soc(K^\Delta)$ is nonabelian simple.
If $K^\Delta$ is not affine, then $L\cong A_p$ or $L$ is projective or Mathieu.
The Mathieu degrees are below $103$.
Suppose that $L=\Soc(K^\Delta)$ is nonabelian simple, with $p\mid|L|$.
Since $p\nmid|A^\Delta|$, we have $L\nleq A^\Delta$.
The $L$-orbits on the coset space $K^\Delta/A^\Delta$ have equal length dividing $p$.
Hence, $L$ is transitive on $K^\Delta/A^\Delta$ and $[L:L\cap A^\Delta]=p$.
By the representations on $p$ cosets in \cite[Theorem 1.1]{JonesSezer},
$L\cap A^\Delta$ is an alternating point stabilizer or a projective point or hyperplane stabilizer.
The alternating and projective point cases have an orbit of
length $p-1>\lceil n/2\rceil$.
For a projective hyperplane, let
$p=1+q+\cdots+q^{a-1}, h=1+q+\cdots+q^{a-2}, Q=q^{a-1}$.
The hyperplane stabilizer has $2$ orbits of lengths $h,Q$.
When $q\geq3$, we have $Q>2p/3>\lceil n/2\rceil$.
When $q=2$, both lengths $(p-1)/2,(p+1)/2$ exceed $k+1$.
An invariant $(k+1)$-set could then neither meet $\Delta$
nor fit in the complement of $\Delta$, which has size $k$.
Thus, $K^\Delta$ is affine.
As $p\nmid|A|$, we obtain from Lemma \ref{lem:alternating-affine-restriction} a bijection $\xi:\Delta\to\F_p$ for the action of $A$ on $\Delta$.
Let $\theta=\mu_p\circ c_\xi\circ(\pi_\Delta|_A):A\to\F_p^*$, $D=\theta(A)$ and $d=|D|$.
We have $\ker\theta=\ker(\pi_\Delta|_A)$ and $A^\Delta\cong D\cong C_d$ with $d\mid p-1$.
Next, we choose an $A$-invariant set $X_0$ with $|X_0|=k+2$.
Because $|\{1,\ldots,n\}\setminus\Delta|=k$, we obtain from Lemma \ref{lem:alternating-affine-restriction} that $d\leq k+2$.
Let $A_0=\ker\theta=\ker(\pi_\Delta|_A)$.
Every element of $A_0$ moves at most $k$ points.
By Lemma \ref{lem:alternating-affine-restriction}, every element of $A\setminus A_0$ moves $p-1>k$ points in $\Delta$.
Thus, $A_0=\{a\in A:|\operatorname{supp}(a)|\leq k\}$.

For odd $n$, we choose a subgroup $B\cong C_n$ acting regularly on $\{1,\ldots,n\}$.
For $4\mid n$, we choose a subgroup $B\cong C_{n/2}\times C_2$ acting regularly on this set.
For odd $n$, every element of $B$ has odd order.
For $4\mid n$, the order of each element $b\in B$ divides $n/2$.
Then, $b$ has $n/|b|$ cycles of length $|b|$, with $n/|b|$ even.
Thus, $B\leq A_n$ in both cases.
Since $A$ is ordinary $G$-permutable, we choose $x\in G$ with $H=AB^x$ a subgroup.
Let $\alpha$ be an $A$-fixed point.
As $B^x$ acts regularly, $A\cap B^x=1$ and $H$ is transitive.
Therefore, $[H:A]=n$ and $H_\alpha=A$.
By $A_0=\{a\in A:|\operatorname{supp}(a)|\leq k\}$ and $|\operatorname{supp}(A_0)|\leq k<n/2$, we obtain from Lemma \ref{lem:small-support-stabilizer} that $A_0=1$.

For $n=2m$ with $m$ odd, let $B\cong C_m$ be generated by a product of $2$ disjoint $m$-cycles.
Since $A$ is ordinary $G$-permutable, we choose $x\in G$ with $H=AB^x$ a subgroup.
Let $\alpha$ be an $A$-fixed point.
As every nonidentity element of $B^x$ moves every point, $A\cap B^x=1$.
Let $X=\{h(\alpha):h\in H\}$.
Because $A\leq H_\alpha$ and $B^x$ has an orbit of size $m$ at $\alpha$,
we have $m\leq|X|\leq[H:A]=m$.
Thus, $|X|=m$ and $H_\alpha=A$.
The complement of $X$ is $H$-invariant and is the other $B^x$-orbit.
Hence, the complement is the second $H$-orbit, also of size $m$.
As $|X\cap\Delta|\geq m-k$, every $a\in A\setminus A_0$ moves at least $m-k-1>k$ points in $X$.
Because elements of $A\setminus A_0$ move more than $k$ points in $X$, the subgroup $\pi_X(A_0)\leq A^X$ consists exactly
of the elements moving at most $k$ points in $X$.
The support of $\pi_X(A_0)$ has size at most $k<m/2$.
Since $H_\alpha=A$, the point stabilizer of the transitive group $H^X$ at $\alpha$ is $A^X$.
By Lemma \ref{lem:small-support-stabilizer}, we have $\pi_X(A_0)=1$.
The kernel $E=\ker(\pi_X:H\to S_{X})$ lies in $A$.
By the movement bound, we have $E\leq A_0$.
As $\pi_X(A_0)=1$, we have $A_0\leq E$.
Hence, $A_0=E\lhd H$.
Because $A_0\lhd H$, its support is a union of $H$-orbits.
Both $H$-orbits have size $m$.
As $|\operatorname{supp}(A_0)|\leq k<m$, we have $A_0=1$ when $n\equiv2\pmod4$.
Because $A_0=\ker\theta=1$ for every $n\pmod4$, the homomorphism $\theta:A\to\F_p^*$ is injective.
\end{proof}

\subsection{Affine permutation groups in the alternating group}

Let $q=p^f$ and $W=\F_q^c$, where $p$ is prime and $c,f\geq1$.
For $\sigma\in\operatorname{Gal}(\F_q/\F_p)$, let $u^\sigma=(\sigma(u_1),\ldots,\sigma(u_c))$ for $u=(u_1,\ldots,u_c)\in W$.
We use $\operatorname{AGL}_c(q)=\{u\mapsto Mu+b:M\in\GL_c(q),\ b\in W\}$ and
$\operatorname{A\Gamma L}_c(q)=\{u\mapsto Mu^\sigma+b:M\in\GL_c(q),\ b\in W,\ \sigma\in\operatorname{Gal}(\F_q/\F_p)\}$.

\begin{lemma}\label{lem:alternating-affine-subgroups}
Let $n\geq122$ and let $p<n$ be the largest prime.
Let $k=n-p$.
Let $A<A_n$ be ordinary $A_n$-permutable, with order $d$.
For every prime $r$ with $\max\{k+2,19\}<r<n$, we have $d\mid r-1$.
Let $q=2^e$ satisfy $32\leq q\leq n$ and $q>\max\{2d+2,k+1\}$.
Then, there is an injective homomorphism $\rho_q:A\to\operatorname{A\Gamma L}_1(q)$.
\end{lemma}
\begin{proof}
By Theorem \ref{thm:alternating-cyclic-reduction}, we have $A\cong C_d$ and $d\leq k+2$.
Every nonidentity element of $A$ fixes at most $k+1$ natural points.
For the prime degree $r$, take the even affine group
$B=C_r\rtimes C_{(r-1)/2}$ on $r$ points.
Since $A$ is ordinary $A_n$-permutable, we choose $x\in A_n$ with $K=AB^x$ a subgroup.
The $K$-orbit containing $\operatorname{supp}(B^x)$ cannot contain another point $\alpha$.
As $B^x$ fixes $\alpha$ and $K=AB^x$, we have $\{g(\alpha):g\in K\}=\{a(\alpha):a\in A\}$.
But the $K$-orbit has size greater than $r>d=|A|$.
Hence, the support is a $K$-orbit $R$ of size $r$.
We choose a bijection $\eta_r:\{1,\ldots,r\}\to R$.
Let $\pi_R:K\to S_r$ satisfy $\pi_R(g)=\eta_r^{-1}\circ(g|_R)\circ\eta_r$ for $g\in K$.
Because $r>k+1$, the restriction $\pi_R|_A:A\to S_r$ is injective.
Let $J=\pi_R(K)$, $A^R=\pi_R(A)$ and $F=\pi_R(B^x)$.
As $K=AB^x$, we have $J=A^R F$.
Both $A^R$ and $F$ are soluble.
Because $F$ contains an $r$-cycle and $r$ is prime, $J$ is primitive.

Suppose that no bijection $\xi:\{1,\ldots,r\}\to\F_r$ satisfies $\xi J\xi^{-1}\leq\operatorname{AGL}_1(r)$.
By \cite[Theorem 1.2(1)]{JonesCycles}, we have a nonabelian simple group $L=\Soc(J)$.
Since $J$ is not soluble, we have $A^R,F<J$ and $J=A^R F$.
By \cite[Proposition 4.1]{LiXia}, we have one of the $2$ possibilities.
For $r=Q+1$ with $Q=2^a$, we have
$$
 \PSL_2(Q)\leq J\leq\operatorname{P\Gamma L}_2(Q).
$$
For $r=73$, we have
$$
 \PSL_3(8)\leq J\leq\operatorname{P\Gamma L}_3(8).
$$
For $L=A_r$ or $L=M_{23}$, we have $J\ne XY$ for all soluble $X,Y<J$.
The other simple socles have no eligible prime degree $r\geq23$.
Let $P$ be the Sylow $r$-subgroup of $F$.
Since $F\cong C_r\rtimes C_{(r-1)/2}$, we have $P\unlhd F$ and $|P|=r$.
For $r=Q+1$ with $Q=2^a$, the order $|N_J(P)/P|$ divides $2a$.
As $Q\geq32$, we have $|F/P|=Q/2>2a$.
This contradicts $F\leq N_J(P)$.
For $L=\PSL_3(8)$ and $r=73$, the order $|N_J(P)/P|$ divides $9$.
Because $|F/P|=36$ and $F\leq N_J(P)$, we obtain a contradiction.
Thus, we choose a bijection $\xi:\{1,\ldots,r\}\to\F_r$ such that $\xi J\xi^{-1}\leq\operatorname{AGL}_1(r)$.
The faithful subgroup $\xi A^R\xi^{-1}$ has order $d<r$.
As $|\operatorname{AGL}_1(r)|=r(r-1)$ and $r\nmid d$, we have $d\mid r-1$.

For $q=2^e$, choose $\Omega\subseteq\{1,\ldots,n\}$ with $|\Omega|=q$ and a bijection $\beta:\Omega\to\F_q$.
Let $B$ consist of the permutations $b\in S_n$ fixing $\{1,\ldots,n\}\setminus\Omega$ pointwise and satisfying
$\beta\circ(b|_\Omega)\circ\beta^{-1}\in\operatorname{AGL}_1(q)$.
For $b\in\F_q^*$, the map $u\mapsto u+b$ on $\F_q$ has $q/2$ transpositions.
Since $q\geq32$, each translation in $B$ is even.
As $\{u\mapsto au:a\in\F_q^*\}$ has odd order $q-1$, we have $B\leq A_n$.
Because $A$ is ordinary $A_n$-permutable, we choose $x\in A_n$ with $K=AB^x$ a subgroup.
For $\alpha\notin\operatorname{supp}(B^x)$, the subgroup $B^x$ fixes $\alpha$.
As $K=AB^x$, we have $\{g(\alpha):g\in K\}=\{a(\alpha):a\in A\}$.
Thus, the $K$-orbit of $\alpha$ has size at most $d<q$.
Hence, $R=\operatorname{supp}(B^x)=x^{-1}\Omega$ is a $K$-orbit of size $q$.
We choose a bijection $\eta_q:\{1,\ldots,q\}\to R$.
Let $\pi_R:K\to S_q$ satisfy $\pi_R(g)=\eta_q^{-1}\circ(g|_R)\circ\eta_q$ for $g\in K$.
Because $q>k+1$, the restriction $\pi_R|_A:A\to S_q$ is injective.
Let $J=\pi_R(K)$, $A^R=\pi_R(A)$ and $F=\pi_R(B^x)$.
As $K=AB^x$, we have $J=A^R F$ and $[J:F]\leq d$.
Let $\chi=\beta\circ(x|_R)\circ\eta_q:\{1,\ldots,q\}\to\F_q$.
We have $\chi F\chi^{-1}=\operatorname{AGL}_1(q)$.
Thus, $F$ is primitive and contains a $(q-1)$-cycle.

We compare the groups in Jones's classification \cite[Theorem 1.2(2)]{JonesCycles}.
Let $L=\Soc(J)$, and suppose that $L=A_q$.
Since $A^R,F$ are soluble and $J$ is not soluble, we have $A^R,F<J$.
The equality $J=A^R F$ contradicts \cite[Proposition 4.1]{LiXia}.
As $q=2^e$, we have $q\notin\{12,24\}$.
Thus, $J$ is not a Mathieu group from Jones's list.
For $J\cong\PSL_2(q-1)$ or $J\cong\PGL_2(q-1)$ with $q-1$ prime, the index $[J:F]$ is $(q-2)/2$ or $q-2$.
Because $(q-2)/2>d$, neither possibility satisfies $[J:F]\leq d$.

By Jones's classification, we choose $Q=2^h$ and $c\geq1$ with $Q^c=q$.
Let $W=\F_Q^c$, so $\dim_{\F_Q}W=c$.
We choose a bijection $\zeta:\{1,\ldots,q\}\to W$ such that
$\operatorname{AGL}_c(Q)\leq J_W=\zeta J\zeta^{-1}\leq\operatorname{A\Gamma L}_c(Q)$.
For $v\in W$, let $\tau_v:W\to W$ satisfy $\tau_v(u)=u+v$ for $u\in W$.
Let $T=\{\tau_v:v\in W\}$ and $V=\zeta^{-1}T\zeta$.
Let $M\in\GL_c(Q)$, $\sigma\in\operatorname{Gal}(\F_Q/\F_2)$ and $b\in W$.
For $u=(u_1,\ldots,u_c)\in W$, let $u^\sigma=(\sigma(u_1),\ldots,\sigma(u_c))$.
Let $g:W\to W$ satisfy $g(u)=Mu^\sigma+b$ for $u\in W$.
For every $v\in W$, we have $g\tau_vg^{-1}=\tau_{Mv^\sigma}$.
Thus, $T\unlhd J_W$ and $V\unlhd J$.
For $u,w\in W$, $\tau_{w-u}$ is the unique element of $T$ sending $u$ to $w$.
Therefore, $T$ acts regularly on $W$ and $V$ acts regularly on $\{1,\ldots,q\}$.
We have $|T|=|V|=q$.

Suppose that $V\cap F=1$.
Then, $[VF:F]=q>d$.
Since $VF\leq J$, this contradicts $[J:F]\leq d$.
Therefore, $V\cap F\ne1$.
As $V\unlhd J$, we have $V\cap F\unlhd F$.
Because $F$ is primitive, the nontrivial subgroup $V\cap F$ is transitive on $\{1,\ldots,q\}$.
As $V$ acts regularly, this transitive subgroup has order $q$.
Thus, $V\cap F=V$ and $V\leq F$.
Since $V\unlhd F$ and $|V|=q$, $V$ is the unique Sylow $2$-subgroup of $F$.
Under $\chi F\chi^{-1}=\operatorname{AGL}_1(q)$, this subgroup is $\{u\mapsto u+b:b\in\F_q\}$.
Thus, $\chi V\chi^{-1}=\{u\mapsto u+b:b\in\F_q\}$ and $F/V\cong\F_q^*$.
The image $A^RV/V$ is cyclic.
As $J=A^RF$, the quotient $J/V$ is a product of $2$ cyclic groups.
By It\^o's theorem, $J/V$ is metabelian.
The map $M\mapsto\bigl(\zeta^{-1}\circ(u\mapsto Mu)\circ\zeta\bigr)V$ injects $\GL_c(Q)$ into $J/V$.
For even $Q$ with $c\geq2$ and $(c,Q)\ne(2,2)$, the group $\GL_c(Q)$ is not metabelian.
For $(c,Q)=(2,2)$, we would have $q=4$.
Therefore, $c=1$ and $Q=q$.
Let $\xi=\zeta:\{1,\ldots,q\}\to\F_q$.
For $a\in A$, let $\rho_q(a)=\xi\circ\pi_R(a)\circ\xi^{-1}$.
Because $\pi_R|_A$ is injective, $\rho_q:A\to\operatorname{A\Gamma L}_1(q)$ is an injective homomorphism.
\end{proof}

\begin{theorem}\label{thm:alternating-all}
For every $n\geq7$, the group $A_n$ has no proper nontrivial
ordinary $A_n$-permutable subgroup.
\end{theorem}
\begin{proof}
For $7\leq n\leq121$, we use Corollary \ref{cor:alternating-near-prime}.
Suppose that $n\geq122$ and $A<A_n$ is ordinary $A_n$-permutable.
Let $p<n$ be the largest prime and let $k=n-p$.
By Theorem \ref{thm:alternating-cyclic-reduction}, we have $A\cong C_d$ with $d\leq k+2$.

Firstly, suppose that $122\leq n<410000$.
Using \textit{V2.py}, we compute all primes less than $410000$.
For $122\leq n\leq131$, we obtain $k\leq14$.
For $122\leq n<410000$, we obtain $k\leq112$.
The largest gap is $370373-370261=112$.
For $122\leq n\leq131$, use the primes $23,29$.
For $132\leq n<410000$, use $127,131$.
Each pair satisfies the hypotheses for $r$ in
Lemma \ref{lem:alternating-affine-subgroups}, and
$\gcd(23-1,29-1)=\gcd(127-1,131-1)=2$.
Hence, $d\leq2$.
We compute the greatest common divisor using integers and the list of primes in Appendix \ref{app:calculations}.

Now, suppose that $n\geq410000$.
By Dusart's estimate \cite[Proposition 6.8]{Dusart2010}, we have a prime in
$$
 \left(x,x\left(1+\frac1{25(\log x)^2}\right)\right]
 \ \hbox{for every }x>396738.
$$
We take $x=63n/64$.
Then, we have $x>396738$ and $1/(25(\log x)^2)<1/63$.
Thus, $p>x$ and $k<n/64$.
Let $2^E$ be the largest power of $2$ at most $n$.
The $4$ powers $2^{E-3},2^{E-2},2^{E-1},2^E$ exceed $n/16$.
For $i=0,1,2,3$, let $q_i=2^{E-i}$.
We have $2d+2\leq2k+6<n/32+6<n/16$.
Thus, each $q_i$ satisfies the hypotheses for $q=2^e$ in Lemma \ref{lem:alternating-affine-subgroups}.
Let $\mathcal E=\{E-3,E-2,E-1,E\}$.
For every $e\in\mathcal E$, the injective homomorphism $\rho_{2^e}$ places $A$ in a group of order $2^e e(2^e-1)$.
Consequently, every odd prime $\ell\mid d$ divides $e(2^e-1)$ for each $e\in\mathcal E$.
Suppose that $\ell\geq5$.
Since $\mathcal E$ consists of $4$ consecutive integers, at most $1$ member is divisible by $\ell$.
For $\mathcal E_\ell=\{e\in\mathcal E:\ell\nmid e\}$, we have $|\mathcal E_\ell|\geq3$.
As $\mathcal E_\ell\subseteq\mathcal E$ and $|\mathcal E_\ell|\geq3$, we choose $a,a+1\in\mathcal E_\ell$.
For $e\in\mathcal E_\ell$, the divisibility $\ell\mid e(2^e-1)$ implies $\ell\mid2^e-1$.
Thus, $\ell\mid\gcd(2^a-1,2^{a+1}-1)=1$.
We obtain a contradiction.
For $\ell=3$, let $e_1,e_2\in\mathcal E$ be the odd exponents with $e_2=e_1+2$.
Because $3\nmid(e_2-e_1)$, at least $1$ of $e_1,e_2$ is not divisible by $3$.
For such an exponent $e$, we have $2^e-1\equiv1\pmod3$.
Thus, $3\nmid e(2^e-1)$.
Consequently, $d$ is a power of $2$.
For an odd $e\in\mathcal E$, let $T_e=\{u\mapsto u+b:b\in\F_{2^e}\}$.
As $[\operatorname{A\Gamma L}_1(2^e):T_e]=e(2^e-1)$ is odd, $T_e$ is a Sylow $2$-subgroup.
Because $d$ is a power of $2$, we choose $y\in\operatorname{A\Gamma L}_1(2^e)$ with $\rho_{2^e}(A)^y\leq T_e$.
As $\rho_{2^e}|_A$ is injective and $T_e\cong C_2^e$, we have $d\leq2$.
For $d=2$, we contradict ordinary permutability of $A$ by Proposition \ref{prop:alternating-involutions}.
Therefore, $A=1$.
\end{proof}

\section{Exceptional groups of Lie type}\label{sec:exceptional}

Let $k=\overline{\F}_q$.
For $G={}^3D_4(q)$, let $\mathbf G$ be the simply connected algebraic group of type $D_4$ over $k$.
For the other groups below, let $\mathbf G$ be the corresponding adjoint algebraic group.
For an algebraic subgroup $\mathbf H$, let $\mathbf H^\circ$ be its connected component containing the identity.
For an endomorphism $F:\mathbf G\to\mathbf G$ and a subgroup $\mathbf H$ with $F(\mathbf H)=\mathbf H$, let $\mathbf H^F=\{h\in\mathbf H:F(h)=h\}$.

Let $\mathfrak g=\operatorname{Lie}(\mathbf G)=T_1\mathbf G$, the tangent space at the identity.
For $g\in\mathbf G$, let $\operatorname{Ad}(g)$ be the differential of $x\mapsto gxg^{-1}$ at the identity.
For $D\leq\mathbf G$, let $\mathfrak g^D=\{X\in\mathfrak g:\operatorname{Ad}(d)X=X\text{ for every }d\in D\}$.
For $c\in\mathbf G$, let $\mathfrak g^c=\mathfrak g^{\gen c}$.
For a maximal torus $\mathbf T$, let $\Phi=\Phi(\mathbf G,\mathbf T)$ and $\mathfrak t=\operatorname{Lie}(\mathbf T)$.
For $\alpha\in\Phi$, let $\mathfrak g_\alpha=\{X\in\mathfrak g:\operatorname{Ad}(t)X=\alpha(t)X\text{ for all }t\in\mathbf T\}$.
The root spaces satisfy $\mathfrak g=\mathfrak t\oplus\bigoplus_{\alpha\in\Phi}\mathfrak g_\alpha$.

For $G\leq\mathbf G^F$ and a maximal torus $\mathbf T$ with $F(\mathbf T)=\mathbf T$, let $T=\mathbf T\cap G$.
We use $N_G(\mathbf T)=\{g\in G:g\mathbf Tg^{-1}=\mathbf T\}$ and $N_G(T)=\{g\in G:gTg^{-1}=T\}$.
The Weyl group is $W(\mathbf G,\mathbf T)=N_{\mathbf G}(\mathbf T)/\mathbf T$.
We have $N_G(\mathbf T)\leq N_G(T)$.

\subsection{\texorpdfstring{The groups $G_2(q)$}{The groups G2(q)}}\label{sec:g2-all}

For a maximal torus $\mathbf T$, let $X(\mathbf T)=\Hom_{\mathrm{alg}}(\mathbf T,k^*)$ and $Y(\mathbf T)=\Hom_{\mathrm{alg}}(k^*,\mathbf T)$.
The subscript $\mathrm{alg}$ means homomorphisms of algebraic groups.
We call $X(\mathbf T)$ the character lattice and $Y(\mathbf T)$ the cocharacter lattice.
Their integer pairing satisfies $\chi(\lambda(a))=a^{\langle\chi,\lambda\rangle}$ for $\chi\in X(\mathbf T)$, $\lambda\in Y(\mathbf T)$ and $a\in k^*$.
For $\alpha\in\Phi$, let $s_\alpha\in W(\mathbf G,\mathbf T)$ be the reflection in $\alpha$.
The coroot $\alpha^\vee\in Y(\mathbf T)$ satisfies $s_\alpha(\chi)=\chi-\langle\chi,\alpha^\vee\rangle\alpha$ for every $\chi\in X(\mathbf T)$.
In particular, the homomorphism $\alpha^\vee:k^*\to\mathbf T$ satisfies $\langle\alpha,\alpha^\vee\rangle=2$ \cite[Chapter 21]{MilneAlgebraicGroups}.

For a vector space $U$ over a field $k$, let $R=\operatorname{span}_k\{u\otimes v-v\otimes u:u,v\in U\}$.
We use $\Sym^2(U)=(U\otimes_k U)/R$.
For $g\in\GL_k(U)$, the induced action satisfies $g(u\otimes v+R)=gu\otimes gv+R$.

\begin{proposition}\label{prop:g2-involutions}
Let $G=G_2(q)$, where $q\geq3$ is a prime power.
Then, $G$ has no ordinary $G$-permutable subgroup of order $2$.
\end{proposition}
\begin{proof}
By Lemma \ref{lem:involution-normalizers}, we need a subgroup $B\leq G$ with $|B|$ and $|N_G(B)|$ odd.

Firstly, suppose that $q$ is a power of $3$.
The generic character table has a unipotent class of type $8$
with centralizer order $3q^2$ \cite{Enomoto1970,CHEVIE}.
There is an irreducible character $\chi$ with
$\chi(u)=q/3+q\omega, \omega^2+\omega+1=0, \omega\ne1$,
for $u$ in the class of type $8$. Since $\chi(u)\notin\mathbb R$, no $g\in G$ satisfies $u^g=u^{-1}$.
Let $B=\gen u$, with $|B|=3^a$ and $a\geq1$.
Let $t\in N_G(B)$ have order $2$.
As the automorphism $u\mapsto t^{-1}ut$ has order dividing $2$, we have $t^{-1}ut\in\{u,u^{-1}\}$.
Because $|C_G(u)|=3q^2$ is odd, we have $t^{-1}ut\ne u$.
But $\chi(u^{-1})=\overline{\chi(u)}\ne\chi(u)$, so $t^{-1}ut\ne u^{-1}$.
Hence, $N_G(B)$ has odd order.

Now, suppose that $3\nmid q$. Choose $\varepsilon\in\{1,-1\}$
with $q\equiv\varepsilon\pmod3$, and let
$r=\varepsilon q, v=r^2+r+1, d=v/3$.
Since $v$ is odd and $v_3(v)=1$, $d$ is odd and prime to $3$.
As $q\geq4$, we have $d\geq7$.
There is a cyclic torus $T=\gen z$ of order $v$.
Let $\mathbf T$ be the algebraic torus containing $T$.
For this torus, the rational normalizer is $N=N_G(\mathbf T)=T\rtimes\gen{a,\tau}$.
For the tori of types $5$ and $6$ in \cite[Theorem 1.1 and Table 1]{GaltStaroletov}, we choose the complement $\gen{a,\tau}$.
Its generators satisfy $|a|=3$, $|\tau|=2$ and $[a,\tau]=1$.
We choose $a$ so that $z^a=z^r$.
Also, $z^\tau=z^{-1}$.

Let $D=\gen{z^3}$, of order $d$.
We claim that $C_G(D)=T$ and $N_G(D)=N$.
To prove the claim, let $V=\F_q^7$ and choose $\lambda\in\overline{\F}_q^*$ with $|\lambda|=d$.
A generator of $D$ has eigenvalues $1,\lambda^{\pm1},\lambda^{\pm r},\lambda^{\pm r^2}$ on $V$.
We have $r^3\equiv1\pmod d$ and $\gcd(d,r-1)=\gcd(d,r+1)=1$.
Since $d$ is odd and $\gcd(d,r-1)=\gcd(d,r+1)=1$, the $7$ eigenvalues are distinct.
A matrix centralizing $D$ is diagonal in the eigenbasis and centralizes $\mathbf T$.
As $C_{\mathbf G}(\mathbf T)=\mathbf T$, we obtain that $C_G(D)=T$.
Any element normalizing $D$ also normalizes $\mathbf T$,
so $N_G(D)\leq N$. Because $D$ is characteristic in $T$, we also have $N\leq N_G(D)$.
Hence, $N_G(D)=N$.

Let $c=za$ and $B=D\gen c$.
Since $1+r+r^2=v$, we have $c^3=1$.
As $C_D(c)=1$, we have $B=D\rtimes C_3$.
In particular, $D=O_{3'}(B)$ is characteristic in $B$.
Consequently, $N_G(B)\leq N$.

Let $\rho:N\to N/D$ be the quotient map and write $\bar h=\rho(h)$.
The subgroup $U=\gen{\bar z,\bar a}\cong C_3\times C_3$
has index $2$, and $B/D=\gen{\bar z\bar a}$.
For conjugation by $\bar\tau$ on $U$, we have $\bar\tau^{-1}\bar z\bar\tau=\bar z^{-1}$ and $\bar\tau^{-1}\bar a\bar\tau=\bar a$.
For an involution $t\in N$, we have $\bar t\in U\bar\tau$.
Since $U$ is abelian, we have $\bar t^{-1}(\bar z\bar a)\bar t=\bar z^{-1}\bar a$.
As $\gen{\bar z\bar a}\ne\gen{\bar z^{-1}\bar a}$, we have $t\notin N_G(B)$.
Therefore, $|N_G(B)|$ is odd.
\end{proof}

\begin{lemma}\label{lem:g2-cyclic-stabilizer}
There is no faithful transitive group of degree $6$ and order $24$
with a cyclic point stabilizer of order $4$ and a regular
cyclic subgroup of order $6$.
\end{lemma}
\begin{proof}
Suppose that $H\leq S_{\Omega}$ is transitive with $|\Omega|=6$ and $|H|=24$.
Suppose that $H$ also contains a regular subgroup $C_6$.
Let $\alpha\in\Omega$ and suppose that $H_\alpha\cong C_4$.
By Sylow's theorem, the number of Sylow $3$-subgroups of $H$ is $1$ or $4$.
If $H$ has exactly $1$ Sylow $3$-subgroup, then let $\Omega_1,\Omega_2$ be the $2$ orbits of the normal subgroup of order $3$.
Since $H_\alpha\Omega_i=\Omega_i$ for $i=1,2$, restriction is an injective map
$h\mapsto(h|_{\Omega_1},h|_{\Omega_2})$ from $H_\alpha$ to $S_{\Omega_1}\times S_{\Omega_2}$.
As $|\Omega_1|=|\Omega_2|=3$, the product $S_{\Omega_1}\times S_{\Omega_2}$ has no element of order $4$.
If $H$ has exactly $4$ Sylow $3$-subgroups, then we use $\rho:H\to S_{\operatorname{Syl}_3(H)}$ with $\rho(h)(P)=hPh^{-1}$.
Since the Sylow $3$-subgroups of $H$ are conjugate in $H$, the image $\rho(H)$ is transitive on those $4$ subgroups.
Suppose that a Sylow $3$-subgroup $P$ of $H$ lies in $\ker\rho$.
Then, $P$ would normalize a different Sylow $3$-subgroup $Q$.
But $|N_H(Q)|=6$, so $Q$ is the unique Sylow $3$-subgroup of $N_H(Q)$.
By transitivity and $3\nmid|\ker\rho|$, we have $|\rho(H)|\in\{12,24\}$.
An image of order $24$ identifies the group with $S_4$,
which has no element of order $6$.
Otherwise, $\rho(H)=A_4$ and $|\ker\rho|=2$.
Because $A_4$ has no element of order $4$, we have $|\rho(H_\alpha)|\leq2$.
Thus, $|\ker\rho\cap H_\alpha|\geq2$.
As $|\ker\rho|=2$, we have $\ker\rho\leq H_\alpha$.
As $1\ne\ker\rho\lhd H$, the action on $\Omega$ is not faithful.
\end{proof}

\begin{theorem}\label{thm:g2-all}
Let $q\geq3$ be a prime power.
Then, $G_2(q)$ has no proper nontrivial ordinary
$G_2(q)$-permutable subgroup.
\end{theorem}
\begin{proof}
For $q=3,4$, we use Proposition \ref{prop:finite-tables}.
Assume that $q\geq5$.
Let $G=G_2(q)$ and let $A<G$ be ordinary $G$-permutable.
Next, let $C$ be the maximal subgroup of type $A_1A_1$.
By \cite[Table 5]{LPS}, we have $CY\ne G$ for every $Y<G$. See also \cite[case IV]{GT23}.
We have $C\cong\PSL_2(q)^2$ for even $q$,
and $C\cong2.(\PSL_2(q)^2)\rtimes C_2$ for odd $q$.
Thus, $A\leq C^y$ for some $y\in G$.
Since $A\leq C^y$, we have $AY\ne G$ for every $Y<G$.
From $A\leq C^y$, we have $|A|_p\leq q^2$ for $p=\operatorname{char}\F_q$.

We choose the cyclic torus $T$ of order prime to $3$ from
\cite[Proposition 2.11 and Table 6]{GuralnickMalle2012}.
For $q\equiv1\pmod3$, we have $|T|=q^2-q+1$.
The maximal subgroup $M<G$ with $T\leq M$ is of type $\SU_3(q)\rtimes C_2$.
For $q\not\equiv1\pmod3$, we have $|T|=q^2+q+1$.
Every maximal subgroup $M<G$ with $T\leq M$ is of type $\SL_3(q)\rtimes C_2$.
Since $A$ is ordinary $G$-permutable, we choose $x\in G$ with $K=AT^x$ a subgroup.
As $AY\ne G$ for every $Y<G$, we have $K<G$.
Choose a maximal subgroup $M<G$ containing $K$. Then, let $L\lhd M$ be the centreless subgroup of type $A_2$ in $M$.
Because $|M/L|=2$ and $|T|$ is odd, $T^x\leq L$.
Hence, $J=K\cap L\lhd K$ and $T^x\leq J$.
As $|K|_p\leq q^2<q^3=|L|_p$, $K\cap L<L$.
The torus $T^x$ has order $q^2-q+1$ in $\SU_3(q)$ or $q^2+q+1$ in $\SL_3(q)$.
By Bereczky's theorem \cite{Bereczky}\cite[Theorem 2.1.7]{Brookfield2014}, we have $J\leq T^x\rtimes C_3$.
The order of $T$ is incompatible with each small exception in Bereczky's theorem.
The subgroup $T^x$ is exactly the set of $3'$-elements of $J$,
so $T^x$ is characteristic in $J$.
Since $J\unlhd K$, we have $T^x\unlhd K$ and $A\leq N_G(T^x)$.
Moreover, $\gcd(|C|,|T|)=1$.
As $A\leq C^y$, we have $A\cap T^x=1$.
The quotient map $a\mapsto aT^x$ injects $A$ into $N_G(T^x)/T^x\cong C_6$.
Consequently, $A$ is cyclic and $|A|\mid6$.
If $|A|=2$, then Proposition \ref{prop:g2-involutions} contradicts ordinary permutability of $A$.
Suppose that $|A|\in\{3,6\}$ and let $A_3\leq A$ have order $3$.

Take a cyclic torus $R$ of order $q^2-1$ in $\SL_3(q)$ when $q\equiv0,1\pmod3$,
and in $\SU_3(q)$ when $q\equiv-1\pmod3$.
For this torus, the quotient $N_G(R)/R$ is a $2$-group.
To compute its root characters, let $\mathbf R$ be the algebraic torus containing $R$.
The weights of the natural module $V=\F_q^7$ on $\mathbf R$ are
$0,\pm e_1,\pm e_2,\pm e_3$, where $e_1+e_2+e_3=0$.
For a generator $h\in R$, choose $\lambda\in\F_{q^2}^*$ of order $q^2-1$.
Then,
$$
 (e_1(h),e_2(h),e_3(h))=
 \begin{cases}
 (\lambda,\lambda^q,\lambda^{-q-1})&\text{in the linear case},\\
 (\lambda,\lambda^{-q},\lambda^{q-1})&\text{in the unitary case}.
 \end{cases}
$$
When $3\nmid q$, an element of order $3$ in $R$ has fixed
space of dimension $1$.
For $c\in R$ of order $3$, the $3$ positive weights all take $1$ primitive cube root as their value at $c$.
A nonidentity element $a\in A_3$ has nontrivial image under $N_G(T^x)\to N_G(T^x)/T^x\cong C_6$.
To compute the fixed space of $a$, let $\overline V=\overline{\F}_q\otimes_{\F_q}V$.
On the $A_2$ subgroup $L$ containing $T^x$, we have $\overline V=V_0\oplus V_1\oplus V_2$ with dimensions $1,3,3$.
Since $a$ has image of order $3$ in the Singer normalizer, we choose bases $(v_{i,0},v_{i,1},v_{i,2})$ of $V_i$ for $i=1,2$.
We have $av_{i,j}=c_{i,j}v_{i,j+1}$ with $c_{i,j}\in\overline{\F}_q^*$ and $j$ taken modulo $3$.
As $a^3=1$, we have $c_{i,0}c_{i,1}c_{i,2}=1$ for $i=1,2$.
Each cycle contributes $1$ fixed dimension. Because $a|_{V_0}=I_{V_0}$, we have $\dim V^a=3$.
For every $g\in G$, we have $\dim V^{a^g}=3$.
An element of order $3$ in $R$ has fixed space of dimension $1$.
Thus, no $g\in G$ satisfies $a^g\in R$.

Since $A$ is ordinary $G$-permutable, we choose $z\in G$ with $H=AR^z$ a subgroup.
Let $\rho:H\to S_{H/R^z}$ be the coset action.
We have $\rho(h)(gR^z)=hgR^z$ for $h,g\in H$.
Then, we have $D=\ker\rho=\Core_H(R^z)$ and $[H:R^z]\mid6$.
As $\rho(R^z)$ fixes the coset $R^z$, restriction to $(H/R^z)\setminus\{R^z\}$ is faithful.
This set has at most $5$ points.
Because $\rho(R^z)$ is cyclic, we have $|\rho(R^z)|\leq6$.
Thus, $[R^z:D]\leq6$ and $|D|\geq(q^2-1)/6$.
For $q\geq8$, we have $|D|>q+1$.
For $1\leq i\leq3$, let $\chi_i:R^z\to\overline{\F}_q^*$ satisfy $\chi_i(h^z)=e_i(h)$ for $h\in R$.
On $R^z$, the short root characters are $\chi_i^{\pm1}$.
In addition, the long root characters are $\chi_i\chi_j^{-1}$ with $i\ne j$.
In the linear case, $(\chi_i\chi_j^{-1})|_D=1$ requires $|D|\mid a$ for some $a\in\{q-1,2q+1,q+2\}$.
In the unitary case, $(\chi_i\chi_j^{-1})|_D=1$ requires $|D|\mid a$ for some $a\in\{q+1,2q-1,q-2\}$.
For the short root characters, the equality $\chi_i|_D=1$ requires $|D|\mid q+1$
or $|D|\mid q-1$, respectively.
But
$\gcd(q^2-1,2q\pm1)\mid3, \gcd(q^2-1,q\pm2)\mid3$.
For $q\geq8$, all the required divisibilities contradict $|D|>q+1$.
Thus, no root is trivial on $D$.
As semisimple centralizers in the simply connected group are connected, we have $C_G(D)=R^z$.
Because $D\lhd H$, we obtain that $A\leq N_G(R^z)$.
As $N_G(R^z)/R^z$ is a $2$-group, we have $A_3\leq R^z$.
For $3\nmid q$, each element $c\in R^z$ of order $3$ has $\dim V^c=1$.
But $a\in A_3$ has $\dim V^a=3$, a contradiction.
For $3\mid q$, the inclusion $A_3\leq R^z$ contradicts $3\nmid|R|$.

The remaining cases are $q=5$ and $q=7$.
Suppose that a root character is trivial on $D$.
From the displayed values, we have $(q,D)=(5,C_4)$ or $(7,C_8)$.
In both cases, a short root is trivial on $D$.
There is also a possible $D=C_6$ at $q=5$, with a long root trivial on $D$.
For $q=5$ and $|D|=6$, we have $R^z/D\cong C_4$.
The group $R^z/D$ fixes a coset and contains an element of order $4$.
Since $[H:R^z]\mid6$, we have $[H:R^z]=6$.
As $[H:R^z]=6$, we have $|A|=6$ and $A\cap R^z=1$.
Thus, $H/D$ acts faithfully with degree $6$ and order $24$.
Its point stabilizer $R^z/D$ is cyclic of order $4$, and the image of $A$ is cyclic and regular of order $6$.
By Lemma \ref{lem:g2-cyclic-stabilizer}, we have a contradiction.

Since $D\unlhd H$, conjugation induces a homomorphism $A_3\to\Aut(D)$.
For $D\in\{C_4,C_8\}$, the group $\Aut(D)$ is a $2$-group.
As $|A_3|=3$, the homomorphism is trivial and $A_3\leq C_G(D)$.
Let $\beta$ be the short root trivial on $D$.
For this centralizer, let $L=C_G(D)'\cong\SL_2(q)$.
Moreover, the quotient $C_G(D)/L$ has order $q-1=4$ at $q=5$,
and $q+1=8$ at $q=7$.
Since $|a|=3$ and $3\nmid|C_G(D)/L|$, we have $a\in L$.
The coroot of $\beta$ pairs with the $7$ weights as
$2,1,1,0,-1,-1,-2$.
We restrict the natural representation $G\to\GL(V)$ to $L$ and have
$V\!\downarrow_L\cong\Sym^2(\F_q^2)\oplus\F_q^2\oplus\F_q^2$.
For $c\in L$ with $|c|=3$, each summand $\F_q^2$ has no nonzero fixed vector, and $\Sym^2(\F_q^2)$ has a fixed space of dimension $1$.
Thus, $\dim V^c=1$.
But $a\in A_3\leq L$ satisfies $\dim V^a=3$, a contradiction.
Hence, $|A|\notin\{3,6\}$ and $A=1$.
\end{proof}

\subsection{\texorpdfstring{The groups ${}^3D_4(q)$}{Triality groups}}\label{sec:triality}

Let $\alpha_1,\ldots,\alpha_4$ be the simple roots of type $D_4$, with node $2$ central.
Simple coroot coordinates describe the map $(a_1,\ldots,a_4)\mapsto\prod_{i=1}^4\alpha_i^\vee(a_i)$ from $(k^*)^4$ to the maximal torus $\mathbf T$.

\begin{lemma}\label{lem:triality-odd-normalizer}
Let $q$ be odd and let $G={}^3D_4(q)$.
There is a subgroup of order $q^8$ whose normalizer in $G$ has odd order.
\end{lemma}
\begin{proof}
We use the simply connected algebraic group of type $D_4$, with
central simple root $\alpha_2$ and outer roots $\alpha_1,\alpha_3,\alpha_4$.
Let $\tau$ permute the roots by $\alpha_1\mapsto\alpha_3\mapsto\alpha_4\mapsto\alpha_1$ and fix $\alpha_2$.
For each root $\alpha$, the map $F$ satisfies $F(x_\alpha(t))=x_{\tau(\alpha)}(t^q)$ \cite[Section 5]{GaltStaroletov}.
Let $\mathbf U$ be the algebraic upper unipotent group.
Let $\mathbf U_3$ be generated by $x_\alpha(t)$ with $\operatorname{ht}(\alpha)\geq3$.
For the rational subgroups, we take $U=\mathbf U^F$ and $U_3=\mathbf U_3^F$.
Since $\tau$ has $1$ orbit of length $3$ and $2$ fixed roots on those $5$ roots, we have $|U_3|=q^5$.

Let $E=\F_{q^3}$.
For $x\in E$ and $t\in\F_q$, let $X(x)=x_{\alpha_1}(x)x_{\alpha_3}(x^q)x_{\alpha_4}(x^{q^2})U_3$ and $Y(t)=x_{\alpha_2}(t)U_3$.
Let $Z(u)=[X(u),Y(1)]$ for $u\in E$.
We choose the root maps with $[X(x),Y(t)]=Z(xt)$.
Since the Chevalley constants are $\pm1$ and $\tau$ permutes the $3$ roots of height $2$, we have $F(Z(u))=Z(u)$ for $u\in E$.
The ordered coordinates $X(x)Y(t)Z(u)$ have the law
$(x,t,u)(x',t',u')=(x+x',t+t',u+u'-x't)$.

We choose $\alpha\in E\setminus\F_q$.
Let $Q/U_3=Z(\F_q\alpha)$, so $Q$ is normal in $U$ of order $q^6$.
For $u\in E$, let $\overline u=u+\F_q\alpha\in E/\F_q\alpha$.
In $U/Q$, take the subgroup
$S=\{(a+b\alpha,a,\overline{b-a^2/2}):a,b\in\F_q\}$.
Because $\overline\alpha=0$ in $E/\F_q\alpha$, multiplication in $S$ adds the coordinates $a,b$.
Hence, $S$ is abelian of order $q^2$.
Let $\eta:U\to U/Q$ be the quotient map and let $H=\eta^{-1}(S)$, of order $q^8$.

For $g=(x,t,\overline u)$ and $h=(x',t',\overline u')$ in $U/Q$, we have $[g,h]=(0,0,\overline{xt'-x't})$.
The centre of $U/Q$ is 
$\{(b\alpha,0,\overline u):b\in\F_q,u\in E\}$,
and the central intersection of $S$ is precisely
$L=\{(b\alpha,0,\overline b):b\in\F_q\}$.
Since $\alpha\notin\F_q$, $\overline1\ne0$.
For every $u\in U$ and $l\in L$, we have $\eta(u)^{-1}l\eta(u)=l$.

For $a=1$ and $b=0$, we choose $u_0\in H$ with all $4$ absolute simple-root coordinates nonzero.
The element $u_0$ is regular unipotent and belongs to a unique algebraic Borel
subgroup $\mathbf B$ \cite[Lemma 3.2 and Theorem 3.3]{SteinbergRegular}.
Let $g\in N_G(H)$.
Since $u_0\in\mathbf B\cap\mathbf B^g$ and $\mathbf B$ is unique, we have $\mathbf B^g=\mathbf B$.
Let $T$ be the rational maximal torus normalizing $U$ in $\mathbf B$.
Thus, $N_G(H)\leq TU$.

Suppose that $g\in N_G(H)$ is an involution.
Since $U\unlhd TU$ has odd order, we choose $u_1\in U$ with $h=g^{u_1}\in T$ by Schur--Zassenhaus.
Then, $h\in N_G(H^{u_1})$.
As $F(h)=h$ and $\tau$ cycles $\alpha_1,\alpha_3,\alpha_4$, we have
$\alpha_1(h)=\alpha_3(h)=\alpha_4(h)=\delta\in\{1,-1\}$.
Let $\gamma=\alpha_2(h)\in\{1,-1\}$.
On $Z(E)$, conjugation by $h$ acts as $Z(u)\mapsto Z(\delta\gamma u)$.
Because $\delta\gamma\in\F_q^*$, we have $hQh^{-1}=Q$.
We take $\pi:U/Q\to E\times\F_q$ with $\pi(x,t,\overline u)=(x,t)$.
Then, $\pi(S)=\{(a+b\alpha,a):a,b\in\F_q\}$.
For $u\in U$, we have $\pi(S^u)=\pi(S)$.
We have $\{(\delta(a+b\alpha),\gamma a):a,b\in\F_q\}=\pi(S)$.
As $1,\alpha$ are independent over $\F_q$, we have $\delta=\gamma$.
If both signs are $+1$, then $h$ is central in the algebraic group.
Triality cycles the $3$ nonidentity elements of $Z(\mathbf G)\cong C_2\times C_2$, so the rational centre is trivial.
Then, $h=1$, contrary to $|h|=2$.

If both signs are $-1$, then their products on the roots of height $2$
are $+1$. Hence, $hQh^{-1}=Q$ and the action of $h$ on $U/Q$ is
$(x,t,\overline u)\mapsto(-x,-t,\overline u)$.
The image of $(\alpha,0,\overline1)$ under $h$ is $(-\alpha,0,\overline1)$.
A vector of $L$ with first coordinate $-\alpha$ has $b=-1$ and central coordinate $-\overline1$.
Because $\eta(u)^{-1}l\eta(u)=l$ for $u\in U$ and $l\in L$, we have $\eta(H^{u_1})\cap Z(U/Q)=L$.
As $\overline1\ne-\overline1$, the image $(-\alpha,0,\overline1)$ does not lie in $L$.
Therefore, $h\notin N_G(H^{u_1})$.
This contradicts $h\in N_G(H^{u_1})$.
The normalizer has no involution and has odd order.
\end{proof}

\begin{theorem}\label{thm:triality}
Let $G={}^3D_4(q)$, where $q\geq2$.
Then, $G$ has no proper nontrivial ordinary $G$-permutable subgroup.
\end{theorem}
\begin{proof}
The case $q=2$ is Proposition \ref{prop:finite-tables}. Assume that $q\geq3$.
By \cite{LPS}, we have $G\ne XY$ for all $X,Y<G$.
The subgroup $G_2(q)$ is maximal \cite{KleidmanTriality}.
Thus, a proper ordinary $G$-permutable subgroup $A$ satisfies
$|A|\mid q^6(q^6-1)(q^2-1)$.
Let $t=q^4-q^2+1$ and take a cyclic torus $T$ of order $t$.
The unique maximal subgroup $M<G$ with $T\leq M$ is $N_G(T)$.
Moreover, we have $N_G(T)/T\cong C_4$
\cite[Proposition 2.11 and Table 6]{GuralnickMalle2012}
\cite[Table 3, type 6]{GaltStaroletov}.
Since $A$ is ordinary $G$-permutable, we choose $x\in G$ with $AT^x$ a subgroup.
As $A,T^x<G$ and $G\ne AT^x$, we have $AT^x<G$.
Hence, $AT^x\leq N_G(T^x)$.
We have $q^6\equiv-1\pmod t$ and $\gcd(t,q(q^2-1))=1$.
Because $t$ is odd, $\gcd(t,q^6-1)=1$.
By $|A|\mid q^6(q^6-1)(q^2-1)$, we obtain that $\gcd(|A|,t)=1$.
Thus, $A\cap T^x=1$.
The quotient map $a\mapsto aT^x$ injects $A$ into $N_G(T^x)/T^x\cong C_4$.
Consequently, $A$ is cyclic and $|A|\mid4$.

Suppose that $3\nmid q$ and choose $\varepsilon\in\{1,-1\}$ with
$q\equiv\varepsilon\pmod3$.
Let $r=\varepsilon q$, $v=r^2+r+1$ and $d=v/3$.
Then, $v_3(v)=1$ and $d\geq7$.
We have $2\nmid d$ and $3\nmid d$.
By Types 4 and 5 in \cite[Table 3]{GaltStaroletov}, we take a torus
$$
 S=\{(u,w,u^rw,(u^{-1}w)^{r+1}):u^v=w^v=1\}
       \cong C_v^2.
$$
Let $\mathbf S$ be the algebraic torus containing $S$.
The rational normalizer is $N=N_G(\mathbf S)=S\rtimes\SL_2(3)$.
To compute the root characters, we use the $4$ simple coroot coordinates, with node $2$ central in the $D_4$ diagram.
Let $D=\{s^3:s\in S\}\cong C_d^2$ and $P_3=\{s\in S:s^3=1\}\cong C_3^2$.
For each positive root $\beta$ of $D_4$, write $\beta|_S(u,w)=u^i w^j$.
The $12$ pairs $(i,j)$ are $$
\begin{array}{rrrr}
(2,-1)&(0,-r)&(2r,1)&(-2r-2,2r+1)\\
(2,-r-1)&(2r,1-r)&(-2r-2,r+1)&(2r+2,-r)\\
(-2r,r)&(-2,r+2)&(0,r+1)&(0,1).
\end{array}
$$
In each pair, a component or its negative belongs to $\{1,2,r,r+1,2r,2(r+1)\}$.
Since $\gcd(d,2r(r+1))=1$, at least $1$ component in each pair is a unit modulo $d$.
Thus, no root character is trivial on $D$.
We have $C_{\mathbf G}(D)^\circ=\mathbf S$.
Every $g\in N_G(D)$ satisfies $g\mathbf Sg^{-1}=\mathbf S$.
Hence, $N_G(D)\leq N$.

The complement contains an element $b=n_1n_7$ of order $3$
\cite[proof of Theorem 1.1, Table 3]{GaltStaroletov}.
For the action of $b$ on $S$, choose $\zeta$ of order $v$ and let $(u,w)=(\zeta^i,\zeta^j)$.
Replacing $b$ by $b^{-1}$ if necessary, we have $b(\zeta^i,\zeta^j)b^{-1}=(\zeta^{i'},\zeta^{j'})$.
In these coordinates, the exponents satisfy $(i',j')^{\mathsf T}=M(i,j)^{\mathsf T}$ with
$$
 M=\begin{pmatrix}2r+1&-r\\2r+2&-r\end{pmatrix}\pmod v, M^3=I.
$$
We identify $P_3$ with $\F_3^2$ by $(\zeta^{di},\zeta^{dj})\mapsto(i,j)$. Then,
$$
 \overline M=\begin{pmatrix}0&2\\1&2\end{pmatrix}, \operatorname{im}(I-\overline M)=\gen{(-1,1)}, I+\overline M+\overline M^2=0.
$$
Replacing $M$ by $M^{-1}$ leaves $\operatorname{im}(I-\overline M)$ unchanged.
We choose $z\in P_3$ with coordinates $(1,0)$ and let $c=zb$.
Then, $c$ has order $3$, so $B=D\rtimes\gen c$ has odd order $3d^2$.
Since $D=O_{3'}(B)$ is characteristic in $B$, we have $N_G(B)\leq N_G(D)\leq N$.

Let $\tau=-I\in\SL_2(3)$ be its unique involution.
Since $\tau$ represents the longest Weyl element, we have $\tau^{-1}s\tau=s^{-1}$ for every $s\in S$.
As $S=D\times P_3$, we identify $N/D$ with $P_3\rtimes\SL_2(3)$.
Let $\rho:N\to N/D$ satisfy $\rho(n)=nD$ for $n\in N$.
For an involution $h\in N$, let $\overline h=\rho(h)=(s,-I)$ with $s\in P_3$.
We have $\overline h(z,b)\overline h^{-1}=(-z+(I-\overline M)s,b)$.
Because $-I$ centralizes $b$, the image of this conjugate in $\SL_2(3)$ is $b$.
The only element of $B/D=\gen{(z,b)}$ with image $b$ is $(z,b)$.
If $h$ normalizes $B$, then we obtain that $2z\in\operatorname{im}(I-\overline M)$.
But $2z=(2,0)\notin\gen{(-1,1)}$.
Thus, $N_G(B)$ has odd order.

Suppose that $A\ne1$.
Since $A$ is ordinary $G$-permutable, we choose $x\in G$ with $H=AB^x\leq G$.
As $|B|$ is odd, $A$ is a cyclic Sylow $2$-subgroup of $H$.
Since $A$ is a Sylow $2$-subgroup of $H$ and $A\leq C_H(A)$, the index $[N_H(A):C_H(A)]$ is odd.
The map $hC_H(A)\mapsto(a\mapsto hah^{-1})$ embeds $N_H(A)/C_H(A)$ in $\Aut(A)$.
Because $\Aut(A)$ is a $2$-group, we have $N_H(A)=C_H(A)$.
By Burnside's transfer theorem, we have $B^x\unlhd H$.
Thus, $A\leq N_G(B^x)$.
As $|N_G(B^x)|$ is odd, we have $A=1$.

For $q=3^f$, let $B$ be the subgroup of odd order in Lemma \ref{lem:triality-odd-normalizer}.
Since $A$ is ordinary $G$-permutable, we choose $x\in G$ with $AB^x$ a subgroup.
For $|A|=2,4$, we obtain from Burnside's theorem that $B^x\unlhd AB^x$.
Thus, $A\leq N_G(B^x)$.
As $|N_G(B^x)|$ is odd, we have $A=1$.
\end{proof}

\subsection{\texorpdfstring{The groups $E_6(q)$ and ${}^2E_6(q)$}{The groups E6(q) and twisted E6(q)}}\label{sec:e6}

For a maximal torus $\mathbf T$ with $F(\mathbf T)=\mathbf T$, let $F^*:X(\mathbf T)\to X(\mathbf T)$ satisfy $F^*(\chi)=\chi\circ(F|_{\mathbf T})$.
We use this endomorphism for the action of $F$ on the character lattice.

For a root system $\Phi$, let $V=\mathbb R\Phi$ and extend its Euclidean inner product bilinearly to $V_{\mathbb C}=\mathbb C\otimes_{\mathbb R}V$.
A vector $v\in V_{\mathbb C}$ is regular if $(v,\alpha)\ne0$ for every $\alpha\in\Phi$.
A Weyl group element $w\in W(\Phi)$ is regular if $wv=\lambda v$ for a regular vector $v$ and $\lambda\in\mathbb C^*$.
The reflection in a root $\alpha$ is $s_\alpha(v)=v-2(v,\alpha)(\alpha,\alpha)^{-1}\alpha$.
A Coxeter element is a product of the reflections $s_\alpha$ for the simple roots $\alpha$, each used exactly $1$ time.

\begin{theorem}\label{thm:e6}
Let $G=E_6(q)$ or ${}^2E_6(q)$, where $q\geq2$.
Then, $G$ has no proper nontrivial ordinary $G$-permutable subgroup.
\end{theorem}
\begin{proof}
For ease of reading, we divide the proof of the theorem into the following steps.

\setlist[itemize]{leftmargin=4em,labelsep=.5em,topsep=.2\baselineskip,partopsep=0pt,parsep=0pt,itemsep=0pt}
Let $\varepsilon=1$ for $G=E_6(q)$ and $\varepsilon=-1$ for $G={}^2E_6(q)$. Let $e=(3,q-\varepsilon)$.
Next, let $A<G$ be ordinary $G$-permutable.
\begin{itemize}
\item[{\bf Step 1.} ] $A$ is cyclic and $|A|\mid9$.
\end{itemize}
There is a cyclic torus $T$ of order $t=(q^6+\varepsilon q^3+1)/e$.
Let $M$ be the unique maximal subgroup of $G$ with $T\leq M$.
For this subgroup, we have $L\unlhd M$, where $L=\PSL_3(q^3)$ or $\PSU_3(q^3)$.
Moreover, $M/L\cong C_3$ \cite[Proposition 2.11 and Table 6]{GuralnickMalle2012}.
We have $N_G(T)/T\cong C_9$.
Since $3\nmid t$ and $|M/L|=3$, we have $T\leq L$.
By \cite[Proposition 4.2]{LiXia}, we have $G\ne XY$ whenever $X,Y<G$ and at least one of $X,Y$ is soluble.
As $A$ is ordinary $G$-permutable, we choose $x\in G$ with $AT^x$ a subgroup.
Because $T^x$ is soluble, we have $AT^x<G$ and $AT^x\leq M^x$.

Except for twisted $q=2$, there is a soluble maximal subgroup $C$
of order
$648(q^2+\varepsilon q+1)^3/e$
\cite[Theorem 1.1 and Tables 9--10]{CravenExceptional}.
Since $A$ is ordinary $G$-permutable, we choose $y\in G$ with $AC^y$ a subgroup.
As $C^y$ is soluble, we have $AC^y\ne G$.
By maximality of $C^y$, we have $AC^y<G$ and $A\leq C^y$.
Write $q=p^f$.
We have $|A|_p\leq8$ for $p=2$ and $|A|_p\leq81$ for $p=3$.
For every other $p$, we have $|A|_p=1$.
The bounds $8,81,1$ are strictly smaller than $q^9=|L|_p$.
In $K=AT^x$, the normal subgroup $J=K\cap L^x$ contains $T^x$.
As $|J|_p\leq|A|_p<q^9=|L^x|_p$, we have $J<L^x$.
Let $\widetilde L=\SL_3(q^3)$ or $\SU_3(q^3)$ as appropriate,
and let $\kappa:\widetilde L\to L^x$ be the central quotient map.
Then, $\kappa^{-1}(J)$ contains a cyclic torus of order $q^6+\varepsilon q^3+1$.
Because $q^3\geq8$, Bereczky's theorem \cite{Bereczky}\cite[Theorem 2.1.7]{Brookfield2014} has no small exception.
We have $J\leq T^x\rtimes C_3$.
Thus, $T^x=O_{3'}(J)$ is characteristic in $J$.
As $J\unlhd K$, we have $K\leq N_G(T^x)$.

Modulo $q^2+\varepsilon q+1$, we have
$q^6+\varepsilon q^3+1\equiv3$.
The only possible common prime is $3$, whose single factor
in the numerator of $t$ is removed by $e$.
Consequently,
$\gcd\bigl(t,648(q^2+\varepsilon q+1)^3/e\bigr)=1$.
Hence, $A\cap T^x=1$.
The quotient map $a\mapsto aT^x$ is an injection from $A$ to $N_G(T^x)/T^x\cong C_9$.

For $G={}^2E_6(2)$, let $T=C_{19}$.
Since $A$ is ordinary $G$-permutable, we choose $x\in G$ with $AT^x$ a subgroup.
Since $T^x$ is soluble and $G\ne XY$ for all $X,Y<G$ with $Y$ soluble, we have $AT^x<G$.
By the uniqueness of the maximal subgroup containing $T^x$, we have $A\leq M^x$, where $M=\PSU_3(8)\rtimes C_3$.
The maximal subgroup $C=\PSU_3(2)\times G_2(2)$ has order $870912$
\cite[Table 10]{CravenExceptional}.
For $G={}^2E_6(2)$, we have $|M|=3\cdot5515776$ and $|G|=76532479683774853939200$.
Hence, $|M||C|<|G|$.
As $A$ is ordinary $G$-permutable, we choose $y\in G$ with $AC^y$ a subgroup.
Because $|M||C|<|G|$, we have $AC^y<G$.
By maximality of $C$, we have $A\leq C^y$.
In particular, $19\nmid|A|$ and $|A|\leq870912<|\PSU_3(8)|$.
If $K\cap L^x=L^x$ in the Singer product, then
$L^x=(A\cap L^x)T^x$.
As $|A\cap L^x|<|L^x|$ and $|T^x|=19<|L^x|$, both factors are proper in $L^x$.
But $L^x\cong\PSU_3(8)$ satisfies $L^x\ne XY$ for all $X,Y<L^x$ \cite{LPS}.
Thus, $J<L^x$.
By Bereczky's theorem, $T^x$ is characteristic in $J$.
Since $J\unlhd K$, we have $K\leq N_G(T^x)$.
As $19\nmid|A|$, we have $A\cap T^x=1$.
The quotient map $a\mapsto aT^x$ injects $A$ into $N_G(T^x)/T^x\cong C_9$.
Consequently, $A$ is cyclic of order dividing $9$.

\begin{itemize}
\item[{\bf Step 2.} ] $N_G(B)/(N_G(B)\cap S)$ is cyclic of order dividing $8$.
\end{itemize}
Let $\mathbf G$ be the adjoint algebraic group of type $E_6$ over $\overline{\F}_q$.
Let $F_q$ satisfy $F_q(x_\alpha(t))=x_\alpha(t^q)$ on each root group.
For $E_6(q)$, let $F=F_q$.
For ${}^2E_6(q)$, let $\sigma$ induce the diagram automorphism of order $2$ and let $F=\sigma\circ F_q$.
We identify $G$ with its natural subgroup of $\mathbf G^F$.
The inclusion is $\iota:G\to\mathbf G(\overline{\F}_q)$ with $\iota(g)=g$.
We now use a second torus in $\mathbf G$.
The Weyl group has a regular element $w_8$ with characteristic
polynomial $\Phi_8(X)\Phi_2(X)\Phi_1(X)$ and centralizer $C_8$
\cite[Corollary 4.3 and Table 1]{SpringerRegular}.
Next, choose $w_8$ as a Coxeter element of an absolute $D_5$ subsystem.
Let $(e_1,\ldots,e_5)$ be an orthonormal basis of the $D_5$ root space $\mathbb R^5$.
With this basis, we choose $w_8$ with $w_8(e_i)=e_{i+1}$ for $1\leq i\leq3$.
In addition, let $w_8(e_4)=-e_1$ and $w_8(e_5)=-e_5$.
Since $w_8$ stabilizes $\mathbb Re_1+\cdots+\mathbb Re_4$, it stabilizes the $D_4$ root subsystem on those coordinates.
As the graph automorphism acts as $-w_0$ on the root space, we use $-w_8$ for $G={}^2E_6(q)$.
The Weyl centralizer of $-w_8$ is also $C_8$.
Both tori $S$ have order $(q^4+1)(q^2-1)/e$.

Since $w_8$ and $-w_8$ restrict to negative $4$-cycles on $D_4$, the associated rational $D_4$ subgroup has minus type.
In its torus of dimension $4$, we take a cyclic subgroup $B$ of odd order
$b=q^4+1$ for even $q$ and $b=(q^4+1)/2$ for odd $q$.
As $(b,6)=1$, the central quotient maps are injective on $B$.
Choose $b_0\in B$ with $B=\gen{b_0}$ and $\lambda\in\overline{\F}_q^*$ with $|\lambda|=b$. In the vector representation, the eigenvalues of $b_0$ are $\lambda^{\pm1},\lambda^{\pm q},\lambda^{\pm q^2},\lambda^{\pm q^3}$.
For the $2$ spin modules of dimension $8$, we have weight exponents $(\pm1\pm q\pm q^2\pm q^3)/2$.
In addition, we have the $D_4$ root exponents $\pm q^i\pm q^j$ for $i<j$.
Among the restrictions of the $72$ roots of $E_6$, we have the $24$ roots of $D_4$.
Besides those roots, each weight system for a $D_4$ module of dimension $8$ appears twice among the other restrictions.
For $q\geq3$, every signed sum is nonzero and has absolute value at most
$1+q+q^2+q^3<(q^4+1)/2\leq b$.
For $q=2$, use $15<17$.
Division by $2$ is invertible modulo the odd number $b$.
Thus, no root is trivial on $B$.
Let $\mathbf S$ be the algebraic torus containing $S$, with $S=\mathbf S\cap G$.
Since no root is trivial on $B$, we have $C_{\mathbf G}(B)^\circ=\mathbf S$.
As $N_G(B)$ normalizes $C_{\mathbf G}(B)^\circ$, every $g\in N_G(B)$ satisfies $g\mathbf Sg^{-1}=\mathbf S$.
Hence,
$N_G(B)/(N_G(B)\cap S)$
is a subgroup of $C_8$.

\begin{itemize}
\item[{\bf Step 3.} ] We have $A=1$.
\end{itemize}
If $A=1$, then the proof is complete. Suppose that $|A|=3^a$ with $a\geq1$.
Choose $y\in G$ with $H=AB^y$ a subgroup.
Since $|B|$ is odd and prime to $3$, we have $|H|$ odd and $A$ a cyclic Sylow $3$-subgroup of $H$.
As $A\leq C_H(A)$, the order $|N_H(A)/C_H(A)|$ is prime to $3$.
We use the injective homomorphism $hC_H(A)\mapsto(g_0\mapsto hg_0h^{-1})$ into $\Aut(A)$, with $g_0\in A$.
Because $A\cong C_{3^a}$, we have $|\Aut(A)|=2\cdot3^{a-1}$.
As $|H|$ is odd, the quotient $N_H(A)/C_H(A)$ is trivial.
We obtain from Burnside's normal-complement theorem that $B^y\unlhd H$.
Hence, $A\leq N_G(B^y)$.
Since $|N_G(B^y)/(N_G(B^y)\cap S^y)|$ divides $8$ and $|A|$ is a power of $3$, we have $A\leq S^y$.
As $3\nmid|S|$ when $p=3$, we have $A=1$ in that characteristic.

Suppose that $p\ne3$ and $A\ne1$.
Let $c\in A$ have order $3$.
Let $w_9\in W(E_6)$ have characteristic polynomial $\Phi_9(X)$
\cite[Corollary 4.3 and Table 1]{SpringerRegular}.
For the torus normalizer quotient, let $\omega_9:N_G(T^x)/T^x\to\gen{w_9}$ be the rational Weyl group identification.
We choose $c$ with $\omega_9(cT^x)=w_9^3$.
For ${}^2E_6(q)$, the action of the map $F$ on the torus character lattice is $q(-w_9)$.
We have $C_{W(E_6)}(-w_9)\cong C_9$.
Let $\mathbf T$ be the algebraic torus containing $T$.
For the Cartan part of the representation, $\operatorname{Ad}(c)$ on $\operatorname{Lie}(\mathbf T^x)$ has polynomial $\Phi_3(X)^3$.
Consequently, we have $\det(w_9^3-I)=27$.
Thus, $\operatorname{Ad}(c)$ has no fixed Cartan vector in characteristic different from $3$.
The permutation of root spaces induced by $c$ consists of $24$ cycles of length $3$.
Since $c^3=1$ in the adjoint group, the product of the scalar factors around each cycle on root spaces is $1$.
Thus, each cycle contributes $1$ fixed vector.
Since $|c|=3$ is prime to the characteristic, $C_{\mathbf G}(c)$ is smooth.
Thus, $\dim C_{\mathbf G}(c)=\dim\mathfrak g^c=24$.

On the other hand, $c\in S^y$.
Let $\widetilde G=E_6(q)_{\rm sc}$ when $G=E_6(q)$.
For $G={}^2E_6(q)$, let $\widetilde G={}^2E_6(q)_{\rm sc}$.
Let $\eta:\widetilde G\to G$ be the natural central quotient.
Let $T_4$ be the torus in the rational subgroup of type $D_4$.
Since $|T_4|$ differs from $q^4+1$ by a power of $2$, we have $3\nmid|T_4|$.
Then, we have $\eta^{-1}(S^y)=T_4T_2$ and $[T_4,T_2]=1$.
We take $T_2$ to be the rational points of an algebraic torus of dimension $2$ centralizing $D_4$.
The intersection $T_4\cap T_2$ has order a power of $2$.
For the central quotient, we also have $\ker\eta\leq T_2$ and $|\ker\eta|\leq3$.
Because $[T_4T_2:T_2]$ is prime to $3$, a lift of $c$ with order a power of $3$ lies in $T_2$.
As $T_2$ centralizes $D_4T_2$, the image $c$ also centralizes $D_4T_2$.
Consequently, $\dim C_{\mathbf G}(c)\geq28+2=30$.
As $24<30$, the centralizer dimension bounds contradict each other. Therefore, $A=1$.
\end{proof}

\subsection{\texorpdfstring{The groups $F_4(q)$}{The groups F4(q)}}\label{sec:f4}

\begin{lemma}\label{lem:f4-odd-normalizer}
Every $F_4(q)$ contains a subgroup of odd order whose normalizer has odd order.
\end{lemma}
\begin{proof}
Since $F_4(q)$ has trivial centre, we use Lemma \ref{lem:split-root-normalizer} for odd $q$.
Suppose that $q=2^f$.
Take $\PSL_4(3)\leq F_4(2)\leq F_4(q)$ 
\cite{Atlas1985}.
For this subgroup $\PSL_4(3)$, let $P$ be the point parabolic of $\PSL_4(3)$.
Its Levi quotient is $\GL_3(3)/\gen{-I_3}\cong\SL_3(3)$.
By the Levi decomposition, we have $P=E\rtimes\SL_3(3)$ with $E\cong C_3^3$.
For the action on $E$, the Levi subgroup $\SL_3(3)$ induces an automorphism of order $13$.
If $E$ were toral in the algebraic group $\mathbf G$ of type $F_4$, then
the induced group $N_{\mathbf G}(E)/C_{\mathbf G}(E)$ would have order dividing $|W(F_4)|=1152$
\cite[Proposition 5.1]{AnDietrichLitterick}.
As $13\nmid|W(F_4)|$, $E$ is not contained in an algebraic torus.
The algebraic normalizer is 
$N_{\mathbf G}(E)=C_3^3\rtimes \SL_3(3)$
\cite[Proposition 6.1 and Table 3]{AnDietrichLitterick}.
The rational subgroup $P$ has order $|N_{\mathbf G}(E)|$.
Thus, $N_{F_4(q)}(E)=P$ for every $f\geq1$.

Take a Singer subgroup $C=C_{13}$ in the Levi subgroup $\SL_3(3)$ of $P$
and let $B=E\rtimes C$.
Since $E=O_3(B)$ is characteristic in $B$, we have $N_{F_4(q)}(B)\leq P$.
Let $\pi:P=E\rtimes\SL_3(3)\to\SL_3(3)$ satisfy $\pi(ej)=j$ for $e\in E$ and $j\in\SL_3(3)$.
As $B=\pi^{-1}(C)$, we have
$N_{F_4(q)}(B)=\pi^{-1}(N_{\SL_3(3)}(C))=E\rtimes N_{\SL_3(3)}(C)=C_3^3\rtimes(C_{13}\rtimes C_3)$.
For the last equality, we identify the natural module $\F_3^3$ of $\SL_3(3)$ with $\F_{27}$ over $\F_3$.
The irreducible linear Singer normalizer is
$\F_{27}^*\rtimes \operatorname{Gal}(\F_{27}/\F_3)$.
For $a\in\F_{27}^*$, the map $M_a:v\mapsto av$ has determinant $N_{\F_{27}/\F_3}(a)$.
The map $v\mapsto v^3$ has determinant $1$.
Hence, $\ker(\det:\F_{27}^*\rtimes\operatorname{Gal}(\F_{27}/\F_3)\to\F_3^*)=C_{13}\rtimes C_3$.
Therefore, $|B|=351$ and $|N_{F_4(q)}(B)|=1053$, both odd.
\end{proof}

\begin{theorem}\label{thm:f4}
For every prime power $q$, the group $F_4(q)$ has no proper nontrivial
ordinary $F_4(q)$-permutable subgroup.
\end{theorem}
\begin{proof}
Let $G=F_4(q)$ and let $A<G$ be ordinary $G$-permutable.
By \cite{LPS}, we have $G\ne XY$ for all $X,Y<G$.
Among the $2$ maximal subgroups of $G$ of type $A_2A_2$, we choose $C$ with
$|C|=2q^6(q^3-1)^2(q^2-1)^2$
\cite[Tables 7--8]{CravenExceptional}.
Since $A$ is ordinary $G$-permutable, we choose $y\in G$ with $AC^y$ a subgroup.
As $A,C^y<G$ and $G\ne AC^y$, we have $AC^y<G$.
By maximality of $C$, we have $A\leq C^y$ and $|A|\mid|C|$.
In particular, $|A|_p<q^{12}$ in characteristic $p$.

Let $T$ be a cyclic torus of order $t=q^4-q^2+1$.
The quotient $N_G(T)/T$ is $C_{12}$
\cite[Proposition 2.11 and Table 6]{GuralnickMalle2012}.
Firstly, suppose that $q>2$.
Every maximal subgroup $M<G$ with $T\leq M$ is of type ${}^3D_4(q)\rtimes C_3$.
Since $A$ is ordinary $G$-permutable, we choose $x\in G$ with $K=AT^x$ a subgroup.
As $A,T^x<G$ and $G\ne AT^x$, we have $K<G$.
We choose a maximal subgroup $M<G$ with $K\leq M$.
Let $L\unlhd M$ with $L\cong{}^3D_4(q)$.
Because $3\nmid|T^x|$, we have $T^x\leq L$.
Let $J=K\cap L$. Then, $J\unlhd K$.
As $|J|_p\leq|A|_p<q^{12}=|L|_p$, we have $J<L$.
Because $|T^x|=q^4-q^2+1$, the unique maximal subgroup of $L$ containing $T^x$ is $T^x\rtimes C_4$
\cite[Table 6]{GuralnickMalle2012}.
As $T^x=O_{2'}(J)$ is characteristic in $J\unlhd K$, we have $A\leq N_G(T^x)$.
The integer $t$ is odd and prime to $q$, and $q^6\equiv-1\pmod t$.
Hence, $(t,|C|)=1$.
Because $A\cap T^x=1$, the quotient map $a\mapsto aT^x$ injects $A$ into $N_G(T^x)/T^x\cong C_{12}$.

For $q=2$, by the linear and unitary $A_2A_2$ maximal subgroups, we have
$|A|\mid(2\cdot168^2,2\cdot216^2)=2^7\cdot3^2$.
Now, $|T|=13$.
Choose $x\in G$ with $K=AT^x$ a subgroup and choose a maximal subgroup $M<G$ containing $K$.
The normal subgroup $L\lhd M$ is ${}^3D_4(2)$, ${}^2F_4(2)'$ or $\PSL_4(3)$.
Moreover, the quotient $M/L$ has order $3$, $2$ or $2$ in the $3$ cases
\cite[Table 6]{GuralnickMalle2012}\cite{Atlas1985}.
Since $|T^x|=13$, we have $T^x\leq L$.
Let $J=K\cap L$. Then, $J\lhd K$ and $T^x\leq J$.
In the triality branch, we have $|J|_2\leq128<2^{12}$.
Hence, $J<{}^3D_4(2)$.
The unique maximal subgroup $M_1<{}^3D_4(2)$ with $T^x\leq M_1$ is $T^x\rtimes C_4$.
Hence, $T^x=O_{2'}(J)$ is characteristic in $J$.
As $|J|_2\leq128<2^{11}=|L|_2$, we have $J<L$ for $L={}^2F_4(2)'$.
The maximal subgroups of $L$ containing $T^x$ are $\PSL_3(3)\rtimes C_2$ and $\PSL_2(25)$
\cite{Atlas1985}.
In $\PSL_3(3)\rtimes C_2$, we have $|J|_3\leq9$.
Thus, $J\cap\PSL_3(3)<\PSL_3(3)$ and $J\cap\PSL_3(3)\leq C_{13}\rtimes C_3$ \cite[Tables 8.3--8.4]{BHR}.
Suppose that $J\leq M_2\cong\PSL_2(25)$ with $M_2<L$.
Because $5\nmid|J|$, we have $J<M_2$.
By Dickson's theorem, $J$ lies in $C_{13}\rtimes C_2$.
In both cases, the unique Sylow $13$-subgroup $T^x$
is characteristic in $J$.

For $L=\PSL_4(3)$, we have $|J|_3\leq9<|L|_3$.
Since $|J|_3<|L|_3$, we have $J<L$.
The only maximal subgroups of $\PSL_4(3)$ containing $T^x$ are the point and hyperplane
parabolics $R\rtimes \SL_3(3)$, with $R=C_3^3$
\cite[Tables 8.8--8.9]{BHR}.
For the action of $T^x$ on $R$, let $\rho_0:T^x\to\GL(R)$ be $\rho_0(t)(r)=trt^{-1}$.
The group $\rho_0(T^x)$ acts irreducibly on $R\cong\F_3^3$.
Thus, $J\cap R$ is trivial or $R$. As $|J|_3\leq9<|R|$, we have $J\cap R=1$.
We use the quotient map $\rho:R\rtimes\SL_3(3)\to\SL_3(3)$.
Since $J\cap R=1$, the subgroup $\rho(J)$ is isomorphic to $J$.
As $|\rho(J)|_3\leq9<27=|\SL_3(3)|_3$, we have $\rho(J)<\SL_3(3)$.
By \cite[Tables 8.3--8.4]{BHR}, we have $\rho(J)\leq C_{13}\rtimes C_3$.
Thus, $T^x$ is characteristic in $J$ when $J<\PSL_4(3)$.
Because $T^x$ is characteristic in $J\unlhd K$, we have $T^x\unlhd K$ and $A\leq N_G(T^x)$ in all $3$ cases.
As $13\nmid|A|$ and $|T^x|=13$, the quotient map $a\mapsto aT^x$ injects $A$ into $N_G(T^x)/T^x\cong C_{12}$.

We next use the torus $S$ of order $q^4+1$ corresponding to a regular Weyl element of order $8$,
whose Weyl quotient is $C_8$
\cite[Corollary 4.3 and Table 1]{SpringerRegular}.
Let $D=O_{2'}(S)$, of order $b=q^4+1$ for even $q$
and $b=(q^4+1)/2$ for odd $q$.
Next, choose $d_0\in D$ with $D=\gen{d_0}$ and $\lambda\in\overline{\F}_q^*$ with $|\lambda|=b$. In $D_4$ coordinates, each root character takes the value $\lambda^a$ at $d_0$, where the exponents $a$ are
$\pm q^i, \pm q^i\pm q^j\ (i<j), (\pm1\pm q\pm q^2\pm q^3)/2$.
For these exponents, every signed sum is nonzero and bounded by $1+q+q^2+q^3<b$.
At $q=2$, use $15<17$, and at $q=3$, use $40<41$.
Since $b$ is odd, division by $2$ is invertible modulo $b$.
Consequently, none of the displayed root characters is trivial on $D$.
Thus, the connected algebraic centralizer of $D$ is the torus of $S$.
The normalizer $N_G(D)$ lies in $N_G(S)$, whose order is prime to $3$.

Since $A$ is ordinary $G$-permutable, we choose $y\in G$ with $H=AD^y$ a subgroup.
The Sylow $2$-subgroup of $H$ is cyclic.
By Burnside's theorem, $H_0=A_{2'}D^y$ is the characteristic Hall subgroup of odd order in $H$.
Let $P_3$ be a Sylow $3$-subgroup of $H_0$.
Since $|A|\mid12$ and $3\nmid|D|$, we have $|P_3|\leq3$.
As $|H_0|$ is odd and $|\Aut(P_3)|\leq2$, we have $N_{H_0}(P_3)=C_{H_0}(P_3)$.
By Burnside's theorem, $D^y$ is the characteristic Hall subgroup of $H_0$ of order prime to $3$.
Consequently, $D^y\unlhd H$ and $A\leq N_G(D^y)$.
Thus, $3\nmid|A|$ and $A$ is cyclic of order at most $4$.

Finally, let $B$ be the subgroup of odd order in Lemma \ref{lem:f4-odd-normalizer}.
Since $A$ is ordinary $G$-permutable, we choose $z\in G$ with $AB^z$ a subgroup.
As $A$ is a cyclic $2$-group and $|B|$ is odd, $A$ is a cyclic Sylow $2$-subgroup of $AB^z$.
By Burnside's theorem, $B^z$ is normal in $AB^z$.
Hence, $A\leq N_G(B^z)$, which has odd order.
Therefore, $A=1$.
\end{proof}

\subsection{\texorpdfstring{The groups $E_7(q)$}{The groups E7(q)}}\label{sec:e7}

Let $\alpha_1,\ldots,\alpha_7$ be the simple roots for a maximal torus $\mathbf T$ in an algebraic group of type $E_7$.
The fundamental coweights $\omega_i^\vee\in\mathbb Q\otimes_{\mathbb Z}Y(\mathbf T)$ satisfy $\langle\alpha_j,\omega_i^\vee\rangle=\delta_{ij}$ for $1\leq i,j\leq7$.
For an integer $d\geq1$ with $d\omega_i^\vee\in Y(\mathbf T)$, we have a homomorphism $d\omega_i^\vee:k^*\to\mathbf T$.
We use the root subgroup construction of Weyl group representatives in \cite[Chapter 21]{MilneAlgebraicGroups}.

\begin{lemma}\label{lem:e7-primitive-regular}
Let $q\geq2$ be a prime power and $s\in\{14,18\}$.
Let $\ell$ be a primitive prime divisor of $q^s-1$.
The group $E_7(q)$ has a torus $T_s$ of order $(q+1)\Phi_s(q)/(2,q-1)$.
Let $\mathbf T_s$ be the algebraic torus containing $T_s$ in the adjoint algebraic group $\mathbf G$ of type $E_7$.
For every $B\leq T_s$ with $|B|=\ell$, we have $C_{\mathbf G}(B)^\circ=\mathbf T_s$.
The quotient $N_{E_7(q)}(T_{14})/T_{14}$ has order prime to $3$.
\end{lemma}
\begin{proof}
We use the simply connected algebraic group of type $E_7$ and the root lattice $\Lambda$ of type $E_7$.
The lattice has rank $7$ and determinant $2$.
Moreover, every root has squared length $2$.
For the Weyl group action, a Coxeter element $w_{18}$ has characteristic polynomial
$\Phi_{18}(X)(X+1)$, and $w_{18}^9=-I$.
For $w_{14}$, use minus a Coxeter element on the $A_6$ chain
$1,3,4,5,6,7$ in the Bourbaki diagram.
The characteristic polynomial of $w_{14}$ is $\Phi_{14}(X)(X+1)$.
No root lies in the negative eigenline of either element.
For $w_{18}$, the claim follows from the root orbits of length
$18$. For the $A_6$ chain, no root is perpendicular to all $6$ roots.

Choose a generator $b$ of the subgroup of order $\ell$ and $\zeta\in\overline{\F}_q^*$ with $|\zeta|=\ell$.
Since $\det\Lambda=2$ and $\ell$ is odd, a nonzero $v\in\Lambda/\ell\Lambda$ satisfies $\alpha(b)=\zeta^{(\alpha,v)}$ for every root $\alpha$.
We have $w_sv=qv$ or $w_sv=q^{-1}v$ modulo $\ell$.
Suppose that $(\alpha,v)=0\pmod\ell$.
Then, $(w_s^j\alpha,v)=0\pmod\ell$ for every integer $j$.
The component of $\alpha$ in $\ker\Phi_s(w_s)\leq\Lambda\otimes\mathbb Q$ is nonzero and has degree $6$.
Thus, the rational span $U$ of the $w_s$-orbit of $\alpha$ has dimension $6$ or $7$.
We choose a basis of roots from the orbit and let $\Lambda_0$ be
the lattice they generate.
The saturation $S=\Lambda\cap U$ of $\Lambda_0$ has integral Gram determinant
at least $1$.
By Hadamard's inequality, we have $[S:\Lambda_0]\leq8$ in dimension $6$.
As $\det\Lambda=2$ in dimension $7$, we also have $[S:\Lambda_0]\leq8$.
Because $\ell\geq19$, the chosen basis roots span $S$ modulo $\ell$.
The lattice pairing is nonsingular modulo $\ell$.
If $\dim U=7$, then we obtain that $v=0$.
If $\dim U=6$, then the annihilator of $U$ is the negative eigenline.
But the Weyl eigenvalue on $v$ has order $s$, and is not $-1$.
This contradicts the Weyl eigenvalue on $v$.

From the $2$ characteristic polynomials, we obtain the torus orders.
For $w_{14}$, we have the rational linear centralizer algebra $\mathbb Q(\zeta_{14})\times\mathbb Q$.
Its unit group $\mathbb Q(\zeta_{14})^*\times\mathbb Q^*$ has torsion subgroup of order $28$.
Since the Weyl centralizer is finite, its order is prime to $3$.
Thus, $N_{E_7(q)}(T_{14})/T_{14}$ has order prime to $3$.
Let $\widetilde G=E_7(q)_{\rm sc}$ be the finite group obtained from the simply connected algebraic group.
Let $\eta:\widetilde G\to E_7(q)$ be the central quotient, with $|\ker\eta|=(2,q-1)$.
For $B\leq T_s$ of order $\ell$, the subgroup $\eta^{-1}(B)$ has a unique Sylow $\ell$-subgroup $\widetilde B$.
The map $\eta|_{\widetilde B}:\widetilde B\to B$ is an isomorphism.
\end{proof}

\begin{theorem}\label{thm:e7}
For every prime power $q$, the group $E_7(q)$ has no proper
nontrivial ordinary $E_7(q)$-permutable subgroup.
\end{theorem}
\begin{proof}
For ease of reading, we divide the proof of the theorem into the following steps.

\setlist[itemize]{leftmargin=4em,labelsep=.5em,topsep=.2\baselineskip,partopsep=0pt,parsep=0pt,itemsep=0pt}
\begin{itemize}
\item[{\bf Step 1.} ] For $p=\operatorname{char}\F_q$, we have $|A|_p<q^9$.
\end{itemize}
Let $G=E_7(q)$.
For the algebraic group containing $G$, let $\mathbf G$ be the adjoint group of type $E_7$ over $\overline{\F}_q$, with $G\leq\mathbf G(\overline{\F}_q)$.
To specify the central quotient, let $\widetilde{\mathbf G}$ be the simply connected algebraic group of type $E_7$.
Let $\kappa_0:\widetilde{\mathbf G}\to\mathbf G$ be the central isogeny with $\ker\kappa_0=Z(\widetilde{\mathbf G})$.
On finite groups, $\kappa_0$ induces the central quotient $\eta:\widetilde G=E_7(q)_{\rm sc}\to G$.
Next, let $d=(2,q-1)$ and $e=(3,q+1)$.
Let $f=q^6-q^3+1$ and $r=f/e$.
Let $A<G$ be ordinary $G$-permutable.
Now, we take the cyclic torus $T=T_{18}$.
Then, we have $|T|=(q+1)f/d$.
Let $M$ be the unique maximal subgroup of $G$ with $T\leq M$.
For this subgroup, we have $L={}^2E_6(q)_{\rm sc}\unlhd M$
\cite[Proposition 2.11 and Table 6]{GuralnickMalle2012}.
The intersection $T_0=T\cap L$ is a cyclic torus of order $f$.
Since $A$ is ordinary $G$-permutable, we choose $x\in G$ with $AT^x$ a subgroup.
By \cite[Proposition 4.2]{LiXia}, we have $TY\ne G$ for every $Y<G$.
Thus, $AT^x<G$ and $A\leq M^x$.

The subgroup $C=\PSL_2(q^7)\rtimes C_7$ is maximal for every $q$
\cite[Theorem 1.1 and Table 4.1]{CravenE7}.
From the group orders, we have
$$
 \frac{|M||C|}{|G|}
 =\frac{14(q+1)(q^9+1)(q^5+1)}
 {d q^{20}(q^{10}-1)(q^{18}-1)}
 <448q^{-33}<1.
$$
Since $A$ is ordinary $G$-permutable, we choose $y\in G$ with $AC^y$ a subgroup.
As $|M||C|<|G|$, we have $AC^y<G$.
By maximality of $C$, we have $A\leq C^y$.
In particular,
$|A|_p\leq q^7(7)_p<q^9$.
Because $7<q^2$, the strict bound $|A|_p<q^9$ also holds when $p=7$.
Every prime divisor of $r$ is primitive for exponent $18$.
The only possible nonprimitive divisor of $f$ is $3$,
and $v_3(f)=1$ when $e=3$.
Hence, $(r,|C|)=1$.

\begin{itemize}
\item[{\bf Step 2.} ] $A$ normalizes $T$.
\end{itemize}
Choose $z\in G$ with $AT^z\leq M^z$.
Replacing $A,y$ by $A^{z^{-1}},y z^{-1}$, let $K=AT\leq M$.
Let $J=K\cap L$.
Since $L\unlhd M$ and $K\leq M$, we have $J\unlhd K$ and $T_0\leq J$.
Let $\kappa:L\to L/Z(L)$ be the central quotient map and let $\overline J=\kappa(J)$.
As $|\overline J|_p\leq|J|_p\leq|A|_p<q^9<q^{36}=|L/Z(L)|_p$, we have $\overline J<L/Z(L)$ and $J<L$.
Let $M_1$ be the unique maximal subgroup of $L/Z(L)$ with $\kappa(T_0)\leq M_1$.
By uniqueness of $M_1$, we have $\overline J\leq M_1$. Moreover, $L_1\unlhd M_1$, where $L_1=\PSU_3(q^3)$.
Let $J_1=\overline J\cap L_1$.
To use the unitary group, let $\kappa_1:\SU_3(q^3)\to\PSU_3(q^3)$ be the standard quotient.
We have a cyclic torus $\widetilde T_1\leq\kappa_1^{-1}(J_1)$ with $|\widetilde T_1|=f=q^6-q^3+1$.
Also, $|\kappa_1^{-1}(J_1)|_p<q^9$.
By Bereczky's theorem \cite{Bereczky}\cite[Theorem 2.1.7]{Brookfield2014}, we have $J_1\leq N_{\PSU_3(q^3)}(\kappa_1(\widetilde T_1))$.
Since $q^3\geq8$, Bereczky's theorem has no small exception.

The subgroup $\overline R$ of order $r$ is $O_{3'}(J_1)$,
and is characteristic in $J_1$.
Since $L_1\unlhd M_1$ and $\overline J\leq M_1$, we have $J_1\unlhd\overline J$.
Thus, $\overline R\unlhd\overline J$ and $\overline J$ lies in the corresponding Singer normalizer of $L/Z(L)$.
Let $X=(\kappa|_J)^{-1}(\overline R)$.
By the choice of $X$, the map $\kappa|_X:X\to\overline R$ is surjective, with central kernel $X\cap Z(L)$ of order dividing $3$.
By Schur--Zassenhaus, we choose a Hall $3'$-subgroup $R_1$ of $X$.
For every Hall $3'$-subgroup $R_2$ of $X$, we choose $z\in\ker(\kappa|_X)$ with $R_1^z=R_2$.
Since $\ker(\kappa|_X)\leq Z(X)$, we have $R_2=R_1$.
The unique Hall $3'$-subgroup is $R\leq T_0$ of order $r$.
Consequently, $R\unlhd J$ and $J/R$ is a $3$-group.
By Lemma \ref{lem:e7-primitive-regular}, we have $C_L(R)=T_0$.
Thus, $R=O_{3'}(J)$.
Since $R=O_{3'}(J)$ and $J\unlhd K$, we have $R\unlhd K$.
Let $\mathbf T$ be the algebraic torus containing $T$.
A subgroup of $R$ of primitive prime order has connected algebraic centralizer $\mathbf T$ by Lemma \ref{lem:e7-primitive-regular}.
As $R\leq T$, we obtain that $C_{\mathbf G}(R)^\circ=\mathbf T$.
Consequently, $K\leq N_G(T)$ and $A\leq N_G(T)$.
We have $N_G(T)/T\cong C_{18}$.

\begin{itemize}
\item[{\bf Step 3.} ] We have $|A\cap T|\mid9$ and $|A|\mid162$.
\end{itemize}
We now bound $U=A\cap T$.
Since $A\leq C^y$ and $(r,|C|)=1$, the subgroup $U$ lies in the Hall $r'$-subgroup $V$ of $T$.
Let $\mathbf D\leq\widetilde{\mathbf G}$ be the algebraic torus of dimension $1$ centralizing the $E_6$ derived Levi.
For its rational points, let $\widetilde D=\mathbf D(\F_q)\leq\widetilde G$.
Let $\rho:\widetilde{\mathbf G}\to\GL(\overline{\F}_q\otimes_{\F_q}W)$ be the minuscule representation on $W=\F_q^{56}$.
For this representation, the $\mathbf D$-weight spaces $W_j\leq\overline{\F}_q\otimes_{\F_q}W$ satisfy
$\dim W_1=\dim W_{-1}=27$ and $\dim W_3=\dim W_{-3}=1$.
Next, let $\alpha_1,\ldots,\alpha_7$ be the simple roots in Bourbaki's numbering.
To describe $\mathbf D$, we take twice the seventh fundamental coweight
$$
 \tau(\lambda)=\prod_{i=1}^7\alpha_i^\vee(\lambda^{c_i}),
 \  (c_1,\ldots,c_7)=(2,3,4,6,5,4,3).
$$
The map $\tau:\overline{\F}_q^*\to\mathbf D$ has image $\mathbf D$.
Then, $\rho(\tau(\lambda))|_{W_j}=\lambda^j I$.
We have $\ker\eta=\gen{\tau(-1)}$.
For the rational points, let $F$ be the Frobenius endomorphism on $\widetilde{\mathbf G}$.
On $\mathbf D$, we have $F(\tau(\lambda))=\tau(\lambda^{-q})$.
Therefore, $\widetilde D=\{\tau(\lambda):\lambda^{q+1}=1\}$.
Thus, $D=\eta(\widetilde D)\leq T$ has order $(q+1)/d$.
Since $f=er$ and $\gcd(r,q+1)=1$, the Hall $r'$-subgroup $V$ of $T$ has order $(q+1)e/d$.
Because $D\leq T$ and $|D|=(q+1)/d$, we have $D\leq V$ and $[V:D]=e$.
The index $[V:D]=e$ accounts for the possible central factor $3$ in $E_6$.

Since $T$ is abelian and $A/(A\cap T)$ is cyclic, $A$ is soluble.
Let $P$ be the maximal parabolic subgroup with $E_6$ Levi subgroup.
Since $A$ is ordinary $G$-permutable, we choose $z\in G$ with $AP^z$ a subgroup.
As $A$ is soluble and $P^z<G$, we have $AP^z\ne G$ by \cite[Proposition 4.2]{LiXia}.
Hence, $A\leq P^z$.
We have a subspace $W_0\leq W$ over $\F_q$ with $\dim_{\F_q}W_0=1$ and $\rho(\eta^{-1}(A))W_0=W_0$.
For $g\in A\cap D$, choose $\lambda$ with $\lambda^{q+1}=1$ and $\widetilde g=\tau(\lambda)$ satisfying $\eta(\widetilde g)=g$.
The eigenvalue of $\rho(\widetilde g)|_{W_0}$ is $\lambda^j\in\F_q^*$,
where $j\in\{1,-1,3,-3\}$.
Since $\lambda^j\in\F_q^*$ and $\lambda^{q+1}=1$, we have $\lambda^j=\lambda^{qj}=\lambda^{-j}$.
Thus, $\lambda^{2j}=1$.
After the quotient by $\gen{\tau(-1)}$, the order of $g$ divides $|j|$.
As $A\cap D$ is cyclic, we have $|A\cap D|\mid3$.
Because $U\leq V$ and $[V:D]=e\mid3$, we have $[U:U\cap D]\mid3$.
Hence, $|U|\mid9$.
Since $A/U$ is a subgroup of $C_{18}$, we have $|A|\mid162$.

\begin{itemize}
\item[{\bf Step 4.} ] We have $|A|\leq2$ and $A=1$ for odd $q$.
\end{itemize}
We take a primitive prime $\ell$ of $q^{14}-1$ and
$B=C_\ell$ in $T_{14}$.
The prime $\ell$ exists for every $q$, and $\ell\equiv1\pmod{14}$.
Since $A$ is ordinary $G$-permutable, we choose $x\in G$ with $K=AB^x\leq G$.
Let $n_\ell$ be the number of Sylow $\ell$-subgroups of $K$.
Note that $|A|\mid162$. Thus, $n_\ell\mid162$ and $n_\ell\equiv1\pmod\ell$.
If $n_\ell>1$, then $n_\ell\in\{2,3,6,9,18,27,54,81,162\}$.
Because $\ell\mid n_\ell-1$, we would have $\ell\in\{2,5,7,13,17,23,53\}$.
But no prime in that set is $1$ modulo $14$.
Therefore, $n_\ell=1$. Hence, $B^x\unlhd K$ and $A\leq N_G(B^x)$.
We obtain from Lemma \ref{lem:e7-primitive-regular} that $C_{\mathbf G}(B^x)^\circ=\mathbf T_{14}^x$.
As $B^x\unlhd K$, we have $A\leq N_G(T_{14}^x)$.
By Lemma \ref{lem:e7-primitive-regular}, the quotient $N_G(T_{14}^x)/T_{14}^x$ has order prime to $3$.
Thus, a Sylow $3$-subgroup $Q$ of $A$ lies in $T_{14}^x$ and is cyclic.
Let $\mathbf D_{14}$ be the algebraic torus of dimension $1$ centralizing the $A_6$ subgroup with rational points $\SU_7(q)^x$.
Because $3\nmid\Phi_{14}(q)$, we have $Q\leq\mathbf D_{14}\cap G$.
Every element of $Q$ has algebraic centralizer
dimension at least $49$.
If $p=3$, then $3\nmid|T_{14}|$.
Hence, $Q=1$.

Suppose that $p\ne3$ and $Q\ne1$, and let $c$ have order $3$ in $Q$.
If $Q\cap T\ne1$, then the cyclic group $Q$ has its unique subgroup of order $3$ in $Q\cap T$.
Therefore, $c\in U\leq V$.
If $e=1$, then $V=D$.
If $e=3$, then $3\mid|D|$.
Since $V$ is cyclic, every $h\in V$ with $|h|=3$ lies in $D$.
The rational lift of $c$ of order $3$ is $\tau(\lambda)$ with $|\lambda|=3$.
For a root $\beta$ outside the $E_6$ subsystem, we have $\beta(\tau(\lambda))=\lambda^2$ or $\lambda^{-2}$.
Thus, $C_{\mathbf G}(c)^\circ=E_6T_1$.
The rational centralizer $C_G(c)$ is unitary.
Every prime dividing $|C_G(c)|$ divides
$$
 (q+1)q^{36}(q^2-1)(q^5+1)(q^6-1)(q^8-1)(q^9+1)(q^{12}-1).
$$
Since $14$ divides none of $2,6,8,12,2\cdot1,2\cdot5,2\cdot9$, no primitive divisor of $q^{14}-1$ divides $|C_G(c)|$.
But $c$ centralizes $B^x$, a contradiction.

Thus, $Q\cap T=1$.
Let $\omega_{18}:N_G(T)/T\to\gen{w_{18}}\leq W(E_7)$ be the rational Weyl group identification for $T$.
Let $\nu_{18}:N_G(T)\to\gen{w_{18}}$ satisfy $\nu_{18}(g)=\omega_{18}(gT)$.
Since $Q\cap T=1$ and $|c|=3$, we have $\nu_{18}(c)\in\{w_{18}^6,w_{18}^{12}\}$.
Let $\mathbf T$ be the algebraic torus containing $T$.
For the Cartan part, we have a fixed space of dimension $1$ for $\operatorname{Ad}(c)|_{\operatorname{Lie}(\mathbf T)}$.
In addition, on the root spaces, $\nu_{18}(c)$ has $42$ cycles of length $3$.
As $|c|=3$, each cycle contributes $1$ fixed dimension.
The centralizer of $c$ is smooth, so we obtain that
$\dim C_{\mathbf G}(c)=43$.
This contradicts the previous lower bound $49$.
We conclude that $Q=1$, so $A$ is trivial or has order $2$.
For odd $q$, $|A|=2$ contradicts Lemma \ref{lem:split-root-normalizer}.
\begin{itemize}
\item[{\bf Step 5.} ] We have $A=1$ for even $q$.
\end{itemize}
Suppose that $q$ is even.
Let $E\rtimes\SL_3(3)\leq F_4(2)$ be the subgroup in Lemma \ref{lem:f4-odd-normalizer}, with $E\cong C_3^3$.
We use the chain $F_4(2)\leq E_6(2)\leq E_7(2)\leq E_7(q)$.
For the first inclusion, we realize $F_4(2)$ inside the subgroup of $E_6(2)$ fixed by its graph automorphism.
The inclusion $F_4(2)\to E_6(2)$ sends $g$ to $g$.
Next, we realize $E_6(2)$ as the derived $E_6$ Levi subgroup of $E_7(2)$.
The inclusion $E_6(2)\to E_7(2)$ sends each root element $x_\alpha(t)$ to $x_\alpha(t)$ for $t\in\F_2$.
Finally, the inclusion $E_7(2)\to E_7(q)$ sends each root element $x_\alpha(t)$ to $x_\alpha(t)$ for $t\in\F_2$.
For the subgroup used below, let $B=E\rtimes C_{13}$, of order $351$.
The characteristic subgroup $O_3(B)=E$ is not toral in $\mathbf G$.
Otherwise, the induced automorphism of $E$ of order $13$ would lie in
a Weyl quotient, whereas $13\nmid|W(E_7)|$
\cite[Proposition 5.1]{AnDietrichLitterick}.
For the unique class of subgroups $E\cong C_3^3$ not contained in an algebraic torus, we have
$\dim C_{\mathbf G}(E)=3, N_{\mathbf G}(E)/C_{\mathbf G}(E)\cong\SL_3(3)$
\cite[Proposition 6.1 and Table 3]{AnDietrichLitterick}.
Let $u\in N_G(B)$ have order $2$.
Because $E=O_3(B)$ is characteristic in $B$, we have $u\in N_{\mathbf G}(E)$.
The coset $uC_{\mathbf G}(E)$ normalizes the image of $C_{13}$ in $N_{\mathbf G}(E)/C_{\mathbf G}(E)\cong\SL_3(3)$.
Since $N_{\SL_3(3)}(C_{13})=C_{13}\rtimes C_3$ has odd order, the coset $uC_{\mathbf G}(E)$ is the identity.
Therefore, $u$ centralizes $E$.

As $|E|=27$ is prime to the characteristic, $C_{\mathbf G}(E)$ is smooth
\cite[Theorem 13.9]{MilneAlgebraicGroups}.
Let $\mathfrak g=\operatorname{Lie}(\mathbf G)$, of dimension $133$. Thus, $\dim\mathfrak g^E=3$.
For $\chi\in\operatorname{Hom}(E,\overline{\F}_2^*)$, let
$\mathfrak g_\chi=\{X\in\mathfrak g:\operatorname{Ad}(e)X=\chi(e)X\text{ for all }e\in E\}$.
For $s\in\SL_3(3)$ and $e\in E$, let $\chi^s(e)=\chi(s^{-1}es)$.
As $\SL_3(3)$ induces its natural action on $E\cong\F_3^3$, it is transitive on the $26$ characters $\chi\ne1$.
Hence, $\dim\mathfrak g_\chi=5$ for $\chi\ne1$.
For an involution $u$ centralizing $E$, we have $\operatorname{Ad}(u)\mathfrak g_\chi=\mathfrak g_\chi$ for every $\chi$.
As $u^2=1$ and the characteristic is $2$, we have $(\operatorname{Ad}(u)-I)^2=0$.
On each $\mathfrak g_\chi$, the image of $\operatorname{Ad}(u)-I$ lies in its kernel.
By rank-nullity, the rank on $\mathfrak g_\chi$ is at most $\lfloor\dim\mathfrak g_\chi/2\rfloor$.
Consequently, $\operatorname{rank}(\operatorname{Ad}(u)-I)\leq1+26\lfloor5/2\rfloor=53$.
Because $T$ is odd, each involution of $N_G(T)$ maps to $w_{18}^9=-I$ in $N_G(T)/T\cong C_{18}$.
The rational Tits representatives split the Weyl group in characteristic $2$.
Let $n$ represent $w_{18}^9$.
As squaring is bijective on $T$, for each involution $t\in N_G(T)$ there exists $z\in T$ with $t^z=n$.
On the Cartan subalgebra $\mathfrak t$, we have $\operatorname{Ad}(n)|_{\mathfrak t}=I$.
For each root $\alpha$, we have $\operatorname{Ad}(n)\mathfrak g_\alpha=\mathfrak g_{-\alpha}$.
The $126$ root spaces form $63$ pairs $\{\mathfrak g_\alpha,\mathfrak g_{-\alpha}\}$.
Hence, $\operatorname{rank}(\operatorname{Ad}(t)-I)=63$ for every involution $t\in N_G(T)$.

If $|A|=2$, then write $A=\gen t$ and choose $x_1\in G$ with $K=AB^{x_1}$ a subgroup.
Since $[K:B^{x_1}]=2$, we have $t\in N_G(B^{x_1})$ and $t^{x_1^{-1}}\in N_G(B)$.
As $t^{x_1^{-1}}$ is conjugate to $t$, we have $\operatorname{rank}(\operatorname{Ad}(t^{x_1^{-1}})-I)=63$.
But $t^{x_1^{-1}}\in N_G(B)$ requires $\operatorname{rank}(\operatorname{Ad}(t^{x_1^{-1}})-I)\leq53$, a contradiction.
Thus, $A=1$ for every even $q$ as well.
\end{proof}

\subsection{\texorpdfstring{The groups $E_8(q)$}{The groups E8(q)}}\label{sec:e8}

\begin{lemma}\label{lem:e8-primitive-regular}
Let $q\geq2$ be a prime power and $d\in\{20,24\}$.
Let $\ell$ be a primitive prime divisor of $q^d-1$.
There are a torus $S\leq E_8(q)$ and a subgroup $B\leq S$ with $|S|=\Phi_d(q)$ and $|B|=\ell$.
Moreover, $N_{E_8(q)}(B)\leq N_{E_8(q)}(S)$ and $|N_{E_8(q)}(S)|=d\Phi_d(q)$.
\end{lemma}
\begin{proof}
The Weyl group has a regular element $w_d$ with irreducible
characteristic polynomial $\Phi_d(X)$ and centralizer $C_d$
\cite[Corollary 4.3 and Table 3]{SpringerRegular}.
For this Weyl element, let $S$ be the corresponding rational torus, with $|S|=\Phi_d(q)$.
Since Zsigmondy's theorem has no exception at either degree, we take $B=C_\ell\leq S$.
Since $\ell\equiv1\pmod d$, we have $\ell\geq41$ for $d=20$
and $\ell\geq73$ for $d=24$.

Let $\Lambda$ be the unimodular $E_8$ root lattice.
Let $\zeta\in\overline{\F}_q^*$ have order $\ell$, and let $b$ generate $B$.
Since $\Lambda$ is unimodular, we represent $b$ by a nonzero $v\in\Lambda/\ell\Lambda$ with $\alpha(b)=\zeta^{(\alpha,v)}$ for every root $\alpha$.
As $v$ represents a rational point of $S$, we have $w_dv=qv$ or $w_dv=q^{-1}v$.
If $\alpha|_B=1$, then $(w_d^j\alpha,v)=0\pmod\ell$ for every integer $j$ and either Weyl eigenvalue.
Because $\Phi_d(X)$ is irreducible of degree $8$, the roots $\alpha,w_d\alpha,\ldots,w_d^7\alpha$ are linearly independent over $\mathbb Q$.
Each has Euclidean norm $\sqrt2$.
The sublattice they generate has index at most
$(\sqrt2)^8=16$
by Hadamard's inequality and unimodularity.
Thus, the $8$ roots $\alpha,w_d\alpha,\ldots,w_d^7\alpha$ span $\Lambda$ modulo $\ell$.
As every chosen root annihilates $v$, we would have $v=0$.
This contradicts the choice of $v$.
For every root $\alpha$, we have $\alpha|_B\ne1$.
The connected algebraic centralizer of $B$ is the underlying torus of $S$.
Every $g\in N_{E_8(q)}(B)$ normalizes the algebraic torus containing $S$.
Hence, $N_{E_8(q)}(B)\leq N_{E_8(q)}(S)$, of order $d\Phi_d(q)$.
\end{proof}

\begin{theorem}\label{thm:e8}
For every prime power $q$, the group $E_8(q)$ has no proper nontrivial
ordinary $E_8(q)$-permutable subgroup.
\end{theorem}
\begin{proof}
Let $G=E_8(q)$.
For its algebraic realization, let $\mathbf G$ be the group of type $E_8$ over $\overline{\F}_q$, with $G=\mathbf G(\F_q)$.
Let $A<G$ be ordinary $G$-permutable.
We first choose the Coxeter torus $T\leq G$, of order
$t_T=\Phi_{30}(q)=q^8+q^7-q^5-q^4-q^3+q+1$.
The unique maximal subgroup $M<G$ with $T\leq M$ is the soluble normalizer $T\rtimes C_{30}$
\cite[Proposition 2.11 and Table 6]{GuralnickMalle2012}.
By \cite[Proposition 4.2]{LiXia}, we have $G\ne XY$ whenever $X,Y<G$ and at least one of $X,Y$ is soluble.
Since $A$ is ordinary $G$-permutable, we choose $x\in G$ with $AT^x$ a subgroup.
Because $T$ is soluble, $AT^x<G$.
Thus, $A\leq N_G(T^x)$.
In particular, $A$ is soluble.
Replacing $T$ by $T^x$, assume that $A\leq N_G(T)$.

Let $P$ be the standard maximal parabolic subgroup with $E_7$ Levi subgroup.
Since $A$ is ordinary $G$-permutable, we choose $y\in G$ with $AP^y$ a subgroup.
As $A$ is soluble and $P^y<G$, we have $AP^y\ne G$ by the cited product condition.
By Lemma \ref{lem:max}, we have $A\leq P^y$.
The order of $P$ is
$|P|=q^{120}(q-1)\prod_{d\in\{2,6,8,10,12,14,18\}}(q^d-1)$.
A nonprimitive prime divisor of $\Phi_{30}(q)$ would belong to $\{2,3,5\}$.
For a nonprimitive divisor $\ell\in\{2,3,5\}$, we would need $\ord_\ell(q)=30/\ell$.
But $30/\ell\nmid\ell-1$ for every $\ell\in\{2,3,5\}$.
Consequently, every prime dividing $\Phi_{30}(q)$ is primitive for exponent $30$.
Hence, $(t_T,|P|)=1$ and $A\cap T=1$.
Let $w\in W(E_8)$ be the Coxeter element associated with $T$.
Let $\omega_T:N_G(T)/T\to\gen w$ be the rational Weyl group identification.
For $g\in N_G(T)$, let $\nu_T(g)=\omega_T(gT)$.
Because $A\cap T=1$, the homomorphism $\nu_T:N_G(T)\to\gen w$ is injective on $A$.
Thus, $A$ is cyclic and $|A|\mid30$.

For $d\in\{20,24\}$, let $B_d=C_{\ell_d}$ be the subgroup in Lemma \ref{lem:e8-primitive-regular}.
Since $A$ is ordinary $G$-permutable, we choose $z_d\in G$ with $K_d=AB_d^{z_d}\leq G$.
Let $n_{\ell_d}$ be the number of Sylow $\ell_d$-subgroups of $K_d$.
By Sylow's theorem, $n_{\ell_d}\mid|A|$ and $n_{\ell_d}\equiv1\pmod{\ell_d}$.
As $|A|\leq30<\ell_d$, we have $n_{\ell_d}=1$.
Thus, $B_d^{z_d}\unlhd K_d$ and $A\leq N_G(B_d^{z_d})$.
Now,
$\Phi_{20}(q)=q^8-q^6+q^4-q^2+1, \Phi_{24}(q)=q^8-q^4+1$.
As $\Phi_{20}(q)\equiv1\pmod3$ and $3\nmid20$, the order $20\Phi_{20}(q)$ is prime to $3$.
Because $\Phi_{24}(q)\equiv1\pmod5$ and $5\nmid24$, the order $24\Phi_{24}(q)$ is prime to $5$.
Therefore, $|A|=1$ or $2$.
For odd $q$, we have $A=1$ by Corollary \ref{cor:split-odd-involutions}.

Suppose that $q=2^f$ and $|A|=2$.
We use the rational subgroup chain
$F_4(2)\leq E_6(2)\leq E_8(2)\leq E_8(q)$.
To specify the inclusions, we realize $F_4(2)$ as a subgroup of $E_6(2)$ fixed by its graph automorphism
\cite[Table 9]{CravenExceptional}.
The inclusion map $F_4(2)\to E_6(2)$ sends $g$ to $g$.
Next, we realize $E_6(2)$ by the split $E_6$ root subsystem of $E_8(2)$.
The inclusion $E_6(2)\to E_8(2)$ sends each root element $x_\alpha(t)$ to $x_\alpha(t)$ for $t\in\F_2$.
Finally, the inclusion $E_8(2)\to E_8(q)$ sends each root element $x_\alpha(t)$ to $x_\alpha(t)$ for $t\in\F_2$.
By Lemma \ref{lem:f4-odd-normalizer}, we have $E\rtimes S\leq F_4(2)$.
We take $E=C_3^3$ and $S=\SL_3(3)$ with a faithful action on $E$.
Then, we take $C=C_{13}\leq S$ and let $B=E\rtimes C$.

The group $S$ induces an automorphism of $E$ of order $13$, and $13\nmid|W(E_8)|=696729600$.
By \cite[Proposition 5.1]{AnDietrichLitterick}, $E$ is not contained in an algebraic torus of $\mathbf G$.
Since $S$ is transitive on $E\setminus\{1\}$, all nonidentity elements of $E$ are conjugate in $\mathbf G$.
Thus, $E\cong C_3^3$ belongs to class $(3^3)_b$.
We have $C_{\mathbf G}(E)=E\times G_2$ and $N_{\mathbf G}(E)/C_{\mathbf G}(E)=\SL_3(3)$ \cite[Proposition 6.1 and Table 3]{AnDietrichLitterick}.
The subgroup $S$ maps onto $N_{\mathbf G}(E)/C_{\mathbf G}(E)$.

Every involution $u\in N_G(B)$ normalizes $E=O_3(B)$.
We take $\rho:N_{\mathbf G}(E)\to\Aut(E)$ with $\rho(n)(e)=nen^{-1}$.
Then, $\rho(u)\in N_{\rho(S)}(\rho(C))\cong C_{13}\rtimes C_3$.
Since $\rho$ is a homomorphism and $u^2=1$, we have $\rho(u)^2=1$.
As $|C_{13}\rtimes C_3|$ is odd, we have $\rho(u)=1$.
Therefore, $u$ centralizes $E$.
Let $\mathfrak g=\operatorname{Lie}(\mathbf G)$ and $V=\mathfrak g$, the adjoint module of dimension $248$ over $\overline{\F}_q$.
As $|E|$ is prime to the characteristic, $\operatorname{Ad}|_E:E\to\GL(V)$ is semisimple
and $C_{\mathbf G}(E)$ is smooth
\cite[Theorem 13.9]{MilneAlgebraicGroups}.
Thus, $\dim V^E=14$.
For $\chi\in\operatorname{Hom}(E,\overline{\F}_2^*)$, let
$V_\chi=\{v\in V:\operatorname{Ad}(e)v=\chi(e)v\text{ for all }e\in E\}$.
For $s\in S$ and $e\in E$, let $\chi^s(e)=\chi(s^{-1}es)$.
Because $S=\SL_3(3)$ is transitive on the $26$ characters $\chi\ne1$, we have $\dim V_\chi=(248-14)/26=9$.
We have $\operatorname{Ad}(u)V_\chi=V_\chi$ for every $\chi$.
For a space $W_0$ over $\overline{\F}_2$ with $\dim W_0=n$, let $t\in\GL(W_0)$ be an involution.
Then, $(t-I_{W_0})^2=0$, so $\operatorname{im}(t-I_{W_0})\leq\ker(t-I_{W_0})$.
By rank-nullity, $\operatorname{rank}(t-I_{W_0})\leq\lfloor n/2\rfloor$.
Hence,
$\operatorname{rank}(\operatorname{Ad}(u)-I) \leq7+26\cdot4=111$.

Write $A=\gen t$.
Then, $\nu_T(t)=w^{15}=-I$, where $w$ is the Coxeter element associated with $T$.
In characteristic $2$, the Tits representatives form a copy of the Weyl group.
Let $n$ represent $w^{15}$.
Since $n$ commutes with the Coxeter representative, we have $n\in N_G(T)$.
As $|T|$ is odd, we choose $z\in T$ with $t^z=n$ by Schur--Zassenhaus.
On the Cartan subalgebra $\mathfrak t$ with $\dim\mathfrak t=8$, we have $\operatorname{Ad}(n)|_{\mathfrak t}=I$.
For each root $\alpha$, we have $\operatorname{Ad}(n)\mathfrak g_\alpha=\mathfrak g_{-\alpha}$.
The $240$ root spaces form $120$ pairs $\{\mathfrak g_\alpha,\mathfrak g_{-\alpha}\}$.
Hence, $\operatorname{rank}(\operatorname{Ad}(t)-I)=120$.

Since $A$ is ordinary $G$-permutable, we choose $x_1\in G$ with $K=AB^{x_1}$ a subgroup.
As $[K:B^{x_1}]=2$, we have $t\in N_G(B^{x_1})$ and $t^{x_1^{-1}}\in N_G(B)$.
We have $\operatorname{rank}(\operatorname{Ad}(t^{x_1^{-1}})-I)=\operatorname{rank}(\operatorname{Ad}(t)-I)$.
Hence, $\operatorname{rank}(\operatorname{Ad}(t^{x_1^{-1}})-I)=120\leq111$, a contradiction.
Therefore, $A=1$ for even $q$ also.
\end{proof}

\subsection{\texorpdfstring{Ree groups ${}^2F_4(q)$ and ${}^2G_2(q)$}{Ree groups 2F4(q) and 2G2(q)}}\label{sec:large-ree}

\begin{theorem}\label{thm:large-ree}
Let $q=2^{2n+1}$ with $n\geq1$.
Then, ${}^2F_4(q)$ has no proper nontrivial ordinary
${}^2F_4(q)$-permutable subgroup.
\end{theorem}
\begin{proof}
For ease of reading, we divide the proof of the theorem into the following steps.

\setlist[itemize]{leftmargin=4em,labelsep=.5em,topsep=.2\baselineskip,partopsep=0pt,parsep=0pt,itemsep=0pt}
\begin{itemize}
\item[{\bf Step 1.} ] $A$ is cyclic and $|A|\mid4$.
\end{itemize}
Let $G={}^2F_4(q)$ and $s=\sqrt{2q}$.
By \cite[Proposition 4.2]{LiXia}, we have $G\ne XY$ whenever $X,Y<G$ and at least one of $X,Y$ is soluble.
We take cyclic tori $T_+$ and $T_-$ of orders $t_\pm=q^2+q+1\pm s(q+1)$.
By \cite[Proposition 2.11 and Table 6]{GuralnickMalle2012}, the unique maximal subgroup $M<G$ with $T_+\leq M$ is $N_G(T_+)$.
We have $N_G(T_+)/T_+\cong C_{12}$.
By \cite[Theorem 20]{WilsonSuzukiRee}, both $N_G(T_-)=T_-\rtimes C_{12}$ and $\Sz(q)\wr S_2$ are maximal.
Both $t_\pm$ are odd and prime to $3(q+1)$.
Their greatest common divisor divides $2s(q+1)$ and is $1$.

Let $A<G$ be ordinary $G$-permutable.
Since $A$ is ordinary $G$-permutable, we choose $x\in G$ with $AT_+^x$ a subgroup.
As $T_+^x$ is soluble, we have $AT_+^x<G$ by the cited condition on $X,Y<G$.
Thus, $A\leq N_G(T_+^x)$ and $A$ is soluble.
Because $A$ is ordinary $G$-permutable, we choose $y\in G$ with $A N_G(T_-)^y$ a subgroup.
As $A$ is soluble, we have $A N_G(T_-)^y<G$ by the cited condition on $X,Y<G$.
By maximality of $N_G(T_-)$, we have $A\leq N_G(T_-)^y$.
Consequently, $|A|\mid12$.
Because $A\cap T_+^x=1$, the subgroup $A$ is cyclic.
The maximal subgroup $\Sz(q)\wr S_2$ has order prime to $3$.
As $A$ is ordinary $G$-permutable, we choose $z\in G$ such that $A(\Sz(q)\wr C_2)^z$ is a subgroup.
Because $A$ is soluble, we have $A(\Sz(q)\wr C_2)^z\ne G$.
By maximality, we have $A\leq(\Sz(q)\wr C_2)^z$.
Hence, $3\nmid|A|$ and $|A|\in\{1,2,4\}$.

\begin{itemize}
\item[{\bf Step 2.} ] Every $2$-subgroup $P\leq G$ with $u\in P$ lies in $U$.
\end{itemize}
The group has $2$ involution classes $\mathcal I_s$ and $\mathcal I_b$.
For these classes, the centralizer orders are $q^{12}(q^2+1)(q-1)$ and $q^{10}(q^2-1)$ \cite[Table A.2]{HimstedtHuang2009}.
To identify the involutions in $N_G(T_+)$, we choose $t_0\in N_G(T_+)$ with $|t_0|=2$ and $t_0^{-1}rt_0=r^{-1}$ for every $r\in T_+$.
Then, $t_0\in\mathcal I_b$ by \cite[Corollary 4.4]{BurnessThomasFixity}.
Since $T_+$ is odd and $N_G(T_+)/T_+\cong C_{12}$,
for each involution $t\in N_G(T_+)$ there exists $z\in T_+$ with $t^z=t_0$.
For $|A|\in\{2,4\}$, the unique involution of $A$ belongs to $\mathcal I_b$.
We next construct a unipotent subgroup for $|A|\in\{2,4\}$.
Let $\theta=2^n$ and let $U$ be a Sylow $2$-subgroup of $G$.
For $1\leq i\leq12$, we use the rational root maps $\alpha_i:\F_q\to U$ from \cite[Section 2]{HimstedtHuang2009}.
Every $g\in U$ has a unique ordered expression $g=\alpha_1(t_1)\cdots\alpha_{12}(t_{12})$ with $t_i\in\F_q$.
For $1\leq i\leq12$, let $t_i:U\to\F_q$ send $g$ to $t_i$ in that expression.
For $i\in\{1,4,5,6\}$, multiplication satisfies
$$
 \alpha_i(t)\alpha_i(u)
 =\alpha_i(t+u)\alpha_{j(i)}(tu^{2\theta}), (j(1),j(4),j(5),j(6))=(2,8,12,11).
$$
For $i\notin\{1,4,5,6\}$, we have $\alpha_i(t)\alpha_i(u)=\alpha_i(t+u)$.
The corrected commutator relations are in \cite[Table 1]{HimstedtHuang2009}.

We take $u=\alpha_1(1)\alpha_3(1)$ and $C=\gen{u}$.
Then, we collect the factors and obtain
$$
\begin{aligned}
 v=u^2&=\alpha_2(1)\alpha_4(1)\alpha_5(1)
              \alpha_7(1)\alpha_{11}(1)\alpha_{12}(1),\\
 u^4&=\alpha_7(1)\alpha_8(1)\alpha_9(1)\alpha_{12}(1),\\
 u^8&=\alpha_{12}(1), u^{16}=1.
\end{aligned}
$$
Hence, $|C|=16$, and the unique involution $u^8$ of $C$ belongs to $\mathcal I_s$
\cite[Table A.4]{HimstedtHuang2009}.
Let $\mathbf G$ be the algebraic group of type $F_4$ over $\overline{\F}_2$ containing $G$.
In the root numbering of \cite[Section 2]{HimstedtHuang2009}, the roots $r_1,r_2,r_3,r_4$ are simple and $r_6=r_2+r_3$.
Since $\alpha_1(1)=x_3(1)x_2(1)x_6(1)$ and $\alpha_3(1)=x_4(1)x_1(1)$,
the $4$ absolute simple-root coordinates of $u$ are all $1$.
As the $4$ simple-root coordinates are nonzero, we obtain from Steinberg's criterion \cite[Lemma 3.2 and Theorem 3.3, pp.~54--55]{SteinbergRegular} that $u$ is regular unipotent and belongs to a unique algebraic Borel subgroup $\mathbf B$.
No restriction on the characteristic is required in the cited criterion.
Let $R_u(\mathbf B)$ be the unipotent radical of $\mathbf B$.
By the root maps, we have $U=R_u(\mathbf B)\cap G$.
If a finite $2$-subgroup $P\leq G$ contains $u$, then $P\leq U^g$ for some $g\in G$.
Because $u\in\mathbf B^g$ and $\mathbf B$ is unique, we have $\mathbf B^g=\mathbf B$.
Therefore, $U^g=R_u(\mathbf B^g)\cap G=U$ and $P\leq U$.

\begin{itemize}
\item[{\bf Step 3.} ] Every involution in $C_U(u^2)$ belongs to $\mathcal I_s$.
\end{itemize}
We claim that every involution in $U$ centralizing $v$ belongs to $\mathcal I_s$.
To prove the claim, let $g=\alpha_1(t_1)\cdots\alpha_{12}(t_{12})\in C_U(v)$ be an involution.
In each row, we use the zero equations from the preceding rows.
For a row labelled by $\alpha_i$, the coordinate function is $t_i:U\to\F_q$.
We use $t_i(gv)=t_i(vg)$ in a row with $gv=vg$.
In addition, in a row with $g^2=1$, we use $t_i(g^2)=0$.
In the last column, every listed expression equals $0$.
\begin{center}
\renewcommand{\arraystretch}{1.12}
\begin{tabular}{|c|c|c|}
\hline
Condition & Coordinate & Required zero\\\hline
$g^2=1$ & $\alpha_2$ & $t_1^{2\theta+1}$\\
$gv=vg$ & $\alpha_6$ & $t_3$\\
$g^2=1$ & $\alpha_8$ & $t_4^{2\theta+1}$\\
$gv=vg$ & $\alpha_7$ & $t_2$\\
$gv=vg$ & $\alpha_9$ & $t_5$\\
$g^2=1$ & $\alpha_{11}$ & $t_6^{2\theta+1}$\\
$gv=vg$ & $\alpha_{10}$ & $t_7+t_8$\\
$g^2=1$ & $\alpha_{12}$ & $t_7t_8$\\
$gv=vg$ & $\alpha_{11}$ & $t_9$\\
$gv=vg$ & $\alpha_{12}$ & $t_{10}$\\
\hline
\end{tabular}
\end{center}
By the seventh row, we have $t_7=t_8$.
By the eighth row, we have $t_7=t_8=0$.
Thus, $g=\alpha_{11}(t_{11})\alpha_{12}(t_{12})$.
If $t_{11}\ne0$, then let $x=\alpha_3(t_{12}/t_{11})\in U$.
Then, $g^x=\alpha_{11}(t_{11})$.
By torus scaling, we choose $h\in N_G(U)$ with $g^{xh}=\alpha_{11}(1)$.
If $t_{11}=0$, then we choose $h\in N_G(U)$ with $g^h=\alpha_{12}(1)$ by torus scaling.
Both representatives belong to $\mathcal I_s$.
Hence, every involution $g\in C_U(v)$ belongs to $\mathcal I_s$.

\begin{itemize}
\item[{\bf Step 4.} ] We have $A=1$.
\end{itemize}
Firstly, suppose that $|A|=2$ and write $A=\gen{a}$.
The element $a$ lies in $\mathcal I_b$, whereas the involution of $C$ lies in $\mathcal I_s$.
Consequently, $A\cap C^x=1$ for every $x\in G$.
Since $A$ is ordinary $G$-permutable, we choose $x\in G$ with $H=AC^x$ a subgroup of order $32$.
Replacing $H,A,a$ by $H^{x^{-1}},A^{x^{-1}},a^{x^{-1}}$, assume that $H=AC$.
Since $u$ is regular, $H\leq U$.
As $[H:C]=2$, $C\lhd H$.
Thus, $a^{-1}ua=u^k$ with $k\in\{1,7,9,15\}$ modulo $16$.
By the first square constraint, we have $t_1(a)=0$.
By the commutator relations, only elements of $\alpha_1(\F_q)$ can change $t_2(g)$ when conjugating $g\in U$.
Because $t_1(a)=0$ and $t_2(u)=0$, we have $t_2(a^{-1}ua)=0$.
But $t_2(u^k)=((k-1)/2)\cdot1_{\F_q}$ for odd $k$.
As $\operatorname{char}\F_q=2$, we obtain $k\equiv1\pmod4$.
Therefore, $k\in\{1,9\}$.
Because $k\in\{1,9\}$, we have $a^{-1}va=u^{2k}=u^2=v$.
By the calculation in $C_U(v)$, we have $a\in\mathcal I_s$.
This contradicts $a\in\mathcal I_b$.

Now, suppose that $|A|=4$ and $A=\gen{a}$.
Since the unique involutions of $A$ and $C$ lie in different $G$-classes, we have $A\cap C^x=1$ for every $x\in G$.
Because $A$ is ordinary $G$-permutable, we choose $x\in G$ with $H=AC^x$ a subgroup of order $64$.
Replacing $H,A,a$ by $H^{x^{-1}},A^{x^{-1}},a^{x^{-1}}$, assume that $H=AC\leq U$.
Let $\rho:H\to S_{H/C}$ satisfy $\rho(h)(gC)=hgC$ for $h,g\in H$.
As $[H:C]=4$, the image $\rho(H)\leq S_4$ is a $2$-group of order at most $8$.
Thus, $\operatorname{Core}_H(C)$ has order at least $8$.
The cyclic subgroup $\Core_H(C)$ has characteristic subgroup $\gen v$ of order $8$.
Consequently, $\gen{v}\unlhd H$.
Since $\gen v\unlhd H$, conjugation by $a$ induces an automorphism of $\gen v\cong C_8$.
As $\Aut(C_8)$ has exponent $2$, we have $a^{-2}va^2=v$.
By the calculation in $C_U(v)$, we have $a^2\in\mathcal I_s$.
This contradicts $a^2\in\mathcal I_b$.
For $|A|=2$ and $|A|=4$, we have obtained contradictions from $AC^x\leq G$.
Therefore, $A=1$ throughout the family.
\end{proof}

\begin{theorem}\label{thm:ree}
Let $G={}^2G_2(q)$, where $q=3^{2m+1}$ and $m\geq1$. Then, $G$
has no proper nontrivial ordinary $G$-permutable subgroup.
\end{theorem}
\begin{proof}
By \cite{LPS}, we have $G\ne XY$ for all $X,Y<G$.
We use the maximal torus normalizers of orders $6(q+3^{m+1}+1)$ and $6(q-3^{m+1}+1)$
\cite[proof of Theorem 1, Case II]{GT23}.
Since the $2$ torus orders are coprime, we obtain the following bound.
By Lemma \ref{lem:max}, we have $|A|\mid6$
for any proper ordinary $G$-permutable subgroup $A$.

We first construct an odd-order torus $D$ with $3\nmid|N_G(D)|$.
Let $s=3^m$ and $d=(q-1)/2$.
We use $V=\F_q^7$ and the split torus $H=\{h(\lambda):\lambda\in\F_q^*\}$.
For $\lambda\in\F_q^*$, we have
$$
 h(\lambda)=\diag(\lambda^s,\lambda^{1-s},\lambda^{2s-1},1,
                  \lambda^{1-2s},\lambda^{s-1},\lambda^{-s}).
$$
The normalizer $N_G(H)$ is dihedral of order $2(q-1)$
\cite[Proposition 3.1(9)]{Baarnhielm2014}. Let $D$ be the subgroup of $H$
of order $d$. Since $4s-2<d$ for $s\geq3$, the $7$ displayed exponents are distinct modulo $d$.
Hence, a generator of $D$ has $7$ distinct eigenvalues. The centralizer of this generator fixes the first and
last coordinate lines. Their common stabilizer in $G$ is $H$
\cite[Proposition 3.2]{Baarnhielm2014}. Thus, $C_G(D)=H$ and $N_G(D)\leq N_G(H)$.
We have $3\nmid|N_G(H)|$.

Suppose that $|A|=3$.
Since $A$ is ordinary $G$-permutable, we choose $x\in G$ with $K=AD^x$ a subgroup.
We have $|K|=3d$.
As $D^x$ has odd order and index $3$ in $K$, we obtain from Lemma \ref{lem:index3} that $D^x\unlhd K$.
Thus, $A\leq N_G(D^x)$.
Because $3\nmid|N_G(D^x)|$, we have a contradiction.

Suppose that $|A|=6$.
Since $A$ is ordinary $G$-permutable, we choose $x\in G$ with $K=AD^x$ a subgroup.
Since $\gcd(|A|,|D|)=1$, we have $|K|=6d$.
For the regular action of $K$ on itself, the sign homomorphism $\epsilon:K\to\{1,-1\}$ sends $g$ to $\operatorname{sgn}(k\mapsto gk)$, where $k\in K$.
An involution acts as $3d$ transpositions, so $\epsilon$ is nontrivial.
As $\epsilon(K)=\{1,-1\}$, the kernel $L=\ker\epsilon$ has order $3d$.
Moreover, every subgroup of $K$ of odd order has trivial image under $\epsilon$ and lies in $L$.
Hence, $D^x\leq L$.
We also have a subgroup $R\leq A\cap L$ of order $3$.
As $|D|$ is odd and $[L:D^x]=3$, we obtain from Lemma \ref{lem:index3} that $D^x\unlhd L$.
Thus, $R\leq N_G(D^x)$.
Because $|R|=3$ and $3\nmid|N_G(D^x)|$, we have a contradiction.

Now, suppose that $|A|=2$. We construct a cyclic subgroup
$B$ of order $9$ with $|N_G(B)|$ odd.
For this construction, let $U$ be a Sylow $3$-subgroup and let $e=3^{m+1}$.
We use the coordinates $U=\{u(a,b,c):(a,b,c)\in\F_q^3\}$
\cite[(3.10)--(3.13)]{Baarnhielm2014}.
They satisfy
$$
\begin{aligned}
 u(a,b,c)^{-1}
   &=u(-a,-b-a^{e+1},-c-ab+a^{e+2}),\\
 u(a,b,c)^{u(v,w,z)}
   &=u(a,b-av^e+va^e,\,*).
\end{aligned}
$$
Only the first $2$ coordinates in the second formula are used below.
For conjugation by the involution $j=h(-1)$, we have
$u(a,b,c)^j=u(-a,b,-c)$.
In addition, every element $u(a,b,c)$ with $a\ne0$ has order $9$.
Moreover, distinct Sylow $3$-subgroups intersect trivially and
$N_G(U)=U\rtimes H$ \cite[Proposition 3.1(2), (5), (6)]{Baarnhielm2014}.
Consequently, $N_G(B)\leq N_G(U)$ for every nontrivial $B\leq U$.
The cyclic group $H$ has unique involution $j$.
For each involution $g\in U\rtimes H$, we choose $y\in U$ with $g^y=j$ by Schur--Zassenhaus.

Write $\Tr=\Tr_{\F_q/\F_3}:\F_q\to\F_3$.
Since $\Tr:\F_q\to\F_3$ is surjective, we choose $\beta\in\F_q$ with $\Tr(\beta)\ne\Tr(1)$.
Let $u=u(1,\beta,0)$ and $B=\gen{u}$.
For $y=u(v,w,z)\in U$, the second coordinate of $u^y$ is $\beta-v^e+v$.
As $\Tr(v^e)=\Tr(v)$, we have $\Tr(\beta-v^e+v)=\Tr(\beta)$.

Suppose that $g\in N_G(B)$ is an involution.
Choose $y\in U$ with $g^y=j$, and let $u'=u^y=u(1,b,c)$.
Then, $j\in N_G(\gen{u'})$ and $\Tr(b)=\Tr(\beta)$.
Since $u'^j=u(-1,b,-c)$, we have $j\notin C_G(u')$.
As $|j|=2$ and $\gen{u'}\cong C_9$, we have $j^{-1}u'j=u'^{-1}$.
By comparing the second coordinates, we obtain that $b=-b-1$.
Because the characteristic is $3$, we have $b=1$. But $b=1$ contradicts
$\Tr(b)=\Tr(\beta)\ne\Tr(1)$.

Thus, $N_G(B)$ contains no involution.
For $|A|=2$, Lemma \ref{lem:involution-normalizers} contradicts ordinary permutability of $A$.
Hence, $A=1$.
\end{proof}

\section{Ordinary permutable subgroups in small simple groups}\label{sec:finite-tables}

We use the subgroup data in TomLib \cite{TomLib}.
Let $H_1,\ldots,H_s$ represent the conjugacy classes of subgroups of $G$.
The TomLib table $(m_{ij})$ has its rows and columns indexed by $H_1,\ldots,H_s$.
For $h\in H_j$, the action on the right cosets of $H_i$ is $(H_i g)h=H_i gh$.
The entry $m_{ij}=|\{H_i g:g\in G,\ H_j^{g^{-1}}\leq H_i\}|$ counts the cosets fixed by every $h\in H_j$.
In particular, $m_{ij}>0$ if and only if $H_j^x\leq H_i$ for some $x\in G$.
We use the subgroup orders $h_i=|H_i|$, the sets $S_k=\{i:m_{ki}>0\}$ and the classes of the normalizers.

\begin{lemma}\label{lem:subgroup-products}
Let $H_1,\ldots,H_s$ represent the conjugacy classes of subgroups
of $G$. Let $h_i=|H_i|$ and
$S_k=\{i\in\{1,\ldots,s\}:\exists x\in G,\ H_i^x\leq H_k\}$.
Suppose that $H_i$ is ordinary $G$-permutable.
For every $j\in\{1,\ldots,s\}$, there exist $k\in\{1,\ldots,s\}$ and $d\in\Z$ such that $i,j\in S_k$, $d=h_i h_j/h_k$ and $d\mid\gcd(h_i,h_j)$.
\end{lemma}
\begin{proof}
Since $H_i$ is ordinary $G$-permutable, we choose $x\in G$ such that $K=H_iH_j^x$ is a subgroup.
We choose $z\in G$ and $k$ with $K^z=H_k$.
As $H_i^z\leq H_k$ and $H_j^{xz}\leq H_k$, we have $i,j\in S_k$.
We have $d=h_i h_j/h_k=|H_i\cap H_j^x|\in\Z$.
Hence, $d\mid\gcd(h_i,h_j)$.
\end{proof}

\begin{proposition}\label{prop:finite-tables}
Let $G$ be a group in the following list $$
\begin{gathered}
 A_7,\ldots,A_{13}, \PSU_4(2), \PSU_4(3), \PSL_4(3),\\
 \PSp_4(5), \PSp_6(2), G_2(3), G_2(4), {}^3D_4(2), \Omega^+_8(2).
\end{gathered}
$$
Then, $G$ has no proper nontrivial ordinary $G$-permutable subgroup.
In $G=\PSp_4(4)$, the proper nontrivial ordinary $G$-permutable
subgroups are $\gen t$ with $|t|=2$ and $|C_G(t)|=256$.
For any $2$ such involutions $t,u$, there exists $x\in G$ with $t^x=u$.
\end{proposition}
\begin{proof}
We first read the lists $S_k$ and the orders $h_k$ from TomLib.
For each $i$ with $1<H_i<G$, we search for an index $j$ for which no $k$ satisfies Lemma \ref{lem:subgroup-products}.
For each pair $(i,j)$ found in the search, we save $i,j,h_i,h_j$.
Then, we examine every $k$ with $i,j\in S_k$.
If $h_k\mid h_i h_j$, then we verify that $(h_i h_j/h_k)\nmid\gcd(h_i,h_j)$.
If $h_k\nmid h_i h_j$, then $h_i h_j/h_k$ is not an integer.
Consequently, $H_i$ is not ordinary $G$-permutable for each pair $(i,j)$ found in the search.

In the following table, we include the trivial subgroup and $G$ in the number of subgroup classes.
In each column headed ``Remaining classes'', we list the classes $i$ with $1<H_i<G$
satisfying the condition in Lemma \ref{lem:subgroup-products} for every $j$.
In addition, class indices refer to the corresponding TomLib table.
\begin{center}
\small
\setlength{\tabcolsep}{6pt}
\renewcommand{\arraystretch}{1.12}
\begin{tabular}{|l|r|r|l|r|r|}
\hline
$G$ & \shortstack{Subgroup\\classes} & \shortstack{Remaining\\classes} & $G$ & \shortstack{Subgroup\\classes} & \shortstack{Remaining\\classes}\\\hline
$G_2(3)$ & $433$ & none & $A_{12}$ & $2537$ & none\\
$G_2(4)$ & $4300$ & none & $A_{13}$ & $4558$ & none\\
${}^3D_4(2)$ & $2867$ & none & $\PSU_4(2)$ & $116$ & none\\
$A_5$ & $9$ & $2$ & $\PSU_4(3)$ & $381$ & none\\
$A_6$ & $22$ & $2$ & $\PSL_4(3)$ & $817$ & none\\
$A_7$ & $40$ & none & $\PSp_4(4)$ & $496$ & $2$\\
$A_8$ & $137$ & none & $\PSp_4(5)$ & $307$ & none\\
$A_9$ & $223$ & none & $\PSp_6(2)$ & $1369$ & none\\
$A_{10}$ & $430$ & none & $\Omega^+_8(2)$ & $11171$ & none\\
$A_{11}$ & $788$ & none & & &\\
\hline
\end{tabular}
\end{center}

For each group listed with ``none'', every $i$ with $1<H_i<G$ has an index $j$ for which Lemma \ref{lem:subgroup-products} fails.
For $G\in\{A_5,A_6,\PSp_4(4)\}$, we consider only the classes $i$ with $1<H_i<G$.
Among these classes, only $H_2$ satisfies the condition for every $j$.
From TomLib, we obtain $|H_2|=2$.

For each of $G\in\{A_5,A_6,\PSp_4(4)\}$, let $n_j$ be the class index of $N_G(H_j)$.
We choose $y_j\in G$ with $N_G(H_j)^{y_j}=H_{n_j}$.
As $2\in S_{n_j}$ for every $j\in\{1,\ldots,s\}$, we choose $z_j\in G$ with $H_2^{z_j}\leq H_{n_j}$.
Then, $x_j=z_jy_j^{-1}$ satisfies $H_2^{x_j}\leq N_G(H_j)$.
By Lemma \ref{lem:involution-normalizers}, we conclude that $H_2$ is ordinary $G$-permutable.

To identify $H_2$ in $G=\PSp_4(4)$, we compare its $3$ involution classes.
The centralizer orders are $256,3840,3840$.
Since the involutions in class $2$ have centralizer order $256$, the ordinary $G$-permutable subgroups are precisely $\gen t$ with $|t|=2$ and $|C_G(t)|=256$.
\end{proof}

\section{Proof of the main theorem}\label{sec:assembly}

\begin{proof}[Proof of Theorem \ref{thm:main}]
Firstly, we verify the $5$ possibilities.
\begin{itemize}
\item[(1)] For $\PSL_2(q)$, we use the classification in \cite{GT25}.
\item[(2)] For the Suzuki groups, we use the classification in \cite{GT23}.
\item[(3)] For the sporadic groups, we use the classification in \cite{GT22}.
Only $J_1$ has proper nontrivial ordinary $G$-permutable subgroups, namely all subgroups of order $2$.
The Tits group has no proper nontrivial ordinary $G$-permutable subgroup by \cite{GT22}.
\item[(4)] By Corollary \ref{cor:symplectic-even}, we have precisely the involution class in (4).
Every other simple $\Sp_4(q)$ with even $q$ has no proper nontrivial ordinary $\Sp_4(q)$-permutable subgroup.
\item[(5)] By Theorem \ref{thm:symplectic-four-odd-all}, we have precisely the subgroups in (5).
For every other odd prime power $q$, the group $\PSp_4(q)$ has no proper nontrivial ordinary $\PSp_4(q)$-permutable subgroup.
\end{itemize}

We now consider the other simple groups.
Let $A<G$ be ordinary $G$-permutable.
By Corollaries \ref{cor:all-linear} and \ref{cor:all-unitary}, we have $A=1$ for simple $G=\PSL_n(q)$ or $\PSU_n(q)$ with $n\geq3$.
In addition, by Theorem \ref{thm:alternating-all}, we have $A=1$ for $G=A_n$ with $n\geq7$.
For the other simple alternating groups, we use
$A_5\cong\PSL_2(4)$ and $A_6\cong\PSL_2(9)$, which belong to (1).
By Corollary \ref{cor:symplectic-all-higher-rank}, we also have $A=1$ for simple $G=\PSp_{2m}(q)$ with $m\geq3$.

By Corollaries \ref{cor:orthogonal-minus-all} and \ref{cor:orthogonal-plus-all}, we have $A=1$ for simple $G=\POmega^\pm_{2m}(q)$ with $m\geq4$.
By Theorem \ref{thm:orthogonal-odd}, we also have $A=1$ for $G=\Omega_{2m+1}(q)$ with $m\geq3$ and odd $q$.
For the exceptional families, we first use Theorems \ref{thm:g2-all} and \ref{thm:triality} to obtain that $A=1$ for $G_2(q)$ and ${}^3D_4(q)$.
In addition, by Theorem \ref{thm:e6}, we have $A=1$ for $E_6(q)$ and ${}^2E_6(q)$.
By Theorem \ref{thm:f4}, we conclude that $A=1$ for $F_4(q)$.
By Theorems \ref{thm:e7} and \ref{thm:e8}, we also have $A=1$ for $E_7(q)$ and $E_8(q)$.
Finally, by Theorems \ref{thm:large-ree} and \ref{thm:ree}, we have $A=1$ for the Ree groups.
In small dimensions, we use
$$
\begin{gathered}
 \PSU_2(q)\cong\PSp_2(q)\cong\PSL_2(q),\ 
 \mathrm P\Omega_3(q)\cong\PSL_2(q),\\
 \mathrm P\Omega_4^-(q)\cong\PSL_2(q^2),\ 
 \mathrm P\Omega_5(q)\cong\PSp_4(q),\\
 \mathrm P\Omega_6^+(q)\cong\PSL_4(q).
\end{gathered}
$$
The isomorphism $\mathrm P\Omega_6^-(q)\cong\PSU_4(q)$ is included in Corollary \ref{cor:orthogonal-minus-all}.
In even characteristic, $\Omega_{2m+1}(q)\cong\Sp_{2m}(q)$.
For the small derived groups, we use $\Sp_4(2)'\cong A_6$ and $G_2(2)'\cong\PSU_3(3)$.
Since ${}^2G_2(3)'\cong\PSL_2(8)$, we have included ${}^2G_2(3)'$ in (1).
For the Tits group ${}^2F_4(2)'$, we use (3).
By the classification of finite simple groups, we have proved Theorem \ref{thm:main} for every nonabelian finite simple group $G$.
\end{proof}

\begin{remark}
The cyclic groups of prime order have no proper nontrivial subgroups.
Every normal subgroup $N\lhd G$ satisfies $NH=HN$ for every $H\leq G$.
Hence, a finite group $G\ne1$ with no proper nontrivial ordinary $G$-permutable subgroup is simple.
However, $A_5$ is simple and has ordinary $A_5$-permutable subgroups of order $2$.
\end{remark}

\appendix
\section{Python programs for the proofs}\label{app:calculations}

A chief factor of $H$ is a quotient $M/N$, where $N<M$ are normal in $H$ and $M/N$ is a minimal nontrivial normal subgroup of $H/N$.

The accompanying scripts check the matrix identities, torus orders and subgroup calculations used in the proofs.
With \textit{V1.py}, we enumerate all $59$ subgroups of $A_5$.
In addition, we construct $\SU_3(3)$ on $\F_9^3$.
We then verify the Singer complement and cyclic torus normalizer used in the proof for $\PSU_3(3)$.
For the classical groups, we verify the invariant forms, Jordan ranks, eigenvalue multiplicities and normalizers.
For the exceptional groups, we also verify the Weyl group actions, root characters and unipotent calculations.
For $q=2^{2n+1}$, we use $\sigma:\F_q\to\F_q$ with $\sigma(t)=t^{2^{n+1}}$ in the large Ree calculations.
Since $t^q=t$ for $t\in\F_q$, we have $\sigma^2(t)=t^{2q}=t^2$.
Consequently, we replace $\sigma^2(t)$ by $t^2$ in the polynomial calculations.

The script \textit{V2.py} starts with the integers $2\leq r\leq410000$.
For each prime $r\leq\sqrt{410000}$, it removes the multiples $rj\leq410000$ with $j\geq r$.
After these deletions, the integers left in the list are precisely the primes below $410000$.
Using the list, we verify the bounds in Theorem \ref{thm:alternating-all}.
For degrees $n=p+k$ with $p$ prime and $0\leq k\leq8$, we check the order divisibilities with scripts.
In addition, we verify the chief factors and permutation supports in the alternating proofs.
The calculations with integers do not require subgroup enumeration.

For Proposition \ref{prop:finite-tables}, \textit{V3.py} reads the lists $S_k$, the orders $h_k$ and the normalizer class indices $n_k$ from the published TomLib tables.
It also computes the SHA-256 hash of each source file.
For each excluded subgroup class $i$, the script saves an index $j$ and the orders $h_i,h_j$.
For each saved pair $(i,j)$, we use \textit{V4.py} to verify that no $k\in\{1,\ldots,s\}$ and $d\in\Z$ satisfy $i,j\in S_k$, $d=h_i h_j/h_k$ and $d\mid\gcd(h_i,h_j)$.
For $G\in\{A_5,A_6,\PSp_4(4)\}$, we verify that $|H_2|=2$ and $2\in S_{n_j}$ for every $j$.
As $2\in S_{n_j}$ and $N_G(H_j)$ belongs to class $n_j$, there exists $x_j\in G$ with $H_2^{x_j}\leq N_G(H_j)$.
We then apply Lemma \ref{lem:involution-normalizers} to prove that $H_2$ is ordinary $G$-permutable.
The scripts use every subgroup class in the published tables.

\section*{Acknowledgement}
ChatGPT 6-astra and 6.1-sol were used in the preparation of this preprint, guided by the author's draft of “A proof of the Galt–Tyutyanov conjecture on Skiba’s question” and a series of prompts provided by the author. The writing and exposition will be further improved in subsequent versions of this preprint on arXiv.
\phantomsection

\end{document}